%% file: thesis.tex
\documentclass[12pt,a4paper]{report}
\usepackage{
  amsmath,
  amsfonts,
  amssymb,
  amstext,
  amsthm,
  pvsymb,
  graphicx, 
  times,
  mathptmx,
  color,
  hyphenat,
  pinlabel}
\usepackage[ansinew]{inputenc}
\input xy
\xyoption{all}
\input defs.tex

\begin{document}
%
%
%%%%%%%%%%%%%%%%%%%%%%% Frontmatter %%%%%%%%%%%%%%%%%%%%%%%%

%%%%%%%% titlepage %%%%%%%%%%%%%%%%%%
%
%
\begin{titlepage}
\setlength{\textheight}{670pt}

\begin{tabular}{lc}
\raisebox{-35pt}{\includegraphics{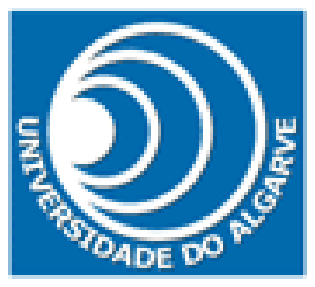}} &
\begin{tabular}{l}
\begin{LARGE}
\textsc{\textbf{Universidade do Algarve}}
\end{LARGE}
\\ \\
\ 
\begin{large} \textsc{\textbf{Faculdade de Ciências e Tecnologia}}
\end{large}
\end{tabular}
\end{tabular}
\bigskip \bigskip \bigskip \bigskip \bigskip \bigskip \bigskip \bigskip \bigskip 
\begin{center}
    \begin{huge}
    \textsc{\textsc{A categorification of the quantum \\[1.2ex] $\sln$-link polynomials using foams}}
    \end{huge}
\end{center}
\bigskip\bigskip
\begin{center}
    \begin{large}
    (Tese para a obtenção do grau de doutor no ramo de
Matemática,\\[1.2ex] especialidade de Geometria e Topologia)
    \end{large}
\end{center}
\bigskip\bigskip \bigskip \bigskip
\bigskip\bigskip \bigskip
\begin{center}
  \begin{Large}
{\bf Pedro Miguel dos Santos Santana Forte Vaz}
  \end{Large}
\end{center}
\bigskip \bigskip \bigskip \bigskip \bigskip \bigskip \bigskip
\begin{center}
\begin{Large}
Tese orientada por Doutor Marco Arien Mackaay
\end{Large}
\end{center}

\bigskip \bigskip \bigskip \bigskip \bigskip \bigskip \bigskip 
\bigskip \bigskip \bigskip \bigskip \bigskip  
\begin{center}
    \begin{Large}
    Faro \\[1.5ex] 2008
    \end{Large}
\end{center}
\end{titlepage}
\newpage
%
%%%%%%%%%%%%%%%%%%%%%%%%%%%%%%%%%%%%
%
%
%
%
\pagenumbering{roman}
%
%
%%%%%%%% 2nd titlepage %%%%%%%%%%%%%%%%%%
%
%
\begin{titlepage}
\vspace*{-14ex}
\begin{center}
\begin{LARGE}
\textsc{\textbf{Universidade do Algarve}}
\end{LARGE}
\end{center}
\medskip
\begin{center}
\begin{large} \textsc{\textbf{Faculdade de Ciências e Tecnologia}}
\end{large}
\end{center}
\bigskip \bigskip \bigskip \bigskip \medskip
\begin{center}
    \begin{huge}
    \textsc{\textsc{A categorification of the quantum \\[1.2ex] $\sln$-link polynomials using foams}}
    \end{huge}
\end{center}
\bigskip
\begin{center}
    \begin{large}
    (Tese para a obtenção do grau de doutor no ramo de Matemática, \\[1.2ex] especialidade de Geometria e Topologia)
    \end{large}
\end{center}
\bigskip

\begin{center}
  \begin{Large}
{\bf Pedro Miguel dos Santos Santana Forte Vaz}
  \end{Large}
\end{center}

\bigskip \bigskip\bigskip
\setlinespacing{1.2}
\n {\bf Orientador:} Doutor  Marco Arien Mackaay, Professor Auxiliar da Faculdade \indent\indent\indent\indent\ \ de Ciências e Tecnologia da Universidade do Algarve 

\bigskip\medskip

\n {\bf Constituição do Júri:}\\[0.8ex]
Presidente: Reitor da Universidade do Algarve    \\[0.8ex]
Vogais: \\
  Doutora Maria Teresa de Lemos Monteiro Fernandes, Professora Catedrática da Faculdade de Ciências da Universidade de Lisboa  \\[0.6ex]
  Doutor José Manuel Vergueiro Monteiro Cidade Mourão, Professor Associado do Instituto Superior Técnico da Universidade Técnica de Lisboa   \\[0.6ex]
  Doutor Roger Francis Picken, Professor Associado do Instituto Superior Técnico da Universidade Técnica de Lisboa \\[0.6ex]     
  Doutor Nenad Manojlovic, Professor Associado da Faculdade de Ciências e Tecnologia da Universidade do Algarve   \\[0.6ex]  
  Doutor Paul Turner, Professor Auxiliar da Heriot-Watt University, Edinburgh, United Kingdom   \\[0.6ex]   
  Doutor Marko Sto\v si\'c, Investigador Auxiliar do Instituto Superior Técnico da Universidade Técnica de Lisboa
\bigskip\medskip
\begin{center}
    \begin{Large}
    Faro \\[1.1ex] 2008
    \end{Large}
\end{center}
\end{titlepage}
%
%
%
%%%%%%%%%%%%%%%%%%%%%%%%%%%%%%%%%%%%
%
%
\setlinespacing{1.0} 
\setcounter{page}{2}
%
%%%%% dedic %%%%%%%
\newpage\ \newpage
\vspace*{12ex}
\hspace{60ex}{\large Para Adriana}
%
%
%
%%%%%%%%%%%%%%%%%%%%% abstract %%%%%%%%%%%%%%%%%%%%%%%%%%%%%%
%
%
\newpage\
%
%%%%%%%%%%%%%%%%
%  Portuguese  %
%%%%%%%%%%%%%%%%
%
\addcontentsline{toc}{chapter}{Resumo}
\chapter*{Resumo}
\begin{center}
\begin{minipage}{0.9\columnwidth}
Nesta tese definimos e estudamos uma categorifica\c{c}\~ao do polin\'omio $\sln$ de enlaces utilizando foams, para $N\geq 3$. 

Para $N=3$ definimos a homologia-$\slt$ de enlaces universal, utilizando foams, que depende de tr\^es par\^ame\-tros. Mostramos que a teoria \'e functorial a menos de um escalar em rela\c{c}\~ao a cobordismos de enlaces. A teoria \'e integral. Mostramos que o seu produto tensorial com $\bQ$ resulta numa teoria que \'e equivalente \`a homologia-$\slt$ universal racional de Khovanov e Rozansky.

Para $N\geq 4$ constru\'imos uma teoria racional que categorifica o polin\'omio $\sln$ de enlaces utilizando foams. A nossa teoria \'e functorial a menos de um escalar em rela\c{c}\~ao a cobordismos de enlaces. Para avaliar foams fechados \'e utilizada a f\'ormula de Kapustin e Li. Mostramos que para qualquer enlace a nossa homologia \'e isomorfa \`a homologia de Khovanov e Rozansky. Conjecturamos que a nossa teoria \'e integral e calculamos a conjecturada homologia-$\sln$ integral para $(2,m)$-enlaces no toro e mostramos que tem tors\~ao de ordem $N$.
\end{minipage}
\end{center}
\bigskip\bigskip\bigskip\bigskip
\n\textbf{Palavras-chave:}\quad Homologia de Khovanov-Rozansky, polin\'omio $\sln$, enlace, cobordismo singular, f\'ormula de Kapustin-Li
\newpage\ \newpage
%
%
%%%%%%%%%%%%%
%  English  %
%%%%%%%%%%%%%
%
\addcontentsline{toc}{chapter}{Abstract}
\chapter*{Abstract}
\begin{center}
\begin{minipage}{0.9\columnwidth}
In this thesis we define and study a categorification of the $\sln$-link polynomial using foams, for $N\geq 3$. 

For $N=3$ we define the universal $\slt$-link homology, using foams, which depends on three parameters and show that it is functorial, up to scalars, with respect to link cobordisms. Our theory is integral. We show that tensoring it with $\bQ$ yields a theory which is equivalent to the rational universal Khovanov-Rozansky $\slt$-link homology.

For $N\geq 4$ we construct a rational theory categorifying the $\sln$-link polynomial using foams. Our theory is functorial, up to scalars, with respect to link cobordisms. To evaluate closed foams we use the Kapustin-Li formula. We show that for any link our homology is isomorphic to the Khovanov-Rozansky homology. We conjecture that the theory is integral and we compute the conjectured integral $\sln$-link homology for the $(2,m)$-torus links and show that it has torsion of order $N$.
\end{minipage}
\end{center}
\bigskip\bigskip\bigskip\bigskip
\n\textbf{Key-words:}\quad Khovanov-Rozansky homology, $\sln$ polynomial, link, singular cobordism, Ka\-pustin-Li formula
%
%
%
%%%%%%%%%%%%%%%%%%%%%%%%%%%%%%%%%%%%%%%%%%%%%%%%%%%%%%%%%%%%%
%
\setlinespacing{1.5}
%
%%%%%% acknowledgements %%%%%%
%
\newpage\
\addcontentsline{toc}{chapter}{Acknowledgments}
\chapter*{Acknowledgments}

First of all I want to thank my supervisor Marco Mackaay from whom I have learned a lot. I owe him a great debt of gratitude. I thank him accepting me as his PhD student, for his involvement in my research, constant interest and encouragement. I shared with him a lot of productive hours. I also thank him for all the valuable advices, of mathematical and non-mathematical nature, that he gave me. This thesis shows only a bit of what he taught me during the last few years. Besides the collaboration with him, much of this thesis consists of the work done in collaboration with Marko Sto\v si\'c. I am glad that I had the opportunity to work with him and I have learned a lot from it. I also thank Paul Turner for his encouragement, support and valuable comments on my research. I am very grateful to Roger Picken for his interest in my research, encouragement and invaluable support.

I would like to thank Mikhail Khovanov for his interest in my work, for the interesting conversations, enlightening exchanges of e-mail and for the many useful comments on my research and on the papers upon which this thesis is based. 

I am grateful to CAMGSD-IST and to the Funda\c{c}\~ao para a Ci\^encia e Tecnologia (IST plurianual funding through the
programme ``Programa Operacional Ci\^{e}ncia, Tecnologia, Ino\-va\c{c}\~{a}o'' (POCTI), cofinanced by the European Community fund FEDER). I would also like to thank Dror Bar-Natan for creating the great symbol font dbnsymb.sty and for sharing it. In this thesis I use a symbol font constructed from it.

Most of all, I thank my whole family and especially my mother in-law Teresinha, my wife Dalila and my daughter Adriana for their encouragement, support and patience during the time this project took place.
\newpage\
%
%%%%%%%%%%%%%%%%%%%%%%%%%%%%%%%%%%%%%%%555
%
\setlinespacing{1.44}
\tableofcontents
\addcontentsline{toc}{chapter}{List of Figures}
\listoffigures

\newpage
%
%%%%%%%%%%%%%%%%%%%%%%% Mainmatter %%%%%%%%%%%%%%%%%%%%%%%%%
%
\setlinespacing{1.5}    \pagenumbering{arabic}
%
%
%
%%%%%%%%% Introduction %%%%%%%%%%%%%%%%
%
\chapter{Introduction}\label{chap:intro}
%
%%%%%%%%%%%%%%%%%%%%%%%%%%%%%%%%%%%%%%%%
%
%
\setlinespacing{1.48}
In a beautiful paper Mikhail Khovanov~\cite{khovanovsl2} initiated the ``second revolution'' in the world of link polynomials (the first one was initiated by Jones in the mid-eighties~\cite{jones}). Khovanov gene\-ralized the Kauffman state-sum model~\cite{kauffman-jones} (see also~\cite{kauffman-kphysics}) of the Jones polynomial~\cite{jones}, also known as the $\mathfrak{sl}(2)$-link polynomial, in order to construct a homology complex of graded mo\-dules from a link diagram whose graded Euler characteristic is the Jones polynomial of that link. The differentials are grading preserving, so the homology forms a bigraded mo\-dule. Khovanov called it a \emph{categorification} of the Jones polynomial (this terminology was introduced by L.~Crane and I.~Frenkel in~\cite{crane-frenkel}). If two diagrams differ by a Reidemeister move then the corresponding complexes are homotopy equivalent. Therefore the homology is a link invariant and is nowadays known as the \emph{Khovanov homology}. The Khovanov homology is strictly stronger than the Jones polynomial since there exist knots with the same Jones polynomial but different Khovanov homologies~\cite{bar-natancat}. Khovanov homology actually defines a projective functor from the category of links and link cobordisms to the category of bigraded  modules~\cite{jacobsson,khovanovcob,bar-natancob}. This means that a link cobordism (in $\bR^4$) between two links gives rise to a homomorphism between their homologies that is defined up to a sign. Using a slight modification of the original Khovanov homology, Clark, Morrison and Walker~\cite{clark-morrison-walker} have fixed the sign problem in the functoriality of Khovanov homology (see also~\cite{caprau}).

One of the basic ingredients in Khovanov homology is a certain Frobenius algebra of rank two whose graded dimension is the value of the Jones polynomial for the unknot. Using a different Frobenius algebra of rank two E.~S.~Lee~\cite{lee} obtained a link homology which is much simpler than the original Khovanov homology (it gives a module of rank two for eve\-ry knot). It is no longer graded but filtered instead. The filtration in Lee's theory induces a spectral sequence whose $E_2$ page is the Khovanov homology and converges to Lee homo\-logy. This spectral sequence was masterfully used by Rasmussen~\cite{rasmussenslice} to define a numerical link invariant, which detects topologically slice knots that are not smoothly slice. It is known that such knots give rise to exotic smooth structures on $\bR^4$, a result that was available only through gauge theory~\cite{donaldson,gompf-stipsicz}. As an application of his invariant Rasmussen gave a short and purely combinatorial proof of the Milnor conjecture, which was first proved by Kronheiner and Mrowka~\cite{kronheiner-mrowka} using gauge theory.

In~\cite{bar-natancob} Bar-Natan found another rank two Frobenius algebra that could be used to construct a link homology (see also~\cite{turner}). Lee and Bar-Natan's work~\cite{lee} suggested that there could be other Frobenius algebras giving rise to link homologies. This question was solved by Mikhail Khovanov in~\cite{khovanovfrob}. He classified all possible Frobenius systems of dimension two which give rise to link homologies via his construction in~\cite{khovanovsl2} and showed that there is a universal one, given by 
$$\quotient{\bZ[X,a,b]}{(X^2-aX-b)},$$
where $a$ and $b$ are parameters. Working over $\bC$, one can take $a$ and $b$ to be complex numbers, obtaining a filtered theory, and study the corresponding homology with coefficients in $\bC$. We refer to the latter as the $\mathfrak{sl}(2)$-link homologies over $\bC$, because they are all deformations of Khovanov's original link homology. Using the ideas in \cite{khovanovfrob,lee,turner}, it was shown by the author in joint work with Marco Mackaay and Paul Turner~\cite{mackaay-turner-vaz} that there are only two isomorphism classes of $\mathfrak{sl}(2)$-link homologies over $\bC$ and that the corresponding isomorphisms preserve the filtration. Given $a,b\in\bC$, the isomorphism class of the corresponding link homology is completely determined by the number of distinct roots of the polynomial $X^2-aX-b$. All $\mathfrak{sl}(2)$-link homologies are isomorphic to Lee homology for $a^2+4b\neq 0$ and isomorphic to Khovanov homology for $a^2+4b=0$. The original Khovanov $\mathfrak{sl}(2)$-link homology $\KH(L,\bC)$ corresponds to the choice $a=b=0$.

Bar-Natan~\cite{bar-natancob} obtained the universal $\mathfrak{sl}(2)$-link homology in a different way (see also~\cite{naot, naotphd}). 
He shows how Khovanov's original construction of the $\mathfrak{sl}(2)$-link 
homology~\cite{khovanovsl2} can be adapted to define a universal functor 
$\cU$ from the category of links, with link cobordisms modulo 
ambient isotopy as morphisms, to the homotopy category of complexes 
in the category of $(1+1)$-dimensional cobordisms modulo a finite set 
of universal relations. Bar-Natan's approach yields a homology complex in a non-abelian category, so we cannot take its homology in that category. However, 
he introduces the {\em tautological homology construction}, which produces a homology theory in the abelian category of graded $\bZ$-modules. Bar-Natan's approach has two main advantages: it allows for a much simpler proof of functoriality and an efficient calculation of the Khovanov homology for a large number of knots~\cite{bar-natanfast}.

In~\cite{khovanovsl3} Khovanov constructed a topological theory categorifying the $\slt$-link polynomial. His construction uses cobordisms with singularities, called foams, modulo a finite set of relations. In~\cite{KR} Khovanov and Rozansky (KR) categorified the $\sln$-link polynomial for arbitrary $N$, the 1-variable specializations of the 2-variable HOMFLY-PT polynomial~\cite{HOMFLY,PT}. Their construction uses the theory of matrix facto\-rizations, a mathematical tool introduced by Eisenbud in~\cite{eisenbud} (see also~\cite{BGS,knorrer,yoshino}) in the study of maximal Cohen-Macaulay modules over isolated hypersurface singularities and used by Kapustin and Li as boundary conditions for strings in Landau-Ginzburg models~\cite{KL}. It was conjectured in~\cite{KR} that for $N=3$ the approach using matrix factorizations yields the same as Khovanov's approach using foams. Shortly after the appearance of Khovanov and Rozansky's paper~\cite{KR} Gornik generalized the results of Lee to $\sln$~\cite{gornik} and based on Gornik's results Lobb~\cite{lobb} and independently Wu~\cite{wu-filt} generalized Rasmussen's $s$-invariant to $\sln$. We remark that besides the results in low-dimension topology referred to above (some more can be found in~\cite{shumakovitch-slice,wu-benn}), link homology has also interesting relations with physics~\cite{gukov-gauge,GIKV,GSV,GW}, geometry~\cite{cautis-kamnitzer,cautis-kamnitzer2,manolescuslices,manolescusympl,seidel-smith,webster-williamson} and representation theory~\cite{sussan,stroppel-springer,mazorchuk-stroppel}.

One of the main features of Khovanov and KR homology is that it is combinatorially defined, so theoretically it is computable. Nevertheless, direct computations are very hard in practice, in part due to the large ranks of the chain groups involved. There are several techniques that allow computation of some homology groups: exact sequences, spectral sequences and complex simplification techniques (see~\cite{bar-natanfast,morrison-nieh,naotphd,rasmussen-diff,turnerspectral,virorem}). Using these techniques Khovanov homology is known for some classes of knots: alternating knots~\cite{lee}, quasi-alternating knots~\cite{manolescu-oszvath}, $(3,q)$-torus knots~\cite{turnerspectral} and some Pretzel knots~\cite{suzuki},  and the KR homology is known for 2-bridge knots~\cite{rasmussen2-bridge}. There also exist computer programs~\cite{barnatan-knotatlas,bar-natancat,shumakovitchkhoho} that compute Khovanov homology of knots up to 50 crossings. Based on Bar-Natan's pioneering computer program in~\cite{bar-natancat} interesting conjectures were made about the properties of Khovanov homology~\cite{bar-natancat,DGR,khovanovpatterns} some of which have already been proved (see for example ~\cite{lee,stosic-propbraid,stosic-thicktknots}). The most powerful program is Bar-Natan's~\cite{barnatan-knotatlas} which is based on simplification techniques developed in~\cite{bar-natanfast, naotphd} following his topological approach to Khovanov homology. One can therefore hope that a topological categorification of the $\sln$-link polynomial will allow us to compute the $\sln$-link homology for more knots.

\medskip
%%%%%%%%
% goal %
%%%%%%%%
My PhD project had to goals. The first one was to construct a universal deformed $\slt$-link homology using foams, extending the work of Khovanov in~\cite{khovanovsl3} and to prove the equivalence between this theory and the corresponding universal deformation for $N=3$ of the theory cons\-tructed by Khovanov and Rozansky in~\cite{KR} using matrix factorizations (this equivalence was conjectured by Khovanov and Rozansky in~\cite{KR} for the undeformed case). The second goal was to construct a combinatorial topological definition of KR link homology, extending to all $N>3$ the work of Khovanov~\cite{khovanovsl3} for $N=3$. Khovanov had to modify 
considerably his original setting for the construction of $\mathfrak{sl}(2)$ 
link homology in order to produce his $\slt$ link homology. It required the 
introduction of singular cobordisms with a particular type of singularity, which he called \emph{foams}. The jump from 
$\slt$ to $\sln$, for $N>3$, requires the introduction of a new type of 
singularity. The latter is needed for proving invariance under the third Reidemeister move. The introduction of the new singularities makes it much harder to evaluate closed foams and we do not know how to do it combinatorially. 
Instead we use the Kapustin-Li formula~\cite{KL}, which was introduced by A.~Kapustin and Y.~Li in~\cite{KL} in the context of topological Landau-Ginzburg models with boundaries and adapted to foams by Khovanov\footnote{We thank M Khovanov for suggesting that we try to use the Kapustin-Li formula.} and Rozansky~\cite{KR-LG}. 
The downside is that our construction does not yet allow us to deduce a (fast) algorithm for computing our link homology.
A positive side-effect is that it allows us to show that for any link our homology is isomorphic to KR homology. Furthermore the combinatorics involved in establishing certain identities among foams gets much harder 
for arbitrary $N$. The theory of symmetric polynomials, in particular Schur polynomials, is 
used to handle that problem.

Although we have not completely achieved our final goal, we believe that we have made good progress towards it. Using the techniques developed by the author in joint work with Marco Mackaay in~\cite{mackaay-vaz} and with Marco Mackaay and Marko Sto\v si\'c in~\cite{mackaay-stosic-vaz} (based on Khovanov~\cite{khovanovsl3} and Bar-Natan~\cite{bar-natancob}) we derive a small set of relations on foams which we show to be sufficient to guarantee that our link homology is homotopy invariant  under the Reidemeister moves and functorial, up to scalars, with respect to link cobordisms. By deriving these relations from the Kapustin-Li formula we prove that these relations are consistent. However, in order to get a purely combinatorial construction we would have to show that they are also sufficient for the evaluation of closed foams, or, equivalently, that they generate the kernel of the Kapustin-Li formula. We conjecture that this holds true, but so far our attempts to prove it have failed. It would be very interesting to have a proof of this conjecture, not just because it would show that our method is completely combinatorial, but also because our theory could then be used to prove that other constructions, using different combinatorics, representation theory or symplectic/complex geometry, give functorial link homologies equivalent to KR (see~\cite{mazorchuk-stroppel} for some work towards it). So far we can only conclude that any other way of evaluating closed foams which satisfies the same relations as ours gives rise to a functorial link homology which categorifies the $\sln$ link polynomial. We conjecture that such a link homology is equivalent to the one presented in this thesis and therefore to KR.  Another open question is the integrality of our link homology. All our relations are defined over the integers and this presents some support to our conjecture that our theory is integral.

%%%%%%%%%%%%%%%%%%%%%%%%%%%%%%%
\subsection*{Outline of the thesis}
The outline of this thesis is as follows:
%%%%%%%%%%%%% KR %%%%%%%%%%%%%%
%
In Chapter~\ref{chap:KR} we describe the Murakami-Ohtsuki-Yamada (MOY) state sum model~\cite{MOY} for the $\sln$-link polynomial. Then we give a review of the theory of matrix factorizations used in this thesis and describe the KR theory introduced in~\cite{KR}. The chapter ends with a description of the universal rational KR $\slt$-link homology.
%
%%%%%%%%%%%% univ3 %%%%%%%%%%%%%
%
In Chapter~\ref{chap:univ3} we define the universal $\slt$-link homology, which depends on 3 parameters, following Khovanov's approach with foams. We work diagrammatically in a category of webs and foams modulo a finite set of relations and prove invariance and functoriality in that category. After applying the tautological functor we can define homology groups. We also state the classification theorem for the $\slt$-link homologies which establishes that there are three isomorphism classes. The first class is the one to which Khovanov's original $\slt$-link homology belongs, the second is the one studied by Gornik in the context of matrix factorizations and the last one is new and can be described in terms of Khovanov's original $\mathfrak{sl}(2)$-link homology. We omit the proof in this thesis (see~\cite{mackaay-vaz}) since the result is beyond the scope of this PhD project.
%
%%%%%%%%%%% foam3mf %%%%%%%%%%%%
%
In Chapter~\ref{chap:foam3mf} we prove that the two universal rational $\slt$ link homologies which were constructed in Chapter~\ref{chap:univ3} using foams, and in Section~\ref{KR:sec:KR-abc} using matrix factorizations, are naturally isomorphic as projective functors from the category of links and link cobordisms to the category of bigraded vector spaces.
%
%%%%%%%%%%%% foamN %%%%%%%%%%%%%
%
In Chapter~\ref{chap:foamN} we use foams to give a topological construction of a rational link homology categorifying the $\mathfrak{sl}(N)$ link invariant, for $N\geq 4$. To evaluate closed foams we use the Kapustin-Li formula adapted to foams by Khovanov and Rozansky~\cite{KR-LG}. We show that for any link our homology is isomorphic to the one constructed by Khovanov and Rozansky in~\cite{KR} and described in Section~\ref{KR:sec:KR}.

\medskip

Most of this thesis is based on papers by the author~\cite{mackaay-stosic-vaz, mackaay-vaz2,mackaay-vaz} in collaboration with Marco Mackaay and Marco Sto\v si\'c. Some sections have been rewritten and some modifications have been made. In Chapter~\ref{chap:univ3} we explain the relation with the integral $U(3)$-equivariant cohomology rings of flag varieties in $\bC^3$ for the convenience of the reader. In the second part of the same chapter we just state the main result of the second part of~\cite{mackaay-vaz}. In Chapter~\ref{chap:foamN} some conventions regarding tensor products of matrix factorizations have been changed. The properties of the Kapustin-Li formula stated in~\cite{KR-LG} and used in~\cite{mackaay-stosic-vaz} are rigorously proved here. Chapter~\ref{chap:foamN} also contains an explicit computation of the integral $\sln$ homology of $(2,m)$-torus links, assuming integrality of the theory. The computation  shows that the conjectured integral $\sln$-link homology has torsion.
%
%
%%%%%%%%%%%%%%%%%%%%%%%%%%%%%%%%%%%%%%%%%%
%
%
%
\setlinespacing{1.5}
%
%
%
%%%%%%%%%%%%%%%%%%%
%                 %
%     chapter     %
%                 %
%%%%%%%%%%%%%%%%%%%
%
%
\chapter{Khovanov-Rozansky homology}\label{chap:KR}
This chapter is introductory and contains the basics of KR theory~\cite{KR} categorifying the $\sln$-link polynomial. We also present its universal deformation in the case of $N=3$. 

We start with the definition of the $\sln$-link polynomial in Section~\ref{KR:sec:slN}. In Section~\ref{KR:sec:matfac} we give a review of the theory of matrix factorizations, which is the main tool used to obtain KR $\sln$-link homology. Matrix factorizations are also needed in Chapter~\ref{chap:foamN} (in a different way) as input for the Kapustin-Li formula. In Section~\ref{KR:sec:KR} we describe KR $\sln$-link homology and its main properties. In Section~\ref{KR:sec:KR-abc} we present the universal rational KR $\slt$-link homology.
%
%

%%%%%%%%%%%%%%%%%%%%%%%%%%%%%%%%%%%%%%%%
%%%                                  %%%
%%%        sl_N                      %%%
%%%                                  %%%
%%%%%%%%%%%%%%%%%%%%%%%%%%%%%%%%%%%%%%%%
\section{Graphical calculus for the $\sln$ polynomial}\label{KR:sec:slN}

In this section we recall some facts about the graphical calculus for the $\sln$-link polynomial. The HOMFLY-PT~\cite{HOMFLY,PT} polynomial of oriented links in $S^3$ is uniquely defined by the skein relation
$$
a
P(\undercrossing)
-a^{-1}
P(\overcrossing)
=z
P(\orsmoothing),
$$
and its value for the unknot. The $\sln$-link polynomial~\cite{turaev} is the 1-variable polynomial corresponding to the specializations $a=q^N$, $z=q-q^{-1}$ with positive integer $N$. We choose the normalization 
$P_N(\bigcirc)=[N]$, where $[N]=(q^N-q^{-N})/(q-q^{-1})$.

Let $D$ be a diagram of a link $L\in S^3$ with $n_+$ positive crossings and $n_-$ negative crossings. 
\begin{figure}[h!]
$$\begin{array}{cclcccl}
\text{positive: } \ 
\figins{-9}{0.3}{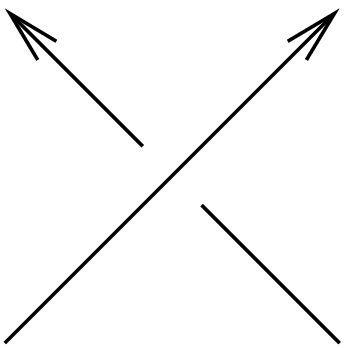} &=&   
\figins{-9}{0.3}{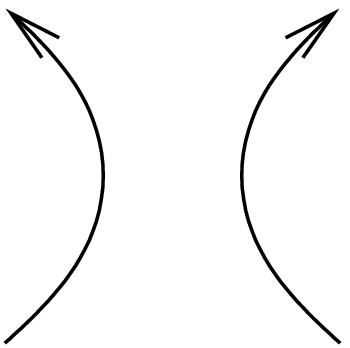}  -   q  
\figins{-9}{0.3}{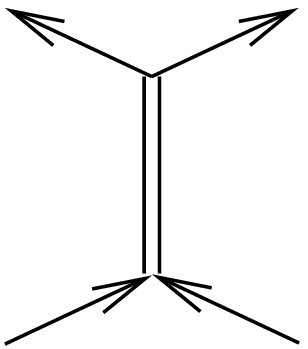} 
&\quad&
\text{negative: } \ 
\figins{-9}{0.3}{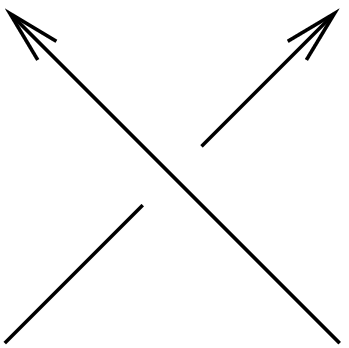} &=& \, 
\figins{-9}{0.3}{dbedge} -  q \
\figins{-9}{0.3}{orsmoothing} \vspace{1ex} 
\\
&& \hspace{1.5ex} 0 \hspace{6.0ex}  1 &&&& \hspace{1.6ex} 0 \hspace{6.6ex} 1
\end{array}$$
\caption{Positive and negative crossings and their 0 and 1-flattening}
\label{KR:fig:flatten}
\end{figure}
Following an approach based on the MOY state sum model constructed in~\cite{MOY} 
we can alternatively define $P_N(D)$ by using flat resolutions of $D$. Each crossing of $D$ can be flattened in two possible ways, 
as shown in Figure~\ref{KR:fig:flatten}, 
where we also show our convention for positive and negative crossings. 
A complete flattening of $D$ is an example of a \emph{web}: a trivalent graph with three  
types of edges: \emph{simple}, \emph{double} and \emph{triple}. Webs can contain closed plane loops (simple, double or triple). Only the simple edges are equipped with an orientation. 
We distinguish the triple edges with a $(*)$, 
as in Figure~\ref{KR:fig:vertices}. Each vertex must satisfy one of the possibilities of Figure~\ref{KR:fig:vertices}. Note that a complete flattening of $D$ never has triple 
edges, but we will need the latter for webs that show up in the proof of invariance under 
the third Reidemeister move.  
\begin{figure}[h!]
$$\xymatrix@R=1.2mm{
\figins{0}{0.32}{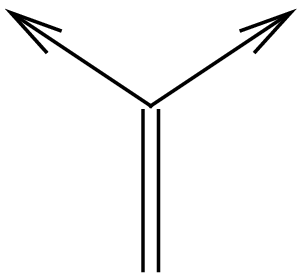} &
\figins{0}{0.32}{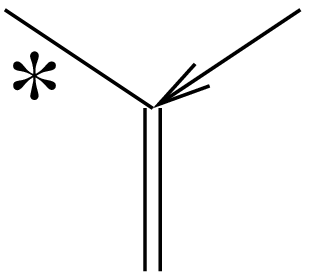} &
\figins{0}{0.32}{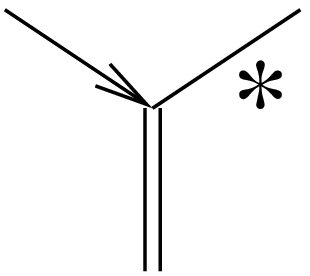}  \\
\figins{0}{0.32}{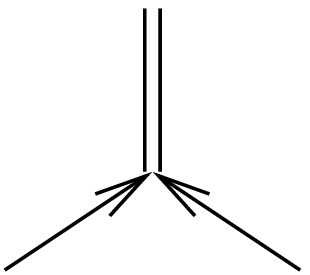}  &
\figins{0}{0.32}{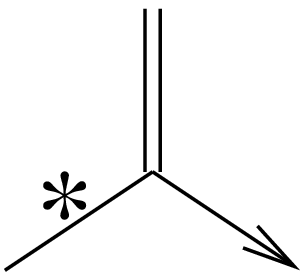} &
\figins{0}{0.32}{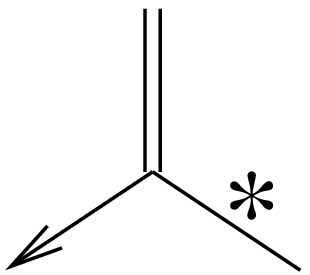}
}$$
\caption{Vertices}
\label{KR:fig:vertices}
\end{figure}

Simple edges correspond to edges labelled 1, double edges to edges labelled 2 and 
triple edges to edges labelled 3 in~\cite{MOY}, 
where edges carry labels from 1 to $N-1$ and label $j$ is associated to the $j$-th exterior power of the fundamental representation of $\sln$~\cite{turaev}.

The \emph{MOY web moves} 
in Figure~\ref{KR:fig:moy} provide a recursive way of assigning to each web $\Gamma$ that only contains 
simple and double edges a unique polynomial
$P_N(\Gamma)\in\bZ[q,q^{-1}]$ with positive coefficients. There are more general 
web moves, which allow for the evaluation of arbitrary webs, but we do not need them here. 
\begin{figure}[h]
%%%%%%%%%%%
% 0 moves %
%%%%%%%%%%%
$$\bigcirc=[N],\quad
\figins{-3.1}{0.16}{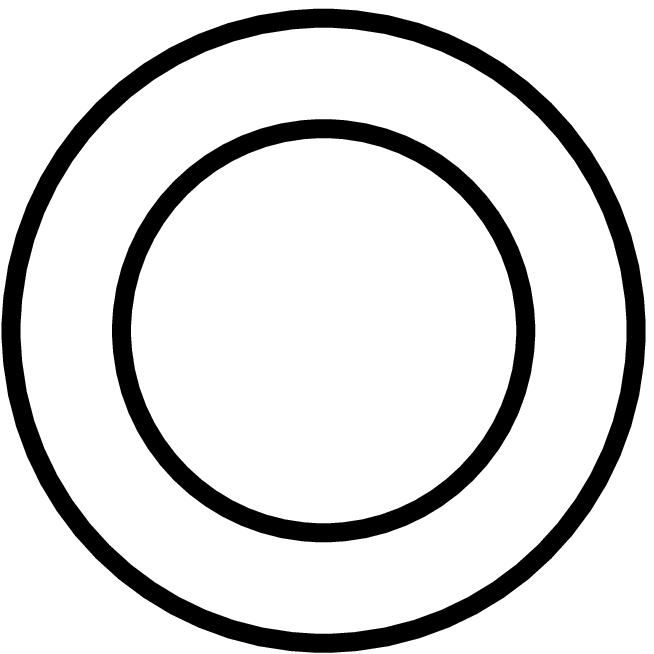}=\qbin{N}{2}$$

%%%%%%%%%%%%%%%
% digon moves %
%%%%%%%%%%%%%%%
$$
\figins{-8}{0.3}{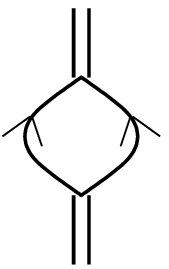} = [2]\ 
\figins{-8}{0.3}{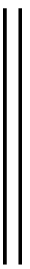}\ , \qquad
\figins{-8}{0.3}{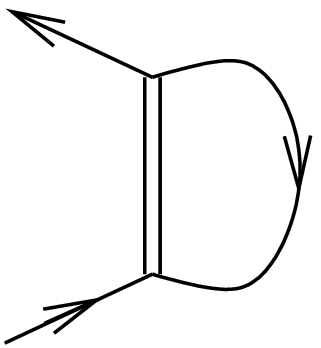} =[N-1]\ 
\figins{-8}{0.3}{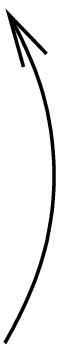}
$$

%%%%%%%%%%%%%%%%
% square moves %
%%%%%%%%%%%%%%%%
$$
\figins{-8}{0.3}{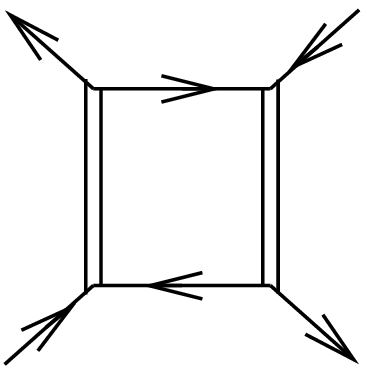} =
\figins{-8}{0.3}{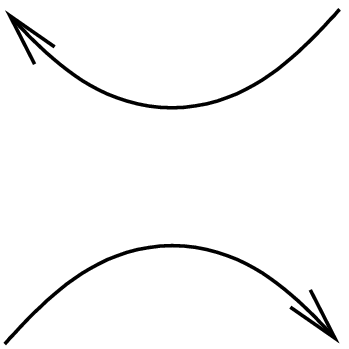} + [N-2]\ 
\figins{-8}{0.3}{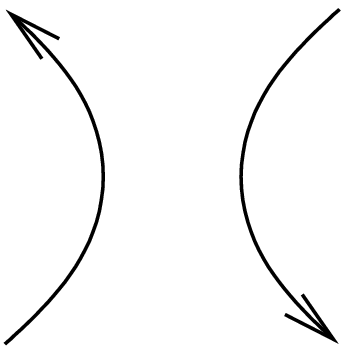}
$$

$$
\figins{-21}{0.65}{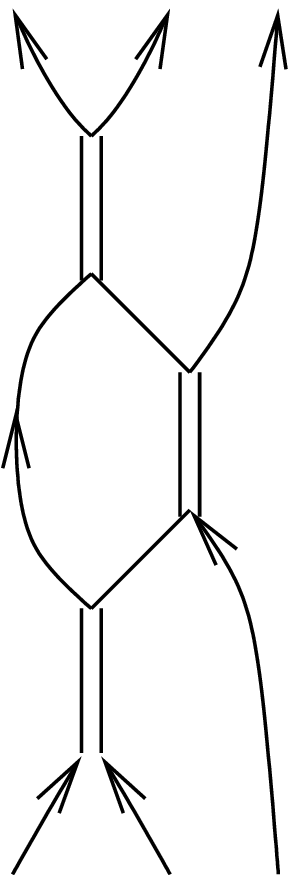} +
\figins{-21}{0.65}{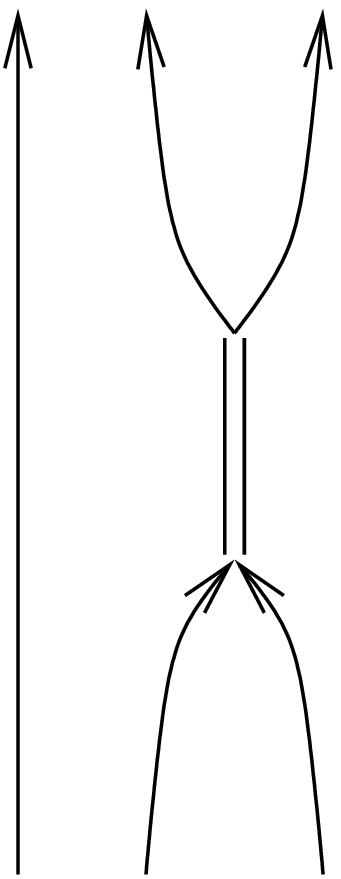} =
\figins{-21}{0.65}{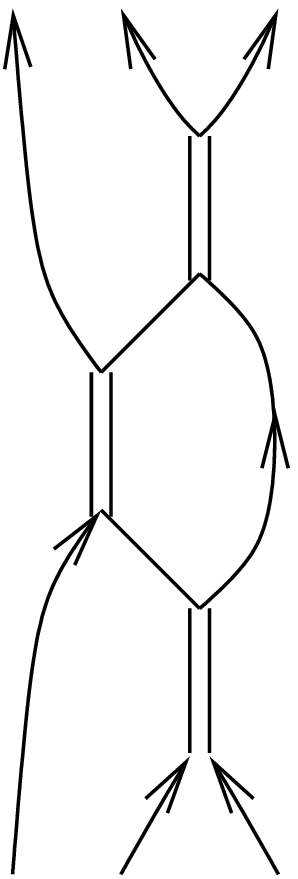}+
\figins{-21}{0.65}{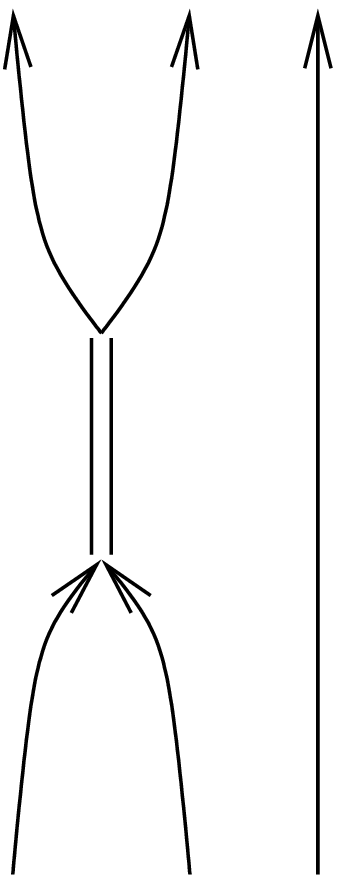}
$$
\caption{MOY web moves}
\label{KR:fig:moy}
\end{figure}
Consistency of the relations in Figure~\ref{KR:fig:moy} is shown in~\cite{MOY}.

Finally let us define the $\sln$ link polynomial. For any $i$ let $\Gamma_i$ denote a 
complete flattening of $D$. Then 
$$P_N(D)=(-1)^{n_-}q^{(N-1)n_+ - Nn_-}\sum_iq^{|i|}P_N(\Gamma_i),$$
where $|i|$ is the number of 1-flattenings in $\Gamma_i$, the sum being over all possible flattenings of $D$.
%
%
%
%
%%%%%%%%%%%%%%%%%%%%%%%%%%%%%%%%%%%%%%%%
%%%                                  %%%
%%%        Matrix Factorizations     %%%
%%%                                  %%%
%%%%%%%%%%%%%%%%%%%%%%%%%%%%%%%%%%%%%%%%
%
\section{Review of matrix factorizations}\label{KR:sec:matfac}
This section contains a brief review of matrix factorizations and the properties 
that will be used throughout this thesis. All the matrix factorizations in this thesis are 
$\bZ/2\bZ\times\bZ$-graded. Let $R$ be a polynomial ring over $\bQ$ in a finite number of variables. We take the 
$\bZ$-degree of each polynomial to be twice its total degree. This way $R$ is 
$\bZ$-graded. Let $W$ be a homogeneous 
element of $R$ of degree $2m$. A matrix factorization of $W$ over $R$ is given by a 
$\bZ/2\bZ$-graded free $R$-module $M=M_0\oplus M_1$ with two $R$-homomorphisms of 
degree $m$
$$M_0\xra{d_0}M_1\xra{d_1}M_0$$
such that $d_1d_0=W\id_{M_0}$ and $d_0d_1=W\id_{M_1}$. We call $W$ the \emph{potential}. 
The $\bZ$-grading of $R$ induces 
a $\bZ$-grading on $M$. The shift functor $\{k\}$ acts on $M$ as
$$M\{k\}=M_0\{k\}\xra{d_0}M_1\{k\}\xra{d_1}M_0\{k\},$$
where its action on the modules $M_0$, $M_1$ means an upward shift by $k$ units on the $\bZ$-grading.

A homomorphism $f\colon M\to M'$ of matrix factorizations of $W$ is a pair of 
maps of the same degree 
$f_i\colon M_i\to M'_i$ ($i=0,1$) such that the diagram
$$\xymatrix{
M_0 \ar[r]^{d_0}\ar[d]_{f_0} & M_1\ar[r]^{d_1}\ar[d]_{f_1} & M_0\ar[d]_{f_0} \\
M'_0 \ar[r]^{d'_0} & M'_1\ar[r]^{d'_1} & M'_0
}$$
commutes. It is an isomorphism of matrix factorizations if $f_0$ and $f_1$ are isomorphisms of the underlying modules.
Denote the set of homomorphisms of matrix factorizations from $M$ to $M'$ by 
$$\Hom_{\mf{}}(M,M').$$
It has an $R$-module structure with the action of $R$ given by 
$r(f_0,f_1)=(rf_0,rf_1)$ for $r\in R$.
Matrix factorizations over $R$ with homogeneous potential $W$ and homomorphisms 
of matrix factorizations form a graded additive category, which we denote by 
$\mf{R}(W)$. If $W=0$ we simply write $\mf{R}$.

Another description of matrix factorizations, which will be used in Chapter~\ref{chap:foamN}, is obtained by assembling the differentials $d_0$ and $d_1$ into an endomorphism $D$ of the $\bZ/2\bZ$-graded free $R$-module $M=M_0\oplus M_1$ such that
$$
D=
\begin{pmatrix} 
0 & d_1 \\ 
d_0 & 0
\end{pmatrix}
\qquad
\deg_{\bZ/2\bZ}D=1
\qquad
D^2=W\id_{M}.
$$
In this case we call $D$ the \emph{twisted differential}.

The free $R$-module $\Hom_R(M,M')$ of graded $R$-module homomorphisms 
from $M$ to $M'$ is a 2-complex 
$$\xymatrix{
\Hom_R^0(M,M')\ar[r]^d &
\Hom_R^1(M,M')\ar[r]^d &
\Hom_R^0(M,M')
}$$ where
\begin{align*}
\Hom_R^0(M,M') &= \Hom_R(M_0,M'_0)\oplus
\Hom_R(M_1,M'_1) \\
\Hom_R^1(M,M') &= \Hom_R(M_0,M'_1)\oplus
\Hom_R(M_1,M'_0)
\end{align*}
and for $f$ in $\Hom_R^i(M,M')$ the differential acts as
$$df=d_{M'}f-(-1)^ifd_{M}.$$
We define 
$$\Ext(M,M')=
\Ext^0(M,M')\oplus\Ext^1(M,M')
=\Ker{d}/\Image{d},$$
and write $\Ext_{(m)}(M,M')$ for the elements of $\Ext(M,M')$ with $\bZ$-degree $m$. Note that for $f\in\Hom_{\mf{}}(M,M')$ we have $df=0$. We say that two homomorphisms $f,\ g\in\Hom_{\mf{}}(M,M')$ are homotopic if there is an element $h\in\Hom_R^1(M,M')$ such that $f-g=dh$.

Denote by $\Hom_{\hmf{}}(M,M')$ the $R$-module of homotopy classes of homomorphisms of matrix factorizations from $M$ to $M'$ and by $\hmf{R}(W)$ the homotopy category of $\mf{R}(W)$.

We denote by $M\brak{1}$ and $M_\bullet$ the factorizations
$$M_1\xra{-d_1}M_0\xra{-d_0}M_1$$
and
$$
(M_0)^*\xra{-\left(d_1\right)^*}
(M_1)^*\xra{\left(d_0\right)^*}
(M_0)^*
$$ 
respectively. Factorization $M\brak{1}$ has potential $W$ while factorization $M_\bullet$ has potential $-W$. We call $M_\bullet$ the \emph{dual factorization} of $M$.

We have
\begin{align*}
\Ext^0(M,M') &\cong  \Hom_{\hmf{}}(M,M') \\
\Ext^1(M,M') &\cong  \Hom_{\hmf{}}(M,M'\brak{1})
\end{align*}

The tensor product $M\otimes_R M_\bullet$ has potential zero and is therefore a 
2-complex. Denoting by $\hy_{\mf{}}$ the homology of matrix factorizations with 
potential zero we have
$$\Ext(M,M')\cong \hy_{\mf{}}(M'\otimes_R M_\bullet)$$
and, if $M$ is a matrix factorization with  $W=0$, 
$$\Ext(R ,M)\cong \hy_{\mf{}}(M).$$

Let $R=\bQ[x_1,\ldots ,x_k]$ and $W\in R$. The Jacobi algebra of $W$ is defined as
\begin{equation}\label{KR:eq:jacobi}
J_W=\quotient{R}{(\partial_1W,\ldots ,\partial_kW)} ,
\end{equation}
where $\partial_i$ means the partial derivative with respect to $x_i$. Writing the differential as a matrix and differentiating both sides of the equation $D^2=W$ with respect to $x_i$ we get $D(\partial_iD)+(\partial_iD)D=\partial_iW$. We thus see that multiplication by $\partial_iW$ is homotopic to the zero endomorphism and that the homomorphism 
$$R\to\End_{\hmf{}}(M),\quad r\mapsto m(r)$$ 
factors through the Jacobi algebra of $W$.

Let $f,g\in\End(M)$. We define the \emph{supercommutator} of $f$ and $g$ as
$$
[f,g]_s=fg-(-1)^{\deg_{\bZ/2\bZ}(f)\deg_{\bZ/2\bZ}(g)}gf.
$$
The \emph{supertrace} of $f$ is defined as
$$\str(f)=\Tr\bigl( (-1)^{\gr}f \bigr)$$
where the \emph{grading operator} $(-1)^{\gr}\in\End(M_0\oplus M_1)$ is given by 
$$(m_0,m_1)\mapsto(m_0,-m_1),\quad m_0\in M_0,\ m_1\in M_1.$$
If $f$ and $g$ are homogeneous with respect to the $\bZ/2\bZ$-grading we have that
$$\str(fg)=(-1)^{\deg_{\bZ/2\bZ}(f)\deg_{\bZ/2\bZ}(g)}\str(gf),$$
and
$$\str\bigl([f,g]_s\bigr)=0.$$

\n There is a canonical isomorphism of $\bZ/2\bZ$-graded $R$-modules
$$\End(M)\cong M\otimes_R M_\bullet.$$
Choose a basis $\{\ket{i} \}$ of $M$ and define a dual basis $\{\bra{j}\}$ of $M_\bullet$ by $\brak{j\vert i}=\delta_{i,j}$, where $\delta$ is the \emph{Kronecker} symbol. 
There is a natural pairing map $M\otimes M_\bullet\to R$ called the \emph{super-contraction} that is given on basis elements $\ket{i}\bra{j}$ by 
$$\ket{i}\bra{j}\mapsto (-1)^{\deg_{\bZ/2\bZ}(\ket{i})\deg_{\bZ/2\bZ}(\bra{j})}\brak{j\vert i}=\delta_{i,j}.$$
The super-contraction induces a map $\End(M)\to R$ which coincides with the supertrace. When $M$ and $M_\bullet$ are factors in a tensor product $(M\otimes_RN)\otimes_R(M_\bullet\otimes_RN_\bullet)$ the super-contraction of $M$ with $M_\bullet$ induces a map $\str_{M}\colon\End(M\otimes_RN)\to\End(N)$ called the \emph{partial super-trace} (w.r.t. $M$).

%
%
%%%%%%%%%%%%%%%%%%%%%%%%%%%
%                         %
%  Koszul factorizations  %
%                         %
%%%%%%%%%%%%%%%%%%%%%%%%%%%
%
\subsection{Koszul Factorizations}
For $a$, $b$ homogeneous elements of $R$, an \emph{elementary Koszul factorization} $\{a,b\}$ over $R$ with potential $ab$ is a factorization of the form
$$R\xra{a}R\bigl\{{\scriptstyle\frac{1}{2}}\bigl(\deg_\bZ b -\deg_\bZ a \bigr)\bigr\}\xra{b}R.$$
When we need to emphasize the ring $R$ we write this factorization as $\{a,b\}_R$.
The tensor product of matrix factorizations $M_i$ with potentials $W_i$ is a matrix factorization with potential $\sum_iW_i$. We restrict to the case where all the $W_i$ are homogeneous of the same degree. Throughout this thesis we use tensor products of elementary Koszul factorizations $\{a_j,b_j\}$ to build bigger matrix factorizations, which we write in the form of a \emph{Koszul matrix} as
$$
\begin{Bmatrix}
a_1 \ , & b_1 \\
 \vdots & \vdots \\
a_k \ , & b_k
\end{Bmatrix}
$$
We denote by $\{\mathbf{a},\mathbf{b}\}$ the Koszul matrix which has columns 
$(a_1,\ldots,a_k)$ and $(b_1,\ldots ,b_k)$. If $\sum\limits_{i=1}^ka_ib_i=0$ then $\{\mathbf{a},\mathbf{b}\}$ is a 2-complex whose homology is an $R/(a_1,\ldots,a_k,b_1,\ldots ,b_k)$-module, since multiplication by $a_i$ and $b_i$ are null-homotopic endomorphisms of $\{\mathbf{a},\mathbf{b}\}$.

Note that the action of the shift $\brak{1}$ on $\{\mathbf{a},\mathbf{b}\}$ 
is equivalent to switching terms in one line of $\{\mathbf{a},\mathbf{b}\}$:
$$\{\mathbf{a},\mathbf{b}\}\brak{1}\cong
\begin{Bmatrix}
 \vdots\  & \vdots \\
a_{i-1}\ ,& b_{i-1} \\
  -b_i\ , &  -a_i   \\
a_{i+1}\ ,& b_{i+1} \\
 \vdots\  & \vdots
\end{Bmatrix}
\bigl\{ {\scriptstyle\frac{1}{2}}\bigl(\deg_\bZ b_i -\deg_\bZ a_i \bigr) \bigr\}.$$
If we choose a different row to switch terms we get a factorization which is 
isomorphic to this one. We also have that
$$\{\mathbf{a},\mathbf{b}\}_\bullet
\cong
\{\mathbf{a},-\mathbf{b}\}\brak{k}\{ s_k \}
,$$
where
$$s_k=\sum_{i=1}^{k}\deg_\bZ a_i -\frac{k}{2}\deg_\bZ W.$$

Let $R=\bQ[x_1,\ldots,x_k]$ and $R'=\bQ[x_2,\ldots,x_k]$. Suppose that 
$W=\sum_i a_ib_i\in R'$ and $x_1-b_i\in R'$, for a certain $1\leq i\leq k$. 
Let $c=x_1-b_i$ and $\{\mathbf{\hat a}^i,\mathbf{\hat b}^i\}$ be the matrix factorization obtained 
from 
$\{\mathbf{a},\mathbf{b}\}$ by deleting the $i$-th row and substituting $x_1$ 
by $c$.   
\begin{lem}[excluding variables]
\label{KR:lem:exvar}
The matrix factorizations $\{\mathbf{a},\mathbf{b}\}$ and 
$\{\mathbf{\hat a}^i,\mathbf{\hat b}^i\}$ are homotopy equivalent. 
\end{lem}
\n In~\cite{KR} one can find the proof of this lemma and its generalization 
with several variables. 

The following lemma contains three particular cases of Proposition~3 in~\cite{KR} 
(see also~\cite{KR2}):
\begin{lem}[Row operations]
We have the following isomorphisms of matrix factorizations 
$$
\begin{Bmatrix}
a_i \ , & b_i \\
a_j \ , & b_j \\
\end{Bmatrix} 
\stackrel{[i,j]_\lambda}{\cong}
\begin{Bmatrix}
a_i-\lambda a_j \ , & b_i  \\
a_j             \ , & b_j+\lambda b_i
\end{Bmatrix}
,\qquad
\begin{Bmatrix}
a_i \ , & b_i \\
a_j \ , & b_j \\
\end{Bmatrix} 
\stackrel{[i,j]'_\lambda}{\cong}
\begin{Bmatrix}
a_i+\lambda b_j \ , & b_i \\
a_j-\lambda b_i \ , & b_j
\end{Bmatrix}
$$ 
for $\lambda\in R$. If $\lambda$ is invertible in $R$, we also have
$$\bigl\{ a_i\ ,\ b_j \bigr\}\stackrel{[i]_\lambda}{\cong}
\bigl\{\lambda a_i\ ,\ \lambda^{-1}b_i \bigr\}.$$
\end{lem}

\begin{proof}
It is straightforward to check that the pairs of matrices
$$
\left[i,j\right]_\lambda = \Biggl(
\begin{pmatrix}
1 & 0 \\ 
0 & 1
\end{pmatrix},\ 
\begin{pmatrix}
1 & -\lambda \\
0 & 1
\end{pmatrix}\Biggr)
,\quad
\left[i,j\right]'_\lambda =\Biggl(
\begin{pmatrix}
       1 & 0 \\ 
-\lambda & 1
\end{pmatrix},\ 
\begin{pmatrix}
1 & 0 \\
0 & 1
\end{pmatrix}\Biggr)
\quad \text{and}\quad
\left[i\right]_\lambda =
(1,\ \lambda )
$$
define isomorphisms of matrix factorizations.
\end{proof}

Recall that a sequence $(a_1,a_2,\ldots ,a_k)$ is called \emph{regular} in $R$ 
if $a_j$ is not a zero divisor in $R/(a_1,a_2,\ldots , a_{j-1})$, for 
$j=1,\ldots,k$.
The proof of the following lemma can be found in~\cite{KR2}.
\begin{lem}\label{KR:lem:regseq-iso}
Let $\mathbf{b}=(b_1,b_2, \ldots ,b_k)$, $\mathbf{a}=(a_1,a_2, \ldots ,a_k)$ and $\mathbf{a'}=(a'_1,a'_2, \ldots ,a'_k)$ be sequences in $R$.
If $\mathbf{b}$ is regular and $\sum_ia_ib_i=\sum_ia'_ib_i$ then the factorizations 
$$
\{\mathbf{a}\ , \mathbf{b} \}\ 
\text{ and }\
\{\mathbf{a'}\ , \mathbf{b}\}
$$ are isomorphic.
\end{lem}

A factorization $M$ with potential $W$ is said to be \emph{contractible} if it is isomorphic to a direct sum of factorizations of the form
$$R\xra{1}R\{{\scriptstyle\frac{1}{2}}\deg_\bZ W \}\xra{W}R\quad 
\text{and}\quad
R\xra{W}R\{-{\scriptstyle\frac{1}{2}}\deg_\bZ W   \}\xra{1}R.$$

%
%
%%%%%%%%%%%%%%%%%%%%%%%
%%%                 %%%
%%%   KR-homology   %%%
%%%                 %%%
%%%%%%%%%%%%%%%%%%%%%%%
\section{KR homology}\label{KR:sec:KR}

To define the KR theory of~\cite{KR} we use the webs of Section~\ref{KR:sec:slN} with a slight modification: every double edge has exactly two simple edges entering one endpoint and two leaving the other.
We also decorate the simple edges with marks such that each simple edge has at least one mark, as the example in Figure~\ref{KR:fig:marks}. We allow 
open webs which have simple edges with only one endpoint glued to the rest of the 
graph. Free ends always count as marks.
\begin{figure}[h!]
\labellist
\small\hair 2pt
\pinlabel $x_1$ at -2 -5
\pinlabel $x_2$ at 57 14
\pinlabel $x_3$ at 61 66
\pinlabel $x_4$ at -1 104
\endlabellist
\centering
\figs{0.6}{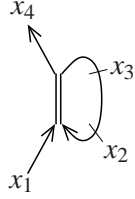}
\caption{A web with marks}
\label{KR:fig:marks}
\end{figure}

Suppose there are $k$ marks on a web $\Gamma$ and let $\mathbf{x}$ denote the set 
$\{x_1,x_2,\ldots ,x_k\}$. Denote by $R$ the polynomial ring $\bQ[\mathbf{x}]$. 
Recall that we have defined a $\bZ$-grading on $R$ which from now on will be referred to as the $q$-grading, because it is related to the powers of $q$ that appear in the polynomial $P_N$ of Section~\ref{KR:sec:slN}.

The polynomial 
$p(x)=x^{N+1}$
is the building block for the potentials in KR theory.
For a general web $\Gamma$ we define the potential as
\begin{equation}\label{KR:eq:KRpot}
W=\sum_{i}s_ip(x_i),
\end{equation}
where $i$ runs over all free ends of $\Gamma$ and $s_i=1$ if the corresponding arc is oriented outward and $s_i=-1$ in the opposite case.

As in~\cite{KR2} we denote by $\hat\Gamma$ the matrix factorization associated to 
a web $\Gamma$. To an oriented arc with marks $x$ and $y$, as in 
Figure~\ref{KR:fig:arc-r}, we assign the potential
$$W=p(x)-p(y)=x^{N+1}-y^{N+1}$$
\begin{figure}[ht!]
\labellist
\small\hair 2pt
\pinlabel $y$ at -15 5
\pinlabel $x$ at 112 6
\endlabellist
\centering
\figs{0.5}{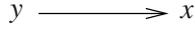}
\caption{An oriented arc}
\label{KR:fig:arc-r}
\end{figure}

\n and the \emph{arc factorization} $\hatarc$, which is given by the Koszul factorization
$$\hatarc=\{\pi_{xy}\ ,\ x-y\}=
R\xra{\pi_{xy}}R\{-2\}\xra{x-y}R
$$
where
$$\pi_{xy}=\frac{x^{N+1}-y^{N+1}}{x-y}=\sum_{j=0}^{N}x^jy^{N-j}.$$

\n To the double edge in Figure~\ref{KR:fig:fatedge} we associate the 
potential $W=p(x_i)+p(x_j)-p(x_k)-p(x_l)$.
\begin{figure}[ht!]
\labellist
\small\hair 2pt
\pinlabel $x_i$ at -9 118
\pinlabel $x_j$ at 59 116
\pinlabel $x_k$ at -9 -9
\pinlabel $x_l$ at 59 -10
\endlabellist
\centering
\figs{0.4}{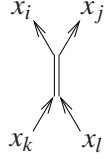}
\caption{A double edge}
\label{KR:fig:fatedge}
\end{figure}

\n Let $g$ be the unique two-variable polynomial such that $g(x+y,xy)=p(x)+p(y)$ and
\begin{align*}
u_{ijkl} &= \frac{g(x_i+x_j,x_ix_j) - g(x_k+x_l,x_ix_j)}{x_i+x_j-x_k-x_l} \\
v_{ijkl} &= \frac{g(x_k+x_l,x_ix_j)-g(x_k+x_l,x_kx_l)}{x_ix_j-x_kx_l}.
\end{align*}
Note that $u_{ijkl}$ and $v_{ijkl}$ are polynomials in $x_i$, $x_j$, $x_k$ and $x_l$.
The \emph{dumbell factorization} $\hatDoubleEdge$ is then defined as 
the tensor product of the factorizations
$$
R\{-1\} \xra{u_{ijkl}} R\{-N\}\xra{x_i+x_j-x_k-x_l} R\{-1\}
$$
and
$$
R \xra{v_{ijkl}}  R\{3-N\}\xra{x_ix_j-x_kx_l} R.
$$
We can write the dumbell factorization as the Koszul matrix 
$$\hatDoubleEdge=
\begin{Bmatrix}
u_{ijkl} \ , & x_i+x_j-x_k-x_l \\
v_{ijkl} \ , & x_ix_j-x_kx_l \\
\end{Bmatrix} 
\{-1\}.$$

The arc and dumbell factorizations are the building blocks to build matrix factorizations for every web. Any web can be built from arcs and dumbells using the operations of \emph{disjoint union}, \emph{identifying free ends} and \emph{mark addition}/\emph{mark removal}. Suppose we have a disjoint union of webs $\Gamma_1\sqcup\Gamma_2$. If $W_1$ and $W_2$ are the potentials of $\hat\Gamma_1$ and $\hat\Gamma_2$ respectively then the factorization assigned to $\Gamma_1\sqcup\Gamma_2$ must have potential $W_1+W_2$, since the set of endpoints of $\Gamma_1\sqcup\Gamma_2$ is the union of the sets of endpoints of $\Gamma_1$ and $\Gamma_2$. Therefore we define 
$$\widehat{\Gamma_1\sqcup\ \Gamma}_2=\hat\Gamma_1\otimes_\bQ\hat\Gamma_2.$$

Now suppose that $\Gamma$ has two free ends marked $x_i$ and $x_j$ and that one of them is oriented outward while the other is oriented inward. Suppose also that they can be glued to each other without crossing the rest of the web. Then we can identify these free ends forming a new diagram. This corresponds to taking the quotient $\hat\Gamma/(x_i-x_j)$. If the web $\Gamma$ is a disjoint union $\Gamma_1\sqcup\Gamma_2$ and if the free ends marked $x_{i_1}$ and $x_{i_2}$ belong to $\Gamma_1$ and $\Gamma_2$ respectively, then we can first take the tensor product $\hat\Gamma_1\otimes_\bQ\hat\Gamma_2$ and then identify the free ends marked $x_{i_1}$ and $x_{i_2}$. It is straightforward that the resulting matrix factorization corresponds to the tensor product
$$\hat\Gamma_1\otimes_{\bQ[y_i]}\hat\Gamma_2,$$
where $y_i$ acts as $x_{i_1}$ on $\hat\Gamma_1$ and $x_{i_2}$ on $\hat\Gamma_2$.

Finally suppose that only the internal marks of $\Gamma$ and $\Gamma'$ are different, that is, one can be obtained from the other by a different placing of marks while the marks at the free ends are kept. Then they have the same potential and $\hat\Gamma$ and $\hat\Gamma'$ are canonically isomorphic in $\hmf{R}(W)$ (see~\cite{KR}).

Let $E$ and $T$ denote the set of arcs between marks and the set of double edges, respectively, of a general web $\Gamma$. The matrix factorization $\hat\Gamma$ is built from the arc and the dumbell 
factorizations as
$$\hat\Gamma=\bigotimes\limits_{e\in E}\hatarc\mspace{-3.5mu}_e\otimes\bigotimes
\limits_{t\in T}\hatDoubleEdge_t.$$
The tensor products are taken over suitable rings, so that as a module $\hat\Gamma$ is free and of finite rank over $R$. It is a matrix factorization with potential given by Equation~\eqref{KR:eq:KRpot}.
 If $\Gamma$ is a closed web then $\hat\Gamma$ is a 2-complex, whose homology is nontrivial in only one of its $\bZ/2\bZ$-degrees~\cite{KR}.

%%%%%%%%%%%%%%%%%%%%%
%   Differentials   %
%%%%%%%%%%%%%%%%%%%%%
\subsection{The maps $\chi_0$ and $\chi_1$}\label{KR:ssec:diffs}
Let $\hatorsmooth$ and $\hatDoubleEdge$ denote the factorizations corresponding to the webs in Figure~\ref{KR:fig:maps-chi}.

\begin{figure}[ht!]
\labellist
\small\hair 2pt
\pinlabel $x_i$ at -9 117
\pinlabel $x_j$ at 57 115
\pinlabel $x_k$ at -9 -10
\pinlabel $x_l$ at 57 -10
\pinlabel $x_i$ at 133 117
\pinlabel $x_j$ at 197 115
\pinlabel $x_k$ at 133 -10
\pinlabel $x_l$ at 197 -10
\pinlabel $\chi_0$ at 95 84
\pinlabel $\chi_1$ at 95 24
\endlabellist
\centering
\figs{0.4}{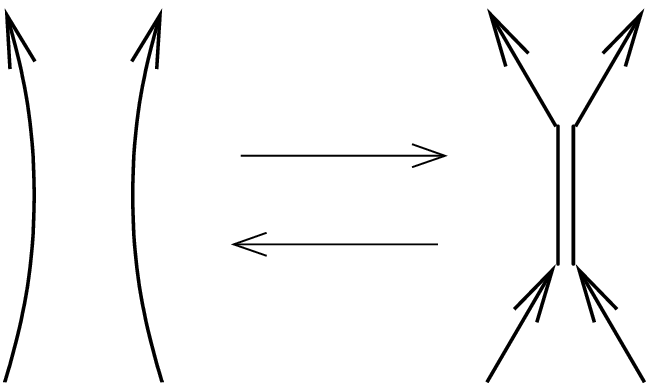}
\caption{Maps $\chi_0$ and $\chi_1$}
\label{KR:fig:maps-chi}
\end{figure}

\n The factorization $\hatorsmooth$ is given in matrix notation by 
$$
\begin{pmatrix}R \\ R\{2-N\} \end{pmatrix}
\stackrel{P_0}{\lra}
\begin{pmatrix}R\{1-N\} \\ R\{1-N\} \end{pmatrix}
\stackrel{P_1}{\lra}
\begin{pmatrix}R \\ R\{2-N\} \end{pmatrix}
$$
with
$$
P_0=\begin{pmatrix}\pi_{ik} & x_j-x_l \\ \pi_{jl} & -x_i+x_k\end{pmatrix},\qquad
P_1=\begin{pmatrix}x_i-x_k & x_j-x_l \\ \pi_{jl} & -\pi_{ik}\end{pmatrix},
$$
and the factorization $\hatDoubleEdge$ is given by 
$$
\begin{pmatrix}R\{-1\} \\ R\{3-2N\} \end{pmatrix}
\stackrel{Q_0}{\lra}
\begin{pmatrix}R\{-N\} \\ R\{2-N\} \end{pmatrix}
\stackrel{Q_1}{\lra}
\begin{pmatrix}R\{-1\} \\ R\{3-2N\} \end{pmatrix}
$$
with
$$
Q_0=\begin{pmatrix}u_{ijkl} & x_ix_j-x_kx_l \\ v_{ijkl} & -x_i-x_j+x_k+x_l
\end{pmatrix},\qquad
Q_1=\begin{pmatrix}x_i+x_j-x_k-x_l & x_ix_j-x_kx_l \\ v_{ijkl} & -u_{ijkl}
\end{pmatrix}.
$$

\n Define the homomorphisms $\chi_0\colon \hatorsmooth\to\hatDoubleEdge$ and $\chi_1\colon\hatDoubleEdge\to\hatorsmooth$ by the pairs of matrices:
$$
\chi_0 = \Biggl(
\begin{pmatrix}x_k-x_j & 0 \\ \alpha & 1 \end{pmatrix},
\begin{pmatrix}x_k & -x_j \\ -1 & 1 \end{pmatrix}
\Biggr)\quad \text{and}\qquad
\chi_1 = \Biggl(
\begin{pmatrix} 1 & 0 \\ -\alpha & x_k-x_j \end{pmatrix},
\begin{pmatrix} 1 & x_j \\ 1 & x_k \end{pmatrix}
\Biggr)
$$
where 
\begin{equation}\label{KR:eq:adiff}
\alpha=-v_{ijkl}+\frac{u_{ijkl}+x_iv_{ijkl}-\pi_{jl}}{x_i-x_k}.
\end{equation}

\n The maps $\chi_0$ and $\chi_1$ have degree 1. A straightforward calculation 
shows that $\chi_0$ and $\chi_1$ are homomorphisms of matrix factorizations, and that
$$\chi_0\chi_1=m(x_k-x_j)\id(\hatDoubleEdge)
\qquad
\chi_1\chi_0=m(x_k-x_j)\id(\hatorsmooth),
$$
where $m(x_*)$ is multiplication by $x_*$.

%%%%%%%%%%%%%%%%%%
%   cobordisms   %
%%%%%%%%%%%%%%%%%%
\subsection{Web cobordisms}\label{KR:ssec:cob-mf}
In this subsection we show which homomorphisms of matrix factorizations we 
associate to the elementary singular web-cobordisms. 
We do not know if these can be put together in a well-defined way for 
arbitrary singular web-cobordisms. 

The elementary web-cobordisms are the \emph{zip}, the \emph{unzip} 
(see Figure~\ref{KR:fig:zipunzip}) 
\begin{figure}[ht!]
\centering
\figins{0}{0.6}{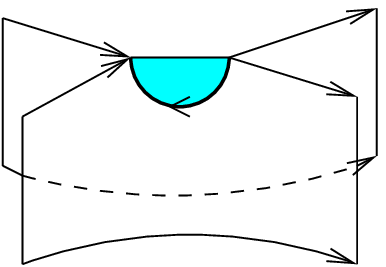} \qquad
\figins{0}{0.6}{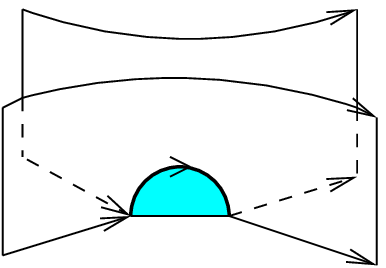}
\caption{The zip and the unzip cobordisms}
\label{KR:fig:zipunzip}
\end{figure}
and the elementary cobordisms in Figure~\ref{KR:fig:cobs}. 
To the zip and the unzip we associate the maps $\chi_0$ and $\chi_1$, 
respectively, as defined in Subsection~\ref{KR:ssec:diffs}. For each 
elementary cobordism in Figure~\ref{KR:fig:cobs} we define a 
homomorphism of matrix factorizations as below.

\begin{figure}[ht!]
\centering
\figins{0}{0.25}{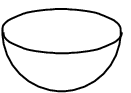}\qquad
\figins{0}{0.25}{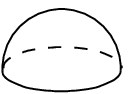}\qquad
\figins{-7}{0.6}{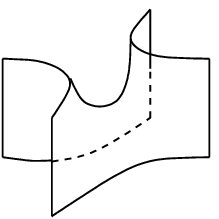}
\caption{Elementary cobordisms}
\label{KR:fig:cobs}
\end{figure}

\n First note that there are isomorphisms
\begin{align}
\hat\emptyset &\cong \bQ\to 0\to\bQ \label{KR:eq:isoe}\\
\hatunknot &\cong  0\to\quotient{\bQ[x]}{(x^N)}\{1-N\}\to 0.\label{KR:eq:isou}
\end{align}

\n The \emph{unit}  map 
$$\imath\colon\bQ\brak{1} \to \hatunknot$$
is the homomorphism of matrix factorizations induced by the 
map (denoted by the same symbol)
$$
\imath\colon\bQ\to\quotient{\bQ[x]}{(x^N)}\{-N+1\}, \quad  1\mapsto 1
$$
using the isomorphisms~\eqref{KR:eq:isoe} and~\eqref{KR:eq:isou}.

Let $\xi$ be a nonzero rational number. The \emph{trace} map 
$$\varepsilon\colon\hatunknot\to \bQ\brak{1}$$
is the homomorphism of matrix factorizations induced by the 
map (denoted by the same symbol)
$$
\varepsilon\colon\quotient{\bQ[x]}{(x^N)}\{-N+1\}\to\bQ,\quad  x^k\mapsto 
\begin{cases}
\xi, & k=N-1 \\
 0 , & k\neq N-1 \end{cases}
$$
using the isomorphisms~\eqref{KR:eq:isoe} and~\eqref{KR:eq:isou}.

Let $\hattwoedgesop$ and $\hathtwoedgesop$ be the factorizations corresponding to the webs in Figure~\ref{KR:fig:saddle}.
\begin{figure}[ht!]
\labellist
\small\hair 2pt
\pinlabel $x_1$ at -4 89
\pinlabel $x_2$ at 89 89
\pinlabel $x_3$ at -4 -6
\pinlabel $x_4$ at 89 -6
\pinlabel $x_1$ at 170 89
\pinlabel $x_2$ at 264 89
\pinlabel $x_3$ at 170 -6
\pinlabel $x_4$ at 264 -6
\pinlabel $\eta$ at 128 56
\endlabellist
\centering
\figs{0.4}{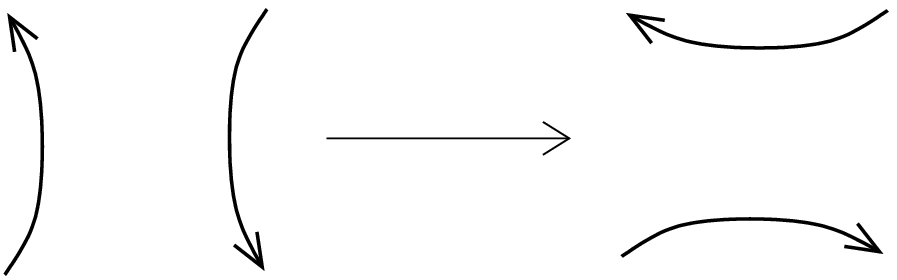}
\caption{Saddle point homomorphism}
\label{KR:fig:saddle}
\end{figure}

\n The matrix factorization $\hattwoedgesop$ is given by 
$$
\begin{pmatrix}R \\ R\{2-2N\} \end{pmatrix}
\xra{\begin{pmatrix}
\pi_{13} & x_4-x_2\\
\pi_{24} & x_3-x_1 
\end{pmatrix}}
\begin{pmatrix}R\{1-N\} \\ R\{1-N\} \end{pmatrix}
\xra{\begin{pmatrix}
x_1-x_3  & x_4-x_2\\
\pi_{24} & -\pi_{13}
\end{pmatrix}}
\begin{pmatrix}R \\R\{2-2N\} \end{pmatrix}
$$
and 
$\hathtwoedgesop\brak{1}$ is given by 
$$
\begin{pmatrix}R\{1-N\} \\ R\{1-N\}  \end{pmatrix}
\xra{\begin{pmatrix}
  x_2-x_1  & x_3-x_4 \\
 -\pi_{34} & \pi_{12}
\end{pmatrix}}
\begin{pmatrix}R \\ R\{2-2N\} \end{pmatrix}
\xra{\begin{pmatrix}
 -\pi_{12} & x_3-x_4\\
 -\pi_{34} & x_1-x_2 
\end{pmatrix}}
\begin{pmatrix}R\{1-N\} \\ R\{1-N\} \end{pmatrix}
$$

\n To the saddle cobordism between the webs $\twoedgesop$ and $\htwoedgesop$ we associate the homomorphism of matrix factorizations $\eta\colon\hattwoedgesop\to \hathtwoedgesop\brak{1}$ described by the pair of matrices
$$\eta_0=
\begin{pmatrix}
 e_{123}+e_{124} & 1 \\
-e_{134}-e_{234} & 1 
\end{pmatrix},
\qquad
\eta_1=
\begin{pmatrix}
 -1 & 1 \\
-e_{123}-e_{234} & -e_{134}-e_{123} 
\end{pmatrix}
$$
where 
$$
e_{ijk}=
\frac{(x_k-x_j)p(x_i)+(x_i-x_k)p(x_j)+(x_j-x_i)p(x_k)}{2(x_i-x_j)(x_j-x_k)(x_k-x_i)}
.
$$

\n The homomorphism $\eta$ has degree $N-1$. In general $\eta$ is only defined up to a sign (see~\cite{KR}). 

%
%
%%%%%%%%%%%%%%%%%%%%%%%%%
%      MOY moves        %
%%%%%%%%%%%%%%%%%%%%%%%%%
%
%
\subsection{MOY web moves}\label{KR:ssec:MOY}

One of the main features of the KR theory is the 
categorification of the MOY web moves. The categorified MOY
moves are described by the homotopy equivalences 
below.
\begin{lem}\label{KR:lem:KR-moy}
We have the following direct sum decompositions:
\begin{align}
\figins{-13}{0.5}{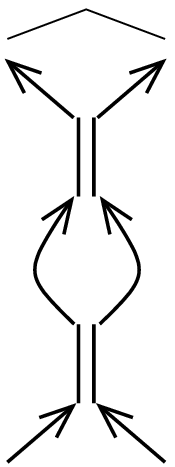}\
&\cong\
\figins{-13}{0.5}{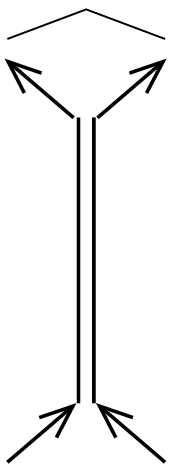}\{-1\}
\oplus
\figins{-13}{0.5}{hatdr1-db}\{1\}
\tag{MOY1}
\\[3ex]
\figins{-13}{0.5}{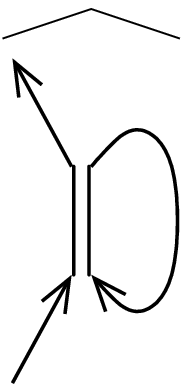}\
&\cong\
\sum\limits_{i=0}^{N-2}\
\figins{-13}{0.5}{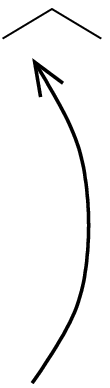}\brak{1}\{2-N+2i\}
\tag{MOY2}
\\[3ex]
\figins{-11}{0.455}{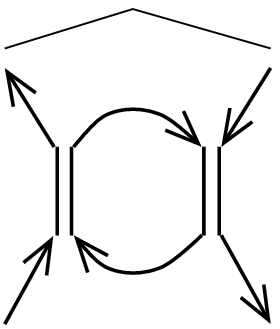}\
&\cong\
\figins{-11}{0.45}{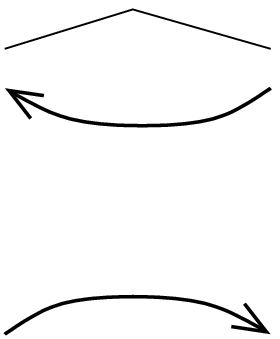}
\oplus\left(
\bigoplus\limits_{i=0}^{N-3}\
\figins{-11}{0.45}{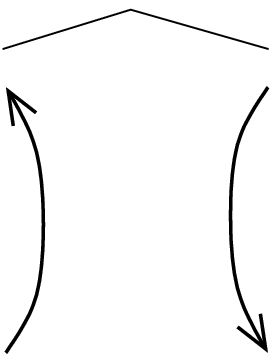}\ 
\brak{1}\{3-N+2i\}
\right)
\tag{MOY3}
\end{align}

\begin{equation}
\hspace{-6ex}
\figins{-19}{0.7}{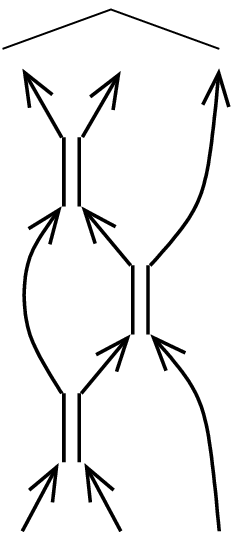}\
\oplus
\figins{-19}{0.7}{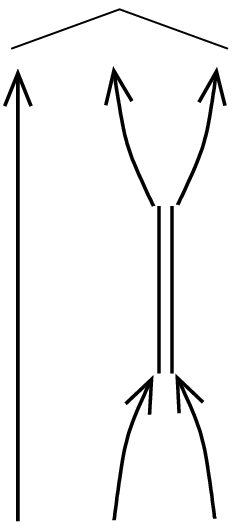}
\cong
\figins{-19}{0.7}{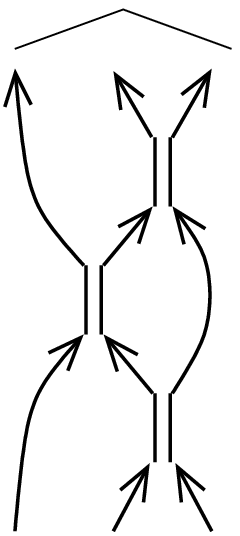}\
\oplus
\figins{-19}{0.7}{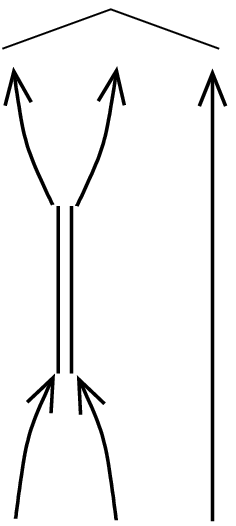}
\tag{MOY4}\label{KR:eq:moy4}
\end{equation}
\end{lem}

\medskip

The last relation is a consequence of two relations involving a triple edge
$$
\figins{-20}{0.7}{hatsquare2a-db}
\cong 
\figins{-20}{0.7}{hatsqr2bb-db}\oplus
\figins{-20}{0.7}{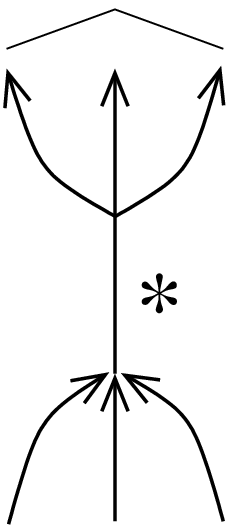}
\qquad , \qquad
\figins{-20}{0.7}{hatsquare2b-db}
\cong 
\figins{-20}{0.7}{hatsqr2aa-db}\oplus
\figins{-20}{0.7}{hattrpl}.
$$ 

\medskip

The factorization assigned to the triple edge in Figure~\ref{KR:fig:trp-KR} has
potential
$$W=p(x_1)+p(x_2)+p(x_3)-p(x_4)-p(x_5)-p(x_6).$$
\begin{figure}[ht!]
\labellist
\small\hair 2pt
\pinlabel $x_1$ at   0 199
\pinlabel $x_2$ at  51 199
\pinlabel $x_3$ at 100 199
\pinlabel $x_4$ at   0 -19
\pinlabel $x_5$ at  51 -19
\pinlabel $x_6$ at 100 -19
\endlabellist
\centering
\figs{0.28}{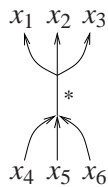}
\caption{Triple edge factorization}
\label{KR:fig:trp-KR}
\end{figure}

\n Let $h$ be the unique three-variable polynomial such that
$$h(x+y+z,xy+xz+yz,xyz)=p(x)+p(y)+p(z)$$
and let
\begin{gather*}
e_1=x_1+x_2+x_3,
\qquad 
e_2=x_1x_2+x_1x_3+x_2x_3,
\quad
e_3=x_1x_2x_3, \\
s_1=x_4+x_5+x_6,
\qquad 
s_2=x_4x_5+x_4x_6+x_5x_6,
\quad
s_3=x_4x_5x_6.
\end{gather*}
Define
\begin{align*}
h_1 &= \frac{h(e_1,e_2,e_3)-h(s_1,e_2,e_3)}{e_1-s_1}, \\
h_2 &= \frac{h(s_1,e_2,e_3)-h(s_1,s_2,e_3)}{e_2-s_2}, \\
h_3 &= \frac{h(s_1,s_2,e_3)-h(s_1,s_2,s_3)}{e_3-s_3},
\end{align*}
so that we have $W=h_1(e_1-s_1)+h_2(e_2-s_2)+h_3(e_3-s_3)$. The matrix factorization 
$\hat\Upsilon$ corresponding to the triple edge is defined by the Koszul 
matrix
$$\hat\Upsilon=
\begin{Bmatrix}
h_1\ ,& e_1-s_1 \\
h_2\ ,& e_2-s_2 \\
h_3\ ,& e_3-s_3
\end{Bmatrix}_R\{-3\},$$
where $R=\bQ[x_1,\ldots ,x_6]$. 
The matrix factorization $\hat\Upsilon$ is the tensor product of 
the matrix factorizations
$$R\xra{h_i}R\{2i-1-N\}\xra{e_i-s_i}R,\qquad i=1,2,3,$$
shifted down by 3.

%%%%%%%%%%%%%%%%%%%%%%%%%
%      KR complex       %
%%%%%%%%%%%%%%%%%%%%%%%%%
\subsection{The KR complex}\label{KR:ssec:cmplx}

Next we define a complex of matrix factorizations for each link diagram. Let $D$ be a diagram of a link $L$. Assigning marks to the free ends of each $\overcrossing$ and $\undercrossing$ we can define complexes of matrix factorizations by
\begin{align*}
\KR(\undercrossing) &= 
0\to\underline{\raisebox{1 pt}{$\hatorsmooth$}}\{1-N\}\xra{\chi_0}\raisebox{1 pt}{$\hatDoubleEdge$}\{-N\}\to 0 \\
\KR(\overcrossing) &= 
0\to\raisebox{1 pt}{$\hatDoubleEdge$}\{N\}\xra{\chi_1}\underline{\raisebox{1 pt}{$\hatorsmooth$}}\{N-1\}\to 0
\end{align*}
where underlined terms correspond to homological degree 0. The complex for a general link diagram $D$ can be obtained from the complexes $\KR(\undercrossing)$ and $\KR(\overcrossing)$ as follows. Put at least one  mark on every arc of $D$ and for each crossing $D_c$ form the complex $\KR(D_c)=\KR(\overcrossing)$ or $\KR(D_c)=\KR(\undercrossing)$, according to whether it is a positive or a negative crossing. The complex $\KR(D)$ is given by the tensor product of $\KR(D_c)$ over all crossings and $\hatarc$ over all arcs between marks. The tensor products are taken over appropriate rings, so that $\KR(D)$ is a finite rank free module over $R$.

An equivalent construction consists of flattening each crossing using the 0 and 1-flattening as in Figure~\ref{KR:fig:flatten}. Let $D$ be a link diagram with $c$ crossings. A complete flattening of $D$ is completely defined by an array $(i_1,\ldots,i_c)$ of 0's and 1's, where we are assuming some ordering of the crossings. The set of complete flattenings of $D$ is in 1-1 correspondence with the vertices of the hypercube $\{0,1\}^c$.  Let $\Gamma_{i_1,\ldots,i_c}$ be a complete flattening of $D$ and let $|i|$ be the number of 1-flattenings in it, that is, the number of 1's in $(i_1,\ldots,i_c)$. We call $|i|$ the \emph{height} of $\Gamma_{i_1,\ldots,i_c}$. To form the complex we start by placing each $\hat{\Gamma}_{i_1,\ldots,i_c}$, shifted by $\{|i|\}$, in the corresponding vertex of $\{0,1\}^c$. Each edge of this decorated hypercube connects two complete flattenings that differ only in the flattening of one crossing of $D$. Let $j$ be that crossing and let $s(j)$ be the number of 1's in $(i_1,\ldots,i_{j-1})$. Orient that edge from the 0- to the 1-flattening and decorate it with the map $(-1)^{s(j)}\chi_i$ where $i=0$ if it is a positive crossing or with $i=1$ if it is a negative crossing. Therefore, every (oriented) edge of the hypercube increases the height by 1 and preserves the $q$-grading. This way every square of the hypercube anti-commutes and we can form a complex of matrix factorizations
$$
\KR^{k,*}(D)=\bigoplus\limits_{|i|=k}\hat{\Gamma}_{i_1,\ldots,i_c}[-c_-]\{(N-1)c_+-Nc_-+|i|\}
$$
where $[\ \ ]$ denotes an upward shift in the homological grading and the horizontal differential $\KR^{k,j}(D)\to\KR^{k+1,j}(D)$ is given by taking the direct sum over all edges of the hypercube starting in $\KR^{k,*}(D)$ and ending in $\KR^{k+1,*}(D)$ using the maps above.

Khovanov and Rozansky prove in~\cite{KR} that if a link diagram $D'$ is obtained from a link diagram $D$ through a sequence of Reidemeister moves then $\KR(D')$ is homotopy equivalent to $\KR(D)$. Therefore, up to homotopy the KR complex of a link does not depend on the link diagram that is used, which justifies the notation $\KR(L)$. The $\bQ$-vector space $\KR(L)$ is $\bZ\times\bZ/2\bZ\times\bZ$-graded, but its homology groups w.r.t. the $\bZ/2\bZ$-grading are nontrivial only in the $\bZ/2\bZ$-degree which is equal to the number of components of $L$ modulo 2. 
This means that taking first the homology of $\KR(L)$ w.r.t. the vertical differential  and then the homology w.r.t. the horizontal differential yields a bigraded $\bQ$-vector space, denoted $\HKR(L)$ and called the
\emph{Khovanov-Rozansky homology} of $L$.
If we want to emphasize a specific value for $N$ we write $\HKR_N(L)$.
We have
$$\HKR(\unknot)\cong
\bigl(\quotient{\bQ[x]}{(x^N)}\bigr)
\brak{1}\{-N+1\}.$$ 

Denote by $\HKR^{i,j}(L)$ the summand of $\HKR(L)$ in bidegree $(i,j)$, where $i$ is the homological grading and $j$ the $q$-grading. From the MOY relations of Lemma~\ref{KR:lem:KR-moy} and from the fact that the differentials in $\KR(L)$ are grading preserving it follows that the graded Euler characteristic of $\HKR(L)$ equals the $\sln$-link polynomial,
$$
P_N(L)=\sum_{i,j\in\bZ}(-1)^iq^j\dim_\bQ\bigl(\HKR^{i,j}_N(L)\bigr).
$$

A second important property of $\HKR$ is that it defines a projective functor from the category $\Link$ of oriented links in $S^3$ and ambient isotopy classes of oriented link cobordisms properly embedded in $S^3\times[0,1]$ to the category $\V$ of $\bZ$-graded rational vector spaces and $\bZ$-graded linear maps. We omit the details in this thesis.

A remark should be made about terminology. KR homology (as the original Khovanov homology) is the homology of a cohomological complex (the differential increases the homological grading) that behaves as homology on cobordisms, that is, it is a covariant functor from $\Link$ to $\V$. That is the reason why we use the term KR homology instead of KR cohomology.

%%%%%%%%%%%%%%%%%%%%%%%
%%%                 %%%
%%%   KR-deformed   %%%
%%%                 %%%
%%%%%%%%%%%%%%%%%%%%%%%
\section{The universal rational KR $\slt$-link homology}
\label{KR:sec:KR-abc}

The theory of Khovanov and Rozansky can be generalized to give the universal rational KR homology for all $N>0$ (see~\cite{gornik,rasmussen-diff,wu-filt}). In this section we describe the universal rational KR homology for $N=3$, which is similar to the one presented in the previous section, but defined over $\bQ[a,b,c,\mathbf{x}]$, where $a$, $b$, $c$ are formal parameters. If $a=b=c=0$ we get back the original KR theory for $N=3$. 
Throughout this section we denote by $R$ the polynomial ring $\bQ[a,b,c,\mathbf{x}]$. We define a $q$-grading on $R$ declaring that 
$$q(1)=0,\quad q(x_i)=2,\,\,\,\text{for all}\,\, i,\quad q(a)=2,\quad q(b)=4,
\quad q(c)=6.$$

The basic polynomial to form the potentials in the universal rational KR theory for $N=3$ is 
\begin{equation}\label{KR:eq:upoly}
p(x)=x^4 -\frac{4a}{3} x^3 - 2b x^2 - 4c x.
\end{equation}

\n The \emph{arc factorization} $\hatarc$ in Figure~\ref{KR:fig:arc-r} has potential
$$W=p(x)-p(y)=x^4-y^4 -\frac{4a}{3} \bigl(x^3-y^3 \bigr) - 2b\bigl(x^2-y^2\bigr) -
4c\bigl(x-y\bigr)
$$
and is defined by the Koszul factorization
$$\hatarc=\{\pi_{xy}\ ,\ x-y\}=
R\xra{\pi_{xy}}R\{-2\}\xra{x-y}R
$$
where
$$\pi_{xy}=\frac{p(x)-p(y)}{x-y}=\frac{x^4-y^4}{x-y}-\frac{4a}{3}\frac{x^3-y^3}{x-y}-2b\frac{x^2-y^2}{x-y}-4c.$$

The \emph{dumbell factorization} $\hatDoubleEdge$ in Figure~\ref{KR:fig:fatedge} has potential $W=p(x_i)+p(x_j)-p(x_k)-p(x_l)$
and is defined by the Koszul matrix
$$\hatDoubleEdge=
\begin{Bmatrix}
u_{ijkl} \ , & x_i+x_j-x_k-x_l \\
v_{ijkl} \ , & x_ix_j-x_kx_l \\
\end{Bmatrix}
\{-1\}.$$

\n Here the polynomials $u_{ijkl}$ and $v_{ijkl}$ are obtained as in Section~\ref{KR:sec:KR} using the polynomial $p$ of Equation~\eqref{KR:eq:upoly}. Explicitly,
\begin{align*}
u_{ijkl} &= \frac{(x_i+x_j)^4-(x_k+x_l)^4}{x_i+x_j-x_k-x_l}-(2b+4x_ix_j)(x_i+x_j+x_k+x_l) \\
  & \quad -\frac{4a}{3}\biggl(\frac{(x_i+x_j)^3-(x_k+x_l)^3}{x_i+x_j-x_k-x_l}-3x_ix_j\biggr)-4c,\\
v_{ijkl} &= 2(x_ix_j+x_kx_l)-4(x_k+x_l)^2 + 4a (x_k+x_l) + 4b.
\end{align*}

The matrix factorization assigned to a bigger web is built exactly as in Section~\ref{KR:sec:KR}, with the ring $\bQ[a,b,c]$ playing the role played by $\bQ$ there.

%%%%%%%%%%%%%%%%%%%%%%%%%%%%%%%%%%%%%%%%%%%%%%%%%%%%%%%%%%%%%%%%%%%%%%
\subsubsection*{The maps $\chi_0$ and $\chi_1$}%\label{mf3:ssec:diffs}
%%%%%%%%%%%%%%%%%%%%%%%%%%%%%%%%%%%%%%%%%%%%%%%%%%%%%%%%%%%%%%%%%%%%%%
The maps $\chi_0$ and $\chi_1$ of Subsection~\ref{KR:ssec:diffs} are defined by the pairs of matrices:
$$
\chi_0 = \Biggl(
2\begin{pmatrix}-x_k+x_j & 0 \\ -\alpha & -1 \end{pmatrix},
2\begin{pmatrix}-x_k & x_j \\ 1 & -1 \end{pmatrix}
\Biggr)\quad \text{and}\qquad
\chi_1 = \Biggl(
\begin{pmatrix} 1 & 0 \\ -\alpha & x_k-x_j \end{pmatrix},
\begin{pmatrix} 1 & x_j \\ 1 & x_k \end{pmatrix}
\Biggr)
$$
with $\alpha$ as in Equation~\eqref{KR:eq:adiff}. A straightforward calculation 
shows that
$$\chi_0\chi_1=m\bigl(2(x_j-x_k)\bigr)\id(\hatDoubleEdge)
\qquad
\chi_1\chi_0=m\bigl(2(x_j-x_k)\bigr)\id(\hatorsmooth).
$$
Note that we are using a sign convention that is different from the one 
%in~\cite{KR} and from the one 
in Subsection~\ref{KR:ssec:diffs} and that the map $\chi_0$ contains an extra factor of 2 compared to the one in Subsection~\ref{KR:ssec:diffs}. These choices are necessary for consistency with the relations of Chapter~\ref{chap:univ3} and do not change the isomorphism class of the homology obtained.

There is another description of the maps $\chi_0$ and $\chi_1$ when the webs in Figure~\ref{KR:fig:maps-chi} are closed, which is due to Rasmussen~\cite{rasmussen-diff} in a more general setting. In this case both $\hatorsmooth$ and $\hatDoubleEdge$ have potential zero. Acting with a row operation on $\hatorsmooth$ we get
$$\hatorsmooth \cong 
\begin{Bmatrix}
\pi_{ik}, & x_i+x_j-x_k-x_l \\
\pi_{jl}-\pi_{ik}, & x_j-x_l
\end{Bmatrix}
.$$
Excluding variable $x_k$ from $\hatorsmooth$ and from $\hatDoubleEdge$ yields
$$\hatorsmooth\cong\bigl\{ \pi_{jl}-\pi_{ik},\ x_j-x_l\bigr\}_R,\qquad
\hatDoubleEdge\cong\bigl\{ v_{ijkl},\ (x_i-x_l)(x_j-x_l)\bigr\}_R,$$
with $R=\bQ[x_i,x_j,x_k,x_l]/(x_k+x_i+x_j-x_l)$. It is straightforward to check that $\chi_0$ and $\chi_1$ correspond to  the maps $(-2(x_i-x_l), -2)$ and $(1, x_i-x_l)$ respectively.
This description will be useful in Section~\ref{mf3:sec:iso}.

%%%%%%%%%%%%%%%%%%%%%%%%%%%%%%%%%%%%%%%%%%%%%%%%%%%%%%%%%%%%
\subsubsection*{Web cobordisms}\label{KR:ssec:cob-KR-abc}
%%%%%%%%%%%%%%%%%%%%%%%%%%%%%%%%%%%%%%%%%%%%%%%%%%%%%%%%%%%%
In the universal rational KR $\slt$-link homology we also have homomorphisms of matrix factorizations which we
associate to the elementary singular web-cobordisms as in Subsection~\ref{KR:ssec:cob-mf}. As in Subsection~\ref{KR:ssec:cob-mf} to the zip and the unzip we associate the maps $\chi_0$ and $\chi_1$, 
respectively.

The \emph{unit} and \emph{trace }maps
$$\imath\colon\bQ[a,b,c]\brak{1} \to \hatunknot$$
$$\varepsilon\colon\hatunknot\to \bQ[a,b,c]\brak{1}$$
are given as in Subsection~\ref{KR:ssec:cob-mf}. They are the homomorphisms of matrix factorizations induced by the 
maps (denoted by the same symbols)
$$\imath\colon\bQ[a,b,c]\to\quotient{\bQ[a,b,c][x]}{(x^3-ax^2-bx-c)}\{-2\},\quad  1\mapsto 1$$
and
$$
\varepsilon\colon\quotient{\bQ[a,b,c][x]}{(x^3-ax^2-bx-c)}\{-2\}\to\bQ[a,b,c], \quad  x^k\mapsto \begin{cases}
-\frac{1}{4}, & k=2 \\
\ \ 0 , & k<2 \end{cases}
$$
using the isomorphisms
\begin{align*}
\hat\emptyset &\cong \bQ[a,b,c]\to 0\to\bQ[a,b,c] \\
\hatunknot &\cong  0\to\quotient{\bQ[a,b,c][x]}{(x^3-ax^2-bx-c)}\{-2\}\to 0.
\end{align*}

The saddle homomorphism $\eta\colon\hattwoedgesop\to \hathtwoedgesop\brak{1}$ is given as in Subsection~\ref{KR:ssec:cob-mf} using the polynomial $p$ of Equation~\eqref{KR:eq:upoly}. In the notation of Subsection~\ref{KR:ssec:cob-mf} $\eta$ is described by the pair of matrices
$$\eta_0=
\begin{pmatrix}
 e_{123}+e_{124} & 1 \\
-e_{134}-e_{234} & 1 
\end{pmatrix},
\qquad
\eta_1=
\begin{pmatrix}
 -1 & 1 \\
-e_{123}-e_{234} & -e_{134}-e_{123} 
\end{pmatrix}
$$
where 
\begin{align*}
e_{ijk} &=
\frac{(x_k-x_j)p(x_i)+(x_i-x_k)p(x_j)+(x_j-x_i)p(x_k)}{2(x_i-x_j)(x_j-x_k)(x_k-x_i)}
\\
  &= \frac{1}{2}\bigl(x_i^2+x_j^2+x_k^2+x_ix_j+x_ix_k+x_jx_k\bigr)-\frac{2a}{3}\bigl(x_i+x_j+x_k\bigr)-b.
\end{align*}
The homomorphism $\eta$ has degree 2. Recall that in general the homomorphism $\eta$ is defined only up to a sign (see Subsection~\ref{KR:ssec:cob-mf}). The following two lemmas can be proved by explicit computation of the corresponding homomorphisms.
\begin{lem}
The composite homomorphisms
$$
\figins{-10.5}{0.4}{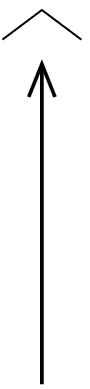}
\xra{\id\otimes\imath}
\figins{-10.5}{0.4}{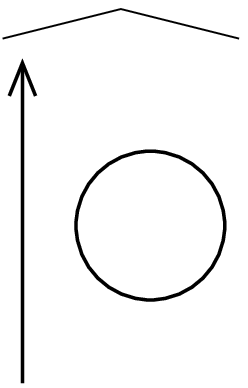}
\xra{\eta}
\figins{-10.5}{0.4}{arc-mf}
\qquad\text{ and }\qquad
\figins{-10.5}{0.4}{arc-mf}
\xra{\eta}
\figins{-10.5}{0.4}{arc-circ-mf}
\xra{\id\otimes\varepsilon}
\figins{-10.5}{0.4}{arc-mf}
$$
correspond to the identity endomorphism of the arc factorization $\hatarc$.
\end{lem}

\begin{lem}
Homomorphism $\eta^2\colon\hattwoedgesop\to \hattwoedgesop$ is homotopic to multiplication by
$$
4\bigl(
-(x_1^2+x_1x_2+x_2^2)
+a(x_1+x_2)
+b\bigr).
$$
\end{lem}

%%%%%%%%%%%%%%%%%%%%%%%%%%%%%%
\subsubsection*{MOY web moves}
%%%%%%%%%%%%%%%%%%%%%%%%%%%%%%
In this case the MOY direct sum decompositions of Subsection~\ref{KR:ssec:MOY} are still satisfied:
\begin{lem}\label{KR:lem:KR-abc-moy}
The universal rational KR theory satisfies the MOY direct sum decompositions of Lemma~\ref{KR:lem:KR-moy}.
\end{lem}

\n Recall that the decomposition~\eqref{KR:eq:moy4} involves a triple edge as in Figure~\ref{KR:fig:trp-KR}. The factorization $\hat\Upsilon$ assigned to the triple edge in Figure~\ref{KR:fig:trp-KR} is obtained the same way as in Subsection~\ref{KR:ssec:MOY} using the polynomial $p(x)$ in Equation~\eqref{KR:eq:upoly}

\medskip

For a link $L$, we denote by $\KR_{a,b,c}(L)$ the universal rational 
KR $\slt$ complex and by $\HKR_{a,b,c}(L)$ the universal rational 
KR $\slt$ homo\-logy. The complex $\KR_{a,b,c}(L)$ is built exactly as in Section~\ref{KR:sec:KR}, over the ring $\bQ[a,b,c]$ instead of $\bQ$.
We have
$$\HKR_{a,b,c}(\unknot)\cong
\bigl(\quotient{\bQ[x,a,b,c]}{(x^3-ax^2-bx-c)}\bigr)
\brak{1}\{-2\}.$$

\n The universal rational KR $\slt$-link homology is still functorial up to $\bQ$-scalars with respect to link cobordisms.

%%%%%%%%%%%%%%%%%%%%%%%%%%%%%%%%%%%%%%%%%%%%%%%%%%%%%%%%%%%%%%%%%%%%%%%%

%%%%%%%%%%%%%%%%%%%%%
%                   %
%      chapter      %
%                   %
%%%%%%%%%%%%%%%%%%%%%
\chapter{The universal $\slt$-link homology}\label{chap:univ3}
%
%
%%%%%%%%%%%%%%%%%%%%%%%%%%%%%%%%%%%%%%%%
%%%                                  %%%
%%%        Introduction              %%%
%%%                                  %%%
%%%%%%%%%%%%%%%%%%%%%%%%%%%%%%%%%%%%%%%%
%
%
In this chapter we construct the universal 
$\slt$-link homology over $\bZ[a,b,c]$. For this universal construction 
we rely heavily on Bar-Natan's~\cite{bar-natancob} work on the universal 
$\mathfrak{sl}(2)$-link homology and Khovanov's~\cite{khovanovsl3} work on his original 
$\slt$-link homology. Our construction is given in Section~\ref{u3:sec:univhom} where we impose a finite set of relations 
on the category of foams, analogous to 
Khovanov's~\cite{khovanovsl3} relations for his $\slt$-link 
homology. These relations enable us to 
cons\-truct a link homology complex which is 
homotopy invariant under the Reidemeister moves and functorial, up to a sign, 
with respect to link cobordisms.
To compute the homology corresponding to our complex 
we use the tautological homology construction like 
Khovanov did in~\cite{khovanovsl3}. We denote this 
universal $\slt$-homology by $U_{a,b,c}(L)$, 
which by the previous results is an invariant of the link $L$. In Section~\ref{u3:sec:isos} we state the classification theorem for the $\slt$-link homologies, as proved by the author in joint work with Marco Mackaay in~\cite{mackaay-vaz}. There are three isomorphism classes in $U_{a,b,c,}(L)$, depending on the number of distinct roots of the polynomial $f(X)=X^3-aX^2-bX-c$. One of them is isomorphic 
to Khovanov's original $\slt$-link homology. The second is isomorphic to Gornik's $\slt$-link homology~\cite{gornik}, and the third one is new and is related to the $\mathfrak{sl}(2)$-Khovanov homology of sublinks $L'\subseteq L$.

%%%%%%%%%%%%%%%%%%%%%%%%%%%%%%%%%%%%%%%%
%%%                                  %%%
%%%    Universal link homology       %%%
%%%                                  %%%
%%%%%%%%%%%%%%%%%%%%%%%%%%%%%%%%%%%%%%%%

\section{The universal $\slt$-link homology}\label{u3:sec:univhom}

Let $L$ be an oriented link in $S^3$ and $D$ a diagram of $L$. In~\cite{khovanovsl3} Khovanov constructed a homological link invariant associated to $\slt$. His construction is based on Kuperberg's graphical calculus for $\slt$~\cite{kupG2,kup}. The construction starts by resolving each crossing of $D$ in two different ways, as in Figure~\ref{u3:fig:resolutions}.
\begin{figure}[ht!]
\medskip
\labellist
\small\hair 2pt
\pinlabel $1$ at 215 95
\pinlabel $0$ at 215 330
\pinlabel $1$ at 600 330
\pinlabel $0$ at 600 95
\endlabellist
\centering
\figs{0.22}{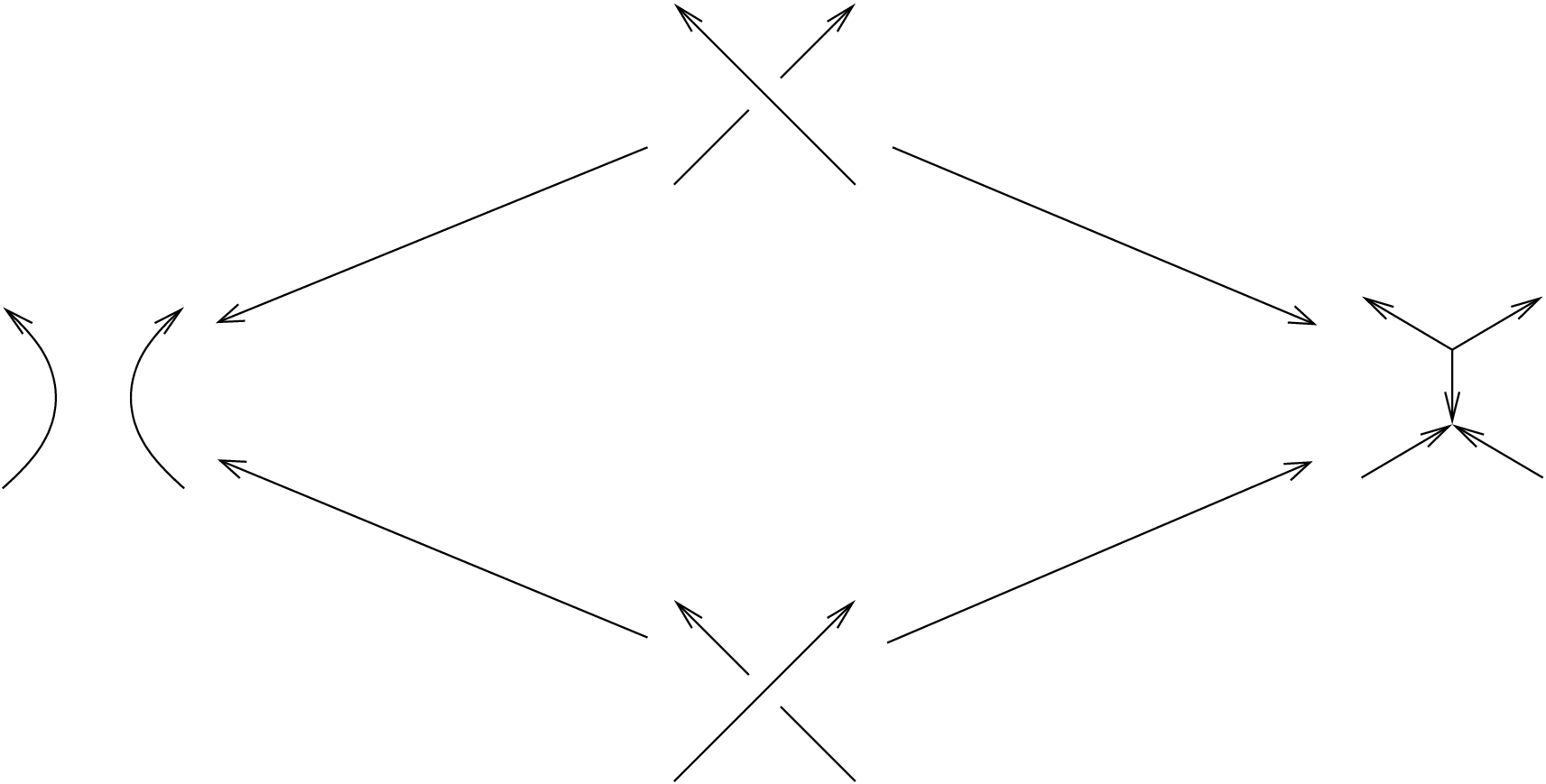}
\caption{0 and 1 resolutions of crossings}
\label{u3:fig:resolutions}
\end{figure}
\n A diagram $\Gamma$ obtained by resolving all crossings of $D$ is an example of an \emph{$\slt$-web}. 
An $\slt$-web is a trivalent planar graph where near each vertex all the edges are simultaneously oriented ``in'' or ``out'' 
(see Figure~\ref{u3:fig:in-out}). We also allow $\slt$-webs without vertices, which are 
oriented loops. 
Note that by definition our $\slt$-webs are closed; there are no vertices with fewer than 3 edges.  

\begin{figure}[ht!]
\centering
\figins{0}{0.4}{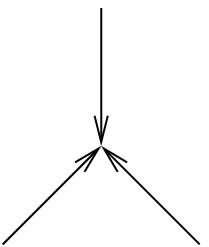}\qquad\quad
\figins{0}{0.4}{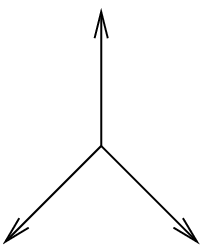}
\caption{``In'' and ``out'' orientations near a vertex}
\label{u3:fig:in-out}
\end{figure}

The \emph{Kuperberg moves} in Figure~\ref{u3:fig:kuper} provide a recursive way of assigning a unique polynomial $P_3$ to every $\slt$-web.
\begin{figure}[ht!]
\begin{gather*}
P_3(\unknot\Gamma) = P_3(\unknot) P_3(\Gamma) 
%\tag{\emph{Circle Removal}}
\\[1.2ex]
P_3(\figins{-2.6}{0.15}{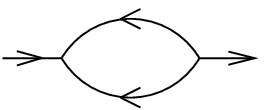}) 
=
q^{-1}P_3(\figwhins{2}{0.025}{0.3}{hthickedge})
+ qP_3(\figwhins{2}{0.025}{0.3}{hthickedge})
%\tag{\emph{Digon Removal}}
%
\\[1.2ex]
P_3\left(\figins{-8.0}{0.3}{squareweb}\right)
= 
P_3\left(\figins{-8.0}{0.3}{vedgesweb}\right)
+
P_3\left(\figins{-8.0}{0.3}{hedgesweb}\right) 
%\tag{\emph{Square Removal}}
\end{gather*}
\caption{Kuperberg moves}
\label{u3:fig:kuper}
\end{figure}
Consistency and sufficiency of the rules in Figure~\ref{u3:fig:kuper} is shown in~\cite{kupG2}. Kuperberg's calculus is previous to the appearance of MOY's calculus~\cite{MOY}. Taking $N=3$, erasing the triple edges and replacing every double edge with a simple edge oriented such that every vertex satisfies one of the possibilities of Figure~\ref{u3:fig:in-out} we obtain Kuperberg's calculus from the MOY calculus of Section~\ref{KR:sec:slN}.

A \emph{foam} is a cobordism with singular arcs between two $\slt$-webs. A singular arc in a foam $f$ is the set of points of $f$ that have a neighborhood homeomorphic to the letter Y times an interval (see the examples in Figure~\ref{u3:fig:ssaddle}). Note that all singular arcs are disjoint and that a closed singular arc is a circle. 
\begin{figure}[ht!]
\medskip
\centering
\figins{0}{0.5}{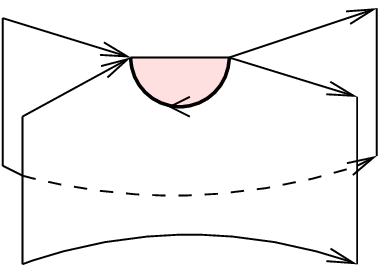}\qquad\qquad
\figins{0}{0.5}{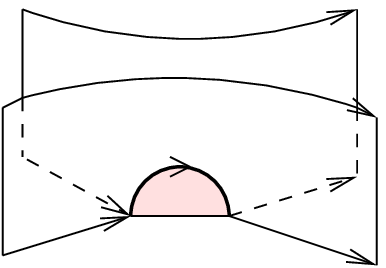}
\caption[Basic foams from $\orsmoothing$ to $\unorsmoothing$ and from $\unorsmoothing$ to $\orsmoothing$]{Basic foams from $\orsmoothing$ to $\unorsmoothing$ (left) and from $\unorsmoothing$ to $\orsmoothing$ (right)}
\label{u3:fig:ssaddle}
\end{figure}
Interpreted as morphisms, we read foams from bottom to top by convention, and the orientation of the singular arcs is by convention as in Figure~\ref{u3:fig:ssaddle}. We call the foams in Figure~\ref{u3:fig:ssaddle} the \emph{zip} and the \emph{unzip}. Foams can have dots that can move freely on the facet to which they belong but are not allowed to cross singular arcs. Let $\bZ[a,b,c]$ be the ring of 
polynomials in $a,b,c$ with integer coefficients.

\begin{defn}
$\Foam$ is the category whose objects are formal direct sums of (closed) $\slt$-webs and whose morphisms are direct sums of $\bZ[a,b,c]$-linear 
combinations of isotopy classes of foams.
\end{defn}

\n $\Foam$ is an additive category and, as we will show, each foam can be given 
a degree in such a way that $\Foam$ becomes a graded additive category, taking also 
$a,b$ and $c$ to have degrees $2,4$ and $6$ respectively. 
For further details about this category see~\cite{khovanovsl3}.

From all different flattenings of all the crossings in $D$ we form a hypercube of flattenings. It has an $\slt$-web in each vertex and 
to an edge between two vertices, given by $\slt$-webs that differ only inside 
a disk $\cD$ around one of the crossings of $D$, we associate the foam that is 
the identity everywhere except inside the cylinder $\cD\times I$, where it 
has a zip or an unzip, with a sign given by the rule in Subsection~\ref{KR:ssec:cmplx}. 
\begin{figure}[ht!]
\quad\qquad\figwins{0}{4.5}{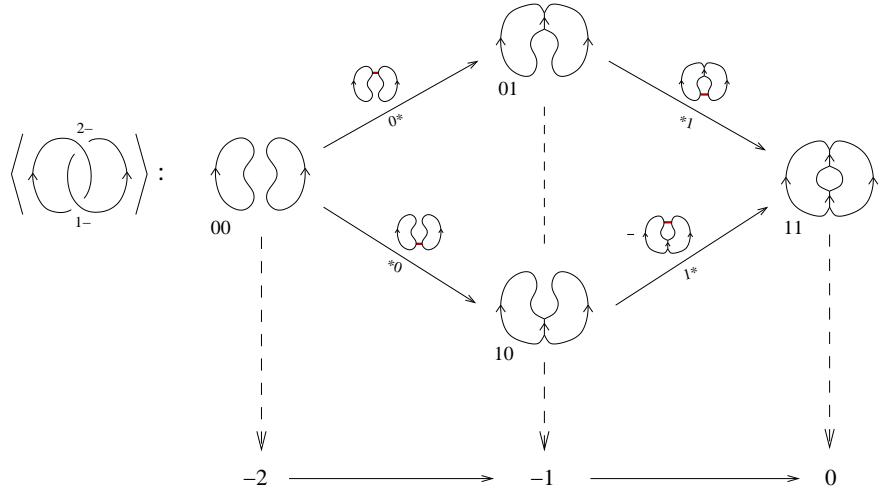}
\caption[The complex of the Hopf link]{The complex of the Hopf link. The homological grading is indicated in the bottom}
\label{u3:fig:hopflink}
\end{figure}
This way we obtain a chain 
complex of $\slt$-web diagrams analogous to the one in~\cite{bar-natancob} which we call 
$\brak{D}$, with ``column vectors'' of $\slt$-webs as ``chain objects'' and 
``matrices of foams'' as ``differentials'' (see Figure~\ref{u3:fig:hopflink} for an example, where we use the short-hand notation $\figins{-2}{0.14}{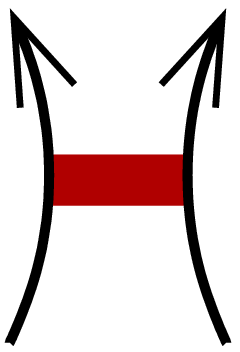}$ for the zip).  We borrow some of the notation 
from~\cite{bar-natancob} and denote by $\kom(\Foam)$ the category of 
complexes in $\Foam$.

In Subsections~\ref{u3:ssec:loc_rel}-\ref{u3:ssec:funct} we first impose a set of local relations on 
$\Foam$. We call this set of relations $\ell$ and denote by 
$\lfoam$ the category $\Foam$ divided by $\ell$. 
We prove that these relations guarantee the invariance of 
$\brak{D}$ under the Reidemeister moves up to homotopy in $\kom(\lfoam)$ 
in a pictorial way, which is analogous to Bar-Natan's proof in~\cite{bar-natancob}. 
Note that the category $\kom(\lfoam)$ is analogous to Bar-Natan's category 
$Kob(\emptyset)=Kom(Mat(Cob^3_{/l}(\emptyset)))$. 
Next we show that up to signs 
$\brak{\;}$ is functorial under 
oriented link cobordisms, i.e. defines a functor from $\Link$ to 
$\kom_{/\pm h}(\lfoam)$.
Here $\kom_{/\pm h}(\lfoam)$ is the 
homotopy category of $\kom(\lfoam)$ modded out by $\pm 1$. In Subsection~\ref{u3:ssec:univhom} we define a tautological functor between 
$\lfoam$ and $\abcModgr$, the category of graded 
$\bZ[a,b,c]$-modules. This induces a homology functor 
$$U_{a,b,c}\colon \Link\to \abcModbg,$$
where $\abcModbg$ is the category of bigraded $\bZ[a,b,c]$-modules.

The principal ideas in this section, as well as most homotopies, are motivated by 
the ones in Khovanov's paper~\cite{khovanovsl3} and 
Bar-Natan's paper~\cite{bar-natancob}.

%%%%%%%%%%%%%%%%%%%%%%%%%%%%%
% Universal local relations %
%%%%%%%%%%%%%%%%%%%%%%%%%%%%%
\subsection{Universal local relations}\label{u3:ssec:loc_rel}
In order to construct the universal theory we divide the category $\Foam$ by the local relations  $\ell=(3D, CN, S, \Theta)$ below.

\begin{gather}
%3Dot reduction
\figins{-7}{0.25}{plan3dot}=
a\figins{-7}{0.25}{plan2dot}+
b\figins{-7}{0.25}{plan1dot}+
c\figins{-7}{0.25}{plan0dot}
\tag{3D}\label{u3:eq:3D}
\\[1.5ex] 
%cutting neck
-\figwhins{-18}{0.55}{0.25}{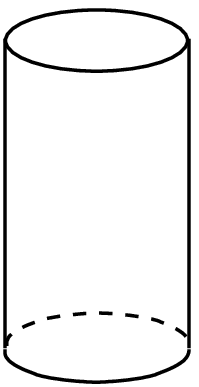}=
\figwhins{-18}{0.55}{0.25}{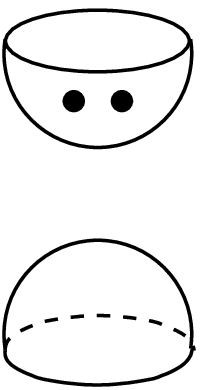}+
\figwhins{-18}{0.55}{0.25}{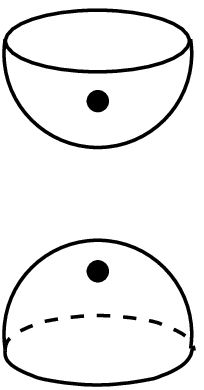}+
\figwhins{-18}{0.55}{0.25}{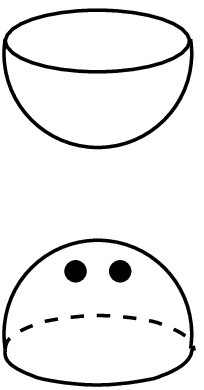}
-a
\left( 
\figwhins{-18}{0.55}{0.25}{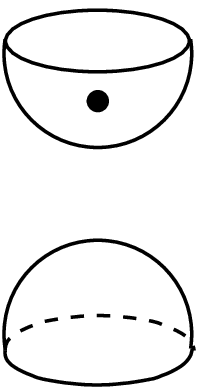}+
\figwhins{-18}{0.55}{0.25}{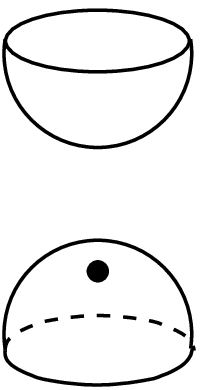}
\right)
-b
\figwhins{-18}{0.55}{0.25}{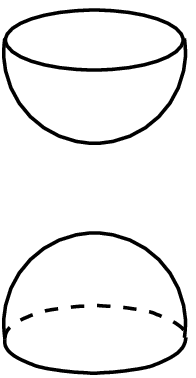}
\tag{CN}\label{u3:eq:CN}
\\[1.5ex]
%S-relation
\figins{-9}{0.35}{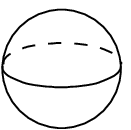}=
\figins{-9}{0.35}{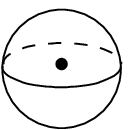}=0,\quad
\figins{-9}{0.35}{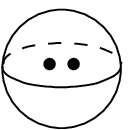}=-1
\tag{S}\label{u3:eq:S}
\end{gather}

\medskip
\n Note that the foams in the $S$-relations are spheres and not theta-foams, which we discuss next. 

Recall from~\cite{khovanovsl3} that theta-foams are obtained by gluing three oriented disks along their boundaries (their orientations must coincide), as shown in Figure~\ref{u3:fig:theta}. 
\begin{figure}[ht!]
\medskip
\labellist
\small\hair 2pt
\pinlabel $\alpha$ at 1 33
\pinlabel $\beta$ at -5 17
\pinlabel $\delta$ at 1 5
\endlabellist
\centering
\figs{1.0}{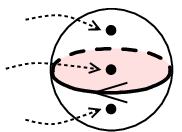}
\caption{A theta-foam}
\label{u3:fig:theta}
\end{figure}
Note the orientation of the singular circle. Let $\alpha$, $\beta$, $\gamma$ denote the number of dots on each facet.
The ($\Theta$) \emph{relation} says that for $\alpha$, $\beta$ or $\gamma\leq 2$
\begin{equation}
\theta(\alpha,\beta,\delta)=
\begin{cases}
1 & (\alpha,\beta,\delta)=(1,2,0)\text{ or a cyclic permutation} \\ 
-1 & (\alpha,\beta,\delta)=(2,1,0)\text{ or a cyclic permutation} \\ 
0 & \text{else}
\end{cases}
\tag{$\Theta$}\label{u3:eq:theta}
\end{equation}
Reversing the orientation of the singular circle reverses the sign of $\theta(\alpha,\beta,\gamma)$. 
Note that when we have three or more dots on a facet of a foam we can use the~\eqref{u3:eq:3D} relation to reduce to the case where it has less than three dots. 

\medskip

The relations $\ell$ above are related to the integral $U(3)$-equivariant cohomology rings of the partial flag varieties of $\bC^3$, which are homogeneous spaces of $\slt$. The~\eqref{u3:eq:3D} relation can be interpreted in the $U(3)$-equivariant cohomology of $\cp{2}$
$$
\hy_{U(3)}(\cp{2},\bZ)=\quotient{\bZ[a,b,c,X]}{(X^3-aX^2-bX-c)}
$$
where a dot corresponds to the generator $X$, which has degree 2 (note that $\bZ[a,b,c]$ is the integral $U(3)$-equivariant cohomology ring of a point).
The~\eqref{u3:eq:S} relation is related to the choice of a trace on $\hy_{U(3)}(\cp{2})$
$$
\varepsilon(X^i)=
\begin{cases}
0 & \text{if }\ i<2 \\
-1 & \text{if }\ i=2
\end{cases}
$$
(for $i>2$ use~\eqref{u3:eq:3D})
This endows $\hy_{U(3)}(\cp{2},\bZ)$ with the structure of a commutative Frobenius algebra.
Denote by $\widehat{X^j}$ the dual of $X^j$ in $\hy_{U(3)}(\cp{2},\bZ)$, \ie $\varepsilon(\hat{X^j}X^i)=\delta_{ij}$. Then for any $z\in\hy_{U(3)}(\cp{2},\bZ)$ we have
$$
z=\sum_{i=0}^{2}\widehat{X^i}\varepsilon(zX^i)
$$
and this gives an interpretation of the~\eqref{u3:eq:CN} relation.

The~\eqref{u3:eq:theta} relation is interpreted in the integral $U(3)$-equivariant cohomology ring of the full flag variety $Fl_3$ of $\bC^3$,
$$
Fl_3=\{(V_1,V_2)|\ V_1\subset V_2\subset \bC^3,\ \dim V_i=i\},
$$
which is
$$
\hy_{U(3)}(FL_3,\bZ)=\quotient{\bZ[a,b,c][X_1,X_2,X_3]}{I}
$$
where each $X_i$ has degree 2 and $I$ is the ideal generated by
$$
X_1+X_2+X_3-a ,\quad X_1X_2+X_1X_3+X_2X_3+b ,\quad X_1X_2X_3-c.
$$
The (non-degenerate) trace on $\hy_{U(3)}(FL_3\,bZ)$ defined by
$$
\varepsilon(X_1X_2^2)=1
$$
satisfies
$$
\varepsilon(X_1^\alpha X_2^\beta X_3^\delta)=\theta(\alpha,\beta,\delta),
$$
where the dots on the facets of the theta foam correspond to the generators $X_1$, $X_2$, $X_3$.

\medskip

A closed foam $f$ can be viewed as a morphism from the empty $\slt$-web to itself which by the relations  (3D, CN, S, $\Theta$) is an element of $\bZ[a,b,c]$. It can be checked, as in~\cite{khovanovsl3}, that this set of relations is consistent and determines uniquely the evaluation of every closed foam 
$f$, denoted $C(f)$. Define a $q$-grading on $\bZ[a,b,c]$ by $q(1)=0$, $q(a)=2$, $q(b)=4$ and $q(c)=6$. As in~\cite{khovanovsl3} we define the \emph{q-grading} of a foam $f$ with $d$ dots by $$q(f)=-2\chi(f)+ \chi(\partial f)+2d,$$ where $\chi$ denotes the Euler characteristic.

\begin{defn}
$\lfoam$ is the quotient of the category $\Foam$ by the local relations $\ell$. 
For $\slt$-webs $\Gamma$, $\Gamma'$ and for families $f_i\in\Hom_{\Foam}(\Gamma,\Gamma')$ and 
$c_i\in\bZ[a,b,c]$ we impose $\sum_ic_if_i=0$ if and only if 
$\sum_ic_iC(g'f_ig)=0$ holds, for all 
$g\in\Hom_{\Foam}(\emptyset,\Gamma)$ and 
$g'\in\Hom_{\Foam}(\Gamma',\emptyset)$.
\end{defn}

\begin{lem}\label{u3:lem:identities}
We have the following relations in $\lfoam$:

\begin{gather}
%4-cyls}
\figwins{-17}{0.65}{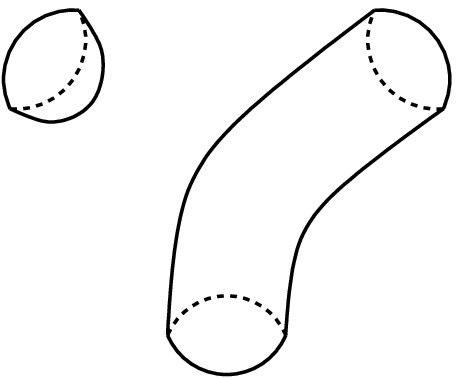}+
\figwins{-17}{0.65}{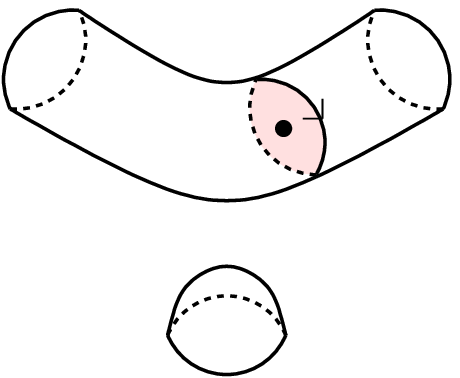}=
\figwins{-17}{0.65}{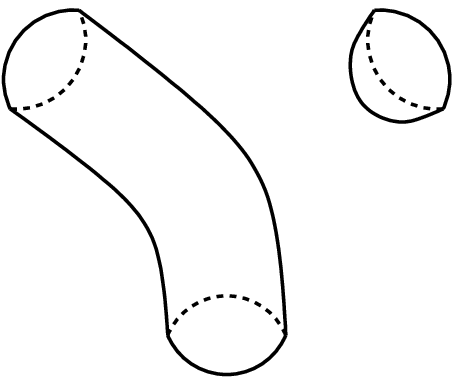}+
\figwins{-17}{0.65}{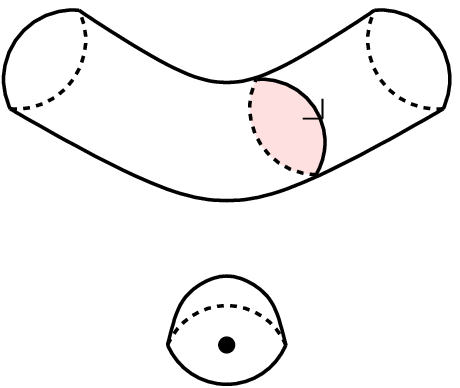}
\tag{4C}\label{u3:eq:4C}
\\[1.8ex]
%Removing disk
\figwhins{-17}{0.55}{0.25}{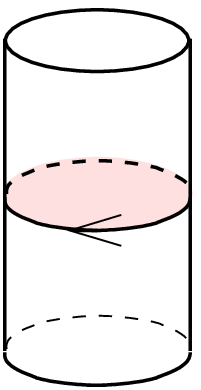}=
\figwhins{-17}{0.55}{0.25}{cnecka1}-
\figwhins{-17}{0.55}{0.25}{cnecka2}
\tag{RD}\label{u3:eq:RD}
\\[1.8ex]
%digon removal
\figins{-20}{0.6}{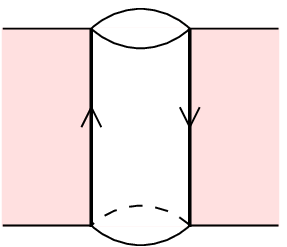}=
\figins{-26}{0.75}{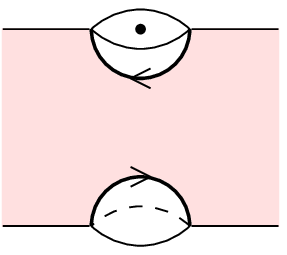}-
\figins{-26}{0.75}{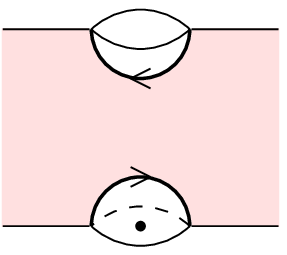}
\tag{DR} \label{u3:eq:DR}
\\[1.8ex]
%Square removal
\figins{-28}{0.8}{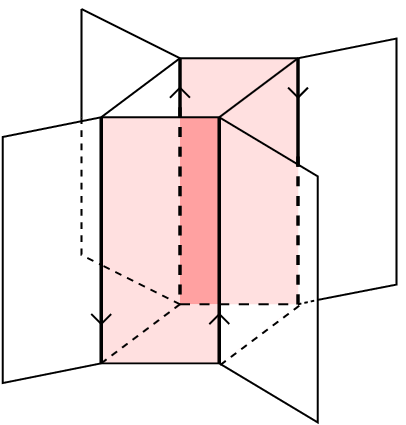}=
-\ \figins{-28}{0.8}{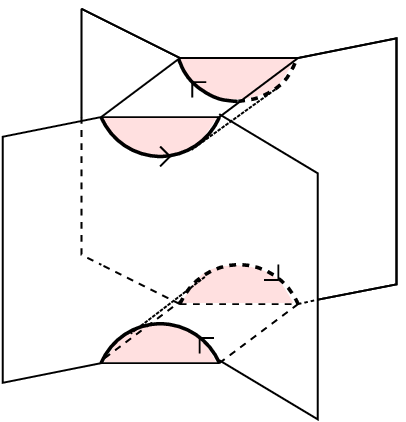}
-\figins{-28}{0.8}{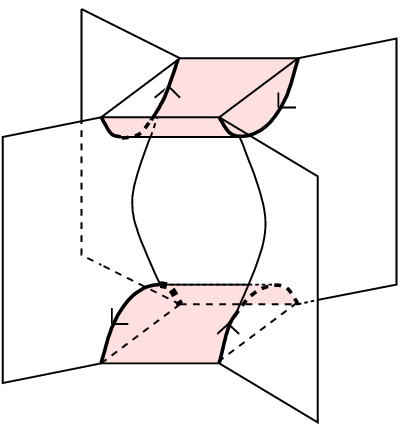}
\tag{SqR}\label{u3:eq:SqR}
\end{gather}
\end{lem}

\medskip

\begin{proof}
Relations~\eqref{u3:eq:4C} and~\eqref{u3:eq:RD} are immediate and follow from~\eqref{u3:eq:CN} 
and~\eqref{u3:eq:theta}. Relations~\eqref{u3:eq:DR} and~\eqref{u3:eq:SqR} are proved as 
in~\cite{khovanovsl3} (see also Lemma~\ref{u3:lem:KhK})
\end{proof}

In Figure~\ref{u3:fig:pdots} we also have a set of useful identities which establish 
the way we can exchange dots between facets. These identities can be used for the simplification of foams and are an immediate consequence of the relations in $\ell$. Note that these are the relations defining $\hy_{U(3)}(FL_3,\bZ)$, if we label the facets 1, 2, 3 around the singular arc and if a dot on facet $i$ corresponds to $X_i$.
\begin{figure}[ht!]
\begin{align*}
\raisebox{-0.32in}{\includegraphics[height=0.6in]{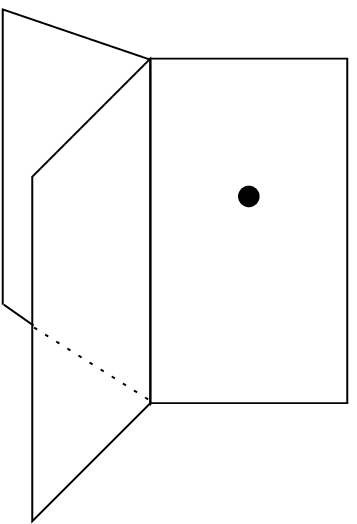}}
\,+\,
\raisebox{-0.32in}{\includegraphics[height=0.6in]{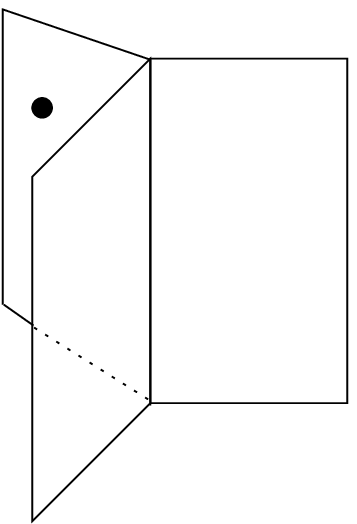}}
\,+\,
\raisebox{-0.32in}{\includegraphics[height=0.6in]{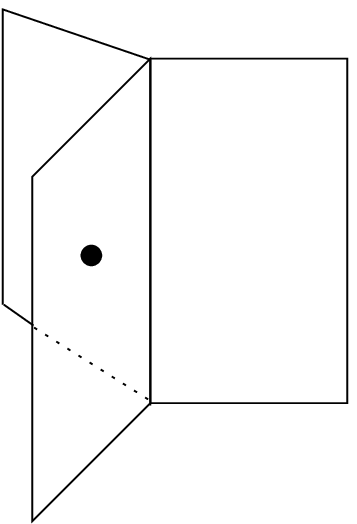}}
\, 
&=
\hspace{1.6ex} a\
\raisebox{-0.32in}{\includegraphics[height=0.6in]{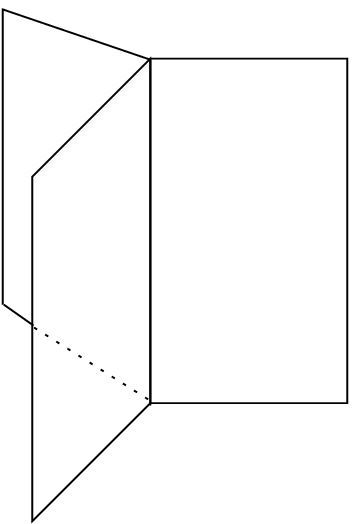}} 
\\[1.5ex]
\raisebox{-0.32in}{\includegraphics[height=0.6in]{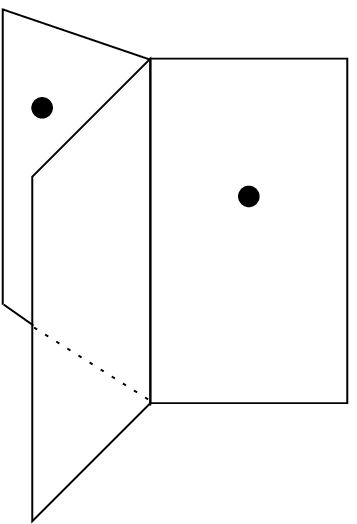}}
\,+\,
\raisebox{-0.32in}{\includegraphics[height=0.6in]{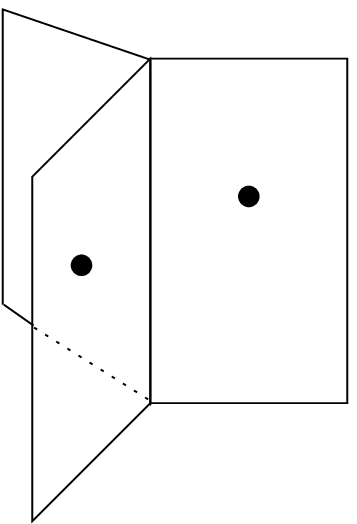}}
\,+\,
\raisebox{-0.32in}{\includegraphics[height=0.6in]{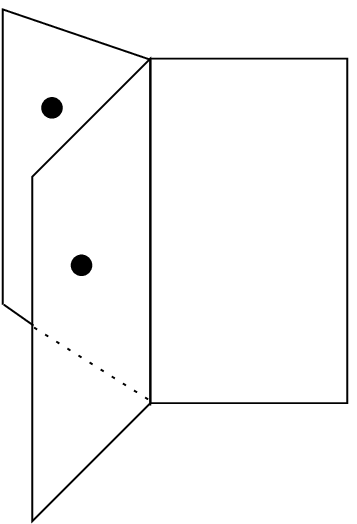}}
\,
&=
-b\
\raisebox{-0.32in}{\includegraphics[height=0.6in]{figs/pdots000}} 
\\[1.5ex]
\raisebox{-0.32in}{\includegraphics[height=0.6in]{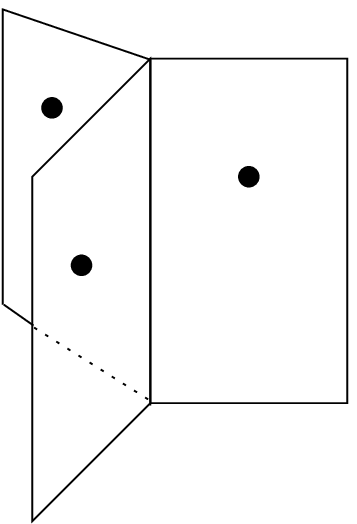}}
\,
&=
\hspace{1.6ex} c\
\raisebox{-0.32in}{\includegraphics[height=0.6in]{figs/pdots000}}
\end{align*}
\caption[Exchanging dots between facets]{Exchanging dots between facets. The relations are the same regardless of the orientation on the singular arcs.}
\label{u3:fig:pdots}
\end{figure}

%%%%%%%%%%%%%%%%%%%%%%%%%%%%%%%%%%%%%%%%
%%%                                  %%%
%%%            Invariance            %%%
%%%                                  %%%
%%%%%%%%%%%%%%%%%%%%%%%%%%%%%%%%%%%%%%%%
\subsection{Invariance under the Reidemeister moves}\label{u3:ssec:thm-inv}

In this subsection we prove invariance of $\brak{\;}$ under the Reidemeister moves. The main result of this section is the following:

\begin{thm}
$\brak{D}$ is invariant under the Reidemeister moves up to homotopy, in other words it is an invariant in $\kom_{/h}(\lfoam)$.
\label{u3:thm-inv}
\end{thm}

\begin{proof}
%%%%%%%%%%%%%%%%%%%%%%
%                    %
%  Reidemeister I    %
%%%%%%%%%%%%%%%%%%%%%%
To prove invariance under the Reidemeister moves we work diagrammatically.

\n\emph{Reidemeister I.} 
Consider diagrams $D$ and $D'$ that differ only in a circular region as in the 
figure below.
$$D=\figins{-13}{0.4}{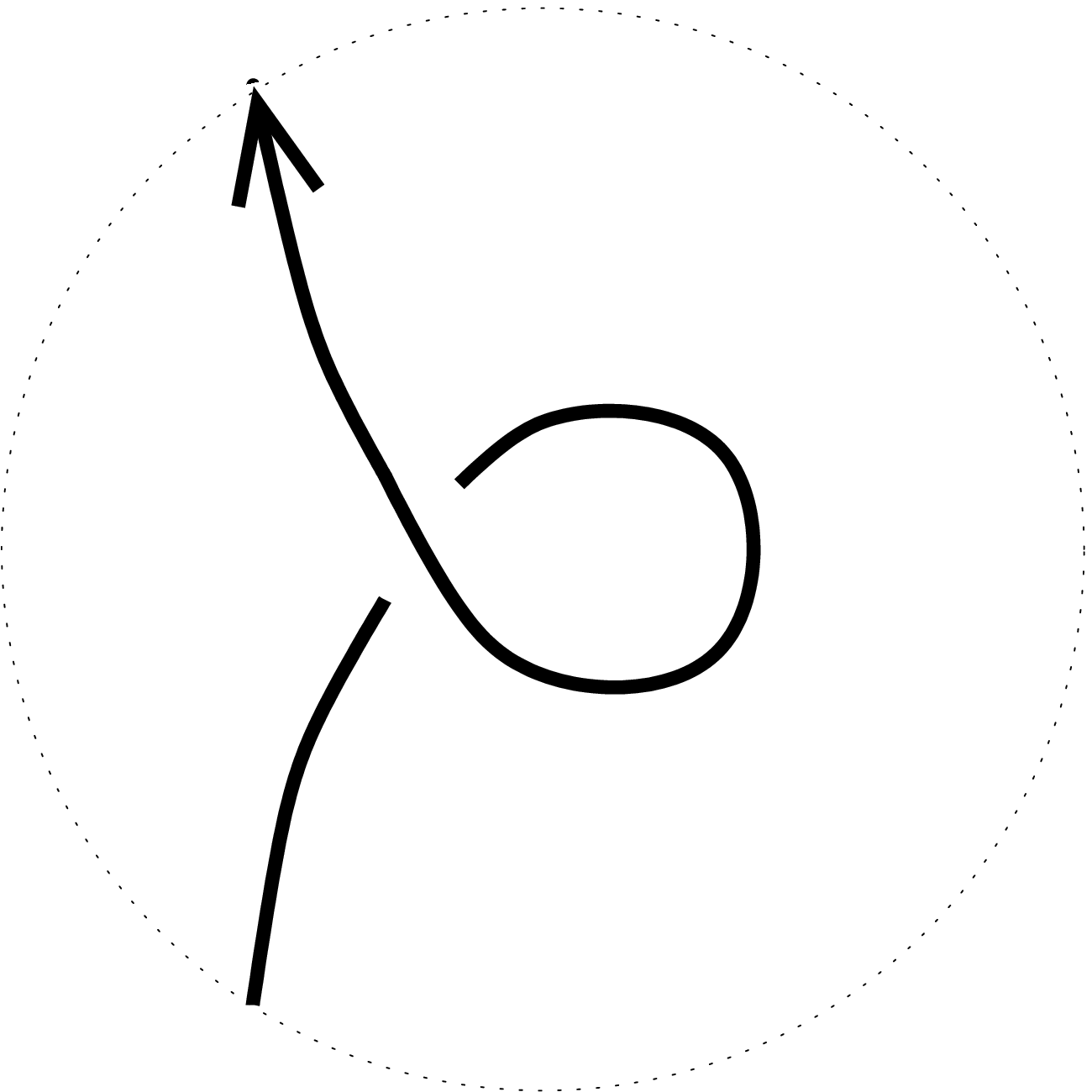}\qquad
D'=\figins{-13}{0.4}{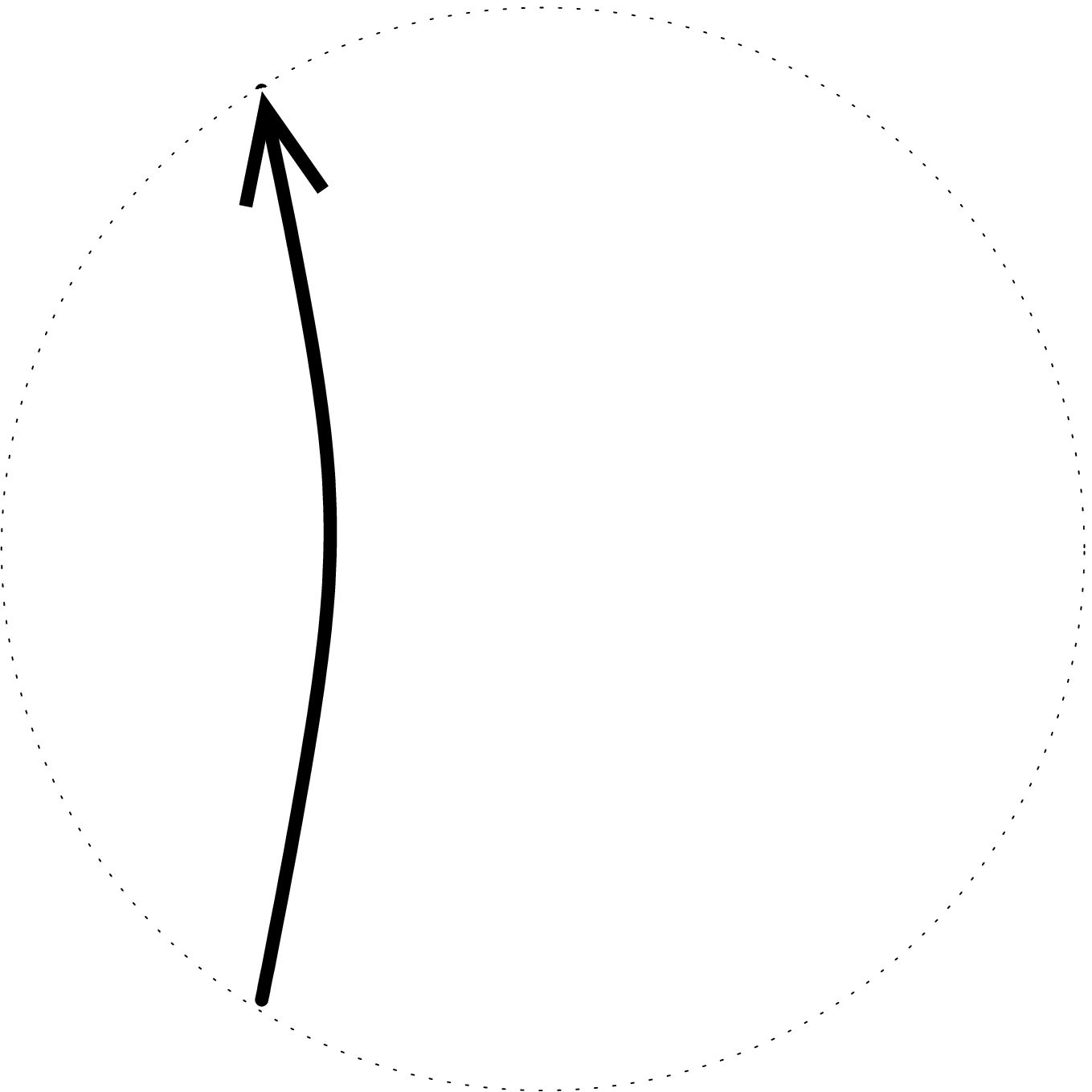}\,$$

\begin{figure}[ht!]
\labellist
\small\hair 2pt
\pinlabel $f^0=-\sum\limits_{i=0}^2$ at -125 305
\pinlabel $_{_{2-i}}$ at 70 225
\pinlabel $_{_{i}}$ at 80 280
\pinlabel $+\ a\sum\limits_{i=0}^1$ at 180 305
\pinlabel $_{_{1-i}}$ at 340 225
\pinlabel $_{_{i}}$ at 350 280
\pinlabel $+\ b$ at 410 305
\pinlabel $g^0=$ at 690 305
\pinlabel $h=-$ at 955 375
\pinlabel $+$ at 1175 370
\pinlabel $d=$ at 960 655
\pinlabel $0$ at 1090 100
\pinlabel $0$ at 1520 305
\endlabellist
\centering
\qquad\figs{0.23}{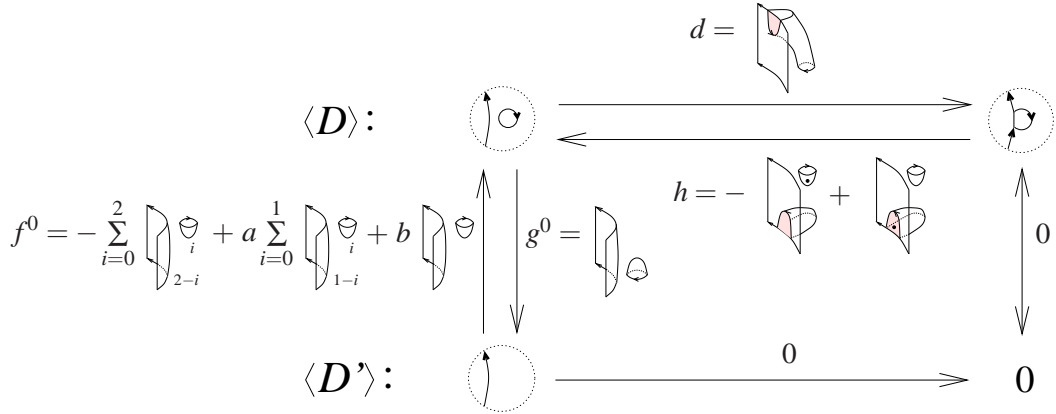}
\caption{Invariance under \emph{Reidemeister I}} \label{u3:fig:RI_Invariance}
\end{figure} 

\n We give the homotopy between complexes $\brak{D}$ and $\brak{D'}$ in Figure~\ref{u3:fig:RI_Invariance}~\footnote{We thank Christian 
Blanchet for spotting a mistake in a previous version of this diagram.}.
By relation~\eqref{u3:eq:S} we have $g^0f^0=Id(\edge)$.
To see that $df^0=0$ holds, one can use the dot exchange relations in 
Figure~\ref{u3:fig:pdots}.
The equality $dh=id(\digon)$ follows from~\eqref{u3:eq:DR} 
(note the orientations on the singular circles). To show that 
$f^0g^0+hd=Id_{\brak{D}^0}$, apply~\eqref{u3:eq:RD} to $hd$ and then cancel all 
terms which appear 
twice with opposite signs. What is left is the sum of $6$ terms which is 
equal to $Id_{\brak{D}^0}$ by~\eqref{u3:eq:CN}. 
Therefore $\brak{D'}$  is homotopy-equivalent to $\brak{D}$.

\bigskip
%%%%%%%%%%%%%%%%%%%%%%
%                    %
%  Reidemeister IIa  %
%%%%%%%%%%%%%%%%%%%%%%
\n\emph{Reidemeister IIa.}
Consider diagrams $D$ and $D'$ that differ in a circular region, as
in the figure below.
$$D=\figins{-13}{0.4}{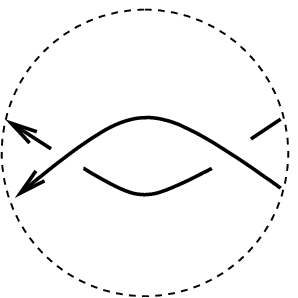}\qquad
D'=\figins{-13}{0.4}{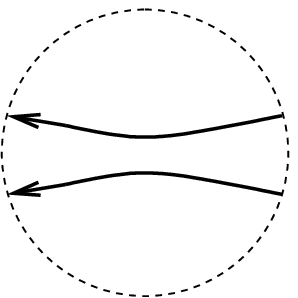}$$

\n We leave to the reader the task of checking that the diagram in 
Figure~\ref{u3:fig:RIIaInvariance} defines a homotopy between the 
complexes $\brak{D}$ and $\brak{D'}$: 
\begin{itemize}
\item $g$ and $f$ are morphisms of complexes (use only isotopies); 
\item $g^1f^1=Id_{\brak{D'}^1}$  (use~\eqref{u3:eq:CN} and~\ref{u3:eq:theta}); 
\item $f^0g^0+hd=Id_{\brak{D}^0}$ and $f^2g^2+dh=Id_{\brak{D}^2}$  (use isotopies);
\item $f^1g^1+dh+hd=Id_{\brak{D}^1}$  (use~\eqref{u3:eq:DR}).
\end{itemize}
\begin{figure}[h!]
\labellist
\small\hair 2pt
\pinlabel $g$ at 0 220
\pinlabel $f$ at 52 220
\pinlabel $0$ at 125 210
\pinlabel $0$ at 783 210
\endlabellist
\centering
\includegraphics[width=5.5in]{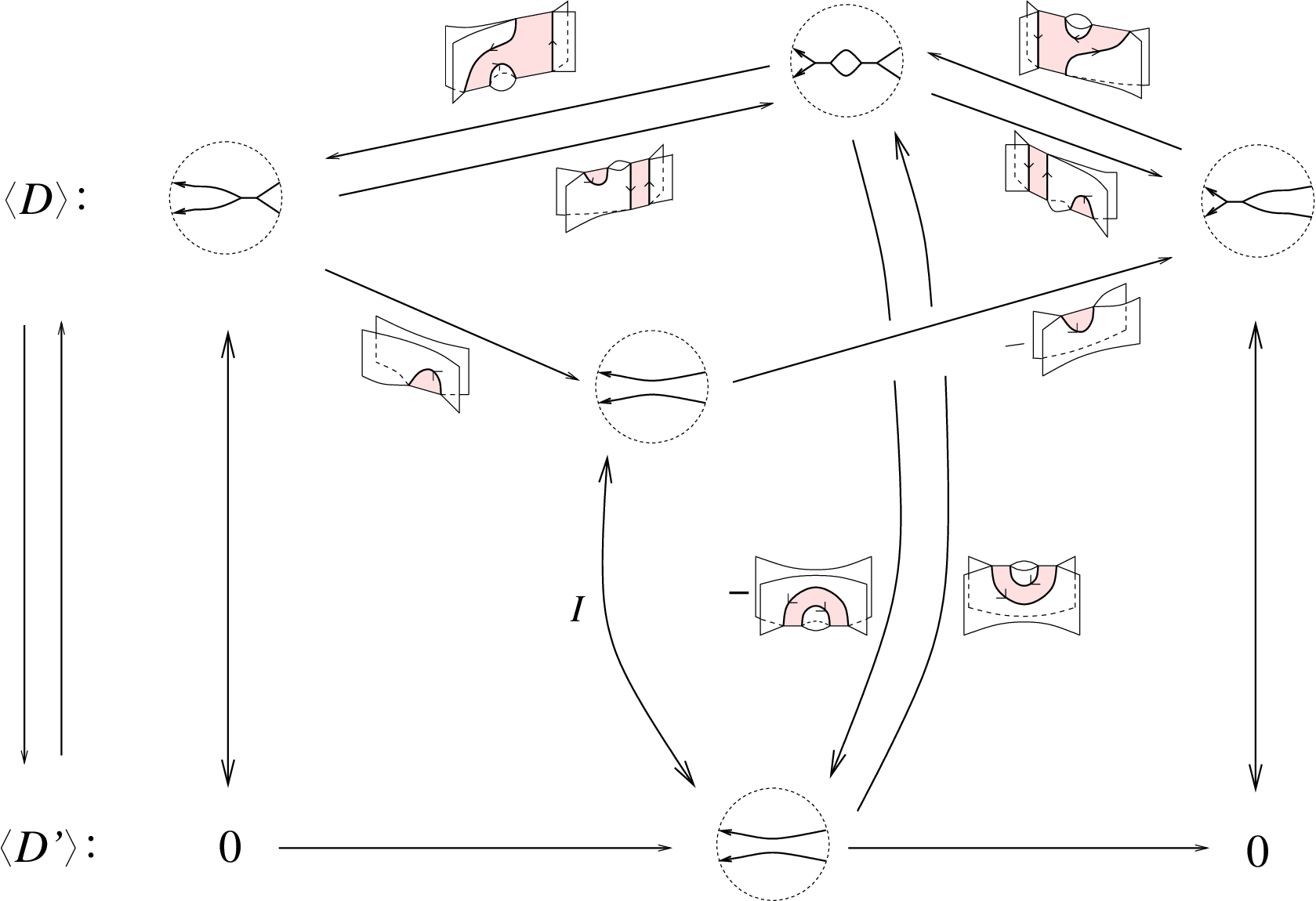}
\caption{Invariance under \emph{Reidemeister IIa} } \label{u3:fig:RIIaInvariance}
\end{figure}

\bigskip
%%%%%%%%%%%%%%%%%%%%%%
%                    %
%  Reidemeister IIb  %
%%%%%%%%%%%%%%%%%%%%%%
\n\emph{Reidemeister IIb.}
Consider diagrams $D$ and $D'$ that differ only in a circular region, as 
in the figure below.
$$D=\figins{-13}{0.4}{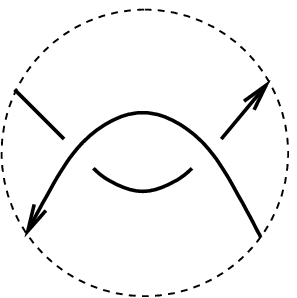}\qquad
D'=\figins{-13}{0.4}{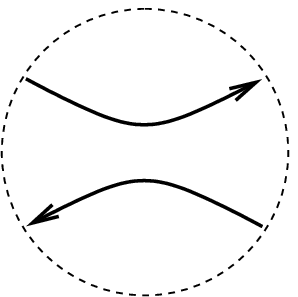}$$
Again, checking that the diagram in Figure~\ref{u3:fig:RIIbInvariance} defines a 
homotopy between the complexes $\brak{D}$ and $\brak{D'}$ is left to the reader:
\begin{itemize}
\item $g$ and $f$ are morphisms of complexes (use only isotopies); 
\item $g^1f^1=Id_{\brak{D'}^1}$  (use~\eqref{u3:eq:RD} and~\eqref{u3:eq:S}); 
\item $f^0g^0+hd=Id_{\brak{D}^0}$ and $f^2g^2+dh=Id_{\brak{D}^2}$ (use~\eqref{u3:eq:RD} and~\eqref{u3:eq:DR});
\item $f^1g^1+dh+hd=Id_{\brak{D}^1}$ (use~\eqref{u3:eq:DR},~\eqref{u3:eq:RD},~\eqref{u3:eq:4C} and~\eqref{u3:eq:SqR}).
\end{itemize}
\begin{figure}[ht!]
\labellist
\small\hair 2pt
\pinlabel $g$ at -5 240
\pinlabel $f$ at 36 240
\pinlabel $0$ at 77 220
\pinlabel $0$ at 603 220
\endlabellist
\centering
\includegraphics[width=5.4in]{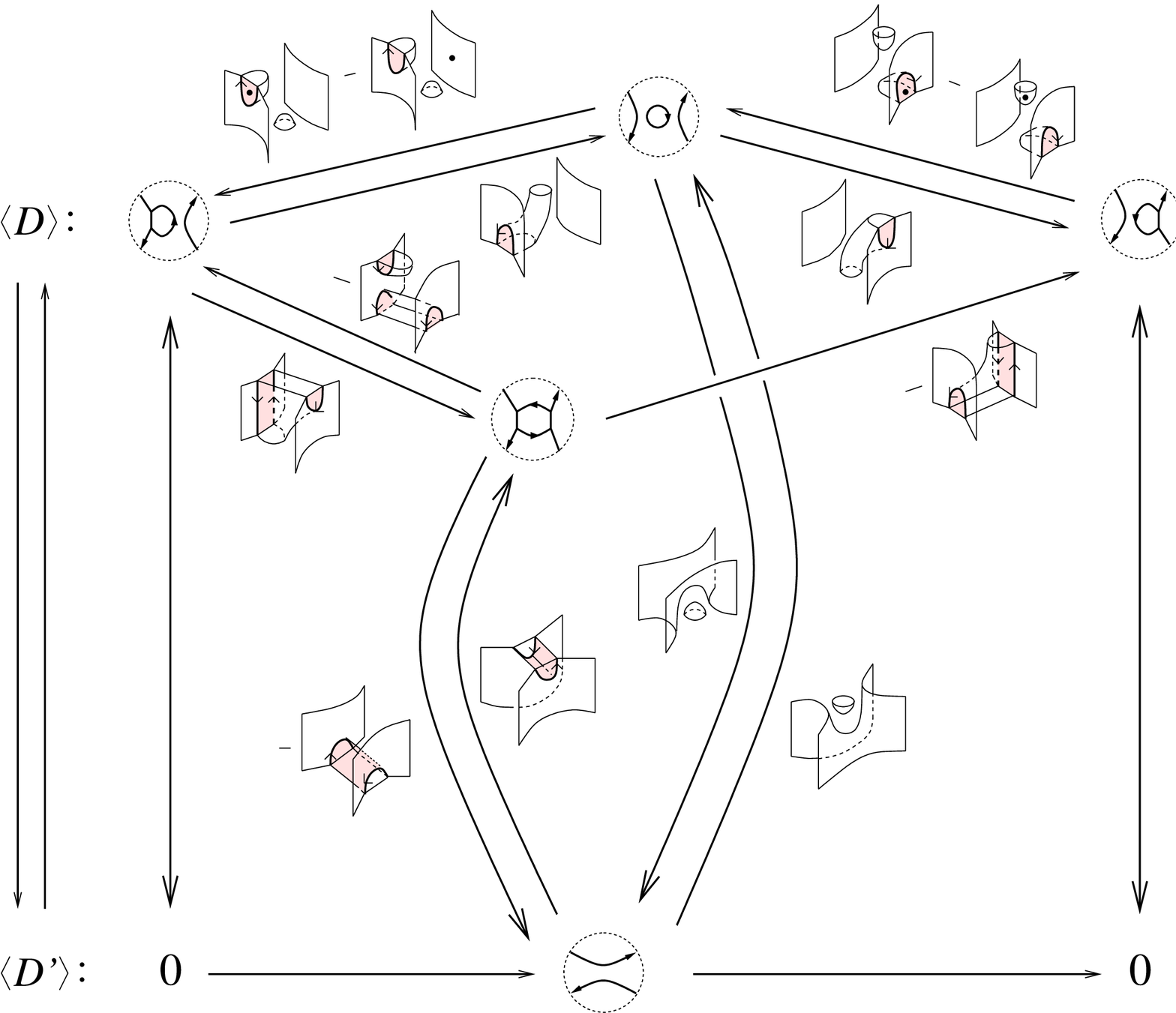}
\caption{Invariance under \emph{Reidemeister IIb}} \label{u3:fig:RIIbInvariance}
\end{figure}

\bigskip\newpage
%%%%%%%%%%%%%%%%%%%%%%
%                    %
%  Reidemeister III  %
%%%%%%%%%%%%%%%%%%%%%%
\n\emph{Reidemeister III.}
Consider diagrams $D$ and $D'$ that differ only in a circular region, as
in the figure below.
$$D=\figins{-13}{0.4}{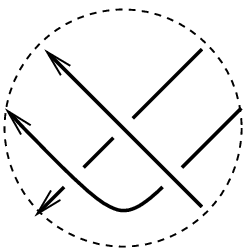}\qquad
D'=\figins{-13}{0.4}{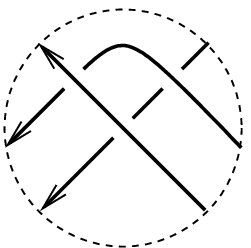}$$
We prove that $\brak{D'}$ is homotopy equivalent to $\brak{D}$ by showing that both complexes are homotopy equivalent to a third complex denoted $\brak{Q}$ (the bottom complex in Figure~\ref{u3:fig:RIIIInvariance}).
\begin{figure}[ht!]
\centering
\figwins{0}{4.4}{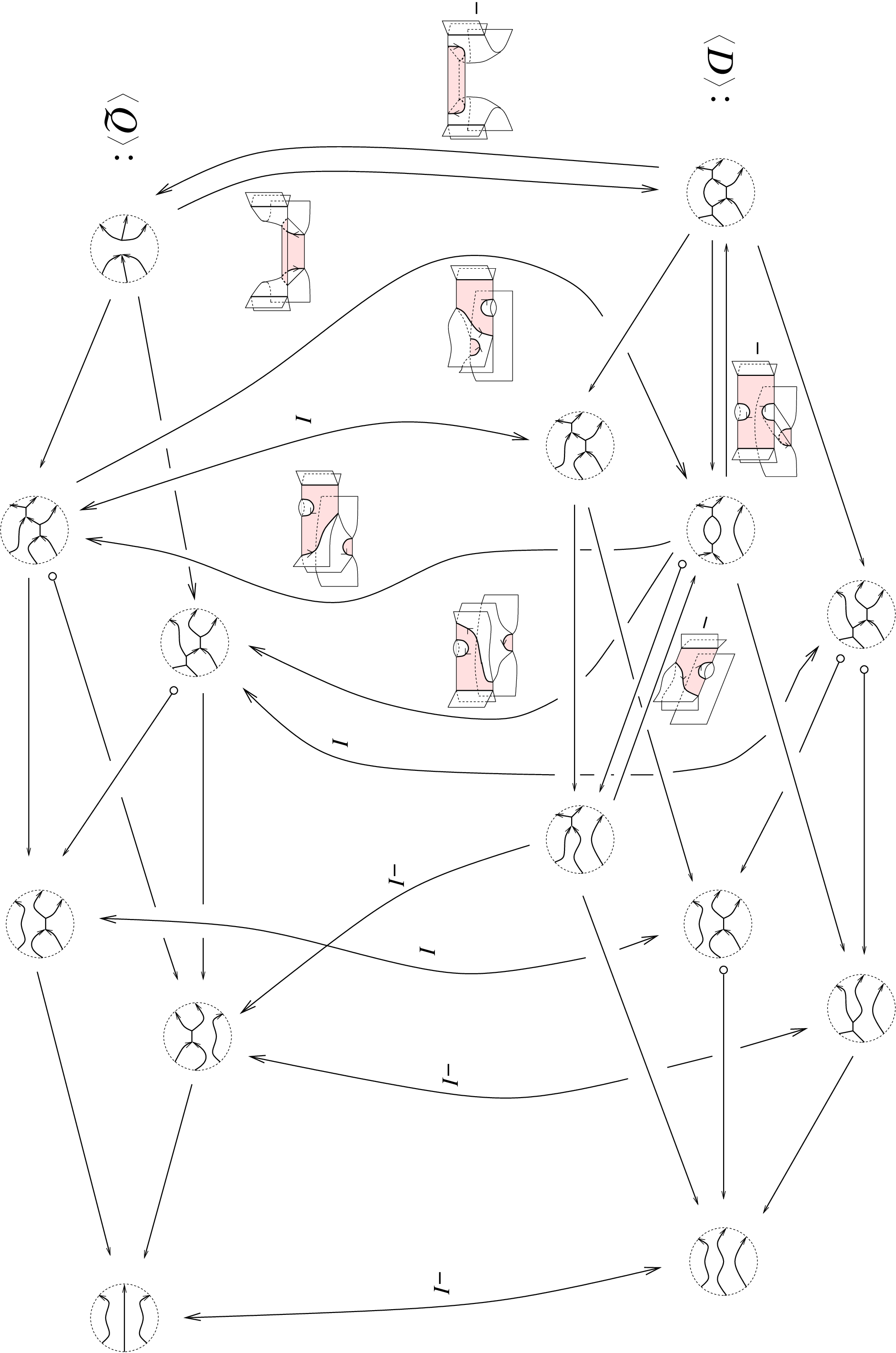}
\caption[Invariance under \emph{Reidemeister III}]{First step of invariance under \emph{Reidemeister III}. A circle attached to the tail of an arrow indicates that the corresponding morphism has a minus sign.}
\label{u3:fig:RIIIInvariance}
\end{figure}

\n Figure~\ref{u3:fig:RIIIInvariance} shows that $\brak{D}$ is homotopy equivalent to $\brak{Q}$. By applying a symmetry relative to a horizontal axis crossing each diagram in $\brak{D}$ we obtain a homotopy equivalence between $\brak{D'}$ and $\brak{Q}$. It follows that $\brak{D}$ is homotopy equivalent to $\brak{D'}$.
\end{proof}

By Theorem~\ref{u3:thm-inv} we can use any diagram of $L$ to obtain 
the invariant in $\kom_{/h}(\lfoam)$ and this justifies the 
notation $\brak{L}$.

%%%%%%%%%%%%%%%%%%%%%%
%                    %
%  Functoriality     %
%%%%%%%%%%%%%%%%%%%%%%
\subsection{Functoriality}\label{u3:ssec:funct}

The construction and the results of the previous sections can be 
extended to the category of tangles, following Bar-Natan's approach 
in~\cite{bar-natancob}. One can then prove the functoriality of $\brak{\;}$ 
as Bar-Natan does. Although we will not give the details of this proof 
some remarks are in order. In the first place, 
we have to consider tangle diagrams in a disk and foams inside a cylinder. 
To ensure additivity of the $q$-grading under lateral composition 
we need to add an extra term to the $q$-grading formula. For a foam between open $\slt$-webs 
with $|b|$ vertical boundary components 
and $d$ dots we have
$$q(f)=-2\chi(f)+ \chi(\partial f)+2d +|b|.$$
Lemma~8.6 in~\cite{bar-natancob} is fundamental in Bar-Natan's proof of the functoriality of the 
universal $\mathfrak{sl}(2)$-link homology. The analogue for $\slt$ follows from the following two lemmas.
\begin{lem}\label{u3:lem:closedfoam}
Let $f$ be a closed foam. If $q(f)<0$, then $f=0$ holds. If $q(f)=0$, then the evaluation of 
$f$ gives an integer. 
\end{lem}
\begin{proof}
Using~\eqref{u3:eq:CN} and~\eqref{u3:eq:RD} we can turn any closed foam into a $\bZ[a,b,c]$-linear 
combination of a disjoint union of spheres and theta foams. Since the grading remains 
unchanged by~\eqref{u3:eq:CN} and~\eqref{u3:eq:RD} it suffices to check the claims for spheres and theta 
foams, which is immediate from the~\eqref{u3:eq:S} and~\eqref{u3:eq:theta}-relations.
\end{proof} 

\begin{lem}
For a crossingless tangle diagram $T$ we have that
$\Hom_{\lfoam}(T,T)$ is zero in negative degrees and $\bZ$ in degree zero.
\end{lem}

\begin{proof}
The set of singular points in every foam $f$ from $T$ to itself consists of a disjoint union 
of circles. Using~\eqref{u3:eq:CN} and~\eqref{u3:eq:RD} we can reduce $f$ to a $\bZ[a,b,c]$-linear 
combination of disjoint unions of vertical disks and closed foams. Note that the $q$-degree of 
a dotted disc is always non-negative. Therefore, if $q(f)<0$, then the closed foams have to 
have negative $q$-degree as well and $f=0$ has to hold by Lemma~\ref{u3:lem:closedfoam}. 
If $q(f)=0$, then $f$ is equal to a $\bZ$-linear combination of undotted discs and closed foams 
of $q$-degree zero, so Lemma~\ref{u3:lem:closedfoam} shows that $f$ is an integer multiple of the 
identity foam on $T$.  
\end{proof}

The proofs of the analogues of lemmas~8.7-8.9 and theorem~5 in~\cite{bar-natancob} follow the 
same reasoning but use the homotopies of our Subsection~\ref{u3:ssec:thm-inv}. We illustrate this 
by showing that $\brak{\ }$ respects the movie move MM13 
(actually it is the mirror of MM13 in~\cite{bar-natancob}):
$$
\raisebox{-0.17in}{\includegraphics[height=0.4in]{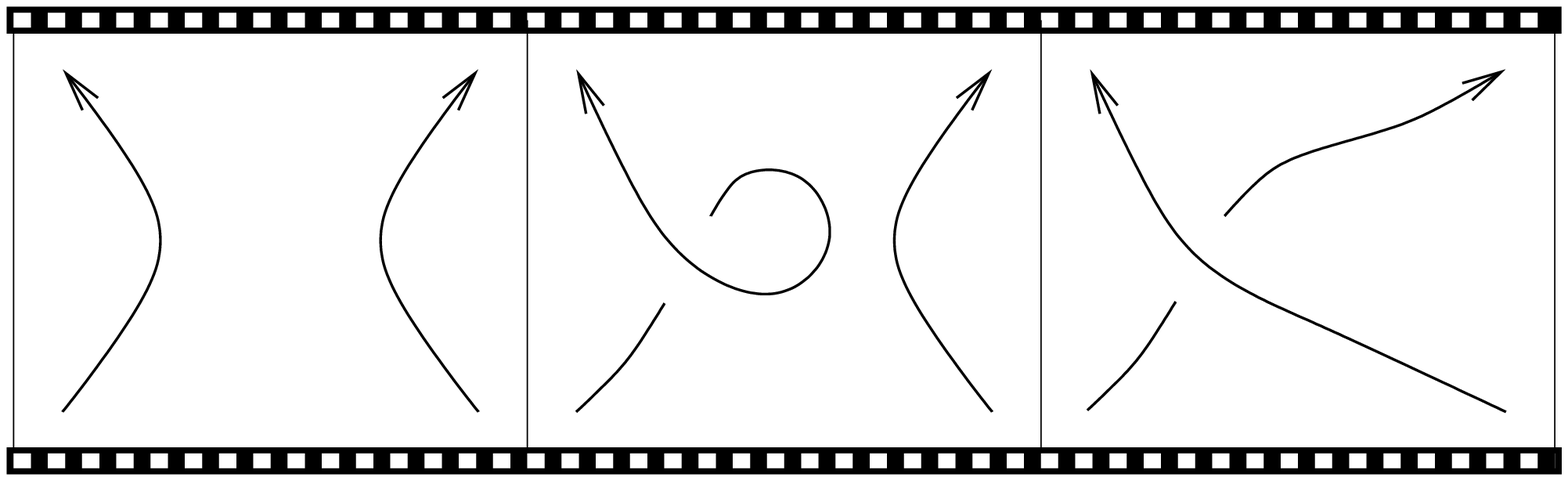}}\sim
\raisebox{-0.17in}{\includegraphics[height=0.4in]{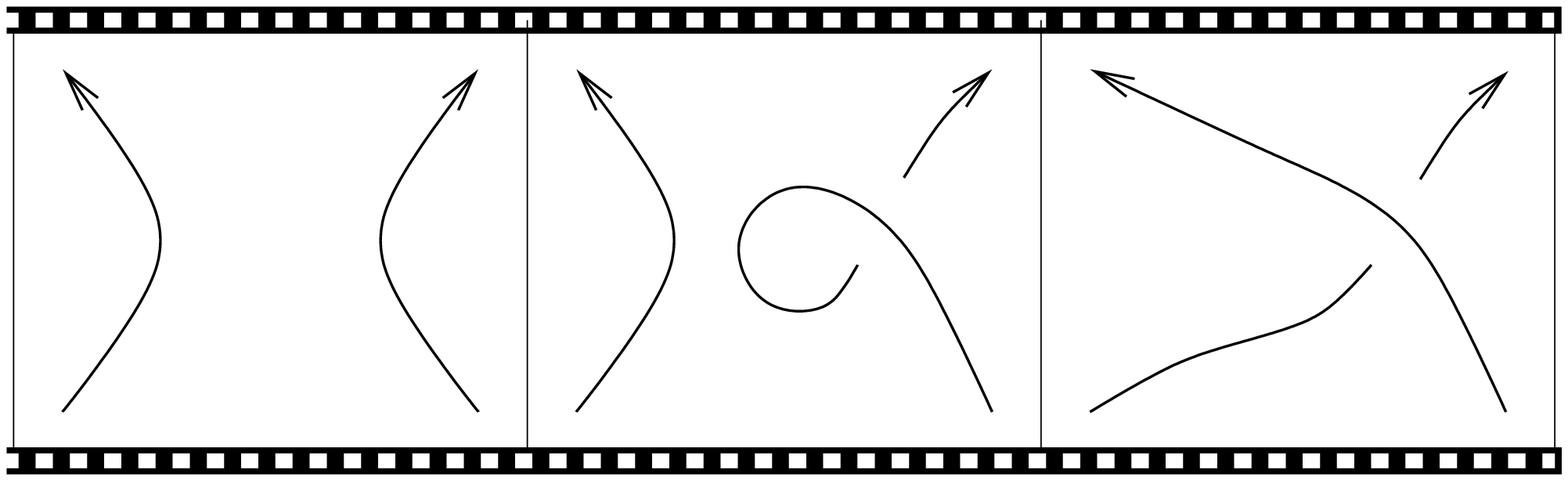}}\ .$$
Going from left to right in homological degree 0 we find the composition 
$$\brak{\raisebox{-0.05in}{\includegraphics[height=0.17in]{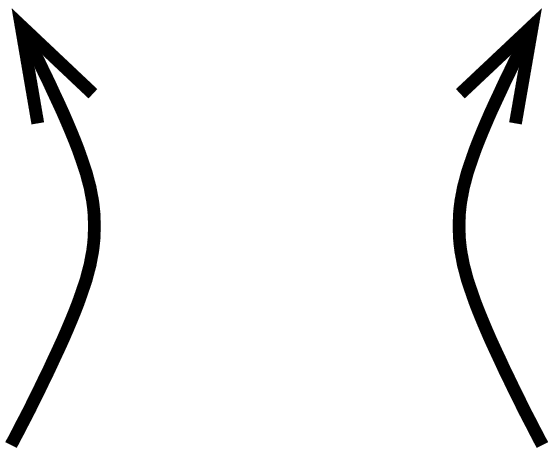}}}
\to
\brak{\raisebox{-0.05in}{\includegraphics[height=0.17in]{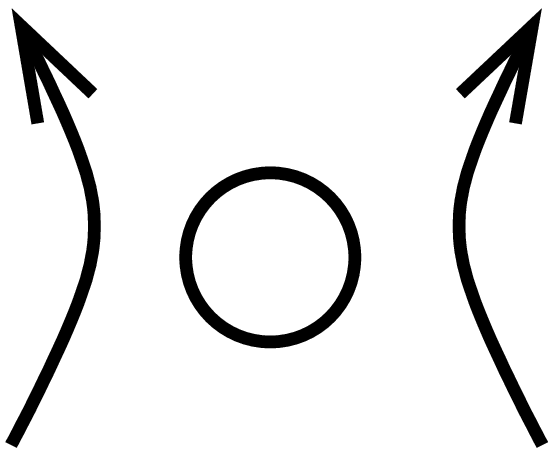}}}
\to
\brak{\raisebox{-0.05in}{\includegraphics[height=0.17in]{figs/mm13w1}}}
$$ in both movies.
The map for the movie on the left-hand side consists of the cobordism $f^0$ of 
Figure~\ref{u3:fig:RI_Invariance} between the left strand and the circle followed by a saddle 
cobordism between the circle and the right strand. For the movie on the right-hand side we 
have $f^0$ between the right strand and the circle followed by a saddle cobordism between the 
circle and the left strand. Both sides are equal to
$$\raisebox{-0.3in}{\includegraphics[height=0.6in]{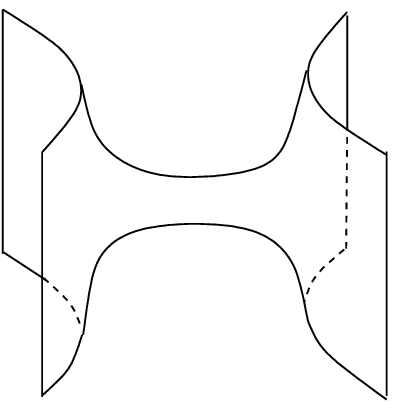}}.$$
Going from right to left and using the cobordism $g^0$ of Figure~\ref{u3:fig:RI_Invariance} we 
obtain the identity cobordism in both movies.

Without giving more details of this generalization, we state the main result. 
Denote by $\kom_{/\pm h}(\lfoam)$ the 
category $\kom_{/h}(\lfoam)$ modded out by $\pm 1$. Then

\begin{prop}
\label{u3:prop:func}
$\brak{\;}$ defines a functor $\Link\to \kom_{/\pm h}(\lfoam)$.
\end{prop}

%%%%%%%%%%%%%%%%%%%%%%%%%%%
%                         %
%  Universal Homology     %
%%%%%%%%%%%%%%%%%%%%%%%%%%%
\subsection{Universal homology}\label{u3:ssec:univhom}

Following Khovanov~\cite{khovanovsl3}, we define a functor $C\colon 
\lfoam\to \abcModgr$, which extends in a straightforward manner to the category $\kom(\lfoam)$.

\begin{defn}
For a closed $\slt$-web $\Gamma$, define $C(\Gamma)=\Hom_{\lfoam}(\emptyset,\Gamma)$. 
From the $q$-grading formula for foams, it follows that $C(\Gamma)$ is graded. 
For a foam $f$ between $\slt$-webs $\Gamma$ and $\Gamma'$ we define the $\bZ[a,b,c]$-linear map 
$$C(f)\colon\Hom_{\lfoam}(\emptyset,\Gamma) \to 
\Hom_{\lfoam}(\emptyset,\Gamma')$$ given by composition, whose degree equals $q(f)$. 
\end{defn}
\n Note that, if we have a disjoint union of $\slt$-webs $\Gamma$ and $\Gamma'$, then 
$C(\Gamma \sqcup \Gamma')\cong C(\Gamma)\otimes C(\Gamma')$. Here, as in the sequel, 
the tensor product is taken over $\bZ[a,b,c]$.

The following relations are a categorified version of Kuperberg's skein relations~\cite{kup} and were used and proved by Khovanov in~\cite{khovanovsl3} to relate his 
$\slt$-link homology to the quantum $\slt$-link invariant. 

\begin{lem}\label{u3:lem:KhK}
\emph{(Khovanov-Kuperberg relations~\cite{khovanovsl3, kup})} We have the following decompositions under 
the functor $C$:
\begin{align}
C(\unknot\Gamma) &\cong C(\unknot)\otimes C(\Gamma) 
\tag{\emph{Circle Removal}}
\\[1.2ex]
C(\figins{-2.6}{0.15}{hdigonweb}) 
 &\cong 
C(\raisebox{2pt}{\includegraphics[height=0.025in,width=0.3in]{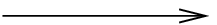}})\{-1\}
\oplus C(\raisebox{2pt}{\includegraphics[height=0.025in,width=0.3in]{figs/hthickedge}})\{1\} 
\tag{\emph{Digon Removal}}
\\[1.2ex]
C\left(\raisebox{-8.0pt}{\includegraphics[height=0.3in]{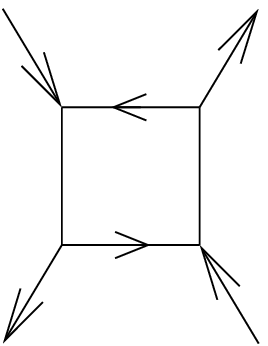}}\right)
 &\cong 
C\left(\raisebox{-8.0pt}{\includegraphics[height=0.3in]{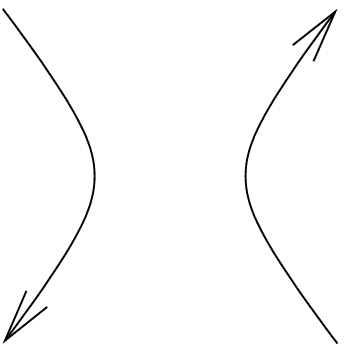}}\right)
\oplus
C\left(\raisebox{-8.0pt}{\includegraphics[height=0.3in]{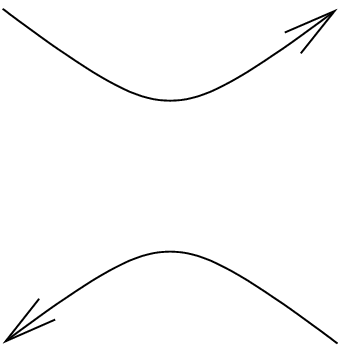}}\right) 
\tag{\emph{Square Removal}}
\end{align}
where $\{j\}$ denotes a positive shift in the $q$-grading by $j$.
\end{lem}

\begin{proof}
\emph{Circle removal} is immediate from the definition of $C(\Gamma)$. \emph{Digon removal} and \emph{square removal} are proved as in~\cite{khovanovsl3}. Notice that \emph{Digon removal} and \emph{square removal} are related to the local relations~\eqref{u3:eq:DR} and~\eqref{u3:eq:SqR} of page~\pageref{u3:eq:RD}.

%digon
To prove \emph{digon removal} we define grading-preserving maps 
\begin{gather*}
\varphi_0\colon C(\figins{-2.6}{0.15}{hdigonweb})\{1\} 
\to C(\raisebox{2pt}{\includegraphics[height=0.025in,width=0.3in]{figs/hthickedge}})
\qquad
\varphi_1\colon C(\figins{-2.6}{0.15}{hdigonweb})\{-1\} 
\to C(\raisebox{2pt}{\includegraphics[height=0.025in,width=0.3in]{figs/hthickedge}})
\\[1.5ex]
\psi_0\colon
C(\raisebox{2pt}{\includegraphics[height=0.025in,width=0.3in]{figs/hthickedge}}) \to
C(\figins{-2.6}{0.15}{hdigonweb})\{1\} 
\qquad
\psi_1\colon C(\raisebox{2pt}{\includegraphics[height=0.025in,width=0.3in]{figs/hthickedge}}) \to
C(\figins{-2.6}{0.15}{hdigonweb})\{-1\}
\end{gather*}
as
$$
\varphi_0= C\left(\figins{-10}{0.3}{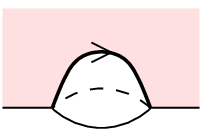}\right),\ \ \ 
\varphi_1= C\left(\figins{-10}{0.3}{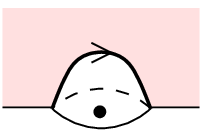} \right),\ \ \
\psi_0= C\left(\figins{-6}{0.3}{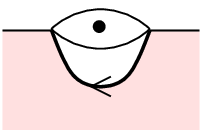}\right),\ \ \
\psi_1=- C\left( \figins{-6}{0.3}{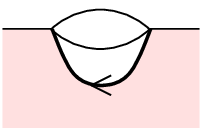}\right).
$$
From the~\eqref{u3:eq:CN} and~\eqref{u3:eq:theta} relations we have that $\varphi_i\psi_j=\delta_{i,j}$ (for $i,j=0,1$) and from the~\eqref{u3:eq:DR} relation it follows that $\psi_0\varphi_0+\psi_1\varphi_1$ is the identity map in $C(\figins{-2.6}{0.15}{hdigonweb})$.

%square
To prove \emph{square removal} we define grading-preserving maps 
\begin{gather*}
\rho_0\colon C\left(\raisebox{-8.0pt}{\includegraphics[height=0.3in]{figs/squareweb}}\right)\to C\left(\raisebox{-8.0pt}{\includegraphics[height=0.3in]{figs/vedgesweb}}\right)
\qquad 
\tau_0\colon C\left(\raisebox{-8.0pt}{\includegraphics[height=0.3in]{figs/vedgesweb}}\right)\to C\left(\raisebox{-8.0pt}{\includegraphics[height=0.3in]{figs/squareweb}}\right)
\\[1.8ex]
\rho_1\colon C\left(\raisebox{-8.0pt}{\includegraphics[height=0.3in]{figs/squareweb}}\right)\to C\left(\raisebox{-8.0pt}{\includegraphics[height=0.3in]{figs/hedgesweb}}\right)
\qquad 
\tau_1\colon C\left(\raisebox{-8.0pt}{\includegraphics[height=0.3in]{figs/hedgesweb}}\right)\to C\left(\raisebox{-8.0pt}{\includegraphics[height=0.3in]{figs/squareweb}}\right)
\end{gather*}
as
\begin{gather*}
\rho_0=  C\left(\figins{-22}{0.7}{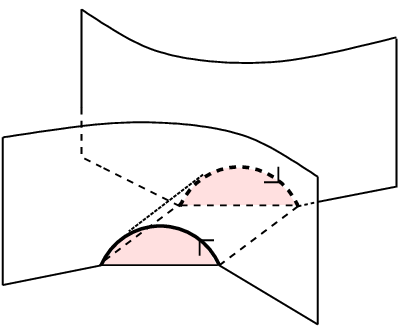}\right)
\qquad
\tau_0=  -C\left(\figins{-22}{0.7}{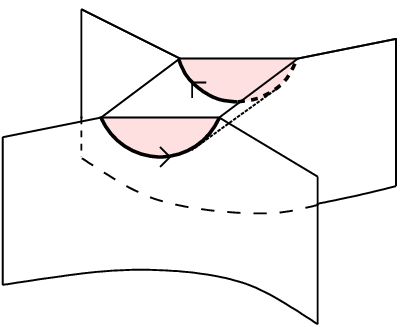}\right)
\\[1.8ex]
\rho_1 = C\left(\figins{-22}{0.7}{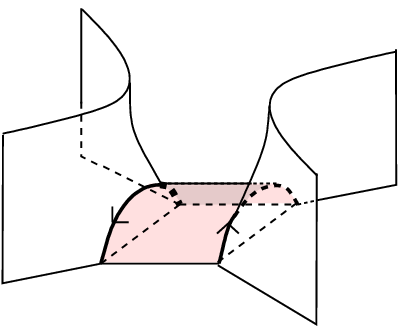}\right)
\qquad
\tau_1 = -C\left(\figins{-22}{0.7}{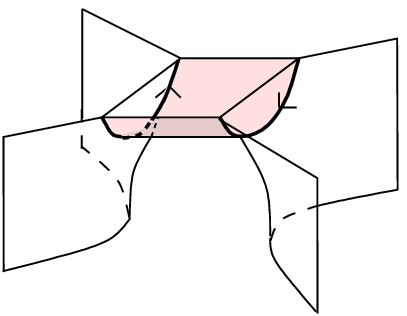}\right)
\end{gather*}
It is straightforward to check that $\rho_i\tau_j=\delta_{i,j}$ using the~\eqref{u3:eq:CN} and the~\eqref{u3:eq:theta} relations. From the~\eqref{u3:eq:SqR} relation it 
follows that 
$\tau_0\rho_0+\tau_1\rho_1=\id\left(
C\left(\raisebox{-8.0pt}{\includegraphics[height=0.3in]{figs/squareweb}}\right)
\right)$.
\end{proof}

\n Let $\uhtgen \ D$ denote the bigraded homology of $C\brak{D}$ and 
$\abcModbg$ the category of 
bigraded $\bZ[a,b,c]$-modules. Proposition~\ref{u3:prop:func} implies

\begin{prop}\label{u3:prop:inv_vec} 
$U_{a,b,c}\colon\Link\to \abcModbg$ is a functor, defined up to a sign.
\end{prop}

\n This justifies the notation $C(L)$ for $C\brak{D}$ and $\uhtgen \ L$ for $\uhtgen \ D$.

%%%%%%%%%%%%%%%%%%%%%%%%%%%%%%%%%%%%%%%%
%%%                                  %%%
%%%    Isomorphism classes           %%%
%%%                                  %%%
%%%%%%%%%%%%%%%%%%%%%%%%%%%%%%%%%%%%%%%%

\section{Isomorphism classes}\label{u3:sec:isos}

In this section we give the isomorphism classes of the $\slt$-link homologies, obtained by the author in joint work with Marco Mackaay in~\cite{mackaay-vaz}. We just state the main result of Section 3 of~\cite{mackaay-vaz}.  
Working over $\bC$ and taking $a,b,c$ to be complex numbers rather than formal parameters we obtain a filtered theory. Using the same construction as in the first part of this chapter we can define $\uht * L$, which is 
the universal $\slt$-homology with coefficients in $\bC$ (we use one $*$ as superscript to emphasize that it is a singly graded theory). 

We could work over $\bQ$ just as well and obtain the same results, 
except that in the proofs we would first have to pass to quadratic 
or cubic field extensions of $\bQ$ to guarantee the existence of the 
roots of $f(X)$ in the field of coefficients of the homology. 
The arguments presented for $\uht * L$ remain valid over those quadratic 
or cubic extensions. The 
universal coefficient theorem then shows that our results hold true 
for the homology defined over $\bQ$.   

There are three isomorphism classes in $\uht * L$. The first class is the one to which Khovanov's original $\slt$-link homology belongs, the second is the one studied by Gornik in the context of matrix factorizations and the last one is new, although in~\cite{DGR} and~\cite{GW} the authors 
make conjectures 
which are compatible with our results, and can be described in terms of Khovanov's original $\mathfrak{sl}(2)$-link homology.

\begin{thm}
There are three isomorphism classes of 
$\uht * L$. For 
a given choice of $a,b,c\in\bC$, the isomorphism class of $\uht * L$ 
is determined by the number of distinct roots of $f(X)=X^3-aX^2-bX-c$: 
\begin{enumerate}
%%%%%%%%%% X^3=0  %%%%%%%%%%%
\item\label{u3:item1} If $f(X)$ has one root (with multiplicity three)    
then  $\uht * L$ is isomorphic to Khovanov's original $\slt$-link homology, 
which in our notation is to $\uhtabc 0 0 0 * L$.

%%%%%%%%%% X^3=1  %%%%%%%%%%%
\item\label{u3:item2} If $f(X)$ has three distinct roots $\uht * L$ is isomorphic to Gornik $\slt$-link homology, which in our notation is $\uhtabc 0 0 1 * L$.

%%%%%%%%%% X^3=X^2 %%%%%%%%%%
\item\label{u3:item3} If $f(X)$ has two distinct roots then
$$\uht i L\cong \bigoplus_{L'\subseteq L} \kh {i-j(L')} * {L'},$$
where $j(L')=2\lk(L',L\backslash L')$. This isomorphism does not take into account the internal grading of the 
Khovanov homology.  
\end{enumerate}
\end{thm}
\n We do not know whether in~\ref{u3:item1} and~\ref{u3:item2} these isomorphisms preserve the filtration. In~\ref{u3:item3} the isomorphism does certainly not preserve the filtration, as can be easily seen by computing some simple examples.

%%%%%%%%%%%%%%%%%%%%%% End of Chapter %%%%%%%%%%%%%%%%%%%%%
%
%
%%%%%%%%%%%%%%%%%%%
%                 %
%     chapter     %
%                 %
%%%%%%%%%%%%%%%%%%%
%
%
% to start in an odd page
\newpage\
\chapter{The foam and the matrix factorization $\slt$ link homologies are equivalent}\label{chap:foam3mf}
%
%
%%%%%%%%%%%%%%%%%%%%%%%%%%%%%%%%%%%%%%%%
%%%                                  %%%
%%%        Introduction              %%%
%%%                                  %%%
%%%%%%%%%%%%%%%%%%%%%%%%%%%%%%%%%%%%%%%%
%
%
In this chapter we prove that the universal rational $\slt$ link homologies which were constructed in Chapter~\ref{chap:univ3}, using foams, and in Section~\ref{KR:sec:KR-abc}, using matrix factorizations, are naturally isomorphic as projective functors from the category of links and link cobordisms to the category of bigraded vector spaces. For $a=b=c=0$ this was conjectured to be true by Khovanov and Rozansky in~\cite{KR}.

One of the main difficulties one encounters when trying to relate both theories mentioned above is that the foam approach uses $\slt$-webs (see Section~\ref{u3:sec:univhom}) whereas the KR theory uses webs (see Section~\ref{KR:sec:slN}). In Khovanov and Rozansky's setup in \cite{KR} there is a unique way to associate a matrix factorization to each web. In general there are several webs that one can associate to an $\slt$-web, so there is no obvious choice of a KR matrix factorization to associate to an $\slt$-web. However, we show that the KR-matrix factorizations for all webs associated to a fixed $\slt$-web are homotopy equivalent and that between two of them there is a canonical choice of homotopy equivalence in a certain sense. This allows us to associate an equivalence class of KR-matrix factorizations to each $\slt$-web. After that it is relatively straightforward to show the equivalence between the foam and the KR $\slt$-link homologies.

\medskip

In Section~\ref{mf3:sec:foam} we change some conventions in the category $\lfoam$ of Chapter~\ref{chap:univ3}. 
Section~\ref{mf3:sec:sl3-mf} is the core of the chapter. 
In this section we show how to associate equivalence classes of 
matrix factorizations to $\slt$-webs and use them to construct a 
link homology that is equiva\-lent to Khovanov and Rozansky's. 
In Section~\ref{mf3:sec:iso} we establish the equivalence between the foam 
$\slt$-link homology and the one presented in Section~\ref{mf3:sec:sl3-mf}.

%%%%%%%%%%%%%%%%%%%%%%%%%%%%%%%%%%%%%%%%
%%%                                  %%%
%%%        Foams                     %%%
%%%                                  %%%
%%%%%%%%%%%%%%%%%%%%%%%%%%%%%%%%%%%%%%%%
\section{New normalization in $\lfoam$}\label{mf3:sec:foam}

This section contains the modifications in the definition of $\lfoam$ that are necessary 
to relate it to Khovanov and Rozansky's universal $\slt$ link homology using matrix 
factorizations.

The modified relations\footnote{We thank Scott Morrison 
for spotting a mistake in the coefficients in a preprint containing the results of this chapter.} are defined over $\bQ$ and are
\begin{gather}
%3Dot reduction
\raisebox{-7pt}{
\includegraphics[height=0.25in]{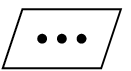}}=
a\raisebox{-7pt}{
\includegraphics[height=0.25in]{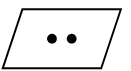}}+
b\raisebox{-7pt}{
\includegraphics[height=0.25in]{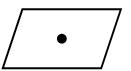}}+
c\raisebox{-7pt}{
\includegraphics[height=0.25in]{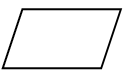}}
\tag{3D}\label{mf3:eq:3D}
\\[1.5ex] 
%cutting neck
\figwhins{-18}{0.55}{0.28}{cylinder}=
4\left(
-\figwhins{-18}{0.55}{0.28}{cneck01}
-\figwhins{-18}{0.55}{0.28}{cneck02}
-\figwhins{-18}{0.55}{0.28}{cneck03}
+a
\left( 
\figwhins{-18}{0.55}{0.28}{cnecka1}+
\figwhins{-18}{0.55}{0.28}{cnecka2}
\right)
+b
\figwhins{-18}{0.55}{0.28}{cneckb}
\right)
\tag{CN}\label{mf3:eq:CN}
\\[1.5ex]
%S-relation
\figwins{-9}{0.35}{sundot}=
\figwins{-9}{0.35}{sdot}=0,\quad
\figwins{-9}{0.35}{sddot}=-\frac{1}{4}
\tag{S}\label{mf3:eq:S}
\end{gather}

\medskip

For $\alpha, \beta, \delta\leq 2$ we put (see Figure~\ref{u3:fig:theta})
\begin{equation}
\theta(\alpha,\beta,\delta)=
\begin{cases}
\ \ \frac{1}{8} & (\alpha,\beta,\delta)=(1,2,0)\text{ or a cyclic permutation} \\ 
-\frac{1}{8} & (\alpha,\beta,\delta)=(2,1,0)\text{ or a cyclic permutation} \\ 
\ \ 0 & \text{else}
\end{cases}
\tag{$\Theta$}\label{mf3:eq:theta}
\end{equation}

\medskip

Using the modified relations $\ell'$ one can prove the modified identities~\eqref{mf3:eq:RD},~\eqref{mf3:eq:DR}.
\begin{gather}
%Removing disk
\figwhins{-17}{0.55}{0.28}{sdisk}=
2\left(
\figwhins{-17}{0.55}{0.28}{cnecka1}-
\figwhins{-17}{0.55}{0.28}{cnecka2}
\right)
\tag{RD}\label{mf3:eq:RD}
\\[1.5ex]
%digon removal
\figins{-20}{0.6}{digonfid-sl3}=
2\left(
\figins{-26}{0.75}{digonf1}-
\figins{-26}{0.75}{digonf2}
\right)
\tag{DR}\label{mf3:eq:DR}
\\[1.5ex]\displaybreak
%Square removal
\figins{-28}{0.8}{square_id-sl3}=
-\ \figins{-28}{0.8}{square_rem1-sl3}
-\figins{-28}{0.8}{square_rem2-sl3}
\tag{SqR}\label{mf3:eq:SqR}
\end{gather}

The Khovanov-Kuperberg relations of Lemma~\ref{u3:lem:KhK} remain the same with this new norma\-lization.

We now show that the theory obtained with this new normalization is equivalent to the one in Chapter~\ref{chap:univ3} after tensoring the latter with $\bQ$. To distinguish the various constructions let $\Qlfoam$ and $U^\bQ_{a,b,c}$ denote the category $\lfoam$ and the universal homology with the new normalization. Both are defined over $\bQ[a,b,c]$. Let also $\Modbg$ denote the category of bigraded $\bQ[a,b,c]$-modules.

\begin{lem}
The categories $\Qlfoam$ and $\lfoam\otimes_\bZ\bQ$ are isomorphic. 
\end{lem}

\begin{proof}
We define a functor $\Xi\colon\Qlfoam\to\lfoam\otimes_\bZ\bQ$. On objects $\Xi$ is the identity. To define $\Xi$ on foams we note that the morphisms in both categories are generated by cups, caps, zip, unzips and saddle point cobordisms. We put
\begin{gather*}
\figins{-6}{0.2}{cup}\mapsto\frac{1}{2}\figins{-6}{0.2}{cup}
\qquad\quad
\figins{-4}{0.2}{cap}\mapsto\frac{1}{2}\figins{-4}{0.2}{cap}
\\[1.5ex]
\figins{-17}{0.5}{ssaddle_ud-sl3}
\mapsto
\figins{-17}{0.5}{ssaddle_ud-sl3}
\qquad\quad
\figins{-17}{0.5}{ssaddle-sl3}
\mapsto\frac{1}{2}
\figins{-17}{0.5}{ssaddle-sl3}
\\[1.5ex]
\figins{-18}{0.6}{saddle}
\mapsto\frac{1}{2}
\figins{-18}{0.6}{saddle}
\end{gather*}
It is straightforward to check that $\Xi$ is a well defined isomorphism of categories.
\end{proof}

Since $\Xi$ commutes with the differentials of both constructions we have
\begin{cor}
The projective functors $U^\bQ_{a,b,c}$ and $U_{a,b,c}\otimes_\bZ\bQ$ from $\Link$ to $\Modbg$ are na\-turally isomorphic.
\end{cor}
%
%
%%%%%%%%%%%%%%
%    MF3     %
%%%%%%%%%%%%%%
\section{A matrix factorization theory for $\slt$-webs}\label{mf3:sec:sl3-mf}

As mentioned in the introduction, the main problem in comparing the foam and 
the matrix factorization $\slt$ 
link homologies is that one has to deal with two different sorts of webs. Therefore it apparently is not clear 
which KR matrix factorization to associate to an $\slt$-web. However, in Proposition 
\ref{mf3:prop:ident-vert} we show that this ambiguity is not problematic. Our proof 
of this result is rather roundabout and requires a matrix 
factorization for each vertex. In this way we associate a matrix factorization 
to each $\slt$-web. For each choice of web associated to a given $\slt$-web the KR-matrix 
factorization is a quotient of ours, obtained by identifying the vertex variables 
pairwise according to the double edges. We then show that for a given $\slt$-web two such 
quotients are always homotopy equivalent. This is the 
main ingredient which allows us to establish the equivalence between the foam and 
the matrix factorization $\slt$ link homologies.        

Recall that an $\slt$-web is an oriented trivalent graph where near each vertex
all edges are simultaneously oriented toward it or away from it. We call the former vertices of 
($-$)-\emph{type} and the latter vertices of ($+$)-\emph{type}. To associate 
matrix factorizations to $\slt$-webs we impose that 
each edge have at least one mark.

\begin{figure}[ht!]
\labellist
\small\hair 2pt
\pinlabel $x$ at -8 113
\pinlabel $y$ at 96 111
\pinlabel $z$ at 42 -12
\pinlabel $v$ at 27  55
\endlabellist
\centering
\figs{0.35}{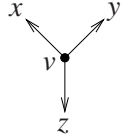}
\caption{A vertex of ($+$)-type}
\label{mf3:fig:3vert-out}
\end{figure}
%
%
%%%%%%%%%%%%%%%%%%
%    3-vertex    %
%%%%%%%%%%%%%%%%%%
%
%
\subsection{The 3-vertex}\label{mf3:ssec:triv}

Let $R$ denote the ring $\bQ[a,b,c][\mathbf{x}]$, where $\mathbf{x}$ is an array of variables indexed by the marks of an $\slt$-web $\Gamma$.
We take the basic polynomial $p(x)$ as in Equation~\eqref{KR:eq:upoly}.

Consider the 3-vertex of ($+$)-type in Figure~\ref{mf3:fig:3vert-out}, with emanating 
edges marked $x,y,z$. The (symmetric) polynomial
$$p(x)+p(y)+p(z)=x^4+y^4+z^4-\frac{4a}{3}\bigl(x^3+y^3+z^3 \bigr)-2b\bigl(x^2+y^2+z^2\bigr)-4c\bigl(x+y+z\bigr)$$
can be written as a polynomial in the elementary symmetric polynomials
$$p_v(x+y+z,xy+xz+yz,xyz)=p_v(e_1,e_2,e_3).$$

Using the methods of Section~\ref{KR:sec:KR-abc} we can obtain a matrix factorization of 
$p_v$, but if we tensor together two of these, then we obtain a matrix factorization 
which is not homotopy equivalent to the dumbell matrix factorization. This can be seen 
quite easily, since the new Koszul matrix has 6 rows and 
only one extra variable. This extra variable can be excluded at the expense of 1 
row, but then we get a Koszul matrix with 5 rows, whereas the dumbell Koszul matrix 
has only 2. To solve this problem we introduce a 
set of three new variables for each vertex\footnote{Khovanov had already observed 
this problem for the undeformed case and suggested to us 
the introduction of one vertex variable in that case.}.
Introduce the \emph{vertex variables} $v_1$, $v_2$, $v_3$ with $q(v_i)=2i$ and define 
the \emph{vertex ring}  
$$R_v=\bQ[a,b,c][x,y,z,v_1,v_2,v_3].$$

\n We define the potential as $$W_v=p_v(e_1,e_2,e_3)-p_v(v_1,v_2,v_3).$$
We have
\begin{align*}
W_v &= \frac{p_v(e_1,e_2,e_3)-p_v(v_1,e_2,e_3)}{e_1-v_1}(e_1-v_1) \\
    &\quad +\frac{p_v(v_1,e_2,e_3)-p_v(v_1,v_2,e_3)}{e_2-v_2}(e_2-v_2) \\
    &\quad +\frac{p_v(v_1,v_2,e_3)-p_v(v_1,v_2,v_3)}{e_3-v_3}(e_3-v_3) \\
    &= g_1(e_1-v_1) + g_2(e_2-v_2) + g_3(e_3-v_3),
\end{align*}
where the polynomials $g_i$ ($i=1,2,3$) have the explicit form
\begin{align}\label{mf3:eqn:gfactor1}
g_1 &= \frac{e_1^4-v_1^4}{e_1-v_1}-4e_2(e_1+v_1)+4e_3-
        \frac{4a}{3}\left(\frac{e_1^3-v_1^3}{e_1-v_1}-3e_2\right) -2b(e_1+v_1)-4c \\
g_2 &= 2(e_2+v_2)-4v_1^2+4av_1 + 4b \\
g_3 &= 4(v_1-a).\label{mf3:eqn:gfactor3}
\end{align}
We define the 3-vertex factorization $\hatYGraph_{v_+}$ as the tensor product of the 
factorizations
$$R_v\xra{g_i} R_v\{2i-4\}\xra{e_i-v_i} R_v,
\qquad (i=1,2,3)$$
shifted by $-3/2$ in the $q$-grading and by $1/2$ in the $\bZ/2\bZ$-grading, which 
we write in the form of the Koszul matrix
$$\hatYGraph_{v_+}=
\begin{Bmatrix}
 g_1\ , & e_1-v_1  \\
 g_2\ , & e_2-v_2  \\
 g_3\ , & e_3-v_3
\end{Bmatrix}_{R_v}
\{-3/2\}\brak{1/2}
.$$

If $\YGraph_v$ is a 3-vertex of $-$-type with incoming edges marked $x,y,z$ 
we define
$$\hatYGraph_{v_-}=
\begin{Bmatrix}
 g_1\ , & v_1-e_1 \\
 g_2\ , & v_2-e_2 \\
 g_3\ , & v_3-e_3
\end{Bmatrix}_{R_v}
\{-3/2\}\brak{1/2},$$
with $g_1$, $g_2$, $g_3$ as above.

\begin{lem}
We have the following homotopy equivalences in $\End_{\mf{}}{(\hatYGraph_{v_\pm})}$:
$$m(x+y+z)\cong m(a),\quad 
m(xy+xz+yz)\cong m(-b),\quad 
m(xyz)\cong m(c).$$
\end{lem}
\proof
For a matrix factorization $\hat{M}$ over $R$ with potential $W$ the homomorphism 
$$R\ra\End_{\mf{}}(\hat{M}),\quad r\mapsto m(r)$$ 
factors through the Jacobi algebra of $W$ and up to shifts, the Jacobi algebra of 
$W_v$ is
$$J_{W_v}\cong\bQ[a,b,c,x,y,z]_{/ 
\{
x+y+z=a,\ 
xy+xz+yz=-b,\ 
xyz=c
\}}.\rlap{\hspace{1.292in}\qedsymbol}$$ 
%end of proof

%
%
%%%%%%%%%%%%%%%%%%%%%%
% vertex composition %
%%%%%%%%%%%%%%%%%%%%%%
%
%
\subsection{Vertex composition}
The elementary $\slt$-webs considered so far can be combined to produce bigger $\slt$-webs. Let $E$ and $V$ denote the set of arcs between marks and the set of vertices of a general $\slt$-web $\Gamma_v$. Denote by $\partial E$ the set of free ends of $\Gamma_v$ and by $(v_{i_1},v_{i_2},v_{i_3})$ the vertex variables corresponding to the vertex $v_i$. We have
\begin{equation*}\label{mf3:eq:bigmf}
\hat\Gamma_v =\bigotimes\limits_{e\in E}\hatarc\mspace{-3.5mu}_e\otimes\bigotimes_{v\in V}\hatYGraph_v.
\end{equation*}
Factorization $\hatarc\mspace{-3.5mu}_e$ is the arc factorization introduced in 
Subsection~\ref{KR:sec:KR-abc}. The tensor product is taken over suitable rings , so that $\hat\Gamma_v$ is a finite rank free $R$-module.

This is a matrix factorization with potential 
$$W=\sum\limits_{i\in\partial E(\Gamma)}s_i p(x_{i})+\sum\limits_{v_j\in V(\Gamma)}s_j^vp_v(v_{j_1},v_{j_2},v_{j_3})=W_\mathbf{x}+W_\mathbf{v}$$ where $s_i=1$ if the corresponding arc is oriented to it or $s_i=-1$ in the opposite case and $s_j^v=1$ if $v_j$ is of positive type and $s_j^v=-1$ in the opposite case.

From now on we only consider open $\slt$-webs in which the number of free ends oriented inwards equals the number of free ends oriented outwards. This implies that the number of vertices of ($+$)-type equals the number of vertices of ($-$)-type.

Let $R$ and $R_\mathbf{v}$ denote the rings $\bQ[a,b,c][\mathbf{x}]$ and 
$R[\mathbf{v}]$ 
respectively. Given two vertices, $v_i$ and $v_j$, of opposite type, we 
can take the quotient by the ideal generated by $v_i-v_j$. 
The potential becomes independent of $v_i$ and $v_j$, because 
they appeared with opposite signs, and we can 
exclude the common vertex variables corresponding to $v_i$ and $v_j$ as in 
Lemma~\ref{KR:lem:exvar}. This is possible because in all our examples the 
Koszul matrices have linear terms which involve vertex and edge variables.  
The matrix factorization which we obtain in this way we represent graphically 
by a virtual edge, as in 
Figure~\ref{mf3:fig:virtualedge}. 
\begin{figure}[h!]
$$
\figins{0}{0.35}{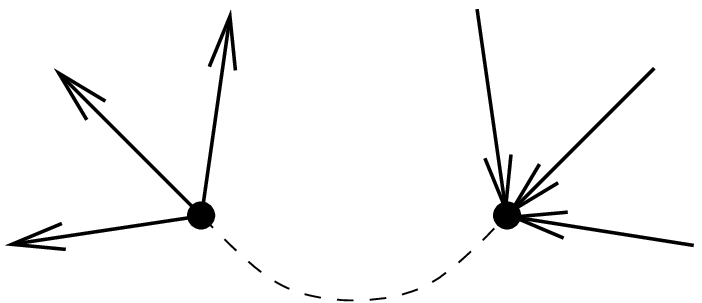}
\qquad\quad
\figins{0}{0.38}{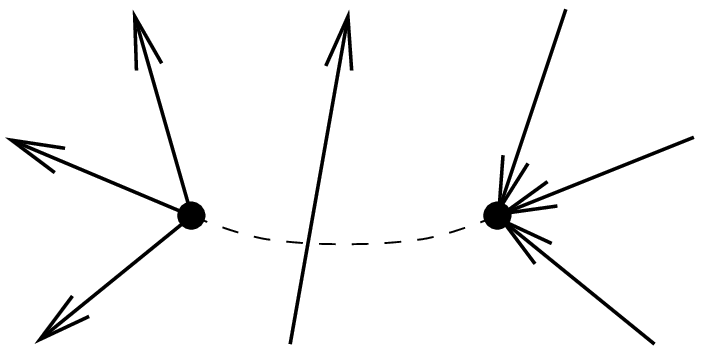}
$$
\caption{Virtual edges}
\label{mf3:fig:virtualedge}
\end{figure}
A virtual edge can cross other virtual edges and ordinary edges and does not have 
any mark.
If we pair every positive vertex in $\Gamma_v$ to a negative one, the above 
procedure leads to a \emph{complete identification} of the vertices of $\Gamma_v$ 
and a corresponding matrix factorization $\zeta(\hat\Gamma_v)$. 
A different complete identification yields a different matrix factorization 
$\zeta'(\hat\Gamma_v)$.

\begin{prop}\label{mf3:prop:ident-vert}
Let $\Gamma_v$ be a closed $\slt$-web. Then $\zeta(\hat\Gamma_v)$ and 
$\zeta'(\hat\Gamma_v)$ are isomorphic, up to a shift in the $\bZ/2\bZ$-grading.
\end{prop}

\n To prove Proposition~\ref{mf3:prop:ident-vert} we need some technical results.
\begin{lem}\label{mf3:lem:KR-3edge}
Consider the $\slt$-web $\Gamma_v$ and the web $\Upsilon$ below.
$$\xymatrix@R=1mm{
\figins{0}{0.7}{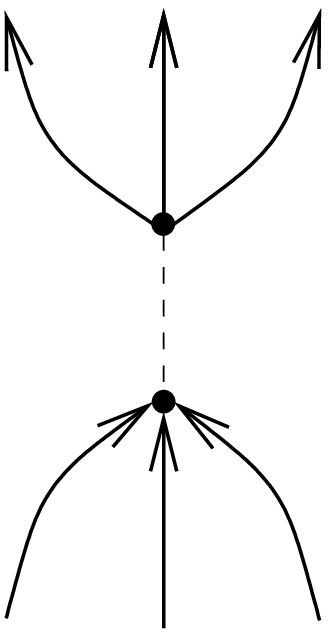} &
\figins{0}{0.7}{trpledge} \\
\Gamma_v & \Upsilon
}.$$
Then $\zeta(\hat\Gamma_v\brak{1})$ is the factorization of the triple edge 
$\hat\Upsilon$ of KR for $N=3$.
\end{lem}
\begin{proof} Immediate. 
\end{proof}

\begin{figure}[ht!]
\raisebox{20pt}{$\Gamma_v =$\quad\ }
\labellist
\small\hair 2pt
\pinlabel $x_i$ at -12 130
\pinlabel $x_j$ at 136 129
\pinlabel $x_k$ at -12 -6
\pinlabel $x_l$ at 136 -6
\pinlabel $x_r$ at 36 66 
\endlabellist
\centering
\figs{0.35}{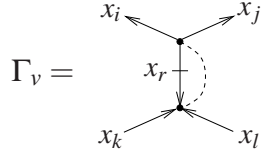}
\caption{A double edge}
\label{mf3:fig:thick-KR}
\end{figure}
\begin{lem}\label{mf3:lem:fatedge}
Let $\Gamma_v$ be the $\slt$-web in Figure~\ref{mf3:fig:thick-KR}.
Then $\zeta(\hat\Gamma_v)$ is isomorphic to the factorization assigned to the 
double edge of KR for $N=3$.
\end{lem}
\begin{proof}
Let $\hat\Gamma_+$ and $\hat\Gamma_-$ be the Koszul factorizations for the upper and
lower vertex in $\Gamma_v$ respectively and $v^\pm_i$ denote the corresponding 
sets of vertex variables.  
We have
$$
\hat\Gamma_+ =
\begin{Bmatrix}
g^+_1\ ,& x_i+x_j+x_r-v^+_1 \\
g^+_2\ ,& x_ix_j+x_r(x_i+x_j)-v^+_2 \\
g^+_3\ ,& x_ix_jx_r-v^+_3
\end{Bmatrix}_{R_{v_+}}\{-3/2\}\brak{1/2}$$
and
$$
\hat\Gamma_- =
\begin{Bmatrix}
g^-_1\ ,& v^-_1 - x_k-x_l-x_r \\
g^-_2\ ,& v^-_2 - x_kx_l-x_r(x_k+x_l) \\
g^-_3\ ,& v^-_3 - x_kx_lx_r
\end{Bmatrix}_{R_{v_-}}\{-3/2\}\brak{1/2}.$$
The explicit form of the polynomials $g^\pm_i$ is given in 
Equations~(\ref{mf3:eqn:gfactor1}-\ref{mf3:eqn:gfactor3}).
Taking the tensor product of $\hat\Gamma_+$ and $\hat\Gamma_-$, identifying vertices $v_+$ and $v_-$ and excluding the vertex variables yields
$$
\zeta(\hat\Gamma_v) = 
(\hat\Gamma_+\otimes\hat\Gamma_-)_{/ \mathbf{v}_+ - \mathbf{v}_-}\cong
\begin{Bmatrix}
g_1\ ,& x_i+x_j-x_k-x_l \\
g_2\ ,& x_ix_j - x_kx_l + x_r(x_i+x_j-x_k-x_l) \\
g_3\ ,& x_r(x_ix_j-x_kx_l) 
\end{Bmatrix}_R\{-3\}\brak{1},
$$
where
$$
g_i=
\left. g_i^+ 
\right|_{\{v_1^+=x_k+x_l+x_r,\ v_2^+=x_kx_l+x_r(x_k+x_l),\ v_3^+=x_rx_kx_l\} }.$$

\n This is a factorization over the ring 
$$R=\bQ[a,b,c][x_i,x_j,x_k,x_l,x_r,\mathbf{v}]_{/I}\cong \bQ[a,b,c][x_i,x_j,x_k,x_l,x_r]$$
where $I$ is the ideal generated by
$$\{
v_1 - x_r - x_k - x_l,\
v_2 - x_kx_l - x_r(x_k+x_l),\
v_3 - x_kx_lx_r
\}.$$
Using $g_3=4(x_r+x_k+x_l-a)$ and acting with the shift functor $\brak{1}$ on the third row one can write
$$
\zeta(\hat\Gamma_v) \cong
\begin{Bmatrix}
g_1\ ,& x_i+x_j-x_k-x_l \\
g_2\ ,& x_ix_j - x_kx_l + x_r(x_i+x_j-x_k-x_l) \\
-x_r(x_ix_j-x_kx_l)\ ,&  -4(x_r+x_k+x_l-a)
\end{Bmatrix}_R\{-1\}
$$
which is isomorphic, by a row operation, to the factorization
$$
\begin{Bmatrix}
g_1 + x_rg_2 \ ,& x_i+x_j-x_k-x_l \\
g_2\ ,& x_ix_j - x_kx_l \\
-x_r(x_ix_j-x_kx_l)\ ,&  -4(x_r+x_k+x_l-a)
\end{Bmatrix}_R\{-1\}.
$$
Excluding the variable $x_r$ from the third row gives
$$
\zeta(\hat\Gamma_v) \cong
\begin{Bmatrix}
g_1 + (a-x_k-x_l) g_2\ ,& x_i+x_j-x_k-x_l \\
g_2 \ ,& x_ix_j - x_kx_l
\end{Bmatrix}_{R'}\{-1\}
$$
where 
$$
R'=\bQ[a,b,c][x_i,x_j,x_l,x_r]_{/x_r-a+x_k+x_l} 
\cong
\bQ[a,b,c][x_i,x_j,x_k,x_l].$$ 
The claim follows from Lemma~\ref{KR:lem:regseq-iso}, since both are factorizations 
over $R'$ with the same potential and the same second column, the terms in which 
form a regular sequence in $R'$. As a matter of fact, using a row operation 
one can write
 $$
\zeta(\hat\Gamma_v) \cong
\begin{Bmatrix}
g_1 + (a-x_k-x_l) g_2+2(a-x_k-x_l)(x_ix_j-x_kx_l) ,& x_i+x_j-x_k-x_l \\
g_2 + 2(a-x_k-x_l)(x_i+x_j-x_k-x_l) ,& x_ix_j - x_kx_l
\end{Bmatrix}_{R'}\{-1\}
$$
and check that the polynomials in the first column are exactly the polynomials $u_{ijkl}$ and $v_{ijkl}$ corresponding to the dumbell factorization of Section~\ref{KR:sec:KR-abc}. 
\end{proof}

\begin{lem}\label{mf3:lem:swapedges}
Let $\Gamma_v$ be a closed $\slt$-web and $\zeta$ and $\zeta'$ two complete identifications 
that only differ in the region depicted 
in Figure~\ref{mf3:fig:swap}, where $T$ is a part of the diagram whose orientation is not 
important. Then there is an isomorphism $\zeta(\hat\Gamma_v)\cong
\zeta'(\hat\Gamma_v)\brak{1}$.
\begin{figure}[ht!]
\bigskip
\labellist
\small\hair 2pt
\pinlabel $x_1$ at -22 136
\pinlabel $x_2$ at -8 194
\pinlabel $x_3$ at 54 212
\pinlabel $x_4$ at 142 212
\pinlabel $x_5$ at 201 194
\pinlabel $x_6$ at 215 136
\pinlabel $x_7$ at 215 54
\pinlabel $x_8$ at 201 0
\pinlabel $x_9$ at 142 -16
\pinlabel $x_{10}$ at 54 -16
\pinlabel $x_{11}$ at -8 0
\pinlabel $x_{12}$ at -24 56
\pinlabel $x_1$ at 320 136
\pinlabel $x_2$ at 330 194
\pinlabel $x_3$ at 388 212
\pinlabel $x_4$ at 482 212
\pinlabel $x_5$ at 540 194
\pinlabel $x_6$ at 554 136
\pinlabel $x_7$ at 554 54
\pinlabel $x_8$ at 540 0
\pinlabel $x_9$ at 482 -16
\pinlabel $x_{10}$ at 388 -16
\pinlabel $x_{11}$ at 330 0
\pinlabel $x_{12}$ at 318 56
\endlabellist
\centering
\figs{0.35}{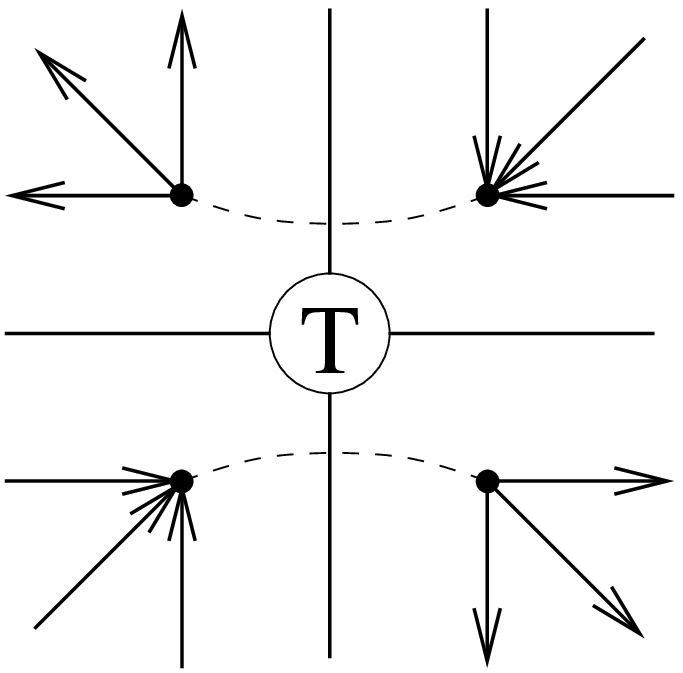}
\qquad\qquad
\figs{0.35}{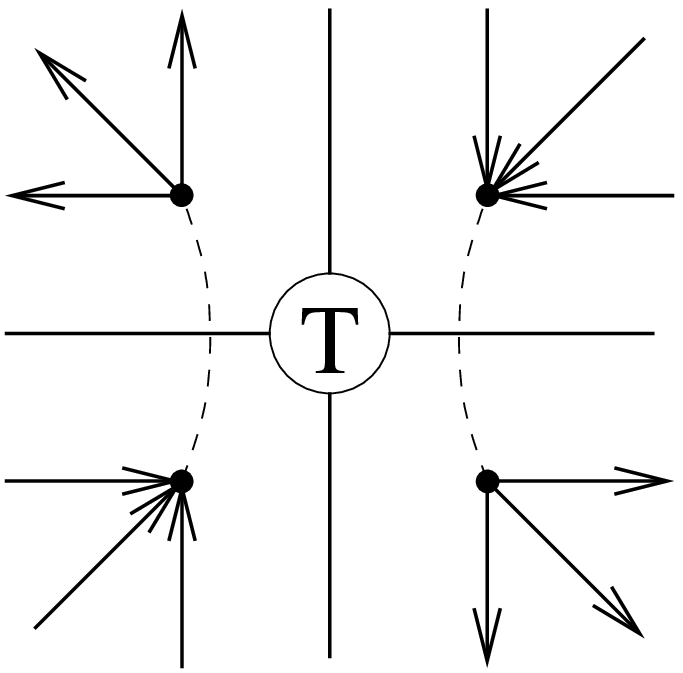}
\put(-154,-30){$\zeta(\hat\Gamma_v)$}\put(-39,-30){$\zeta'(\hat\Gamma_v)$}
\caption{Swapping virtual edges}
\label{mf3:fig:swap}
\end{figure}
\end{lem}

\begin{proof}
Denoting by $\hat{M}$ the tensor product of $\hat T$ with the factorization 
corresponding to the part of the diagram not depicted in Figure~\ref{mf3:fig:swap} 
we have
$$\zeta(\hat\Gamma_v) \cong \hat M\otimes
\begin{Bmatrix}
g_1\ , & x_1+x_2+x_3-x_4-x_5-x_6 \\
g_2\ , & x_1x_2+(x_1+x_2)x_3-x_5x_6-x_4(x_5+x_6)\\
g_3\ , & x_1x_2x_3-x_4x_5x_6 \\
g_1'\ , & x_7+x_8+x_9-x_{10}-x_{11}-x_{12} \\
g_2'\ , & x_8x_9+x_7(x_8+x_9)-x_{10}x_{11}-(x_{10}+x_{11})x_{12} \\
g_3'\ , & x_7x_8x_9-x_{10}x_{11}x_{12}
\end{Bmatrix}\{-6\}$$
with polynomials $g_i$ and $g_i'$ ($i=1,2,3$) given by 
Equations~(\ref{mf3:eqn:gfactor1}-\ref{mf3:eqn:gfactor3}). Similarly
$$\zeta'(\hat\Gamma_v) \cong \hat M\otimes
\begin{Bmatrix}
h_1\ , & x_7+x_8+x_9-x_4-x_5-x_6 \\
h_2\ , & x_8x_9+x_7(x_8+x_9)-x_5x_6-x_4(x_5+x_6) \\
h_3\ , & x_7x_8x_9-x_4x_5x_6  \\
h_1'\ , & x_1+x_2+x_3-x_{10}-x_{11}-x_{12} \\
h_2'\ , & x_1x_2+(x_1+x_2)x_3-x_{10}x_{11}-(x_{10}+x_{11})x_{12}\\
h_3'\ , & x_1x_2x_3-x_{10}x_{11}x_{12}
\end{Bmatrix}\{-6\}$$
where the polynomials $h_i$ and $h_i'$ ($i=1,2,3$) are as above. The factorizations 
$\zeta(\hat\Gamma_v)$ and $\zeta'(\hat\Gamma_v)$ have potential zero.
Using the explicit form   
$$g_3=h_3=4(x_4+x_5+x_6-a),\qquad g_3'=h_3'=4(x_{10}+x_{11}+x_{12}-a)$$
we exclude the variables $x_4$ and $x_{12}$ from the third and sixth rows in 
$\zeta(\hat\Gamma_v)$ and $\zeta'(\hat\Gamma_v)$. This operation transforms 
the factorization 
$\hat M$ into the factorization $\hat M'$, which again is a tensor factor 
which $\zeta(\hat\Gamma_v)$ and 
$\zeta'(\hat\Gamma_v)$ have in common. Ignoring common overall shifts we obtain
$$\zeta(\hat\Gamma_v) \cong \hat M'\otimes
\begin{Bmatrix}
g_1\ , & x_1+x_2+x_3-a \\
g_2\ , & x_1x_2+(x_1+x_2)x_3-x_5x_6-(x_5+x_6)(a-x_5-x_6)\\
g_1'\ , & x_7+x_8+x_9-a \\
g_2'\ , & x_8x_9+x_7(x_8+x_9)-x_{10}x_{11}-(x_{10}+x_{11})(a-x_{10}-x_{11})
\end{Bmatrix}$$
and
$$\zeta'(\hat\Gamma_v) \cong \hat M'\otimes
\begin{Bmatrix}
h_1\ , & x_7+x_8+x_9-a \\
h_2\ , & x_8x_9+x_7(x_8+x_9)-x_5x_6-(x_5+x_6)(a-x_5-x_6) \\
h_1'\ , & x_1+x_2+x_3-a \\
h_2'\ , & x_1x_2+(x_1+x_2)x_3-x_{10}x_{11}-(x_{10}+x_{11})(a-x_{10}-x_{11})
\end{Bmatrix}.$$

\n Using Equation~\ref{mf3:eqn:gfactor1} we see that $g_1=h_1'$ and $g_1'=h_1$ and therefore, 
absorbing in $\hat M'$ the corresponding Koszul factorizations, we can write
$$\zeta(\hat\Gamma_v) \cong \hat M''\otimes\hat{K}
\quad\text{and}\quad 
\zeta'(\hat\Gamma_v) \cong \hat M''\otimes\hat{K}'$$
where
$$\hat{K}=
\begin{Bmatrix}
g_2\ , & x_1x_2+(x_1+x_2)x_3-x_5x_6-(x_5+x_6)(a-x_5-x_6)\\
g_2'\ , & x_8x_9+x_7(x_8+x_9)-x_{10}x_{11}-(x_{10}+x_{11})(a-x_{10}-x_{11})
\end{Bmatrix}$$
and
$$\hat{K}'=
\begin{Bmatrix}
h_2\ , & x_8x_9+x_7(x_8+x_9)-x_5x_6-(x_5+x_6)(a-x_5-x_6) \\
h_2'\ , & x_1x_2+(x_1+x_2)x_3-x_{10}x_{11}-(x_{10}+x_{11})(a-x_{10}-x_{11})
\end{Bmatrix}.$$

\n To simplify notation define the polynomials $\alpha_{i,j,k}$ and $\beta_{i,j}$ by
$$
\alpha_{i,j,k} = x_ix_j+(x_i+x_j)x_k, \qquad
\beta_{i,j}    = x_ix_j+(x_i+x_j)(a-x_i-x_j).
$$
In terms of $\alpha_{i,j,k}$ and $\beta_{i,j}$ we have
\begin{equation}\label{mf3:eq:swapK1}
\hat{K}=
\begin{Bmatrix}
2(\alpha_{1,2,3}+\beta_{5,6}) +4b\ , & \alpha_{1,2,3}-\beta_{5,6} \\
2(\alpha_{7,8,9}+\beta_{10,11})+4b\ , & \alpha_{7,8,9}-\beta_{10,11}
\end{Bmatrix}
\end{equation}
and
\begin{equation}\label{mf3:eq:swapK2}
\hat{K}'=
\begin{Bmatrix}
2(\alpha_{7,8,9}+\beta_{5,6})+4b\ , & \alpha_{7,8,9}-\beta_{5,6} \\
2(\alpha_{1,2,3}+\beta_{10,11})+4b\ , & \alpha_{1,2,3}-\beta_{10,11})
\end{Bmatrix}.
\end{equation}

Factorizations $\hat{K}$ and $\hat{K}'\brak{1}$ can now be written in matrix form as 
$$\hat{K}=
\begin{pmatrix}R \\ R\end{pmatrix}
\xra{P}
\begin{pmatrix}R \\ R\end{pmatrix}\xra{Q}\begin{pmatrix}R \\ R\end{pmatrix}
\quad\text{and}\quad
\hat{K}'\brak{1}=
\begin{pmatrix}R \\ R\end{pmatrix}
\xra{P'}
\begin{pmatrix}R \\ R\end{pmatrix}\xra{Q'}\begin{pmatrix}R \\ R\end{pmatrix},$$
where
\begin{align*}
P &=
\begin{pmatrix}2(\alpha_{1,2,3}+\beta_{5,6})+4b & \alpha_{7,8,9}-\beta_{10,11} \\
2(\alpha_{7,8,9}+\beta_{10,11})+4b & -\alpha_{1,2,3}+\beta_{5,6}\end{pmatrix}
\\[1.2ex]
Q &=
\begin{pmatrix}\alpha_{1,2,3}-\beta_{5,6} & \alpha_{7,8,9}-\beta_{10,11} \\
2(\alpha_{7,8,9}+\beta_{10,11})+4b & -2(\alpha_{1,2,3}+\beta_{5,6})-4b\end{pmatrix}
%\end{align*}
\intertext{and}
%\begin{align*}
P' &=
\begin{pmatrix}-\alpha_{7,8,9}+\beta_{5,6} & -\alpha_{1,2,3}+\beta_{10,11} \\
-2(\alpha_{1,2,3}+\beta_{10,11})-4b & 2(\alpha_{7,8,9}+\beta_{5,6})+4b\end{pmatrix}
\\[1.2ex]
Q' &=
\begin{pmatrix}
-2(\alpha_{7,8,9}+\beta_{5,6})-4b & -\alpha_{1,2,3}+\beta_{10,11} \\
-2(\alpha_{1,2,3}+\beta_{10,11})-4b & \alpha_{7,8,9}-\beta_{5,6}\end{pmatrix}
\end{align*}

\n Define a homomorphism $\psi=(f_0,f_1)$ from $\hat{K}$ to $\hat{K}'\brak{1}$ by the pair of matrices
$$\Biggl(
\begin{pmatrix}
1 & -1/2 \\ -1 & -1/2
\end{pmatrix},
\begin{pmatrix}
1/2 & -1/2 \\ -1 & -1
\end{pmatrix}
\Biggr).$$
It is immediate that $\psi$ is an isomorphism with inverse $(f_1,f_0)$.
It follows that $1_{\hat{M}''}\otimes\psi$ defines an isomorphism between $\zeta(\hat{\Gamma}_v)$ and $\zeta'(\hat{\Gamma}_v)\brak{1}$.
\end{proof}

Although having $\psi$ in this form will be crucial in the proof of Proposition~\ref{mf3:prop:ident-vert} an alternative description will be useful in Section~\ref{mf3:sec:iso}. Note that we can reduce $\hat K$ and $\hat K'$ in 
Equations~\eqref{mf3:eq:swapK1} and~\eqref{mf3:eq:swapK2} further by using the row 
operations $[1,2]_1\circ[1,2]'_{-2}$. We obtain 
$$\hat{K}\cong
\begin{Bmatrix}
-4(\alpha_{7,8,9}-\beta_{5,6}), & \alpha_{1,2,3}-\beta_{5,6} \\
2(\alpha_{7,8,9}+\beta_{10,11}+\alpha_{1,2,3}-\beta_{5,6}), & \alpha_{1,2,3}+\alpha_{7,8,9}-\beta_{5,6}-\beta_{10,11} 
\end{Bmatrix}$$
and
$$\hat{K}'\cong
\begin{Bmatrix}
-4(\alpha_{1,2,3}-\beta_{5,6}), & \alpha_{7,8,9}-\beta_{5,6} \\
2(\alpha_{7,8,9}+\beta_{10,11}+\alpha_{1,2,3}-\beta_{5,6}), & \alpha_{1,2,3}+\alpha_{7,8,9}-\beta_{5,6}-\beta_{10,11} 
\end{Bmatrix}
.$$
Since the second lines in $\hat{K}$ and $\hat{K}'$ are equal we can 
write 
$$\hat{K}\cong
\{
-4(\alpha_{7,8,9}-\beta_{5,6}),\ \alpha_{1,2,3}-\beta_{5,6}\}\otimes \hat{K}_2$$
and
$$\hat{K}'\cong
\{
-4(\alpha_{1,2,3}-\beta_{5,6}),\ \alpha_{7,8,9}-\beta_{5,6}\} \otimes\hat{K}_2.$$
An isomorphism $\psi'$ between $\hat{K}$ and $\hat{K'}\brak{1}$ can now be 
given as the tensor product between $\bigl( -m(2),\ -m(1/2)\bigr)$ and the 
identity homomorphism of $\hat{K}_2$. 

\begin{cor}\label{mf3:cor:swap}
The homomorphisms $\psi$ and $\psi'$ are equivalent.
\end{cor}

\begin{proof}
The first thing to note is that we obtained the homomorphism 
$\psi$ by first writing the differential $(d_0,d_1)$ in $\hat{K}'$ as 
$2\times 2$ matrices and then its shift $\hat{K}'\brak{1}$ using 
$(-d_1,-d_0)$, but in the computation of $\psi'$ we switched the terms 
and changed the signs in the first line of the Koszul matrix corresponding
 to $\hat{K}'$. The two factorizations obtained are isomorphic by a 
non-trivial isomorphism, which is given by
 $$T=\Biggl(
\begin{pmatrix}
-1 & 0 \\ 0 & 1
\end{pmatrix},\
\begin{pmatrix}
-1 & 0 \\ 0 & 1
\end{pmatrix}
\Biggr).$$
Bearing in mind that $\psi$ and $\psi'$ have $\bZ/2\bZ$-degree 1 and using
$$
[1,2]_\lambda = 
\Biggl(\begin{pmatrix}1 &0 \\ 0 & 1\end{pmatrix},\
\begin{pmatrix}1 & -\lambda \\ 0 & 1\end{pmatrix}\Biggr),
\qquad
[1,2]'_{\lambda} = 
\Biggl(\begin{pmatrix}1 &0 \\ -\lambda & 1\end{pmatrix},\
\begin{pmatrix}1 & 0 \\ 0 & 1\end{pmatrix}\Biggr),
$$
it is straightforward to check that the composite homomorphism
$T[1,2]_{1}[1,2]'_{-2}\psi[1,2]'_{2}[1,2]_{-1}$ is
$$\Biggr(
\begin{pmatrix}
-2 & 0 \\ 0 & -1/2
\end{pmatrix},\
\begin{pmatrix}
-1/2 & 0 \\ 0 & -2
\end{pmatrix}
\Biggr)$$
which is the tensor product of $\bigl( -m(2),\ -m(1/2)\bigr)$ and the 
identity homomorphism of $\hat{K}_2$.
\end{proof}

\begin{proof}[Proof of Proposition~\ref{mf3:prop:ident-vert}]
We claim that $\zeta'(\hat\Gamma_v)\cong\zeta(\hat\Gamma_v)\brak{k}$ with $k$ a nonnegative integer. We transform $\zeta'(\hat\Gamma_v)$ into $\zeta(\hat\Gamma_v)\brak{k}$ by repeated application of Lemma~\ref{mf3:lem:swapedges} as follows. Choose a pair of vertices connected by a virtual edge in $\zeta(\hat\Gamma_v)$. Do nothing if the same pair is connected by a virtual edge in $\zeta'(\hat\Gamma_v)$ and use Lemma~\ref{mf3:lem:swapedges} to connect them in the opposite case. Iterating this procedure we can turn $\zeta'(\hat\Gamma_v)$ into $\zeta(\hat\Gamma_v)$ with a shift in the $\bZ/2\bZ$-grading by ($k\mod 2$) where $k$ is the number of times we applied Lemma~\ref{mf3:lem:swapedges}. 

It remains to show that the shift in the $\bZ/2\bZ$-grading is independent of 
the choices one makes. To do so we label the vertices of $\Gamma_v$ of 
($+$)-type and ($-$)-type by $(v_1^+,\ldots,v_k^+)$ and 
$(v_1^-,\ldots,v_k^-)$ respectively. Any complete identification of vertices in  
$\Gamma_v$ is completely determined by an ordered set 
$J_\zeta=(v_{\sigma(1)}^-,\ldots ,v_{\sigma(k)}^-)$,  
with the convention that $v_j^+$ is connected through a virtual edge to 
$v_{\sigma(j)}^-$ for $1\leq j\leq k$. Complete identifications of the vertices in 
$\Gamma_v$ are therefore in one-to-one correspondence with the elements of 
the symmetric group on $k$ letters $S_k$. Any transformation of $\zeta'(\hat\Gamma)$ into 
$\zeta(\hat\Gamma)$ by repeated application of Lemma~\ref{mf3:lem:swapedges} 
corresponds to a sequence of elementary transpositions whose composite is equal to 
the quotient of the permutations corresponding to $J_{\zeta'}$ and $J_\zeta$. 
We conclude that the shift in the $\bZ/2\bZ$-grading is well-defined, because 
any decomposition of a given permutation into elementary transpositions 
has a fixed parity.
\end{proof}

In Section~\ref{mf3:sec:iso} we want to associate an equivalence class of 
matrix factorizations to each closed web and an equivalence class of 
homomorphism to each foam. In order to do that consistently, we 
have to show that the isomorphisms in the proof of 
Proposition~\ref{mf3:prop:ident-vert} are canonical in a certain sense 
(see Corollary~\ref{mf3:cor:caniso}).

Choose an ordering of the vertices of 
$\Gamma_v$ such that $v^+_i$ is paired with $v^-_i$ for all $i$ and let 
$\zeta$ be the corresponding vertex identification. Use the linear entries in the 
Koszul matrix of $\zeta(\hat{\Gamma}_v)$ to exclude one variable corresponding to an 
edge entering in each vertex of $(-)$-type, as in the proof of 
Lemma~\ref{mf3:lem:swapedges}, so that the resulting Koszul factorization 
has the form $\zeta(\hat{\Gamma}_v)=\hat{K}_{lin}\otimes\hat{K}_{quad}$ where 
$\hat{K}_{lin}$ (resp. $\hat{K}_{quad}$) consists of the lines in 
$\zeta(\hat{\Gamma}_v)$ having linear (resp. quadratic) terms as its right entries. 
From the proof of Lemma~\ref{mf3:lem:swapedges} we see that changing a pair of 
virtual edges leaves $\hat{K}_{lin}$ unchanged. 

Let $\sigma_i$ be the element of $S_k$ corresponding to the elementary transposition, 
which sends the complete identification $(v^-_1,\ldots ,v^-_i,v^-_{i+1},\ldots,v^-_k)$ to $(v^-_1,\ldots ,v^-_{i+1},v^-_i,\ldots,v^-_k)$, 
and let $$\Psi_i=1_{/i}\otimes \psi$$ be the corresponding isomorphism of matrix 
factorizations from the proof of Lemma~\ref{mf3:lem:swapedges}. The homomorphism 
$\psi$ only acts 
on the $i$-th and $(i+1)$-th lines 
in $\hat{K}_{quad}$ and $1_{/i}$ is the identity morphism on the remaining lines. 
For the composition $\sigma_i\sigma_j$ we have the composite homomorphism 
$\Psi_i\Psi_j$.

\begin{lem}\label{mf3:lem:repsym}
The assignment $\sigma_i\mapsto \Psi_i$ defines a representation of $S_k$ on 
$\zeta(\hat{\Gamma})_0\oplus\zeta(\hat{\Gamma})_1$.
\end{lem}
\begin{proof}
Let $\hat{K}$ be the Koszul factorization corresponding to the lines $i$ and $i+1$ in $\zeta(\hat{\Gamma}_v)$ and let $\ket{00}$, $\ket{11}$, $\ket{01}$ and $\ket{10}$ be the standard basis vectors of $\hat{K}_0\oplus\hat{K}_1$.
The homomorphism $\psi$ found in the proof of Lemma~\ref{mf3:lem:swapedges} can be written 
as only one matrix acting on $\hat{K}_0\oplus\hat{K}_1$:
$$\psi=\begin{pmatrix}
0 & 0 & \frac{1}{2} & -\frac{1}{2}  \\[1ex]
0 & 0 & -1 & -1  \\[1ex]
1 & -\frac{1}{2} & 0 & 0 \\[1ex]
-1 & -\frac{1}{2} & 0 & 0
\end{pmatrix}.$$
We have that $\psi^2$ is the identity matrix and therefore it follows that $\Psi_i^2$ 
is the identity homomorphism on $\zeta(\hat{\Gamma}_v)$. It is also immediate that 
$\Psi_i\Psi_j=\Psi_j\Psi_i$ for $|i-j|>1$. To complete the proof we need to show that 
$\Psi_i\Psi_{i+1}\Psi_i=\Psi_{i+1}\Psi_i\Psi_{i+1}$, which we do by explicit 
computation of the corresponding matrices. Let $\hat{K}'$ be the Koszul matrix 
consisting of the three lines $i$, $i+1$ and $i+2$ in $\hat{K}_{quad}$. To show 
that $\Psi_i\Psi_{i+1}\Psi_i=\Psi_{i+1}\Psi_i\Psi_{i+1}$ is equivalent to showing 
that $\psi$ satisfies the Yang-Baxter equation
\begin{equation}\label{mf3:eq:YB}
(\psi\otimes 1)(1\otimes\psi)(\psi\otimes 1)=(1\otimes\psi)(\psi\otimes 1)
(1\otimes\psi),
\end{equation}
with $1\otimes\psi$ and $\psi\otimes 1$ acting on $\hat{K}'$. Note that, in general, 
the tensor product of two homomorphisms of matrix factorizations $f$ and $g$ is 
defined by 
$$(f\otimes g)\ket{v\otimes w}=(-1)^{\deg_{\bZ/2\bZ}(g)\deg_{\bZ/2\bZ}(v)}\ket{fv\otimes gw}.$$ Let $\ket{000}$, 
$\ket{011}$, $\ket{101}$, $\ket{110}$, $\ket{001}$, $\ket{010}$, $\ket{100}$ and 
$\ket{111}$ be the standard basis vectors of $\hat{K}'_0\oplus\hat{K}'_1$. 
With respect to this basis the homomorphisms $\psi\otimes 1$ and $1\otimes\psi$ have 
the form of block matrices
$$\psi\otimes 1=
\left(
\begin{array}{c|c}
 0  & 
\begin{array}{cccc}
0 & \scriptstyle{\frac{1}{2}} & -\scriptstyle{\frac{1}{2}} & 0  \\[1ex] 
1 & 0     & 0 & -\scriptstyle{\frac{1}{2}}   \\[1ex]
-1 & 0     & 0 & -\scriptstyle{\frac{1}{2}}  \\[1ex]
 0  & -1 & -1 & 0  
\end{array}
\\ \hline
\begin{array}{cccc}
0  & \scriptstyle{\frac{1}{2}} & -\scriptstyle{\frac{1}{2}} & 0 \\[1ex] 
1 & 0     & 0 & -\scriptstyle{\frac{1}{2}}  \\[1ex]
-1 & 0     & 0 & -\scriptstyle{\frac{1}{2}}  \\[1ex]
0  & -1 & -1 & 0
\end{array} & 0
 \end{array}\right)$$
and
$$1\otimes\psi=
\left(
\begin{array}{c|c}
 0  & 
\begin{array}{cccc}
\scriptstyle{\frac{1}{2}} & -\scriptstyle{\frac{1}{2}} & 0 & 0  \\[1ex] 
-1 & -1     & 0 & 0   \\[1ex]
0 & 0     & -1 & \scriptstyle{\frac{1}{2}}  \\[1ex]
 0  & 0 & 1 & \scriptstyle{\frac{1}{2}} 
\end{array}
\\ \hline
\begin{array}{cccc}
1  & -\scriptstyle{\frac{1}{2}} & 0 & 0 \\[1ex] 
-1 & -\scriptstyle{\frac{1}{2}}     & 0 & 0  \\[1ex]
0 & 0     & -\scriptstyle{\frac{1}{2}} & \scriptstyle{\frac{1}{2}}   \\[1ex]
0  & 0 & 1 & 1
\end{array} & 0
 \end{array}\right)
.$$ 
\medskip

\n By a simple exercise in matrix multiplication we find that both sides 
in Equation~\eqref{mf3:eq:YB} are equal and it follows that 
$\Psi_i\Psi_{i+1}\Psi_i=\Psi_{i+1}\Psi_i\Psi_{i+1}$.
\end{proof}

\begin{cor}
\label{mf3:cor:caniso}
The isomorphism $\zeta'(\hat\Gamma_v)\cong\zeta(\hat\Gamma_v)\brak{k}$ 
in Proposition~\ref{mf3:prop:ident-vert} is uniquely determined by $\zeta'$ and $\zeta$. 
\end{cor}
\begin{proof} Let $\sigma$ be the permutation that relates $\zeta'$ and $\zeta$. 
Recall that in the proof of Proposition~\ref{mf3:prop:ident-vert} we 
defined an isomorphism $\Psi_{\sigma}\colon\zeta'(\hat\Gamma_v)\to 
\zeta(\hat\Gamma_v)\brak{k}$ by writing $\sigma$ as a product of transpositions. 
The choice of these transpositions is not unique in general. However, 
Lemma~\ref{mf3:lem:repsym} shows that $\Psi_{\sigma}$ only depends on $\sigma$.     
\end{proof}

From now on we write $\hat\Gamma$ for the equivalence class of $\hat\Gamma_v$ under 
complete vertex identification. Graphically we represent the vertices of 
$\hat\Gamma$ as 
in Figure~\ref{mf3:fig:g-equiv}.
\begin{figure}[ht!]
\labellist
\small\hair 2pt
\pinlabel $\hat\Gamma_v$ at 50 -34
\pinlabel $\hat\Gamma$ at 300 -34
\endlabellist
\centering
\figs{0.3}{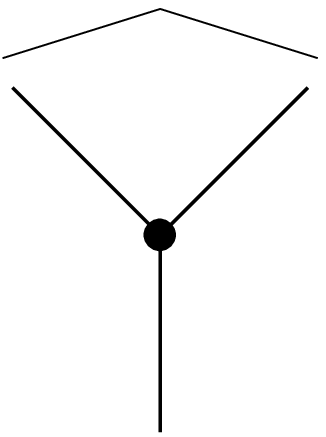}\qquad\qquad
\figs{0.3}{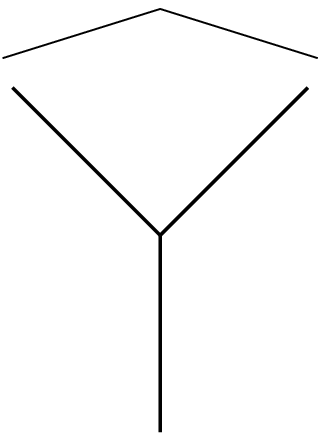}
\bigskip
\caption{A vertex and its equivalence class under vertex identification}
\label{mf3:fig:g-equiv}
\end{figure}
\n We need to neglect the $\bZ/2\bZ$ grading, which we do by imposing 
that $\hat\Gamma$, for any closed $\slt$-web $\Gamma$, have only homology in degree zero, 
applying a shift 
if necessary.
  
Let $\Gamma$ and $\Lambda$ be arbitrary webs. We have to define the 
morphisms between $\hat\Gamma$ and $\hat\Lambda$. Let 
$\zeta(\hat\Gamma_v)$ and $\zeta'(\hat\Gamma_v)$ be representatives of $\hat\Gamma$ 
and $\zeta(\hat\Lambda_v)$ and $\zeta'(\hat\Lambda_v)$ be representatives of 
$\hat\Lambda$. Let 
$$f\in 
\Hom_{\mf{}}\bigl(\zeta(\hat\Gamma_v),\zeta(\hat\Lambda_v)\bigr)$$ 
and $$g\in \Hom_{\mf{}}\bigl(\zeta'(\hat\Gamma_v),\zeta'(\hat\Lambda_v)\bigr)$$ 
be two homomorphisms. 
We say that $f$ and $g$ are equivalent, denoted by $f\sim g$, if and only if 
there exists a commuting square
$$\xymatrix@C=15mm{
\zeta(\hat\Gamma_v) 
\ar[d]^f\ar[r]^(0.45){\cong}
& \zeta'(\hat\Gamma_v) 
\ar[d]^g \\
\zeta(\hat\Lambda_v) 
\ar[r]^(0.45){\cong}
& 
\zeta'(\hat\Lambda_v) 
}$$ 
with the horizontal isomorphisms being of the form as discussed in 
Proposition~\ref{mf3:prop:ident-vert}. 
The composition rule for these equivalence classes of homomorphisms, which relies 
on the choice of representatives within each class, is well-defined by 
Corollary~\ref{mf3:cor:caniso}. Note that we can take well-defined linear combinations of 
equivalence classes of homomorphisms by taking linear combinations of their 
representatives, as long as the latter have all the same source and 
the same target. By Corollary~\ref{mf3:cor:caniso}, homotopy equivalences are 
also well-defined on equivalence classes. We take 
$$\Hom(\hat\Gamma,\hat\Lambda)$$
to be the set of equivalence classes of homomorphisms of matrix factorizations 
between $\hat\Gamma$ and $\hat\Lambda$ modulo homotopy equivalence.

\begin{defn} 
$\widehat{\Qlfoam}$ is the additive category whose objects are equivalence classes of matrix factorizations associated to webs. Two such matrix factorizations are equivalent if the underlying webs are complete vertex identifications of the same $\slt$-web. The morphisms in $\widehat{\Qlfoam}$ are equivalence classes of homomorphisms of matrix factorizations modulo homotopy as defined above.
\end{defn}

Note that we can define the homology of ${\hat\Gamma}$, for any 
closed $\slt$-web $\Gamma$. This group is well-defined up to isomorphism and we denote it 
by $\hat\hy(\Gamma)$.

Next we show how to define a link homology using the objects and morphisms in 
$\widehat{\Qlfoam}$. For any link $L$, first take the universal rational KR complex $\KR_{a,b,c}(L)$. The $i$-th chain group $\KR_{a,b,c}^i(L)$ is 
given by the direct sum of cohomology groups of the form $\hy(\Gamma_v)$, where 
$\Gamma_v$ is a total flattening of $L$. By the remark above it makes sense 
to consider $\widehat{\KR}_{a,b,c}^i(L)$, for each $i$. The differential 
$d^i\colon \KR_{a,b,c}^i(L)\to\KR_{a,b,c}^{i+1}(L)$ induces a map 
$$\hat d^i\colon\widehat{\KR}_{a,b,c}^i(L)\to\widehat{\KR}_{a,b,c}^{i+1}(L),$$
for each $i$. The latter map is well-defined and therefore the homology 
$$\widehat{\HKR}^i_{a,b,c}(L)$$ is well-defined, for each $i$. 

Let $u\colon L\to L'$ be a link cobordism. Khovanov and Rozansky~\cite{KR} constructed 
a cochain map which induces a homomorphism 
$$\HKR_{a,b,c}(u)\colon\HKR_{a,b,c}(L)\to
\HKR_{a,b,c}(L').$$ The latter is only defined up to a $\bQ$-scalar. The induced 
map 
$$\widehat{\HKR}_{a,b,c}(u)\colon\widehat{\HKR}_{a,b,c}(L)\to
\widehat{\HKR}_{a,b,c}(L')$$ is also well-defined up to a $\bQ$-scalar. The 
following result follows immediately:

\begin{lem} $\HKR_{a,b,c}$ and $\widehat{\HKR}_{a,b,c}$ are naturally isomorphic 
as projective functors from $\Link$ to $\Modbg$. 
\end{lem}

\n In the next section we will show that $U^\bQ_{a,b,c}$ and $\widehat{\HKR}_{a,b,c}$ 
are naturally isomorphic as projective functors. 

By Lemma~\ref{KR:lem:KR-abc-moy} we also get the following

\begin{lem}\label{mf3:lem:KhK-mf}
We have the Khovanov-Kuperberg decompositions in $\widehat{\Qlfoam}$:
\begin{gather}
\widehat{\unknot\Gamma} 
\cong
\hatunknot\otimes_{\bQ[a,b,c]}
\hat\Gamma 
\tag{\emph{Disjoint Union}}
\\[1.2ex]
\figins{-9.5}{0.4}{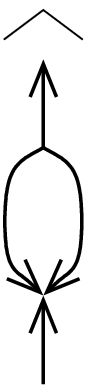}   
\cong 
\figins{-9.5}{0.4}{arc-mf}\{-1\}\oplus 
\figins{-9.5}{0.4}{arc-mf}\{1\}
\tag{\emph{Digon Removal}}
\\[1.2ex]
\figins{-8.5}{0.4}{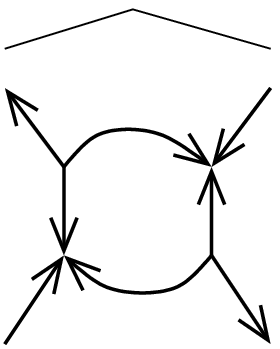}
\cong
\figins{-8.5}{0.405}{sqr1a}\oplus
\figins{-8.5}{0.4}{sqr1b}
\tag{\emph{Square Removal}}
\end{gather}
\end{lem}

\n Although Lemma~\ref{mf3:lem:KhK-mf} follows from Lemma~\ref{KR:lem:KR-abc-moy} and  
Lemma~\ref{mf3:lem:fatedge}, an explicit proof will be useful in the sequel.
\begin{proof}
\emph{Disjoint union} is a direct consequence of the definitions. To prove
%%%%%%
% dr %
%%%%%%
\emph{Digon removal} define the grading-preserving homomorphisms
$$
\alpha\colon\figins{-9.5}{0.4}{arc-mf}\{-1\}   
\to
\figins{-9.5}{0.4}{digon-mf}
\qquad\qquad
\beta\colon\figins{-9.5}{0.4}{digon-mf}   
\to
\figins{-9.5}{0.4}{arc-mf}\{1\}
$$
by Figure~\ref{mf3:fig:dr-mf}.
\begin{figure}[h!]
\labellist
\small\hair 2pt
\pinlabel $\alpha\colon$ at -58 188
\pinlabel $x_1$ at -4 234
\pinlabel $x_2$ at -4 140
\pinlabel $\imath$ at 94 198
\pinlabel $x_1$ at 162 234
\pinlabel $x_2$ at 162 140
\pinlabel -- at 235 184
\pinlabel $x_3$ at 255 184
\pinlabel $x_1$ at 396 234
\pinlabel $x_2$ at 396 140
\pinlabel -- at 398 186
\pinlabel $x_4$ at 380 186
\pinlabel -- at 419 186
\pinlabel $x_3$ at 440 186
\pinlabel $\chi_0$ at 314 199
\pinlabel $\beta\colon$ at -58 48
\pinlabel $x_1$ at -4 94
\pinlabel $x_2$ at -4 0
\pinlabel -- at 0 46
\pinlabel $x_4$ at -18 46
\pinlabel -- at 23 46
\pinlabel $x_3$ at 42 46
\pinlabel $\chi_1$ at 94 59
\pinlabel $x_1$ at 162 94
\pinlabel $x_2$ at 162 0
\pinlabel -- at 235 46
\pinlabel $x_3$ at 255 46
\pinlabel $x_1$ at 396 94
\pinlabel $x_2$ at 396 0
\pinlabel $\epsilon$ at 314 58
\endlabellist
\centering
\figs{0.4}{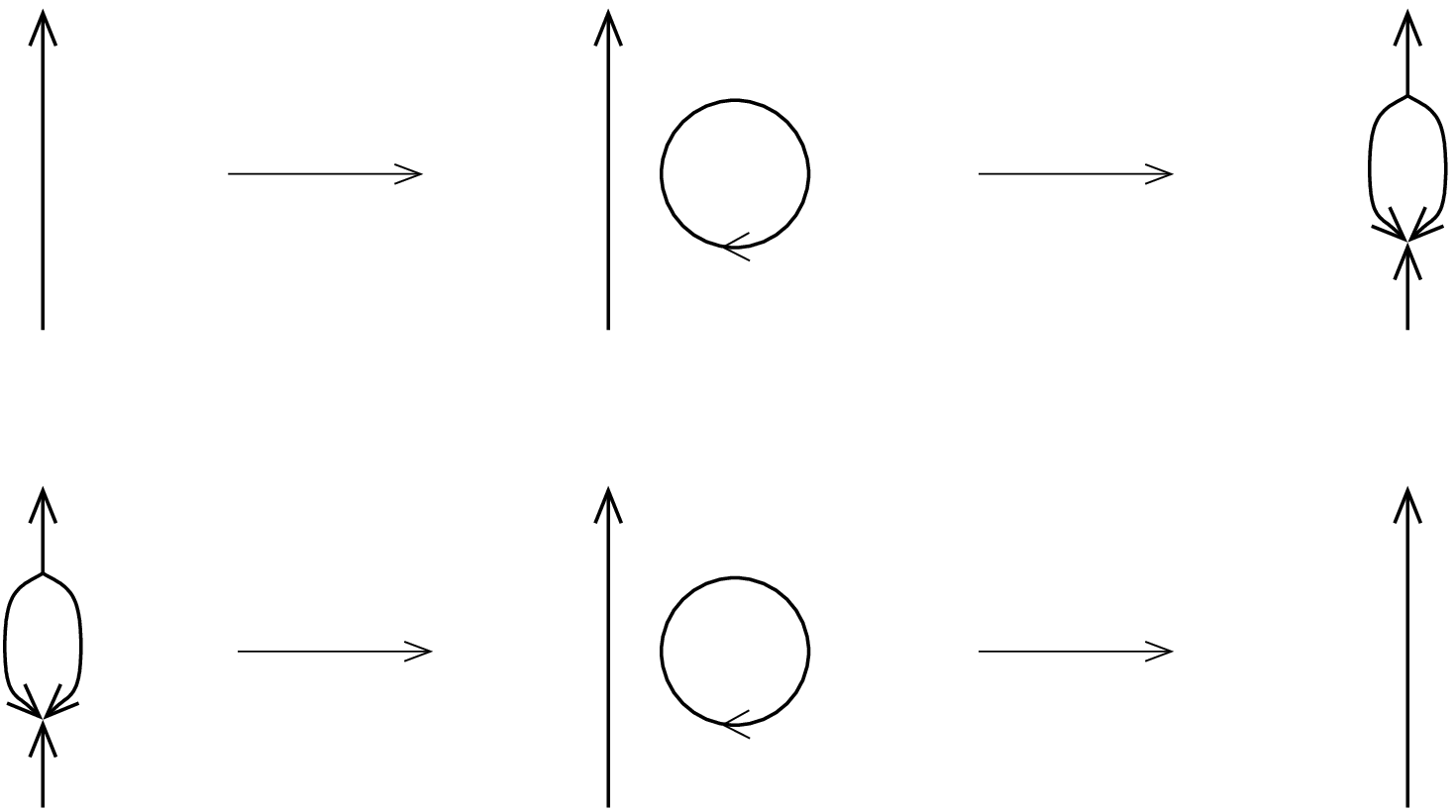}
\caption{Homomorphisms $\alpha$ and $\beta$}
\label{mf3:fig:dr-mf}
\end{figure}

\n If we choose to create the circle on the other side of the arc in 
$\alpha$ we obtain a homomorphism homotopic to $\alpha$ and the same holds for 
$\beta$. Define the homomorphisms 
$$
\alpha_0\colon
\figins{-9.5}{0.4}{arc-mf}\{-1\}   
\to
\figins{-9.5}{0.4}{digon-mf}
\qquad\qquad
\alpha_1\colon
\figins{-9.5}{0.4}{arc-mf}\{1\}   
\to
\figins{-9.5}{0.4}{digon-mf}
$$
by $\alpha_0=2\alpha$ and $\alpha_1=2\alpha\circ m(-x_2)$. Note that 
the homomorphism $\alpha_1$ is homotopic to the homomorphism 
$2\alpha\circ m(x_1+x_3-a)$.
Similarly define
$$
\beta_0\colon\figins{-9.5}{0.4}{digon-mf}   
\to
\figins{-9.5}{0.4}{arc-mf}\{-1\}
\qquad\qquad
\beta_1\colon\figins{-9.5}{0.4}{digon-mf}   
\to
\figins{-9.5}{0.4}{arc-mf}\{1\}
$$
by $\beta_0=-\beta\circ m(x_3)$ and $\beta_1=-\beta$. A simple calculation shows that $\beta_j\alpha_i=\delta_{ij}\id(\hatarc$).
Since the cohomologies of the factorizations $\hatdigon$ and 
$\hatarc\{-1\}\oplus\hatarc\{1\}$ have the 
same graded dimension (see~\cite{KR}) we have that
 $\alpha_0+\alpha_1$ and $\beta_0+\beta_1$ are homotopy inverses of each other and 
that $\alpha_0\beta_0+\alpha_1\beta_1$ is homotopic to the identity in 
$\End(\hatdigon)$.
%%%%%%%
% sqr %
%%%%%%%
To prove (\emph{Square removal}) define grading preserving homomorphisms
$$
\xymatrix@R=3mm{
\psi_0\colon\figins{-9.5}{0.4}{square1-mf}\lra
\figins{-9.5}{0.405}{sqr1a},
&
\psi_1\colon\figins{-9.5}{0.4}{square1-mf}\lra
\figins{-9.5}{0.4}{sqr1b},
\\
\varphi_0\colon\figins{-9.5}{0.4}{sqr1a}\lra
\figins{-9.5}{0.405}{square1-mf},
&
\varphi_1\colon\figins{-9.5}{0.4}{sqr1b}\lra
\figins{-9.5}{0.405}{square1-mf},
}
$$
by the composed homomorphisms below
$$
\xymatrix@C=12mm{
\figins{0}{0.4}{square1-mf}
\ar[r]^{\chi_1\chi_1'} 
\ar@<10pt>@/^1.6pc/[rr]^{\psi_0} &
\figins{0}{0.4}{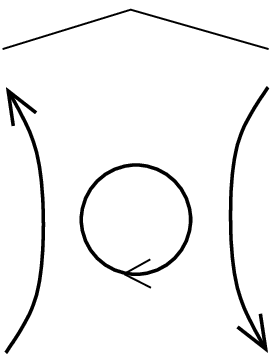}
\ar[r]^{-\varepsilon}
\ar@<6pt>[l]^{\chi_0\chi_0'} &
\figins{0}{0.405}{sqr1a}
\ar@<6pt>[l]^{\imath}
\ar@<10pt>@/^1.6pc/[ll]^{\varphi_0} 
\\
}
\qquad\qquad
\xymatrix@C=12mm{
\figins{0}{0.4}{square1-mf}
\ar[r]^{\overline{\chi}_1\overline{\chi}_1'} 
\ar@<10pt>@/^1.6pc/[rr]^{\psi_1} &
\figins{0}{0.4}{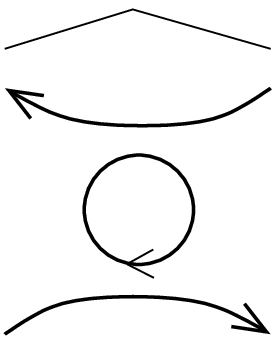}
\ar[r]^{-\varepsilon}
\ar@<6pt>[l]^{\overline{\chi}_0\overline{\chi}_0'} &
\figins{0}{0.405}{sqr1b}
\ar@<6pt>[l]^{\imath}
\ar@<10pt>@/^1.6pc/[ll]^{\varphi_1} 
\\
}.$$
We have that $\psi_0\varphi_0=\id(\hattwoedgesop)$ and 
$\psi_1\varphi_1=\id(\hathtwoedgesop)$. We also have 
$\psi_1\varphi_0=\psi_0\varphi_1=0$ because 
$\Ext(\hattwoedgesop,\hathtwoedgesop)\cong\widehat{\HKR}_{a,b,c}(\unknot)\{4\}$ which is zero in 
$q$-degree zero and so any homomorphism of degree zero between $\hattwoedgesop$ and 
$\hathtwoedgesop$ is homotopic to the zero homomorphism. Since the cohomologies of 
$\hatsquare$ and $\hattwoedgesop\oplus\hathtwoedgesop$ have the same graded dimension 
(see~\cite{KR}) we have that $\psi_0+\psi_1$ and $\varphi_0+\varphi_1$ are homotopy 
inverses of each 
other and that $\varphi_0\psi_0+\varphi_1\psi_1$ is homotopic to the identity in 
$\End(\hatsquare)$.
\end{proof}

%
%
%
%%%%%%%%%%%%%%%%%%%%%%%%%%%%%%%%%%%%%%%%
%%%                                  %%%
%%%        Isomorphism               %%%
%%%                                  %%%
%%%%%%%%%%%%%%%%%%%%%%%%%%%%%%%%%%%%%%%%
%
%
\section{The equivalence functor}\label{mf3:sec:iso}

In this section we first define a functor 
$$\widehat{}\,\,\colon \Qlfoam\to \widehat{\Qlfoam}.$$
 Then we show that this functor is well-defined and an isomorphism 
of categories. Finally we show that this implies that the link 
homology functors $U^\bQ_{a,b,c}$ and $\widehat{\HKR}_{a,b,c}$ from $\Link$ to 
$\Modbg$ are naturally isomorphic.

On objects the functor is defined by $$\Gamma\to\hat\Gamma,$$
as explained in the previous section. 
Let $f\in \Hom_{\Qlfoam}(\Gamma,\Gamma')$. Suffice it to consider the case in which 
$f$ can be given by one singular cobordism, also denoted $f$. 
If $f$ is given by a linear combination of 
singular cobordisms, one can simply extend the following arguments to all terms. 
Slice $f$ up between critical 
points, so that each slice contains one elementary foam, i.e. a zip or unzip, 
a saddle-point cobordism, or a cap or a cup, 
glued horizontally to the identity foam on the rest of the source and target $\slt$-webs. 
For each slice choose compatible 
vertex identifications on the source and target $\slt$-webs, such that in the region where 
both $\slt$-webs are isotopic the vertex 
identifications are the same and in the region where they differ the 
vertex identifications are such that we can apply 
the homomorphism of matrix factorizations $\chi_0,\chi_1,\eta, \imath$ or $\epsilon$. 
This way we get a homomorphism of matrix factorizations for each slice. We can take 
its \,\,$\widehat{}$\,\, equivalence class. Composing all these morphisms gives a 
morphism $\hat{f}$ between $\hat\Gamma_v$ and 
$\hat\Gamma'_v$. For its definition 
we had to choose a representative singular cobordism of the foam $f$, a way to slice 
it up and complete vertex identifications for the source and target of each 
slice. Of course we have to show that  
$\hat{f}\in\Hom_{\widehat{\Qlfoam}}(\hat\Gamma,\hat\Gamma')$ is independent of those 
choices. But to do that we first need a result about the homomorphism of matrix factorizations $\eta$ induced by a saddle-point cobordism. In general $\eta$ is defined up to a sign (see~\cite{KR}) but in the case of $\slt$-webs we have the following
\begin{lem}
The map $\hat\eta$ is well defined for closed $\slt$-webs. 
\end{lem}
\begin{proof}
Let $\Gamma$ and $\Gamma'$ be two closed $\slt$-webs and $\Sigma\colon \Gamma\to\Gamma'$ a 
cobordism which is the identity everywhere except for one saddle-point. By a slight 
abuse of notation, let $\hat\eta$ denote the homomorphism of matrix factorizations 
which corresponds to $\Sigma$. Note that $\Gamma$ and $\Gamma'$ have the same number 
of 
vertices, which we denote by $v$. Our proof that $\hat\eta$ is well-defined proceeds 
by induction on $v$. If $v=0$, then the lemma holds, because $\Gamma$ consists of 
one circle and $\Gamma'$ of two circles, or vice-versa. These circles have no marks 
and are therefore indistinguishable. To each circle we associate the 
complex $\hatunknot$ and $\hat\eta$ corresponds to the product or the coproduct in 
$\bQ[a,b,c][X]/{X^3-aX^2-bX-c}$. Note that as soon as we mark the two circles, they 
will become distinguishable, and a minus-sign creeps in when we switch them. 
However, this minus-sign then cancels against the minus-sign 
showing up in the homomorphism associated to the saddle-point cobordism.  

Let $v>0$. This part of our proof uses some ideas from the proof of Theorem 2.4 in~\cite{JK}. Any $\slt$-web can be seen as lying on a 2-sphere. Let $V,E$ and $F$ denote the 
number of vertices, edges and faces of an $\slt$-web, where a face is a connected component 
of the complement of the $\slt$-web in the 2-sphere. Let $F=\sum_iF_i$, where 
$F_i$ is the number of faces with $i$ edges. Note that we only have faces with an 
even number of edges. It is easy to see that the following 
equations hold:
\begin{align*}
3V &= 2E\\
V-E+F &= 2\\
2E &= \sum_iiF_i.
\end{align*}    
Therefore, we get
$$6=3F-E=2F_2+F_4-F_8-\ldots,$$
which implies 
\begin{equation}
\label{mf3:eq:lb}
6\leq 2F_2+F_4.
\end{equation}
\n This lower bound holds for any $\slt$-web, in particular for $\Gamma$ and $\Gamma'$. 

Note that for $F_2=3$ and $F_4=0$ we have a theta-web, which is the 
intersection of a theta-foam and a plane. Since all edges have 
to have the same orientation at both vertices, there is no way to apply 
a saddle point cobordism to this web. For all other values of $F_2$ and 
$F_4$ there is always a 
digon or a square in $\Gamma$ and $\Gamma'$ 
on which $\hat\eta$ acts as the identity, i.e. which does not get changed by the 
saddle-point in the cobordism to which $\hat\eta$ corresponds. To see how 
this follows from~\eqref{mf3:eq:lb}, just note that one saddle-point cobordism 
never involves more than three squares, one digon and two squares, or 
two digons and one square. Since the MOY-moves 
in Lemma~\ref{mf3:lem:KhK-mf} are all given by isomorphisms which correspond to the 
zip and the 
unzip and the birth and the death of a circle (see the proof of 
Lemma~\ref{mf3:lem:KhK-mf}), this shows that there is 
always a set of MOY-moves which can be applied both to $\Gamma$ and $\Gamma'$ 
whose target $\slt$-webs, say $\Gamma_1$ and $\Gamma'_1$, have less 
than $v$ vertices, and which commute with $\hat\eta$. Here we denote the homomorphism 
of matrix factorizations corresponding to the saddle-point cobordism between 
$\Gamma_1$ and $\Gamma'_1$ by $\hat{\eta}$ again. By induction, 
$\hat\eta\colon\hat{\Gamma}_1\to \hat{\Gamma'}_1$ is well-defined. 
Since the MOY-moves commute with $\hat\eta$, we conclude that 
$\hat\eta\colon\hat\Gamma\to\hat\Gamma'$ is well-defined.  
\end{proof}

\begin{lem}\label{mf3:lem:widehat} 
The functor\,\, $\widehat{}$\,\, is well-defined.  
\end{lem}
\begin{proof} The fact that $\hat{f}$ does not depend on the vertex identifications 
follows immediately from Corollary~\ref{mf3:cor:caniso} and the equivalence 
relation $\sim$ on the Hom-spaces in $\widehat{\Qlfoam}$.

Next we prove that $\hat{f}$ does not depend on the way we have sliced it up. 
By Lemma~\ref{mf3:lem:KhK-mf} we know that, for any closed $\slt$-web $\Gamma$, 
the class $\hat\Gamma$ is homotopy equivalent to a direct sum of terms of the form 
$\hatunknot^k$. Note that $\widehat{\Ext}(\emptyset,\unknot)$ is generated by 
$X^s\imath$, for 
$0\leq s\leq 2$, and that all maps in the proof of Lemma~\ref{mf3:lem:KhK-mf} 
are induced by cobordisms with a particular slicing. 
This shows that $\widehat{\Ext}(\emptyset,\Gamma)$ 
is generated by maps of the form $\hat{u}$, where $u$ is a cobordism between 
$\emptyset$ and $\Gamma$ with a particular slicing. A similar result holds for 
$\widehat{\Ext}(\Gamma,\emptyset)$. Now let $f$ and $f'$ be given by the 
same cobordism between $\Gamma$ and $\Lambda$ but with different slicings. If 
$\hat{f}\ne\hat{f'}$, then, by the previous arguments, there exist maps  
$\hat{u}$ and $\hat{v}$, where $u\colon \emptyset\to \Gamma$ and 
$v\colon\Lambda\to\emptyset$ are cobordisms with particular slicings, 
such that $\widehat{vfu}\ne\widehat{vf'u}$. 
This reduces the question of independence of 
slicing to the case of closed cobordisms. Note that we already know that 
\,\,$\widehat{}$\,\, is well-defined on the parts that do not involve 
singular circles, because it is the generalization of a 2d TQFT. 
It is therefore easy to see that \,\,$\widehat{}$\,\, respects the 
relation (CN). Thus we can cut up any closed singular cobordism near the 
singular circles to obtain a linear combination of closed singular 
cobordisms isotopic to spheres and theta-foams. The spheres do not have 
singular circles, 
so \,\,$\widehat{}$\,\, is well-defined on them and it is easy to check 
that it respects the relation (S). 

Finally, for theta-foams we do have to 
check something. There is one basic Morse move that can be applied to 
one of the discs of a theta-foam, which we show in 
Figure~\ref{mf3:fig:SingMorse}. We have to show that \,\,$\widehat{}$\,\, 
is invariant under this Morse move. 

\begin{figure}[ht!]
\centering
\figins{0}{0.65}{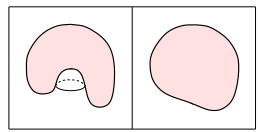}
\caption{Singular Morse move}
\label{mf3:fig:SingMorse}
\end{figure}

\n In other words, we have to show that the composite homomorphism $\Phi$ in 
Figure~\ref{mf3:fig:smorse} is homotopic to the identity. It suffices to do 
the computation on the homology. 
\begin{figure}[h!]
\centering
\labellist
\small\hair 2pt
\pinlabel $x_1$ at  -3 82
\pinlabel $x_2$ at  44 37
\pinlabel $x_1$ at  96 83
\pinlabel $x_2$ at 145 26
\pinlabel $x_2$ at 234 15
\pinlabel $x_3$ at 219 09
\pinlabel $x_2$ at 329 15
\pinlabel $x_3$ at 314 09
\pinlabel $x_3$ at 407 12
\endlabellist
\figs{0.73}{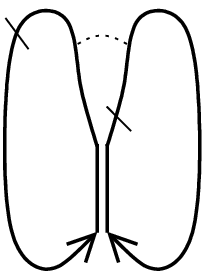}\raisebox{26 pt}{$\xra{\ \chi_0\ }$}
\figs{0.745}{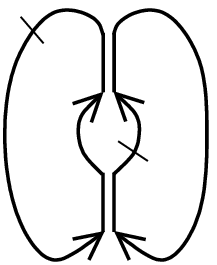}\raisebox{26 pt}{$\xra{\ \psi\ }$}
\figs{0.67}{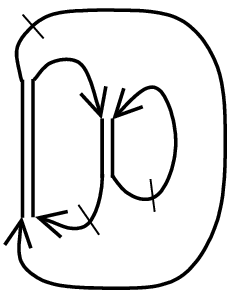}\raisebox{26 pt}{$\xra{\ \chi_1\ }$}
\figs{0.67}{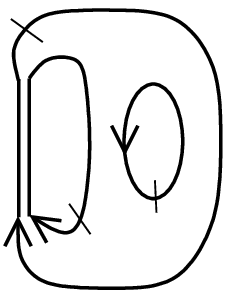}\raisebox{26 pt}{$\xra{\ \varepsilon\ }$}
\figs{0.67}{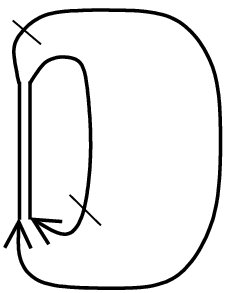}\raisebox{26 pt}{$\xra{\ \tilde{\psi}\ }$}
\figs{0.73}{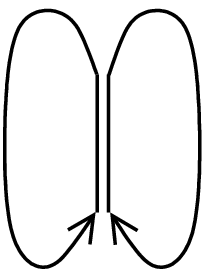}
\caption[Homomorphism $\Phi$]{Homomorphism $\Phi$. To avoid cluttering only some marks are shown}
\label{mf3:fig:smorse}
\end{figure}
First we note that the theta $\slt$-web has 
homology only in $\bZ/2\bZ$-degree $0$. From the remark at the end of  
the definition of the maps $\chi_0$ and $\chi_1$ in Section~\ref{KR:sec:KR-abc} it follows that $\chi_0$ is equivalent to  
multiplication by $-2(x_1-x_2)$ and $\chi_1$ to multiplication by 
$x_2-x_3$, where we used the fact that $\psi$ has $\bZ/2\bZ$-degree $1$. 
From Corollary~\ref{mf3:cor:swap} we have that $\psi$ is equivalent to 
multiplication by $-2$ and from the definition of vertex identification it 
is immediate that $\tilde{\psi}$ is the identity. Therefore we have that
$$\Phi=\varepsilon\bigl(4(x_2-x_3)(x_1-x_2)\bigr)=1.$$
It is also easy to check that \,\,$\widehat{}$\,\, respects the relation 
($\Theta$). 

Note that the arguments above also show that, for an open foam $f$, 
we have $\hat{f}=0$ if $u_1fu_2=0$ for all singular cobordisms 
$u_1\colon \emptyset\to \Gamma_v$ and $u_2\colon\Gamma'_v\to\emptyset$. 
This proves that \,\,$\widehat{}$\,\, is well-defined on foams, which are 
equivalence classes of singular cobordisms. 
\end{proof}

\begin{cor}\label{mf3:cor:widehat}
\,\,$\widehat{}$\,\, is an isomorphism of categories.  
\end{cor}
\begin{proof} On objects \,\,$\widehat{}$\,\, is clearly a bijection. On morphisms it 
is also a bijection by Lemma~\ref{mf3:lem:KhK-mf} and the proof of Lemma~\ref{mf3:lem:widehat}.
\end{proof}

\begin{thm}
The projective functors $U^\bQ_{a,b,c}$ and $\widehat{\HKR}_{a,b,c}$ from $\Link$ to 
$\Modbg$ are naturally isomorphic. 
\end{thm}
\begin{proof}
Let $D$ be a diagram of $L$, $C_{\Qlfoam}(D)$ the complex for $D$ constructed with 
foams in Section~\ref{mf3:sec:foam} and $\widehat{KR}_{a,b,c}(D)$ the complex constructed with 
equivalence classes of matrix factorizations in Section~\ref{mf3:sec:sl3-mf}.
From Lemma~\ref{mf3:lem:widehat} and Corollary~\ref{mf3:cor:widehat} it follows that for 
all $i$ we have isomorphisms of graded $\bQ[a,b,c]$-modules 
$C^i_{\Qlfoam}(D)\cong \widehat{KR}^i_{a,b,c}(D)$ where $i$ is 
the homological degree. By a slight abuse of notation we denote these isomorphisms 
by $\,\,\widehat\,\,$ too. The differentials in $\widehat{KR}_{a,b,c}(D)$ are 
induced by $\chi_0$ and $\chi_1$, which are exactly the maps that we associated to 
the zip and the unzip. This shows that $\,\,\widehat{}\,\,$ commutes with the 
differentials in both complexes and therefore that it defines an isomorphism of 
complexes.  

The naturality of the isomorphism between the two functors follows from 
Corollary~\ref{mf3:cor:widehat} and the fact that all elementary link cobordisms 
are induced by the elementary foams and their respective images w.r.t.  
$\,\,\widehat\,\,$. 
\end{proof}

%%%%%%%%%%%%%%%%%%%%%%%%%%%%%%%%%%%%%%%%%%%%%%%%%%%
%
%
%%%%%%%%%%%%%%%%%%%
%                 %
%     chapter     %
%                 %
%%%%%%%%%%%%%%%%%%%
\chapter{$\sln$-link homology ($N\geq 4$) using foams and the Kapustin-Li formula}
\label{chap:foamN}
%
%
%%%%%%%%%%%%%%%%%%%%%%%%%%%%%%%%%%%%%%%%
%%%                                  %%%
%%%        Introduction              %%%
%%%                                  %%%
%%%%%%%%%%%%%%%%%%%%%%%%%%%%%%%%%%%%%%%%

In this chapter we use foams, as in~\cite{bar-natancob, khovanovsl3, mackaay-vaz}, for an 
almost completely combinatorial topological construction of 
a rational link homology categorifying the 
$\sln$ link 
polynomial for $N\geq 4$. To evaluate closed foams we use the KL formula adapted to foams by Khovanov and Rozansky~\cite{KR-LG}. Our theory is functorial under link cobordisms up to scalars. We show that for any link our homology is isomorphic to KR homology of~\cite{KR} and described in Section~\ref{KR:sec:KR}.

In Section~\ref{slN:sec:partial flags} we recall some 
basic facts about Schur polynomials and the cohomology of partial flag 
varieties. In Section~\ref{slN:sec:pre-foam} we define pre-foams and their 
grading. In Section~\ref{slN:sec:KL} we explain the KL formula for 
evaluating closed pre-foams and compute the spheres and the theta-foams. 
In Section~\ref{slN:sec:foamN} we derive a set of basic relations in the 
category $\foam$, which is the quotient of the category of pre-foams by 
the kernel of the KL evaluation. In Section~\ref{slN:sec:invariance} 
we show that our link homology complex is homotopy invariant under the 
Reidemeister moves. In Section~\ref{slN:sec:functoriality} we show that our link 
homology complex extends to a link homology functor. In 
Section~\ref{slN:sec:taut-functor} we show that our link homology, obtained from 
our link homology complex using the tautological functor, 
categorifies the $\sln$-link polynomial and that it is isomorphic to the Khovanov-Rozansky link homology. In Section~\ref{slN:sec:2m-torus} we conjecture that the theory is integral. This conjecture is supported by the set of local relations introduced in Section~\ref{slN:sec:foamN}, which are defined over $\bZ$. Finally, using some techniques of web and complex simplification, we compute the conjectured integral $\sln$-homology of $(2,m)$-torus links and show that it has torsion of order $N$.
%
%
%
%%%%%%%%%%%%%%%%%%%%%%%%%%%%%%%%%%%%%%
%%%%%                        %%%%%%%%%
%%%%% Partial flag varieties %%%%%%%%%
%%%%%                        %%%%%%%%%  
%%%%%%%%%%%%%%%%%%%%%%%%%%%%%%%%%%%%%%
%
\section{Schur polynomials and the cohomology of partial flag varieties}
\label{slN:sec:partial flags}
In this section we recall some basic facts about Schur polynomials and 
the cohomology of partial flag varieties which we need in the rest of 
this thesis.  

%%%%%%%%%%%%%%%%%%%%%%%%%%%%%%
\subsection{Schur polynomials}
\label{slN:sec:Schur}
%%%%%%%%%%%%%%%%%%%%%%%%%%%%%%
A nice basis for homogeneous symmetric polynomials is given by the Schur 
polynomials. If $\lambda=(\lambda_1,\ldots,\lambda_k)$ is a partition such that $\lambda_1\ge\ldots\ge \lambda_k\ge 0$, then the Schur polynomial $\pi_{\lambda}(x_1,\ldots,x_k)$ is given by the following expression:
\begin{equation}
\pi_{\lambda}(x_1,\ldots,x_k)=\frac{\det\bigl(x_i^{\lambda_j+k-j}\bigr)}{\Delta},
\label{slN:eq:sur}
\end{equation}
where $\Delta=\prod_{i<j}(x_i-x_j)$, and by $\det(x_i^{\lambda_j+k-j})$, we have denoted the determinant of the $k\times k$ matrix whose $(i,j)$ entry is equal to $x_i^{\lambda_j+k-j}$. Note that the elementary symmetric polynomials 
are given by $\pi_{1,0,0,\ldots,0}, \pi_{1,1,0,\ldots,0},\ldots,  
\pi_{1,1,1,\ldots,1}$. There are multiplication rules for the Schur polynomials which show that 
any $\pi_{\lambda_1,\lambda_2,\ldots,\lambda_k}$ can be expressed in terms of the 
elementary symmetric polynomials. 

If we do not specify the variables of the Schur polynomial $\pi_{\lambda}$, we will assume that these are exactly $x_1,\ldots,x_k$, with $k$ being the length of $\lambda$, i.e. $$\pi_{\lambda_1,\ldots,\lambda_k}:=\pi_{\lambda_1,\ldots,\lambda_k}(x_1,\ldots,x_k).$$ 

In this thesis we only use Schur polynomials of two and 
three variables. In the case of two variables, the Schur polynomials are 
indexed by pairs of nonnegative integers $(i,j)$, such that $i\ge j$, and~\eqref{slN:eq:sur} becomes
$$\pi_{i,j}=\sum_{\ell =j}^i{x_1^\ell x_2^{i+j-\ell}}.$$
Directly from \emph{Pieri's formula} we obtain the following multiplication rule for the Schur polynomials in two variables:
\begin{equation}
\pi_{i,j}\pi_{a,b}=\sum {\pi_{x,y}},
\label{slN:eq:mnoz2}
\end{equation}
where the sum on the r.h.s. is over all indices $x$ and $y$ such that $x+y=i+j+a+b$ and $a+i\ge x\ge \max(a+j,b+i)$. Note that this implies $\min(a+j,b+i)\ge y\ge b+j$. Also, we shall write $\pi_{x,y}\in\pi_{i,j}\pi_{a,b}$ if $\pi_{x,y}$ belongs to the sum on the r.h.s. of~\eqref{slN:eq:mnoz2}. Hence, we have that $\pi_{x,x}\in\pi_{i,j}\pi_{a,b}$ iff $a+j=b+i=x$ and $\pi_{x+1,x}\in\pi_{i,j}\pi_{a,b}$ iff $a+j=x+1$, $b+i=x$ or $a+j=x$, $b+i=x+1$.

We shall need the following combinatorial result which expresses the 
Schur polynomial in three variables as a combination of Schur polynomials 
of two variables.
For $i\ge j\ge k \ge 0$, and the triple $(a,b,c)$ of nonnegative integers, we define
$$(a,b,c)\sqsubset(i,j,k),$$
if $a+b+c=i+j+k$, $i\ge a \ge j$ and $j\ge b \ge k$. We note that this implies that $i\ge c\ge k$, and hence $\max\{a,b,c\}\le i$.
\begin{lem}\label{slN:lem1}
$$\pi_{i,j,k}(x_1,x_2,x_3)=\sum_{(a,b,c)\sqsubset(i,j,k)}{\pi_{a,b}(x_1,x_2)x_3^c}.$$
\end{lem}

\begin{proof}
From the definition of the Schur polynomial, we have
$$
\pi_{i,j,k}(x_1,x_2,x_3)={\frac{(x_1x_2x_3)^k}{(x_1-x_2)(x_1-x_3)(x_2-x_3)}}
\det
\begin{pmatrix}
x_1^{i-k+2} & x_1^{j-k+1} & 1 \\
x_2^{i-k+2} & x_2^{j-k+1} & 1 \\
x_3^{i-k+2} & x_3^{j-k+1} & 1
\end{pmatrix}
.$$
After subtracting the last row from the first and the second one of the last determinant, we obtain
$$
\pi_{i,j,k}={\frac{(x_1x_2x_3)^k}{(x_1-x_2)(x_1-x_3)(x_2-x_3)}}
\det
\begin{pmatrix}
x_1^{i-k+2}-x_3^{i-k+2} & x_1^{j-k+1}-x_3^{j-k+1} \\
x_2^{i-k+2}-x_3^{i-k+2} & x_2^{j-k+1}-x_3^{j-k+1} 
\end{pmatrix}
,$$
and so
$$\pi_{i,j,k}={\frac{(x_1x_2x_3)^k}{x_1-x_2}}
\det
\begin{pmatrix}
\sum_{m=0}^{i-k+1} x_1^{m}x_3^{i-k+1-m} & \sum_{n=0}^{j-k} x_1^{n}x_3^{j-k-n}\\
\sum_{m=0}^{i-k+1} x_2^{m}x_3^{i-k+1-m} & \sum_{n=0}^{j-k} x_2^{n}x_3^{j-k+n}
\end{pmatrix}
.$$
Finally, after expanding the last determinant we obtain
\begin{equation}
\pi_{i,j,k}=\frac{(x_1x_2x_3)^k}{x_1-x_2}\sum_{m=0}^{i-k+1}\sum_{n=0}^{j-k}{(x_1^mx_2^n-x_1^nx_2^m)x_3^{i+j-2k+1-m-n}}.
\label{slN:eq:pf1}
\end{equation}
We split the last double sum into two: the first one when 
$m$ goes from $0$ to $j-k$, denoted by $S_1$, and the other 
one when $m$ goes from $j-k+1$ to $i-k+1$, denoted by $S_2$. 
To show that $S_1=0$, we split the double sum further into 
three parts: when $m<n$, $m=n$ and $m>n$. Obviously, each 
summand with $m=n$ is equal to $0$, while the summands of the sum for $m<n$ are exactly the opposite of the summands of the sum for $m>n$. Thus, by replacing only $S_2$ instead of the double sum in~\eqref{slN:eq:pf1} and after rescaling the indices $a=m+k-1$, $b=n+k$, we get
\begin{align*}
\pi_{i,j,k} &= \frac{(x_1x_2x_3)^k}{x_1-x_2}\sum_{m=j-k+1}^{i-k+1}\sum_{n=0}^{j-k}{(x_1^mx_2^n-x_1^nx_2^m)x_3^{i+j-2k+1-m-n}} \\
&= \sum_{a=j}^i\sum_{b=k}^j{\pi_{a,b}x_3^{i+j+k-a-b}}=\sum_{(a,b,c)\sqsubset(i,j,k)}\pi_{a,b}x_3^c,
\end{align*}
as wanted.  
\end{proof}

Of course there is a multiplication rule for three-variable Schur 
polynomials which is compatible with~\eqref{slN:eq:mnoz2} and the lemma above, 
but we do not want to discuss it here. For details see~\cite{Fulton-Harris}.

%%%%%%%%%%%%%%%%%%%%%%%%%%%%%%%%%%%%%%%%%%%%%%%%%%%%%
\subsection{The cohomology of partial flag varieties}
%%%%%%%%%%%%%%%%%%%%%%%%%%%%%%%%%%%%%%%%%%%%%%%%%%%%%
\label{slN:ssec:hflag}
In this chapter the rational cohomology rings of partial flag varieties 
play an essential role. The partial flag variety $Fl_{d_1,d_2,\ldots,d_l}$, 
for $1\le d_1<d_2<\ldots<d_l=N$, is defined by 
$$Fl_{d_1,d_2,\ldots,d_l}=\{V_{d_1}\subset V_{d_2}\subset\ldots\subset 
V_{d_l}=\bC^N|\dim(V_i)=i\}.$$
A special case is $Fl_{k,N}$, the Grassmannian 
variety of all $k$-planes in $\bC^N$, also denoted $\cG_{k,N}$. The dimension of the partial 
flag variety is given by
$$\dim Fl_{d_1,d_2,\ldots,d_l}=N^2-\sum_{i=1}^{l-1}(d_{i+1}-d_{i})^2-d_1^2.$$
The rational cohomology rings of the partial flag varieties are well known 
and we only recall those facts that we need in this thesis. 
\begin{lem}
$\hy(\cG_{k,N})$ is isomorphic to the vector space generated by 
all $\pi_{i_1,i_2,\ldots,i_k}$ modulo the relations    
\begin{equation}
\pi_{N-k+1,0,\ldots,0}=0,\quad\pi_{N-k+2,0,\ldots,0}=0,\quad \ldots\quad,\
 \pi_{N,0,\ldots,0}=0,
\label{slN:eq:gras}
\end{equation}
where there are exactly $k-1$ zeros in the multi-indices of the Schur 
polynomials. 
\end{lem}
\n A consequence of the multiplication rules for Schur polynomials is 
that
\begin{cor} The Schur polynomials $\pi_{i_1,i_2,\ldots,i_k}$, 
for $N-k\geq i_1\geq i_2\geq \ldots\geq i_k\geq 0$, form a basis 
of $\hy(\cG_{k,N})$
\end{cor}
\n Thus, the dimension of $\hy(\cG_{k,N})$ is $\binom{N}{k}$, and up to a 
degree shift, 
its graded dimension is $\qbin{N}{k}$.

Another consequence of the multiplication rules is that 
\begin{cor} The Schur polynomials $\pi_{1,0,0,\ldots,0}, 
\pi_{1,1,0,\ldots,0},\ldots,\pi_{1,1,1,\ldots,1}$ (the elementary 
symmetric polynomials) generate $\hy(\cG_{k,N})$ as a ring. 
\end{cor}

Furthermore, we can introduce a non-degenerate trace form on 
$\hy(\cG_{k,N})$ 
by giving its values on the basis elements 
\begin{equation}
\varepsilon(\pi_{\lambda})=
\begin{cases}
(-1)^{\lfloor\frac{k}{2}\rfloor}, & \lambda=(N-k,\ldots,N-k)\\
0, & \text{else} 
\end{cases}
.\label{slN:eq:trag}
\end{equation}
This makes $\hy(\cG_{k,N})$ into a commutative Frobenius algebra. 
One can compute the basis dual to $\{\pi_{\lambda}\}$ in 
$\hy(\cG_{k,N})$, 
with respect to $\epsilon$. It is given by    
\begin{equation}
\label{slN:eq:db}
\widehat{\pi}_{\lambda_1,\ldots,\lambda_k}=(-1)^{\lfloor\frac{k}{2}\rfloor}
\pi_{N-k-\lambda_k,\ldots,N-k-\lambda_1}.
\end{equation}

We can also express the cohomology rings of the partial 
flag varieties $Fl_{1,2,N}$ and $Fl_{2,3,N}$ in terms of Schur polynomials. 
Indeed, we have 
\begin{equation}\label{slN:eq:hflag}
\begin{split}
\hy(Fl_{1,2,N}) &=
\quotient{\bQ[x_1,x_2]}{ (\pi_{N-1,0},\pi_{N,0}) }, \\
\hy(Fl_{2,3,N}) &=
\quotient{\bQ[x_1+x_2,x_1x_2,x_3]}{ (\pi_{N-2,0,0},
\pi_{N-1,0,0},\pi_{N,0,0}) }.
\end{split}
\end{equation}

The natural projection map $p_1:Fl_{1,2,N}\to\cG_{2,N}$ induces 
\begin{equation}\label{slN:eq:proj1}
p^*_1:\hy(\cG_{2,N})\to \hy(Fl_{1,2,N}),
\end{equation}
which is just the inclusion of the polynomial rings. Analogously, 
the natural 
projection map $p_2:Fl_{2,3,N}\to\cG_{3,N}$, 
induces 
\begin{equation}\label{slN:eq:proj2}
p^*_2:\hy(\cG_{3,N})\to \hy(Fl_{2,3,N}),
\end{equation}
which is also given by the inclusion of the polynomial rings.

%%%%%%%%%%%%% End of Part 1 %%%%%%%%%%%%%%%%%%
%
%%%%%%%%%%%%%% Part 2 %%%%%%%%%%%%%%%%%
%
%
%%%%%%%%%%%%%%%%%%%%%%%%%%%%%%%%%%%%%%%%
%%%                                  %%%
%%%        Foams                     %%%
%%%                                  %%%
%%%%%%%%%%%%%%%%%%%%%%%%%%%%%%%%%%%%%%%%
\section{Pre-foams}
\label{slN:sec:pre-foam}

In this section we begin to define the foams we will work with. 
The philosophy behind these foams will be explained in 
Section~\ref{slN:sec:KL}. 
To categorify the $\sln$ link polynomial we need singular 
cobordisms with two types of singularities. The basic examples are given in 
Figure~\ref{slN:fig:elemfoams}. These foams are composed of three types of facets: simple, double and 
triple facets. The double facets are coloured and the triple facets are marked 
to show the difference.
\begin{figure}[h!]
\medskip
\centering
\figins{0}{0.7}{ssaddle} \qquad
\figins{0}{0.7}{ssaddle_ud} \qquad
\figins{-2}{0.8}{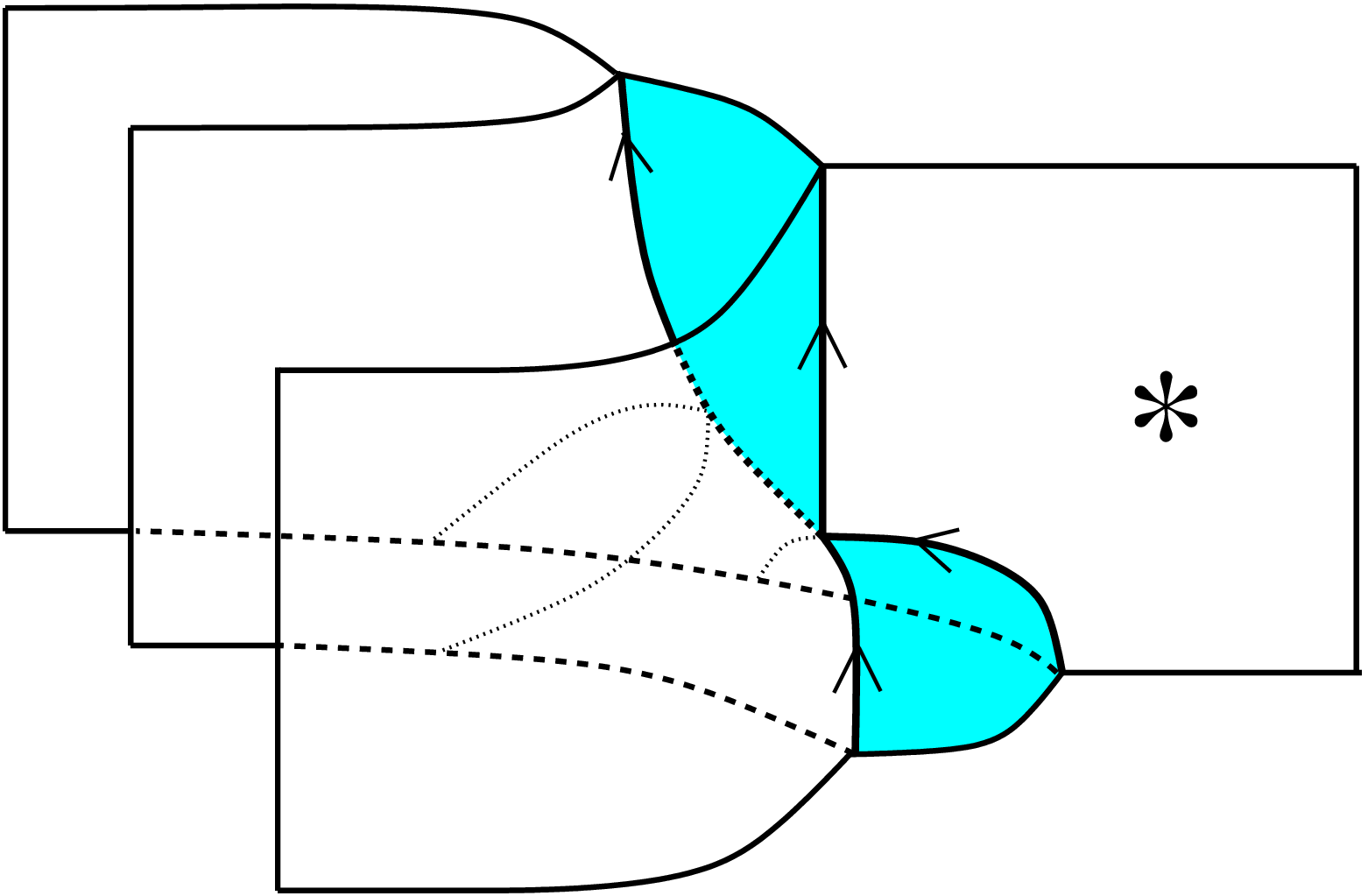} \qquad
\figins{-2}{0.8}{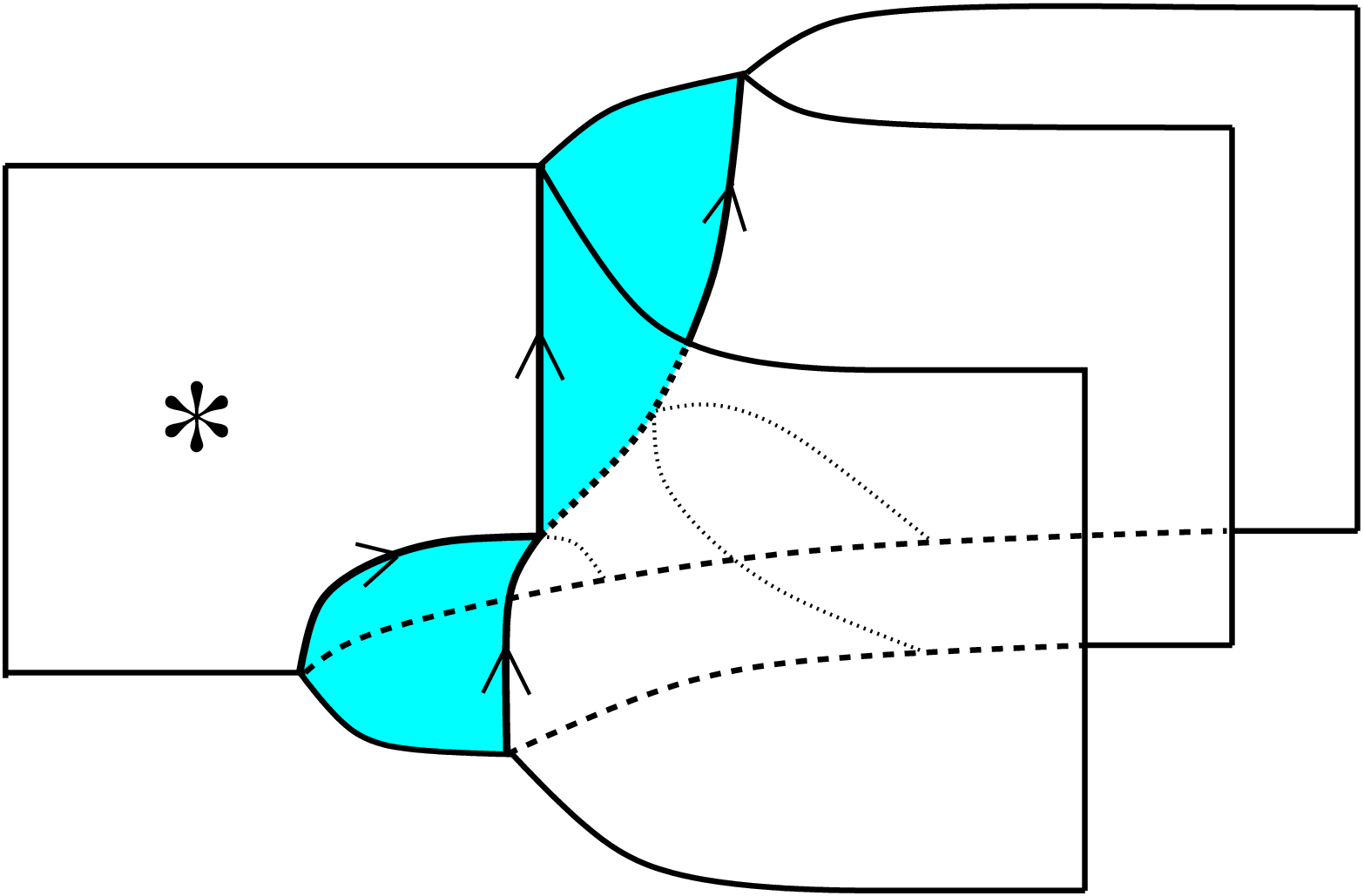}
\caption{Some elementary foams}
\label{slN:fig:elemfoams}
\end{figure}
Intersecting such a foam with a plane results in a web, as long as the plane avoids the 
singularities where six facets meet, such as on the right in Figure~\ref{slN:fig:elemfoams}. 

We adapt the definition of a world-sheet foam given in~\cite{Rozansky} to our setting.

\begin{defn}
\label{slN:defn:pre-foam}
Let $\sk$ be a finite closed oriented $4$-valent graph, 
which may contain disjoint circles. We assume 
that all edges of $\sk$ are oriented.  
A cycle in $\sk$ is defined to be a circle or a closed sequence of edges which form a 
piece-wise linear circle. 
Let $\Sigma$ be a compact orientable possibly disconnected surface, 
whose connected components are white, coloured or marked, also denoted by simple, double or 
triple. Each component can have a boundary consisting of several disjoint circles and can have 
additional decorations which we discuss below.     
A closed \emph{pre-foam} $u$ is the identification space $\Sigma/\sk$ obtained by glueing 
boundary circles of $\Sigma$ to cycles in $\sk$ such that every edge and circle in $\sk$ is glued 
to exactly three boundary circles of $\Sigma$ and such that for any point $p\in \sk$:    
\begin{enumerate}
\item if $p$ is an interior point of an edge, then $p$ has a neighborhood homeomorphic to the 
letter Y times an interval with exactly one of the facets being double, and at most one of them 
being triple. For an example see 
Figure~\ref{slN:fig:elemfoams};
\item if $p$ is a vertex of $\sk$, then it has a neighborhood as shown on the r.h.s. in 
Figure~\ref{slN:fig:elemfoams}. 
\end{enumerate}
We call $\sk$ the \emph{singular graph}, its edges and vertices \emph{singular arcs} and 
\emph{singular vertices}, and the connected components of $u - \sk$ the \emph{facets}.

Furthermore the facets can be decorated with dots. A simple facet can only have black 
dots ($\bdot$), a double facet can also have white dots ($\wdot$), and a triple facet besides 
black and white dots can have double dots ($\bwdot$). Dots can move freely on a facet 
but are not allowed to cross singular arcs. 
See Figure~\ref{slN:fig:simpl-2pol} for examples of pre-foams. 
\end{defn}

\begin{figure}[h!]
$$\xymatrix@R=1mm{
 \figins{0}{0.8}{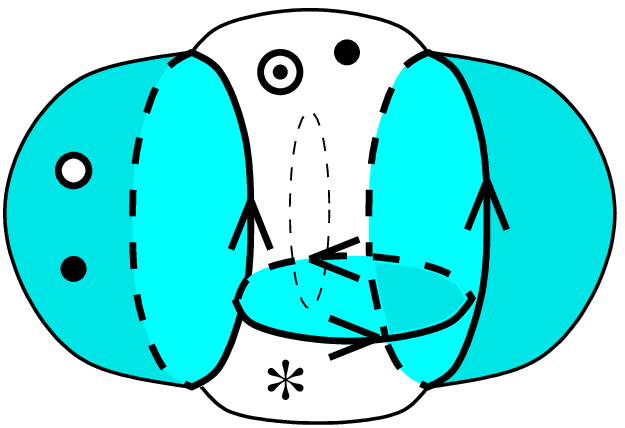} &
 \figins{-34}{0.8}{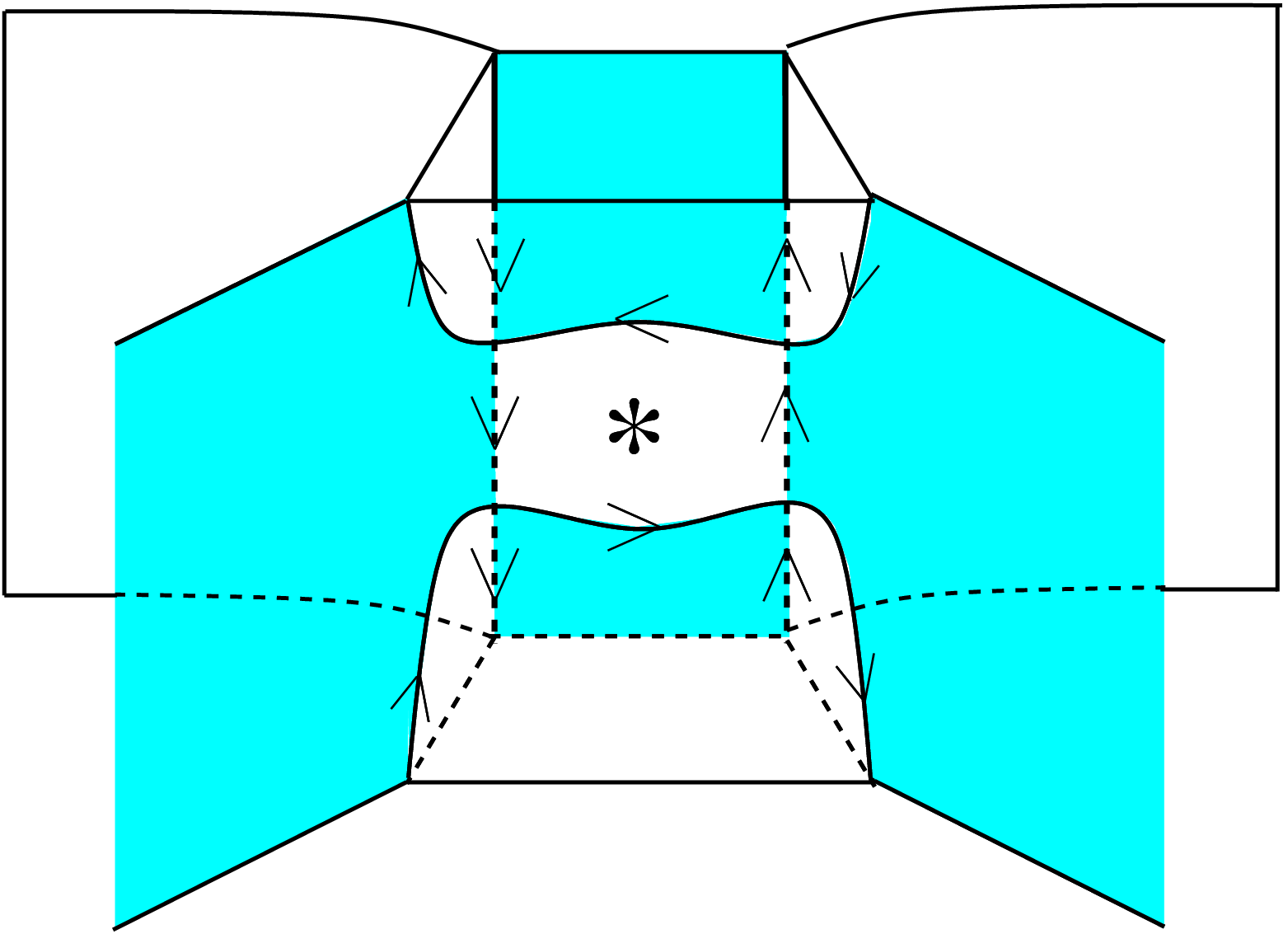} \\
a) & b)
}$$
\caption[A closed and an open pre-foam]{a) A pre-foam.\ b) An open pre-foam}
\label{slN:fig:simpl-2pol}
\end{figure}

Note that the cycles to which the boundaries of the simple and the triple facets are glued are 
always oriented, whereas the ones to which the boundaries of the double facets are glued are not. 
Note also that there are two types of singular vertices. Given a singular vertex $v$, 
there are precisely two singular edges which meet at $v$ and bound a triple facet: one oriented 
toward $v$, denoted $e_1$, and one oriented away from $v$, denoted $e_2$. 
If we use the ``left hand rule'', then the cyclic ordering of the facets 
incident to $e_1$ and $e_2$ is either $(3,2,1)$ and $(3,1,2)$ respectively, or the other way 
around. We say that $v$ is of type I in the first case and of type II in the second case (see Figure~\ref{slN:fig:vertextype}).
\begin{figure}[h!]
\labellist
\small\hair 2pt
\pinlabel $e_1$ at 772 290
\pinlabel $e_2$ at 620 535
\pinlabel $e_3$ at 75 300
\pinlabel $e_4$ at 280 70
\pinlabel $*$ at 780 450
\pinlabel $e_1$ at 1180 304
\pinlabel $e_3$ at 1882 290
\pinlabel $e_2$ at 1730 535
\pinlabel $e_4$ at 1380 70
\pinlabel $*$ at 1350 380
\endlabellist
\centering
\figs{0.13}{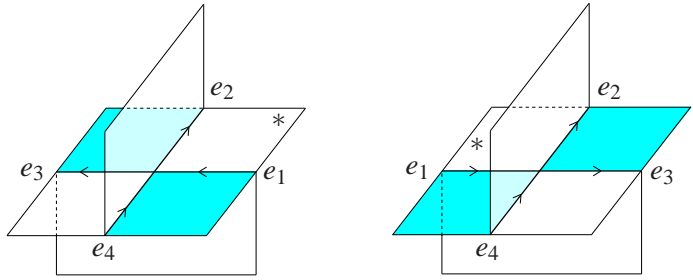}
\caption{Singular vertices of type I and type II}
\label{slN:fig:vertextype}
\end{figure}
When we go around a triple facet we see that there have to 
be as many singular vertices of 
type I as there are of type II for the cyclic orderings of the facets to match up. This shows 
that for a closed pre-foam the number of singular vertices of type I is equal to the number 
of singular vertices of type II.     

We can intersect a pre-foam $u$ generically by a plane $W$ in order to get a web, 
as long as the plane avoids the vertices of $\sk$. 
The orientation of $\sk$ determines the orientation of the simple edges of the web 
according to the convention in Figure~\ref{slN:fig:orientations}.

\begin{figure}[h!]
$$\xymatrix@R=1mm{
\figins{0}{0.85}{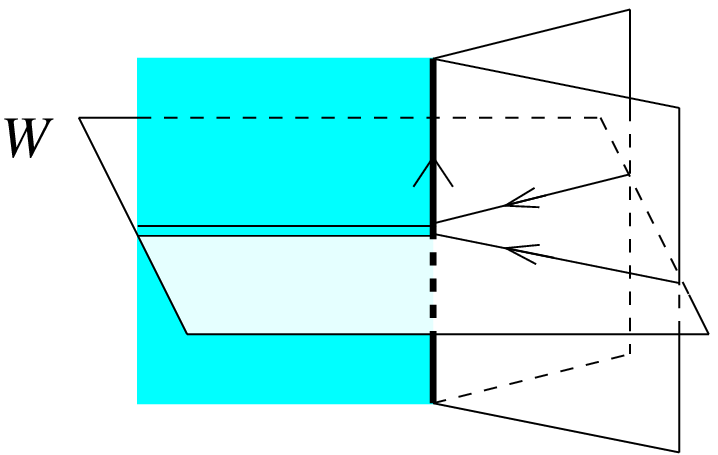} & 
\figins{0}{0.85}{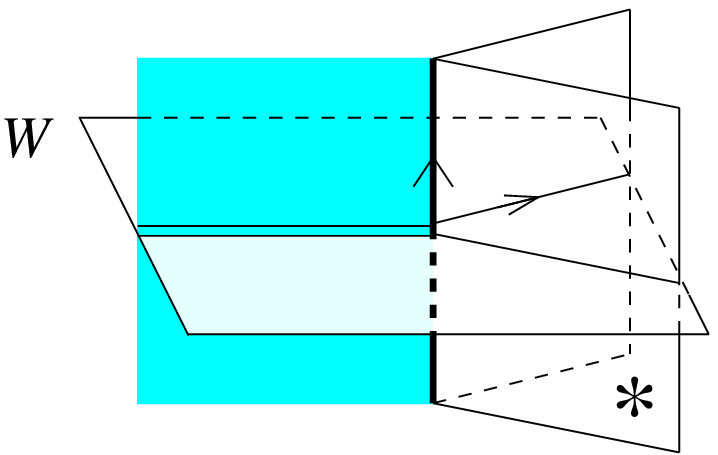} &
\figins{0}{0.85}{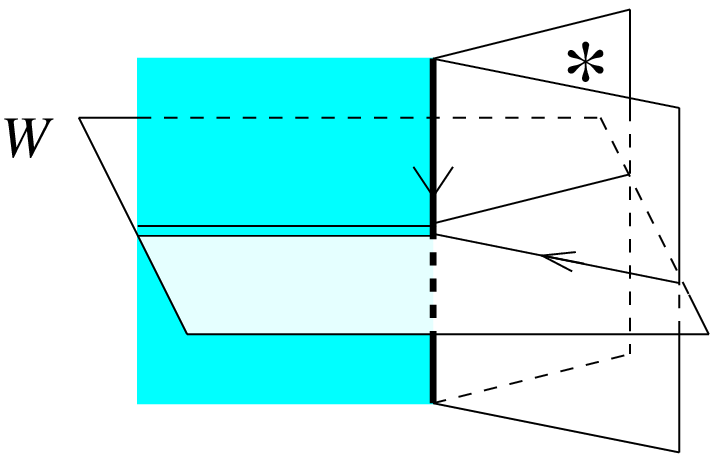} \\
\figins{0}{0.85}{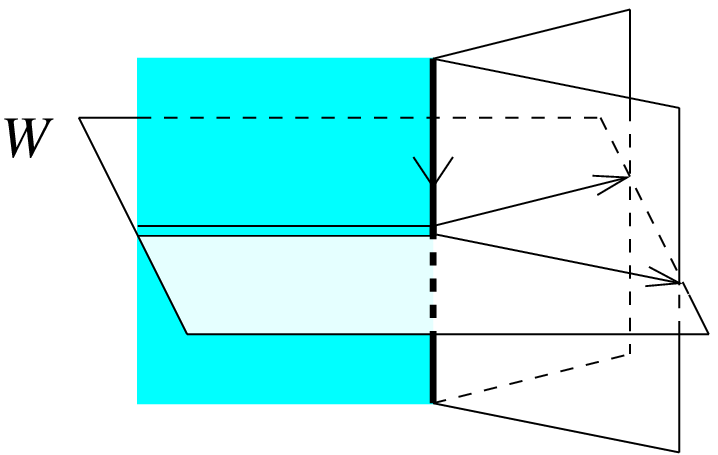}  & 
\figins{0}{0.85}{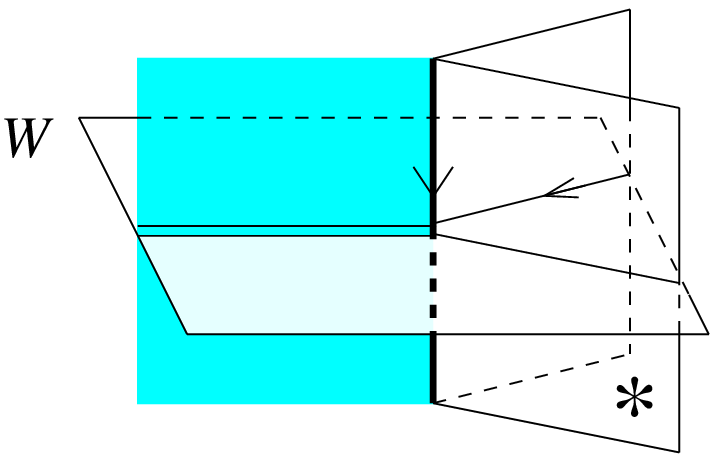} &
\figins{0}{0.85}{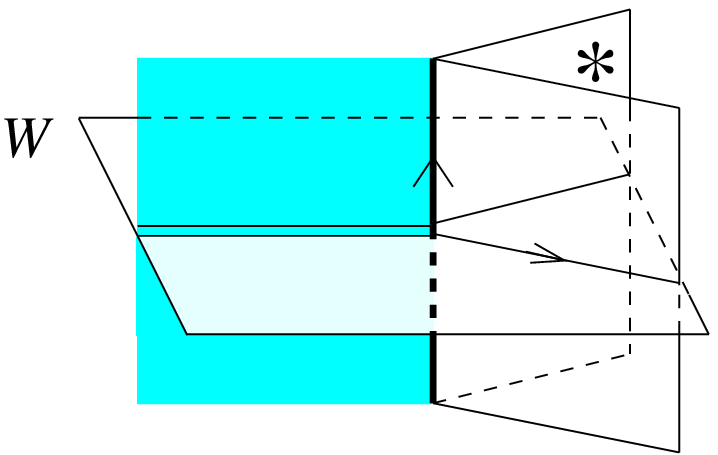}
}$$
\caption{Orientations near a singular arc}
\label{slN:fig:orientations}
\end{figure}

Suppose that for all but a finite number of values $i\in ]0,1[$, 
the plane $W\times {i}$ intersects $u$ generically. Suppose also that 
$W\times 0$ and $W\times 1$ intersect $u$ generically and outside the vertices 
of $\sk$. We call $W\times I\cap u$ an \emph{open} pre-foam.   
Interpreted as morphisms we read open pre-foams from bottom to top, and their 
composition consists of placing one pre-foam on top of the other, as long as their boundaries 
are isotopic and the orientations of the simple edges coincide.

\begin{defn} Let $\PF$ be the category whose objects are 
closed webs and whose morphisms are $\bQ$-linear combinations of 
isotopy classes of pre-foams with the obvious identity pre-foams and 
composition rule.  
\end{defn}

%%%%%%%%%
% foams %
%%%%%%%%%

We now define the $q$-degree of a pre-foam. 
Let $u$ be a pre-foam, $u_1$, $u_2$ and $u_3$ the disjoint union of its 
simple and double and marked facets respectively and 
$\sk(u)$ its singular graph. Define the partial $q$-gradings of $u$ as 
\begin{align*}
q_i(u)    &= \chi(u_i)-\frac{1}{2}\chi(\partial u_i\cap\partial u), 
\qquad i=1,2,3 \\
q_{\sk}(u) &= \chi(\sk(u))-\frac{1}{2}\chi(\partial \sk(u)).
\end{align*}
where $\chi$ is the Euler characteristic and $\partial$ denotes the boundary.

\begin{defn}
Let $u$ be a pre-foam with $d_\bdot$ dots of type $\bdot$, $d_\wdot$ dots of type 
$\wdot$ and $d_\bwdot$ dots of type $\bwdot$. The $q$-grading of $u$ is given 
by
\begin{equation*}
q(u)= -\sum_{i=1}^{3}i(N-i)q_i(u) - 2(N-2)q_{\sk}(u) + 
2d_\bdot + 4d_\wdot +6d_\bwdot.
\end{equation*}
\end{defn}

The following result is a direct consequence of the definitions. 
\begin{lem}
$q(u)$ is additive under the glueing of pre-foams.
\end{lem}

%%%%%%%%%%%%%%%%%%%%%%%%%%%%%%%%%%%
%%% Kapustin-Li and closed foams%%%
%%%%%%%%%%%%%%%%%%%%%%%%%%%%%%%%%%%

\section{The Kapustin-Li formula and the evaluation of closed pre-foams}
\label{slN:sec:KL}
Let us briefly recall the philosophy behind the pre-foams. Losely speaking, to each closed 
pre-foam should correspond an element in the cohomology ring of a configuration space of planes 
in some big $\bC^M$. The singular graph imposes certain conditions on those planes. The 
evaluation of a pre-foam should correspond to the evaluation of the corresponding element in 
the cohomology ring. Of course one would need to find a consistent way of choosing the volume forms 
on all of those configuration spaces for this to work. However, one encounters a difficult 
technical problem when working out the details of this philosophy. Without explaining all the 
details, we can say that the problem can only be solved by figuring out what to associate to 
the singular vertices. Ideally we 
would like to find a combinatorial solution to this problem, but so far it has eluded us. That 
is the reason why we are forced to use the KL formula.

We denote a simple facet with $i$ dots by 
$$\figins{-8}{0.3}{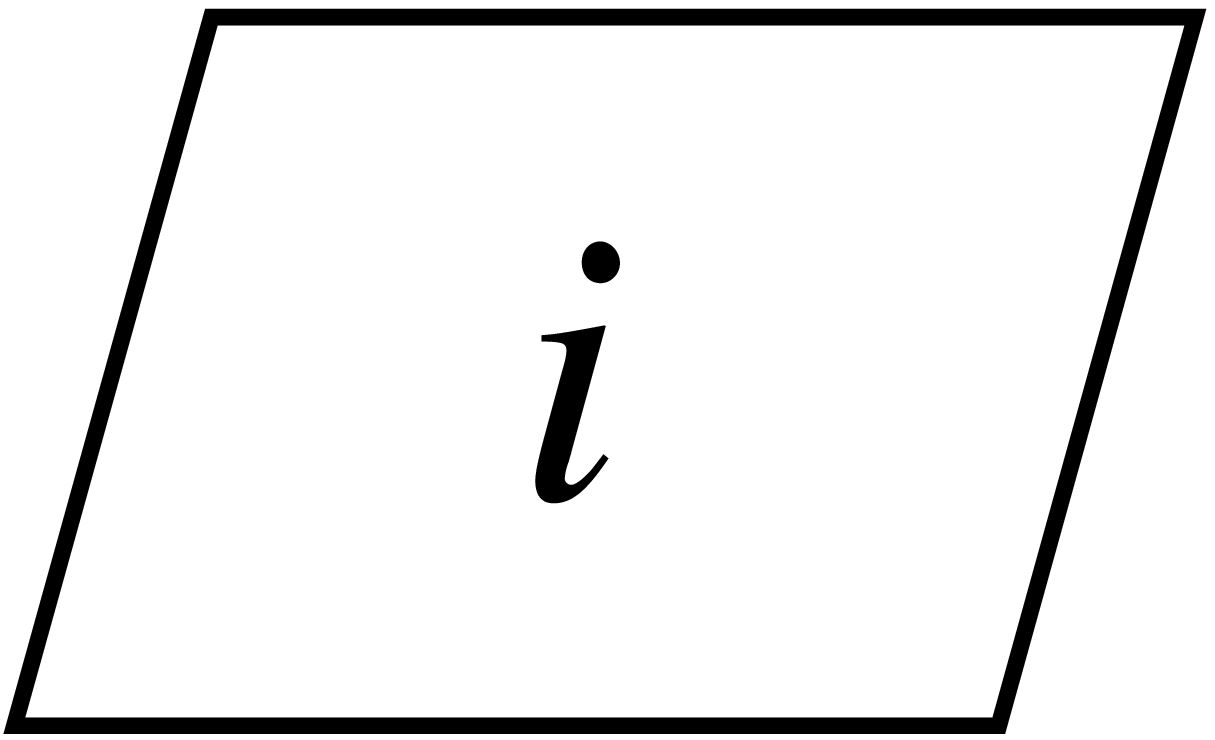}.$$  
Recall that $\pi_{k,m}$ can be expressed in terms of 
$\pi_{1,0}$ and $\pi_{1,1}$. In the philosophy explained above, the latter should correspond to 
$\bdot$ and $\wdot$ on a double facet respectively. We can then define  
$$\figins{-8}{0.3}{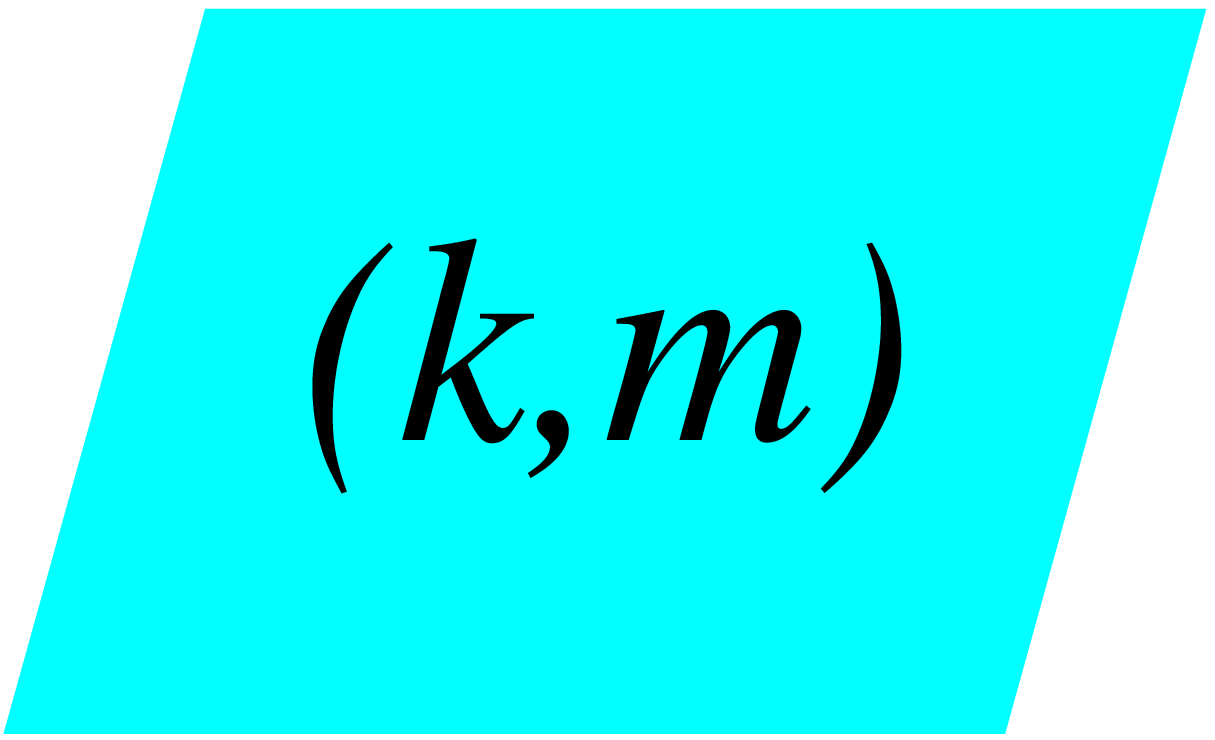}$$
as being the linear combination of dotted double facets corresponding to the expression of 
$\pi_{k,m}$ in terms of $\pi_{1,0}$ and $\pi_{1,1}$. 
Analogously we expressed $\pi_{p,q,r}$ in terms of $\pi_{1,0,0}$, $\pi_{1,1,0}$ and $\pi_{1,1,1}$ 
(see Section~\ref{slN:sec:partial flags}).  
The latter correspond to $\bdot$, $\wdot$ and $\bwdot$ on a triple facet respectively, so 
we can make sense of   
$$\figins{-8}{0.3}{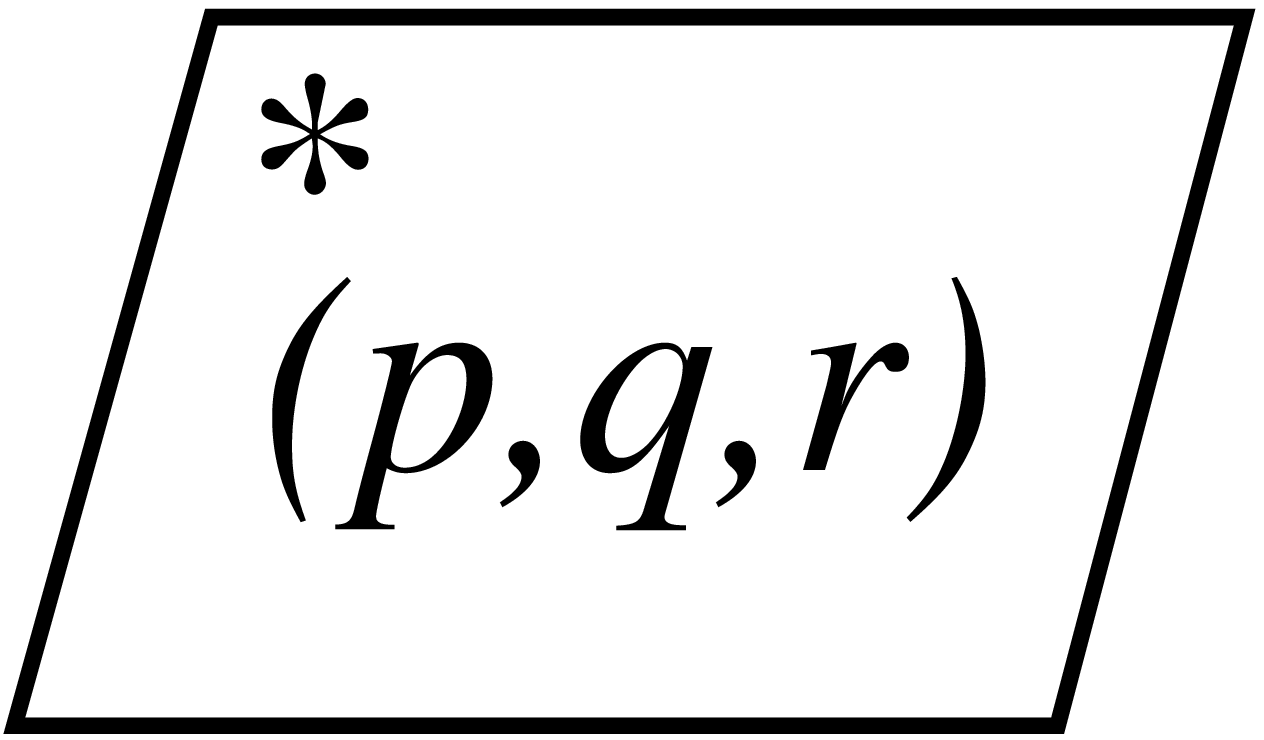}.$$
Our dot conventions and the results in 
Proposition~\ref{slN:prop:principal rels1} 
will allow us to use decorated facets in exactly the same way as we did Schur 
polynomials in the cohomology rings of partial flag varieties.

In the sequel, we shall give a definition of the KL formula for the evaluation of pre-foams and state some of its basic properties. The KL formula was introduced by A.~Kapustin and Y.~Li~\cite{KL} to generalize Vafa's work~\cite{vafa} in the context of the evaluation of 2-dimensional TQFTs to the case of smooth surfaces with boundary. It was later extended to the case of pre-foams by M.~Khovanov and L.~Rozansky in~\cite{KR-LG}, who interpreted singular arcs as boundary conditions as in~\cite{KL}. Khovanov and Rozansky adapted the KL formula to a general sort of foam. In this thesis we have to specify the input data which allows us to use it for the evaluation of our pre-foams. The normalization is ours and is used to obtain integral relations.

%%%%%%%%%%%%%%%%%%%%%%%%%%%%%%%%%%
\subsection{The general framework}
%%%%%%%%%%%%%%%%%%%%%%%%%%%%%%%%%%
Let $u=\Sigma/s_{\gamma}$ be a closed pre-foam with singular graph $s_{\gamma}$ and without any dots on it. Let $F$ denote an arbitrary $i$-facet, $i\in\{1,2,3\}$, with a $1$-facet being a simple facet, a $2$-facet being a double facet and a $3$-facet being a triple facet.

Each $i$-facet can be decorated with dots, which correspond to generators of the rational cohomology ring of the Grassmannian $\cG_{i,N}$, i.e. $\hy(\cG_{i,N},\bQ)$. Alternatively, we can associate to every $i$-facet $F$, $i$ variables $x^F_1\ldots,x^F_i$, with $\deg x^F_i=2i$, and the potential $W(x^F_1,\ldots,x^F_i)$, which is the polynomial defined such that
$$W(\sigma_1,\ldots,\sigma_i)=y_1^{N+1}+\ldots+y_i^{N+1},$$
where $\sigma_j$ is the $j$-th elementary symmetric polynomial in the variables $y_1,\ldots,y_i$. The Jacobi algebra $J_W$$$J_W=\quotient{\bQ[x^F_1,\ldots,x^F_i]}{ ( \partial_iW )},$$ 
where $\partial_iW$ denote the ideal generated by the partial derivatives of $W$, is isomorphic to the rational cohomology ring of the Grassmannian $\cG_{i,N}$. Note that up to a multiple the top degree nonvanishing element in this Jacobi algebra is $\pi_{N-i,\ldots,N-i}$ (multiindex of length $i$), i.e. the polynomial 
in variables $x^F_1,\ldots,x^F_i$ which gives $\pi_{N-i,\ldots,N-i}$ after replacing the variable $x^F_j$ by 
$\pi_{1,\ldots,1,0,\ldots,0}$ with exactly $j$ $1$'s, $1\le j\le i$ (see also 
Subsection~\ref{slN:sec:Schur}). We define the trace 
(volume) form $\epsilon$ on $\hy(\cG_{i,N},\bQ)$ by giving it on the basis of the Schur polynomials: 
$$\epsilon(\pi_{j_1,\ldots,j_i})=
\begin{cases}(-1)^{\lfloor\frac{i}{2}\rfloor}&\text{if}\quad (j_1,\ldots,j_i)=(N-i,\ldots,N-i)\\
0&\text{else}
\end{cases}.
$$

The KL formula associates to $u$ an 
element in the product of the cohomology rings of the Jacobi 
algebras $J$, over all the facets in the pre-foam. 
Alternatively, we can see this element as a polynomial, $KL_u\in J$, in 
all the variables associated to the facets. Now, let us put some dots on 
$u$. Recall that a dot corresponds to an elementary symmetric polynomial. 
So a linear combination of dots on $u$ is equivalent to a polynomial, $f$, 
in the variables of the dotted facets. Let $\varepsilon$ denote the product of the trace forms $\varepsilon_{J_i}$ over all facets of $u$.
The value of this dotted pre-foam we define to be
\begin{equation}\label{slN:eq:ev}
\KL{u}:=\epsilon\biggl(\prod_{F}\dfrac{\det(\partial_i\partial_j W_F)^{g(F)}}{(N+1)^{g'(F)}}
KL_u\, f\biggr).
\end{equation}
\n The product is over all facets $F$ and $W_F$ is the potential 
associated to $F$. For any $i$-facet $F$, $i=1,2,3$, the symbol $g(F)$ 
denotes the genus of $F$ and $g'(F)=ig(F)$. If $u$ is a closed surface without singularities we define $KL_u=1$ and $\KL{\ }$ reduces to an extension to colored closed surfaces of the formula introduced by Vafa in~\cite{vafa}.
The \emph{Vafa factor}
$$\prod_{F}\dfrac{\det(\partial_i\partial_j W_F)^{g(F)}}{(N+1)^{g'(F)}}$$
computes the contribution of the handles in the facets of $u$.

Having explained the general idea, we are left with defining the element 
$KL_u$ for a dotless pre-foam. For that we have to explain Khovanov and 
Rozansky's extension of the KL formula to pre-foams~\cite{KR-LG}, 
which uses the theory of matrix factorizations.

%%%%%%%%%%%%%%%%%%%%%%%%%%%%%%%%%%%%
\subsection{Decoration of pre-foams}
%%%%%%%%%%%%%%%%%%%%%%%%%%%%%%%%%%%%
\label{slN:ssec:decoration}

As we said, to each facet we associate certain variables (depending on the 
type of facet), a potential and the corresponding Jacobi algebra. 
If the variables associated to a facet $F$ are $x_1,\ldots,x_i$, then we 
define $R_F=\bQ[x_1,\ldots,x_i]$. It is immediate that the KL formula gives zero if the argument of $\varepsilon$ in Equation~\ref{slN:eq:ev} contains an element of $\partial_iW_F$: for any $Q\in\bigotimes\limits_{F}R_F$ we have that
\begin{equation}
\varepsilon(Q\partial_iW_F)=0.
\end{equation}

Now we consider the edges. To each edge we associate a matrix factorization whose potential is equal to the signed sum of the potentials of the facets that are glued along this edge. We define it to be a certain tensor product of Koszul factorizations. In the cases we are interested in there are always three facets glued along an edge, with two possibilities: either two simple facets and one double facet, or one simple, one double and one triple facet. 
\begin{figure}[h!]
\medskip
\labellist
\small\hair 2pt
\pinlabel $x^{N+1}$ at 134 43
\pinlabel $y^{N+1}$ at 144 139
\pinlabel $W(s,t)$ at 44 100
\endlabellist
\centering
\figs{0.58}{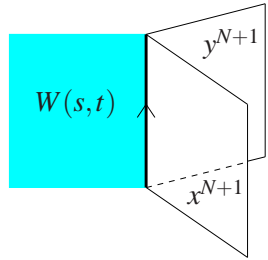}
\caption{Singular edge of type $(1,1,2)$}
\label{slN:fig:orientation-MF1}
\end{figure}
In the first case, we denote the variables of the two simple facets by 
$x$ and $y$ and take the potentials to be $x^{N+1}$ and $y^{N+1}$ respectively, according to the convention in Figure~\ref{slN:fig:orientation-MF1}. To 
the double facet we associate the variables $s$ and $t$ and the potential 
$W(s,t)$. To the edge we associate the matrix factorization 
which is the tensor product of Koszul factorizations given by
\begin{equation}\label{slN:eq:MF1}
MF_1= 
\begin{Bmatrix}
A', & x+y-s \\
B', & xy-t
\end{Bmatrix},
\end{equation}
where $A'$ and $B'$ are given by
\begin{align*}
A'&= \frac{W(x+y,xy)-W(s,xy)}{x+y-s},\\
B'&= \frac{W(s,xy)-W(s,t)}{xy-t}. 
\end{align*}
Note that $(x+y-s)A'+(xy-t)B'=x^{N+1}+y^{N+1}-W(s,t)$.

In the second case, the variable of the simple facet is $x$ and 
the potential is $x^{N+1}$, the variables of 
the double facet are $s$ and $t$ and the potential is $W(s,t)$, and the 
variables of the triple face are $p$, $q$ and $r$ and the potential is 
$W(p,q,r)$. 
\begin{figure}[h!]
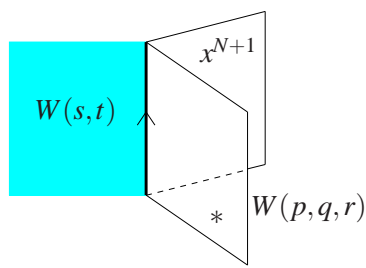

\medskip
\labellist
\small\hair 2pt
\pinlabel $x^{N+1}$ at 142 140
\pinlabel $W(s,t)$ at 44 100
\pinlabel $W(p,q,r)$ at 194 38
\pinlabel $*$ at 134 32
\endlabellist
\centering
\figs{0.58}{orientation-MF}
\caption{Singular edge of type $(1,2,3)$}
\label{slN:fig:orientation-MF2}
\end{figure}

\n Define 
the polynomials
\begin{align}
A &= \frac{W(x+s,xs+t,xt)-W(p,xs+t,xt)}{x+s-p},\label{slN:eq:MF21}\\
B &= \frac{W(p,xs+t,xt)-W(p,q,xt)}{xs+t-q},\\
C &= \frac{W(p,q,xt)-W(p,q,r)}{xt-r},
\end{align}
so that 
$$(x+s-p)A+(xs+t-q)B+(xt-r)C=x^{N+1}+W(s,t)-W(p,q,r).$$ 
To such an edge we associate the matrix factorization given by the 
following tensor product of Koszul factorizations:
\begin{equation}\label{slN:eq:MF2}
MF_2=
\begin{Bmatrix}
A, & x+s-p  \\
B, & xs+t-q \\
C, & xt-r
\end{Bmatrix}.
\end{equation}

\medskip

In both cases, if the edges have the opposite orientation we associate the 
 matrix factorizations $(MF_1)_\bullet$ and $(MF_2)_\bullet$ respectively.

\medskip

Next we explain what we associate to a singular vertex. First of all,
for each vertex $v$, we define its local graph $\gamma_v$ to be 
the intersection of a small sphere centered at $v$ with the pre-foam. 
Then the vertices of $\gamma_v$ correspond to the edges of $u$ that are 
incident to $v$, to which we had associated matrix factorizations. 
In this thesis all local graphs $\gamma_v$ are in fact tetrahedrons. However, recall 
that there are two types of vertices (see the remarks below Definition~\ref{slN:defn:pre-foam}). 
Label the six facets that are incident to a vertex $v$ by the numbers 
$1,2,3,4,5$ and $6$. Furthermore, 
denote the edge along which are glued the facets $i$, $j$ and $k$ by 
$(ijk)$. Denote the matrix factorization associated to the edge 
$(ijk)$ by $M_{ijk}$, if the edge points toward $v$, and by 
$(M_{ijk})_\bullet$, if the edge points away from $v$. Note that $M_{ijk}$ and 
$(M_{ijk})_\bullet$ are both defined over $R_i\otimes R_j \otimes R_k$. 

Now we can take the tensor product of these four matrix factorizations, 
over the polynomial rings of the facets of the pre-foam, that correspond to the 
vertices of $\gamma_v$. This way we obtain the matrix factorization $M_v$, 
whose potential is equal to $0$, and so it is a 
2-complex and we can take its homology.

To each vertex $v$ we associate an 
element $O_v\in \hy_{\mf{}}(M_v)$. More precisely, if $v$ is of type I, then
\begin{equation}\label{slN:eq:vertex}
\begin{split}
\hy_{\mf{}}(M_v) &\cong 
\Ext\left(
MF_1(x,y,s_1,t_1)\otimes_{s_1,t_1} MF_2(z,s_1,t_1,p,q,r)\right. , \\
&\quad\qquad\ \left. 
MF_1(y,z,s_2,t_2)\otimes_{s_2,t_2} MF_2(x,s_2,t_2,p,q,r)\right).
\end{split}
\end{equation}

If $v$ is of type II, then 
\begin{equation}\label{slN:eq:vertex2}
\begin{split}
\hy_{\mf{}}(M_v) &\cong  
\Ext\left(
MF_1(y,z,s_2,t_2)\otimes_{s_2,t_2} MF_2(x,s_2,t_2,p,q,r)\right. ,\\
&\quad\qquad\ \left. 
MF_1(x,y,s_1,t_1)\otimes_{s_1,t_1} MF_2(z,s_1,t_1,p,q,r)\right).
\end{split}
\end{equation}
Both isomorphisms hold up to a global shift in $q$. Note that 
$$
MF_1(x,y,s_1,t_1)\otimes_{s_1,t_1} MF_2(z,s_1,t_1,p,q,r) \simeq
MF_1(y,z,s_2,t_2)\otimes_{s_2,t_2} MF_2(x,s_2,t_2,p,q,r),
$$
because both tensor products are homotopy equivalent to the factorization
$$
\begin{Bmatrix}
*, & x+y+z-p \\
*, & xy+xz+yz-q \\
*, & xyz-r
\end{Bmatrix}.
$$
We have not specified the l.h.s. of the latter Koszul matrix, because of Lemma~\ref{KR:lem:regseq-iso}. If $v$ is of type I, we take $O_v$ to be the cohomology class of 
a fixed degree $0$ homotopy equivalence
$$
w_v\colon MF_1(x,y,s_1,t_1)\otimes_{s_1,t_1} MF_2(z,s_1,t_1,p,q,r)\to
MF_1(y,z,s_2,t_2)\otimes_{s_2,t_2} MF_2(x,s_2,t_2,p,q,r).
$$
The choice of $O_v$ is unique up to a scalar, because the graded dimension 
of the $\Ext$-group in \eqref{slN:eq:vertex} is equal to 
$$q^{3N-6}\qdim(\hy(M_v))=q^{3N-6}[N][N-1][N-2]=1+q(\ldots),$$
where $(\ldots)$ is a polynomial in $q$. Note that $M_v$ is homotopy equivalent to 
the matrix factorization which corresponds to the closure of $\Upsilon$ in~\cite{KR}, which 
allows one to compute the graded dimension above using the results in the latter paper.
If $v$ is of type II, we take $O_v$ to be the cohomology class of the homotopy 
inverse of $w_v$. Note that a particular choice of $w_v$ fixes $O_v$ for 
both types of vertices and that the value of the KL formula for a closed 
pre-foam does not depend on that choice because there are as many singular vertices of type I 
as there are of type II (see the remarks below Definition~\ref{slN:defn:pre-foam}). 
We do not know an explicit formula for $O_v$. Although such a formula would be very interesting 
to have, we do not need it for the purposes of this thesis.

%%%%%%%%%%%%%%%%%%%%%%%%%%%%%%%%%%%%%%%%%%%%%%%%%%%%%%%%%%%%%%%%%%%%%%%%%%%%%%
\subsection{The KL derivative and the evaluation of closed pre-foams}
%%%%%%%%%%%%%%%%%%%%%%%%%%%%%%%%%%%%%%%%%%%%%%%%%%%%%%%%%%%%%%%%%%%%%%%%%%%%%%

From the definition, every boundary component of each facet $F$ is 
either a circle or a cyclicly ordered finite sequence
of edges, such that the beginning of the next edge corresponds to the end 
of the previous
edge. For every boundary component choose an edge $e$ and denote the 
differential of the matrix factorization associated to this edge by $D_e$.   
Let $R_F=\bQ[x_1,\dots, x_k]$.
The KL derivative of $D_e$ in the 
variables $x_1,\ldots,x_k$ associated to the facet $F$, is an element 
from $\End(M)\cong M\otimes M_\bullet$, given by:
\begin{equation}\label{slN:eq:tw}
O_{F,e}=\partial D\hat{_e} =\frac{1}{k!}\sum_{\sigma\in S_k} 
(\sgn{\sigma}){\partial_{\sigma(1)
}D_{e} \partial_{\sigma(2)}D_{e}\ldots\partial_{\sigma(k)}D_{e}},
\end{equation}
where $S_k$ is the symmetric group on $k$ letters, 
and $\partial_i D$ is the partial derivative of $D$ with respect to 
the variable $x_i$. For all the other edges $e'$ in the boundary of $F$ we take $O_{F,e'}$ to be the identity. Denote the set of facets whose boundary contains $e$ by $F(e)$.
For every edge define $O_e\in\End(M)$ as the composite
$$O_e=\prod\limits_{F\in F(e)}O_{F,e}.$$
The order of the factors in $O_e$ is irrelevant as we will prove it in Lemma~\ref{slN:lem:operorder}.

Let $\cV$ and $\cE$ be the sets of all vertices and all edges of the singular graph $\sk$ of a pre-foam $u$. Denote the matrix factorization associated to an edge $e$ by $M_e$ ($M_e=MF_1$ if $e$ is of type $(1,1,2)$ and $M_e=MF_2$ if $e$ is of type $(1,2,3)$). Recall that the factorization $M_v$ associated to a singular vertex is the tensor product of the matrix factorizations associated to the edges that are incident to $v$. Consider the factorization $M_{\sk}$ given by the tensor product
\begin{equation}
M_{\sk}=\biggl(\bigotimes\limits_{v\in\cV}M_v\biggr)\otimes
\biggl(\bigotimes\limits_{e\in\cE}\ M_e\otimes(M_e)_\bullet\biggr).
\end{equation}
From the definition of $M_v$ we see that we can group all the factorizations involved in pairs of mutually dual factorizations: for every edge $e$ we can pair $M_e$ coming from $M_e\otimes(M_e)_\bullet$ with $(M_e)_\bullet$ coming from $M_v$ and $(M_e)_\bullet$ from $M_e\otimes(M_e)_\bullet$ can be paired with $M_e$ coming from $M_v$. Using super-contraction on each pair we get a map
$$
\phi_\gamma\colon M_{\sk}\to\bQ[\mathbf{x}_u],
$$
where $\mathbf{x}_u$ is the set of variables associated to all the facets of $u$.
\begin{defn}\label{slN:def:KL}
$
KL_u=\phi_\gamma
\biggl( 
\bigl(\bigotimes\limits_{v\in\cV}O_v\bigr)\otimes
\bigl(\bigotimes\limits_{e\in\cE}O_e\bigr)
\biggr).
$
\end{defn}

Note that the $O_e$ and $O_v$ can be seen as tensors with indices associated to the facets that meet at $e$ and $v$ respectively. So we can super-contract all the tensor factors $O_e$ and $O_v$, with respect to a particular facet $F$, along a cycle that bounds $F$. From Definition~\ref{slN:def:KL}  we see that if we do this for all boundary components of all facets we also get $KL_u$.

\begin{lem}\label{slN:lem:operorder}
$KL_u$ does not depend on the order of the factors in $O_e$.
\end{lem}
\begin{proof}
Let $e$ be an edge in the boundary of facets $F$ and $F'$. Since the potential $W_e$ is a sum of the individual potentials associated to the facets that are glued along $e$, each depending on its own set of variables, we have $\partial_i\partial_j'W_e=0$. Therefore, applying $\partial_i\partial'_j$ to both sides of the relation $D_e^2=W_e$ gives
$$[\partial_iD_e,\partial'_jD_e]_s=-[D_e,\partial_i\partial'_jD_e]_s,$$
and the term on the r.h.s. is annihilated after the super-contraction because it is a coboundary. This means that the KL derivatives of $D$ w.r.t. different facets super-commute.
\end{proof}

\begin{lem}
$KL_u$ does not depend on the choice of the preferred edges.
\end{lem}
\begin{proof}
%\proof
It suffices to prove the claim for only one facet $F$ with one boundary component. Label the edges that bound $F$ by $e_1$, \ldots, $e_k$ and take $e_1$ as the preferred edge of $F$. Suppose first that $F$ is a simple or a triple facet, so that its boundary consists of an oriented cycle of $\sk$. Suppose also that $e_i$ is oriented from $v_i$ to $v_{i+1}$. Since $[O_{F,e},O_{F',e}]_s=0$ for every $F'\neq F$ we can assume that $O_{e_1}=O_{F,e_1}$ without loss of generality. The contribution to $KL_u$ of the facet $F$ is given by
$$
\str_{W_F}\bigl( \partial D\hat{_{e_1}}O_{v_1}O_{v_2}\dots O_{v_k}\bigr),
$$
where $\str_{W_F}$ is the partial supertrace w.r.t. the indices associated to $F$.

The relevant part of a small neighborhood of the vertex $v_1$ is depicted in Figure~\ref{slN:fig:vertexfacet}, where only the facet $F$ is shown.
\begin{figure}[h!]
\labellist
\small\hair 2pt
\pinlabel $e_1$ at 112 10
\pinlabel $e_2$ at 100 300
\pinlabel $e'_2$ at 2 300
\pinlabel $e'_1$ at -20 30
\pinlabel $v_1$ at 28 160
\pinlabel $F$ at 200 150
\endlabellist
\centering
\figs{0.25}{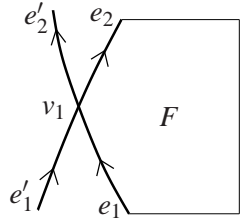}
\caption{Singular vertex}
\label{slN:fig:vertexfacet}
\end{figure}
From Equation~\eqref{slN:eq:vertex} it follows that $O_v$ can be seen as a homomorphism from $M_e(e_1)\otimes M_e(e'_1)$ to $M_e(e_2)\otimes M_e(e'_2)$, where $(e_i)$ denotes the variables associated to the facets that are glued along $e$. Therefore we have that $[D,O_v]_s=0$, where $D=D_{e_1}+D_{e'_1}+D_{e_2}+D_{e'_2}$ and we are using the convention that the composite of two non-composable homomorphisms is zero. Note that $\partial_iD=\partial_iD_{e_1}+\partial_iD_{e_2}$ since $e'_1$ and $e'_2$ are not variables associated to $F$. Therefore $[D,O_v]_s=0$ implies
\begin{equation}
\label{slN:eq:Dcomm}
[\partial_iD,O_v]_s = -[D,\partial_iO_v]_s
\end{equation}
by partial differentiation w.r.t. a variable of $F$. This implies
$$
\str_{W_F}\bigl( \partial D\hat{_{e_1}}O_{v_1}O_{v_2}\dots O_{v_k}\bigr)
=
\str_{W_F}\bigl( O_{v_1}\partial D\hat{_{e_2}}O_{v_2}\dots O_{v_k}\bigr),
$$
since terms involving the r.h.s. of Equation~\eqref{slN:eq:Dcomm} get killed by $\str$.

Now suppose that $F$ is a double facet. The boundary of $F$ is not an oriented cycle in $\sk$. Suppose a small neighborhood of $v$ has a part as depicted in Figure~\ref{slN:fig:vertexfacet2}.
\begin{figure}[h!]
\labellist
\small\hair 2pt
\pinlabel $e_1$ at 112 10
\pinlabel $e_2$ at 100 300
\pinlabel $e'_2$ at -15 300
\pinlabel $e'_1$ at -20 30
\pinlabel $v_1$ at 28 160
\pinlabel $F$ at 200 150
\endlabellist
\centering
\figs{0.25}{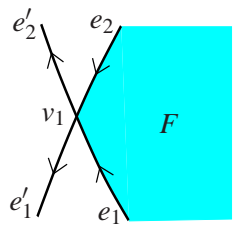}
\caption{Double facet near a singular vertex}
\label{slN:fig:vertexfacet2}
\end{figure}
In this case $O_v$ can be seen as a homomorphism from $M_e(e_1)\otimes M_e(e'_1)_\bullet$ to $M_e(e_2)\otimes M_e(e'_2)_\bullet$, so that $D$ and $O_v$ super-commute, where $D=D_{e_1}+(D_{e'_1})_\bullet + (D_{e_2})_\bullet + D_{e'_2}$. Taking a partial derivative of both sides of the relation $[D,O_v]_s$ relative to a variable associated to $F$ we obtain that 
$$
\str_{W_F}\bigl( \partial D\hat{_{e_1}}O_{v_1}O_{v_2}\dots O_{v_k}\bigr)
=
\str_{W_F}\bigl( O_{v_1}\partial (D\hat{_{e_2}})_\bullet O_{v_2}\dots O_{v_k}\bigr),
%\rlap{\hspace*{11.4ex}\qed}
$$
which proves the claim.
\end{proof}

%%%%%%%%%%%%%%%%%%%%%%%%%%%%%%
\subsection{Some computations}
\label{slN:ssec:compKL}
%%%%%%%%%%%%%%%%%%%%%%%%%%%%%%
In this subsection we compute the KL evaluation of some closed pre-foams that will be used in the sequel.

%%%%%%%%%%%%%%%%%%%%%%%
\subsubsection{Spheres}
%%%%%%%%%%%%%%%%%%%%%%%

The values of dotted spheres are easy to compute. Note that for any sphere with dots $f$ 
the KL formula gives 
$$\epsilon(f).$$ 
Therefore for a simple sphere we get $1$ if $f=x^{N-1}$, for a double 
sphere we get 
$-1$ if $f=\pi_{N-2,N-2}$ and for a triple sphere we get $-1$ if 
$f=\pi_{N-3,N-3,N-3}$. 

Note that the evaluation of spheres corresponds to the trace on the cohomology of the Grassmannian $\hy(\cG_{i,N})$ for $i=1,2,3$ in Equation~\eqref{slN:eq:db}.

%%%%%%%%%%%%%%%%%%%%%%%%%%%%%%%%%%%%%%%%%%%%%%%%
\subsubsection{Dot conversion and dot migration}
%%%%%%%%%%%%%%%%%%%%%%%%%%%%%%%%%%%%%%%%%%%%%%%%

The pictures related to the dot conversion and dot migration relations can be found in Proposition~\ref{slN:prop:principal rels1}. Since $KL_u$ takes values in the tensor product of the 
Jacobi algebras of the potentials associated to the facets of $u$, we see that 
for a simple facet we have $x^N=0$, for a double facet $\pi_{i,j}=0$ if $i\geq N-1$, and 
for a triple facet $\pi_{p,q,r}=0$ if $p\geq N-2$. We call these the 
\emph{dot conversion relations}. The dot conversion relations are related to the relations defining the cohomology ring of the Grassmannian $\cG_{k,N}$ for $k=1,2,3$ in Equation~\eqref{slN:eq:gras}.

To each edge along which two simple facets with variables $x$ and $y$ and one double facet with the variables $s$ and $t$ are glued, we 
associated the matrix factorization $MF_1$ with 
entries $x+y-s$ and $xy-t$. Therefore $\Ext(MF_1,MF_1)$ 
is a module over $R/ ( x+y-s,xy-t ) $. Hence, 
we obtain the \emph{dot migration relations} along this edge. Analogously, to the other type of singular edge along 
which are glued a simple facet with variable $x$, 
a double facet with variable $s$ and $t$, and a triple 
facet with variables $p$, $q$ and $r$,
we associated the matrix factorization $MF_2$. Note that $\Ext(MF_2,MF_2)$ is a module over 
$R/ ( x+s-p, xs+t-q,xt-r ) $, which gives us the \emph{dot migration relations} along this edge.
The dot migration relations are related to the relations in the cohomology ring of the partial flag varieties $Fl_{1,2,N}$ and $Fl_{2,3,N}$ in Equation~\eqref{slN:eq:hflag} under the projection maps in Equations~\eqref{slN:eq:proj1} and~\eqref{slN:eq:proj2}.

%%%%%%%%%%%%%%%%%%%%%%%%%%%%%%%%%%%%%%%%%%%%
\subsubsection{The $(1,1,2)$-theta pre-foam}
%%%%%%%%%%%%%%%%%%%%%%%%%%%%%%%%%%%%%%%%%%%%

Recall that $W(s,t)$ is the polynomial such that 
$W(x+y,xy)=x^{N+1}+y^{N+1}$. More precisely, 
we have
\begin{equation*}
W(s,t)=\sum_{i+2j=N+1}a_{ij}s^it^j,
\end{equation*}
with $a_{N+1,0}=1$, $a_{N+1-2j,j}=\frac{(-1)^j}{j}(N+1)\binom{N-j}{j-1}$, for $2\le 
2j\le N+1$, and $a_{ij}=0$ otherwise. In particular $a_{N-1,1}=-(N+1)$. We have
\begin{align*}
W'_1(s,t) &= \sum_{i+2j=N+1}ia_{ij}s^{i-1}t^j,\\
W'_2(s,t) &= \sum_{i+2j=N+1}ja_{ij}s^it^{j-1}.
\end{align*}
 
By $W'_1(s,t)$ and $W'_2(s,t)$, we denote the partial derivatives of $W(s,t)$ with respect to the first and the second variable, respectively.

\begin{figure}[h!]
\centering
\figs{0.25}{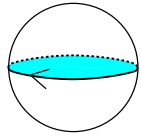}
\caption{A dotless (1,1,2)-theta pre-foam}
\label{slN:fig:theta112}
\end{figure}

Consider the $(1,1,2)$-theta pre-foam of Figure~\ref{slN:fig:theta112}. According to the conventions of Subsection~\ref{slN:ssec:decoration} we have variables $x$ and $y$ on the lower and upper simple facets respectively, and the variables $s$ and $t$ on the double facet. To the singular circle we assign the  
matrix factorization
\begin{equation*}
MF_{1}=
\begin{Bmatrix}
A', & x+y-s  \\
B', & xy-t
\end{Bmatrix}.
\end{equation*}

Recall that  
\begin{align}
A' &= \frac{W(x+y,xy)-W(s,xy)}{x+y-s}, \label{slN:eq:Adef}\\
B' &= \frac{W(s,xy)-W(s,t)}{xy-t}. \label{slN:eq:Bdef}
\end{align}
Hence, the differential of this matrix factorization is 
given by the following 4 by 4 matrix:
\begin{equation*}
D=
\begin{pmatrix}
  0 & D_1 \\
D_0 & 0
\end{pmatrix},
\end{equation*}
where
\begin{equation*}
D_0=
\begin{pmatrix}
A', & xy-t \\
B', & s-x-y
\end{pmatrix}
,\quad
D_1=
\begin{pmatrix}
x+y-s, & xy-t \\
B',    & -A'
\end{pmatrix}.
\end{equation*}

\n Note that we are using a convention for tensor products of matrix factorizations that is different from the one in~\cite{mackaay-stosic-vaz}. The KL formula assigns the polynomial, $KL_{\Theta_1}(x,y,s,t)$, which is given by the supertrace of the 
twisted differential of $D$
\begin{equation*}
KL_{\Theta_1}=\str\left(\partial_x D\partial_y D \frac{1}{2}(\partial_s D\partial_t D-\partial_t D \partial_s D)\right).
\end{equation*}
Straightforward computation gives
\begin{equation}\label{slN:eq:KL1}
KL_{\Theta_1}=-B'_s(A'_x-A'_y)-(A'_x+A'_s)(B'_y+xB'_t)+(A'_y+A'_s)(B'_x+yB'_t),
\end{equation}
where by $A'_i$ and $B'_i$ we have denoted the partial derivatives with respect to the variable $i$.
From the definitions~\eqref{slN:eq:Adef} and~\eqref{slN:eq:Bdef} we have 
\begin{align*}
A'_x-A'_y &= (y-x)\frac{W'_2(x+y,xy)-W'_2(s,xy)}{x+y-s},\\
A'_x+A'_s &= \frac{W'_1(x+y,xy)-W'_1(s,xy)+y(W'_2(x+y,xy)-W'_2(s,xy))}{x+y-s},\\
A'_y+A'_s &= \frac{W'_1(x+y,xy)-W'_1(s,xy)+x(W'_2(x+y,xy)-W'_2(s,xy))}{x+y-s},\\
B'_s &= \frac{W'_1(s,xy)-W'_1(s,t)}{xy-t},\\
B'_x+yB'_t &= y\frac{W'_2(s,xy)-W'_2(s,t)}{xy-t},\\
B'_y+xB'_t &= x\frac{W'_2(s,xy)-W'_2(s,t)}{xy-t}.
\end{align*}
After substituting this back into~\eqref{slN:eq:KL1}, we obtain
\begin{equation}\label{slN:eq:eq1}
KL_{\Theta_1}=(x-y)\det\begin{pmatrix}
\alpha &\beta\\
\gamma & \delta 
\end{pmatrix},
\end{equation}  
where
\begin{align*}
\alpha &= \frac{W'_1(x+y,xy)-W'_1(s,xy)}{x+y-s},\\
\beta &= \frac{W'_2(x+y,xy)-W'_2(s,xy)}{x+y-s},\\
\gamma &= \frac{W'_1(s,xy)-W'_1(s,t)}{xy-t},\\
\delta &= \frac{W'_2(s,xy)-W'_2(s,t)}{xy-t}.
\end{align*}
From this formula we see that $KL_{\Theta_1}$ is homogeneous of degree $4N-6$ (remember that 
$\deg x=\deg y=\deg s=2$ and $\deg\,t=4$).

Since the evaluation is in the product of the Grassmannians corresponding to the three disks, i.e. in the ring
$\bQ[x]/(x^N) \times \bQ[y]/(y^N) \times \bQ[s,t] / ( W'_1(s,t),W'_2(s,t) ) $, we have 
$x^N=y^N=0=W'_1(s,t)=W'_2(s,t)$. Also, we can express the monomials in $s$ and $t$ as linear 
combinations of the Schur polynomials $\pi_{k,l}$ (writing $s=\pi_{1,0}$ and $t=\pi_{1,1})$), and 
we have $W'_1(s,t)=(N+1)\pi_{N,0}$ and $W'_2(s,t)=-(N+1)\pi_{N-1,0}$. Hence, we can write 
$KL_{\Theta_1}$ as
$$KL_{\Theta_1}=(x-y)\sum_{N-2\ge k\ge l\ge 0} {\pi_{k,l} p_{kl}(x,y)},$$
with $p_{kl}$ being a polynomial in $x$ and $y$. We want to determine which combinations of dots 
on the simple facets give rise to non-zero evaluations, so our aim is to compute the coefficient of $\pi_{N-2,N-2}$ in the sum on the r.h.s. of the above equation 
(i.e. in the determinant in~\eqref{slN:eq:eq1}). For degree reasons, this coefficient is of degree zero, 
and so we shall only compute the parts of $\alpha$, $\beta$, $\gamma$ and $\delta$ which do not 
contain $x$ and $y$. We shall denote these parts by putting 
a bar over the Greek letters. Thus we have
\begin{align*}
\bar{\alpha} &= (N+1)s^{N-1},\\
\bar{\beta} &= -(N+1)s^{N-2},\\
\bar{\gamma} &= \sum_{i+2j=N+1,\,j\ge 1}ia_{ij}s^{i-1}t^{j-1},\\
\bar{\delta} &= \sum_{i+2j=N+1,\,j\ge 2}ja_{ij}s^it^{j-2}.\\
\end{align*}
Note that we have 
$$t\bar{\gamma}+(N+1)s^N=W'_1(s,t),$$
and 
$$t\bar{\delta}-(N+1)s^{N-1}=W'_2(s,t),$$
and so in the cohomology ring of the Grassmannian $\cG_{2,N}$, we have $t\bar{\gamma}=-(N+1)s^N$ and $t\bar{\delta}=(N+1)s^{N-1}$.
On the other hand, by using $s=\pi_{1,0}$ and $t=\pi_{1,1}$, we obtain that in 
$\hy(\cG_{2,N})\cong\bQ[s,t]/ ( \pi_{N-1,0},\pi_{N,0} ) $, the following holds:
$$s^{N-2}=\pi_{N-2,0}+tq(s,t),$$
for some polynomial $q$, and so
$$s^{N-1}=s^{N-2}s=\pi_{N-1,0}+\pi_{N-2,1}+stq(s,t)=t(\pi_{N-3,0}+sq(s,t)).$$
Thus, we have
\begin{align}
\det
\begin{pmatrix}
\bar{\alpha} &\bar{\beta}\\
\bar{\gamma} &\bar{\delta} 
\end{pmatrix} &= 
(N+1)(\pi_{N-3,0}+sq(s,t))t\bar{\delta} +(N+1)\pi_{N-2,0}\bar{\gamma} 
+(N+1)q(s,t)t\bar{\gamma} \notag \\
&= (N+1)^2(\pi_{N-3,0}+sq(s,t))s^{N-1} + (N+1)\pi_{N-2,0}\bar{\gamma} 
- (N+1)^2 q(s,t) s^N \notag \\
&= (N+1)^2\pi_{N-3,0}s^{N-1}+ (N+1)\pi_{N-2,0}\bar{\gamma}. \label{slN:eq:eq2}
\end{align}
Since 
$$\bar{\gamma}=(N-1)a_{N-1,1}s^{N-2} + t r(s,t)$$
holds in the cohomology ring of the Grassmannian $\cG_{2,N}$ 
for some polynomial $r(s,t)$, we have
$$\pi_{N-2,0}\bar{\gamma}=\pi_{N-2,0}(N-1)a_{N-1,1}s^{N-2}=
-\pi_{N-2,0}(N-1)(N+1)s^{N-2}.$$
Also, we have that for every $k\ge 2$,
$$s^k= \pi_{k,0}+(k-1)\pi_{k-1,1}+t^2 w(s,t),$$
for some polynomial $w$. Replacing this in~\eqref{slN:eq:eq2} and bearing in mind that $\pi_{i,j}=0$, for $i\ge N-1$, we get
\begin{align*}
\det
\begin{pmatrix}
\bar{\alpha} &\bar{\beta}\\
\bar{\gamma} &\bar{\delta} 
\end{pmatrix}
&=  (N+1)^2 s^{N-2} (\pi_{N-2,0}+\pi_{N-3,1}-(N-1)\pi_{N-2,0}) \\
&=  (N+1)^2 (\pi_{N-2,0}+(N-3)\pi_{N-3,1}+\pi_{2,2}w(s,t)) (\pi_{N-3,1}-(N-2)\pi_{N-2,0}) \\
&=  -(N+1)^2 \pi_{N-2,N-2}.  
\end{align*}
Hence, we have
$$KL_{\Theta_1}=(N+1)^2 (y-x) \pi_{N-2,N-2} + 
\sum_{\substack{N-2\ge k \ge l\ge 0 \\ N-2>l} }
c_{i,j,k,l} \pi_{k,l} x^i y^j.$$
Recall that in the 
product of the Grassmannians corresponding to the three 
disks, i.e. in the ring 
$\bQ[x]/(x^N) \times \bQ[y]/(y^N) \times \bQ[s,t]/ ( \pi_{N-1,0},\pi_{N,0} )$, we have 
$$\epsilon(x^{N-1}y^{N-1}\pi_{N-2,N-2})=-1.$$
Therefore the only monomials $f$ in $x$ and $y$ such that 
$\KL{KL_{\Theta_1}f}\ne 0$ 
are $f_1=x^{N-1}y^{N-2}$ and $f_2=x^{N-2}y^{N-1}$, and  
$\KL{KL_{\Theta_1}f_1}=-(N+1)^2$ and 
$\KL{KL_{\Theta_1}f_2}=(N+1)^2$. Thus, we have that the value of 
the theta pre-foam with unlabelled 2-facet is nonzero only when the 
upper 1-facet has $N-2$ dots and the lower one has $N-1$ 
dots (and has the value $(N+1)^2$) and when the 
upper 1-facet has $N-1$ dots and the lower one has $N-2$ dots 
(and has the value $-(N+1)^2$). The evaluation of this theta foam with 
other labellings can be obtained from the result above by dot migration.

Up to normalization the KL evaluation of the $(1,1,2)$-theta pre-foam corresponds to the trace on the cohomology ring of the partial flag variety $Fl_{1,2,N}$  in Equation~\eqref{slN:eq:hflag} given by $\varepsilon(x_1^{N-2}x_2^{N-1})=1$, and where $x_1$ and $x_2$ correspond to the dots in the upper and lower facet respectively.

%%%%%%%%%%%%%%%%%%%%%%%%%%%%%%%%%%%%%%%%%%%%
\subsubsection{The $(1,2,3)$-theta pre-foam}
%%%%%%%%%%%%%%%%%%%%%%%%%%%%%%%%%%%%%%%%%%%%

For the theta pre-foam in Figure~\ref{slN:fig:theta123} the method is the same as in the previous case, just the computations are more complicated.
\begin{figure}[h!]
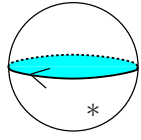

\labellist
\small\hair 2pt
\pinlabel $*$ at 130 32
\endlabellist
\centering
\figs{0.25}{theta-pre}
\caption{A dotless (1,2,3)-theta pre-foam}
\label{slN:fig:theta123}
\end{figure}
In this case, we have one 1-facet, to which we associate the 
variable $x$, one 2-facet, with variables $s$ and $t$ and 
the 3-facet with variables $p$, $q$ and $r$. Recall that the 
polynomial $W(p,q,r)$ is such that 
$W(a+b+c,ab+bc+ac,abc)=a^{N+1}+b^{N+1}+c^{N+1}$. We denote 
by $W'_i(p,q,r)$, $i=1,2,3$, the partial derivative of $W$ 
with respect to $i$-th variable. Also, let $A$, $B$ and $C$ 
be the polynomials given by
\begin{align}
A &= \frac{W(x+s,xs+t,xt)-W(p,xs+t,xt)}{x+s-p},\label{slN:eq:1}\\
B &= \frac{W(p,xs+t,xt)-W(p,q,xt)}{xs+t-q},\\
C &= \frac{W(p,q,xt)-W(p,q,r)}{xt-r}.\label{slN:eq:3}
\end{align}

To the singular circle of this theta pre-foam, we associated 
the matrix factorization 
(see Equations~\eqref{slN:eq:MF21}-\eqref{slN:eq:MF2}):
$$MF_2=
\begin{Bmatrix}
A, &  x+s-p  \\
B, &  xs+t-q \\
C, &  xt-r
\end{Bmatrix}.$$
The differential of this matrix factorization is the 8 by 8 matrix
\begin{equation*}
D=
\begin{pmatrix}
  0 & D_1 \\
D_0 & 0
\end{pmatrix}
,
\end{equation*}
where
\begin{equation*}
D_0=
\left(\begin{array}{c|c}
  d_0   & (xt-r)\id_2  \\
\hline
C\id_2  & -d_1 
\end{array}
\right),
\quad\quad
D_1=\left(\begin{array}{c|c}
    d_1 & (xt-r)\id_2 \\
\hline
C\id_2  & -d_0 \\
\end{array}
\right).
\end{equation*}
Here $d_0$ and $d_1$ are the differentials of the matrix factorization
$$
\begin{Bmatrix}
A, &  x+s-p  \\
B, & xs+t-q
\end{Bmatrix},$$
i.e.
$$d_0=
\begin{pmatrix}
 A & xs+t-q  \\
 B & p-x-s
\end{pmatrix},\quad\quad
d_1=
\begin{pmatrix}
x+s-p & xs+t-q \\
 B    &  -A
\end{pmatrix}.$$
The KL formula assigns to this theta pre-foam the polynomial 
$KL_{\Theta_2}(x,s,t,p,q,r)$ 
given as the supertrace of the twisted differential of $D$, i.e.
$$KL_{\Theta_2}=\str\left(\partial_x D\frac{1}{2}(\partial_s D\partial_t D-\partial_t D \partial_s D)
\partial_3D\hat{\ }\right),$$
where
\begin{align*}
\partial_3D\hat{\ } &= \frac{1}{3!}\left(\partial_p D\partial_qD 
\partial_rD-\partial_p D\partial_rD \partial_qD+\partial_q D\partial_rD \partial_pD\right. \\
& \qquad\quad\left.-\partial_q D\partial_pD \partial_rD+\partial_r D\partial_pD \partial_qD-\partial_r D\partial_qD \partial_pD\right).
\end{align*}
After straightforward computations and some grouping, we obtain
\begin{align*}
KL_{\Theta_2}
&= (A_p+A_s)\bigl[(B_t+B_q)(C_x+tC_r)-(B_x+sB_q)(C_t+xC_r)-(B_x-sB_t)C_q\bigr] \\
& \quad +\ (A_p+A_x)\bigl[(B_s+xB_q)(C_t+xC_r)+(B_s-xB_t)C_q\bigr] \\
& \quad +\ (A_x-A_s)\bigl[B_p(C_t+xC_r)-(B_t+B_q)C_p+B_pC_q\bigr] \\
& \quad -\ A_t\bigl[((B_s+xB_q)+B_p)(C_x+tC_r)+((B_s+xB_q) \\ 
& \quad -\ (B_x+sB_q))C_p+((sB_s-xB_x)+(s-x)B_p)C_q\bigr].
\end{align*}
In order to simplify this expression, we introduce the following polynomials
\begin{align*}
a_{1i} &= \frac{W'_i(x+s,xs+t,xt)-W'_i(p,xs+t,xt)}{x+s-p},\quad i=1,2,3, \\
a_{2i} &= \frac{W'_i(p,xs+t,xt)-W'_i(p,q,xt)}{xs+t-q},\qquad\qquad i=1,2,3, \\
a_{3i} &= \frac{W'_i(p,q,xt)-W'_i(p,q,r)}{xt-r},\qquad\qquad\qquad i=1,2,3.
\end{align*}

\medskip

Then from~\eqref{slN:eq:1}-\eqref{slN:eq:3}, we have
\begin{align*}
A_x+A_p&=a_{11}+sa_{12}+ta_{13}, &  A_p+A_s&=a_{11}+xa_{12}, \\
A_x-A_s&=(s-x)a_{12}+ta_{13}, & A_t&=a_{12}+xa_{13}, \\
B_p&=a_{21}, & B_s-xB_t&=-x^2a_{23}, \displaybreak[0] \\
sB_s-xB_x&=xta_{23}, & B_x-sB_t&=(t-sx)a_{23}, \\
B_t+B_q&=a_{22}+xa_{23}, & B_x+sB_q&=sa_{22}+ta_{23}, B_s+xB_q=xa_{22}, \\
C_p&=a_{31}, & C_q&=a_{32}, \\
 C_x+tC_r&=ta_{33}, & C_t+xC_r&=xa_{33}.
\end{align*}

\medskip

Using this $KL_{\Theta_2}$ becomes
\begin{equation}
\label{slN:eq:detheta123}
KL_{\Theta_2}=(t-sx+x^2)\det
\begin{pmatrix}
a_{11} & a_{12} & a_{13} \\
a_{21} & a_{22} & a_{23} \\
a_{31} & a_{32} & a_{33}
\end{pmatrix} .
\end{equation}

Now the last part follows analogously as in the case of the 
$(1,1,2)$-theta pre-foam. For degree reasons  
the coefficient of $\pi_{N-3,N-3,N-3}$ in the latter 
determinant is of degree zero, and one can obtain that it is 
equal to $(N+1)^3$. Thus, the coefficient of $\pi_{N-3,N-3,N-3}$ in $KL_{\Theta_2}$ is 
$(N+1)^3(t-sx+x^2)$ from which we obtain the value of the theta pre-foam when the 3-facet is 
undotted. For example, we see that 
$$\varepsilon\bigl(KL_{\Theta_2}\pi_{1,1}(s,t)^{N-3}x^{N-1}\bigr)=(N+1)^3.$$
It is then easy to obtain the values when the 3-facet is 
labelled by $\pi_{N-3,N-3,N-3}(p,q,r)$ using dot migration.  
The example above implies that 
$$\varepsilon\bigl(KL_{\Theta_2}\pi_{N-3,N-3,N-3}(p,q,r)x^2\bigr)=(N+1)^3.$$

Up to normalization the KL evaluation of the $(1,2,3)$-theta pre-foam corresponds to the trace on the cohomology ring of the partial flag variety $Fl_{2,3,N}$  in Equation~\eqref{slN:eq:hflag} given by $\varepsilon(x^2\pi_{N-3,N-3,N-3})=1$,  where $\pi_{N-3,N-3,N-3}$ correspond to a linear combination of dots in the triple facet and $x$ corresponds to a dot in the upper simple facet (see Section~\ref{slN:sec:partial flags}).

\medskip

For $N=3$, using the explicit formula for $W(p,q,r)$ we see that the determinant~\eqref{slN:eq:detheta123} is zero, which means that the $(1,2,3)$-theta pre-foams would evaluate to zero, independently of the dots they may have. That is why we restrict the construction in this chapter to the case of $N\geq 4$.

%%%%%%%%%%%%%%%%%%%%%%%%%%
\subsection{Normalization}
\label{slN:sec:norm}
%%%%%%%%%%%%%%%%%%%%%%%%%%

It will be convenient to normalize the KL evaluation. Let $u$ be a closed pre-foam 
with graph $\Gamma$. Note that $\Gamma$ has two types of edges: the ones incident to two 
simple facets and one double facet and the ones incident to one simple, one double and one triple facet. 
Edges of the same type form cycles in $\Gamma$. Let $e_{112}(u)$ be the total number of cycles in $\Gamma$ 
with edges of the first type and $e_{123}(u)$ the total number of cycles with edges of the second type. 
We normalize the KL formula by dividing $KL_u$ by 
$$(N+1)^{2e_{112}+3e_{123}}.$$
In the sequel we only use this normalized KL 
evaluation keeping the same notation $\KL{u}$.
%Note that the numbers $e_{112}(u)$ and $e_{123}(u)$ are invariant under the relation~\eqref{slN:eq:MP}. 
Note that with this normalization the KL-evaluation in the examples above always gives 
$0,-1$ or $1$.

%%%%%%%%%%%%%%%%%%%%%%%%%%%%%%%%%
\subsection{The glueing property}
\label{slN:sec:glueing}
%%%%%%%%%%%%%%%%%%%%%%%%%%%%%%%%%

We now consider the glueing property of the KL formula, which is an important property of TQFT's.

Suppose that $u$ is a pre-foam with boundary $\Gamma$. We decorate the facets, singular arcs and singular vertices of $u$ as in Subsection~\ref{slN:ssec:decoration}. Recall that the orientations of the singular arcs of $u$ induce an orientation of $\Gamma$ (see Figure~\ref{slN:fig:orientations}). To each vertex $\nu$ of $\Gamma$ we associate the matrix factorization which is the matrix factorization associated to the singular arc of $u$ that is bounded by $\nu$. To each circle in $\Gamma$ we associate the Jacobi algebra of the corresponding facet in $\bZ/2\bZ$-degree $i\pmod 2$, where $i=1,2,3$. Then define the matrix factorization $M_\Gamma$ as the tensor product of all the matrix factorizations of its vertices as given above and Jacobi algebras $J_i$ in $\bZ/2\bZ$-degree $i\pmod 2$ for all (if any) circles in $\Gamma$. The tensor product is taken over suitable rings so that $M_\Gamma$ is a free  module over $R$ of finite rank, where $R$ is the polynomial ring with rational coefficients in the variables of the facets of $u$ that are bounded by $\Gamma$. The factorization $M_\Gamma$ has potential zero, since for every edge $e$ of $\Gamma$ the individual potential $W_e$ appears twice in $W_\Gamma$ (one for each vertex bounding $e$) with opposite signs.
 The homology 
\begin{equation}\label{slN:eq:extg}
\hy_{\mf{}}(M_\Gamma)\cong\Ext(R,M_\Gamma)
\end{equation}
is finite-dimensional and coincides with the one in~\cite{KR} after using Lemma~\ref{KR:lem:exvar} to exclude the variables associated to all double and triple edges of $\Gamma$. 

Let $u$ be an open pre-foam whose boundary consists of two parts $\Gamma_1$ and $\Gamma_2$, and denote by $M_1$ and $M_2$ the matrix factorizations associated to $\Gamma_1$ and $\Gamma_2$ respectively. We say that $F$ is an \emph{interior facet} of $u$ if $\partial F\cap \partial u=\emptyset$. Restricting $KL_u$ to the interior facets of $u$ and doing the same to $\varepsilon$ in Equation~\eqref{slN:eq:ev} we see that the KL formula associates to $u$ an element of $\Ext(M_1,M_2)$. 

If $u'$ is another pre-foam whose boundary consists of $\Gamma_2$ and $\Gamma_3$, then it 
corresponds to an element of $\Ext(M_2,M_3)$, while the element associated to the pre-foam $uu'$, which is obtained by glueing the pre-foams $u$ and $u'$ along $\Gamma_2$, is equal to the composite of the elements associated to $u$ and $u'$. 

On the other hand, we can also see $u$ as a morphism from the empty web 
to its boundary $\Gamma=\Gamma_2\sqcup\Gamma_1^*$, where $\Gamma_1^*$ is equal to $\Gamma_1$ but with the opposite orientation. In that case, the KL formula 
associates to it an element from 
$$\Ext\bigl(R,M_{\Gamma_2}\otimes \bigl(M_{\Gamma_1}\bigr)_\bullet\bigr)\cong \hy_{\mf{}}(\Gamma).$$ 
Both ways of applying the KL formula are equivalent up to a global 
$q$-shift by corollary 6 in \cite{KR}.

In the case of a pre-foam $u$ with corners, i.e. a pre-foam with two horizontal boundary components 
$\Gamma_1$ and $\Gamma_2$ which are connected by vertical edges, one has to ``pinch'' the 
vertical edges. This way one can consider $u$ to be a morphism from the empty set to 
$\Gamma_2\cup_{\nu}\Gamma_1^*$, where $\cup_\nu$ means that the webs 
are glued at their vertices. The same observations as above hold, except that 
$M_{\Gamma_2}\otimes\bigl(M_{\Gamma_1}\bigr)_\bullet$ is now the tensor product over the polynomial ring 
in the variables associated to the horizontal edges with corners.

The KL formula also has a general property that will be useful later. The KL formula defines a duality pairing between $\Hom_{\foam}(\emptyset,\Gamma)$ and $\Hom_{\foam}(\Gamma,\emptyset)$ as
\begin{equation}\label{slN:eq:dualpairing}
(a,a')=\KL{a'a},
\end{equation}
for $a\in\Hom_{\foam}(\emptyset,\Gamma)$ and $a'\in\Hom_{\foam}(\Gamma,\emptyset)$.
From the duality pairing it follows that 
$$\Hom_{\foam}(\emptyset,\Gamma^*)=\Hom_{\foam}(\Gamma,\emptyset).$$

The duality pairing also defines a canonical element 
$$\psi_{\Gamma,\Gamma^*}\in\Hom_{\foam}(\emptyset,\Gamma^*)\otimes\Hom_{\foam}(\emptyset,\Gamma)$$ 
by
$$
(\psi_{\Gamma,\Gamma^*},a\otimes a') = (a,a')
$$
Introducing a basis $\{a_i\}$ of $\Hom_{\foam}(\emptyset,\Gamma)$ and its dual basis $\{a_j^*\}$ of $\Hom_{\foam}(\Gamma,\emptyset)$ we have
$$
\psi_{\Gamma,\Gamma^*}=\sum_ja_j\otimes a_j^*.
$$

Suppose that a closed pre-foam $u$ contains two points $p_1$ and $p_2$ such that intersecting $u$ with disjoint spheres centered in $p_1$ and $p_2$ result in two webs $\Gamma_1$ and $\Gamma_2$ and that $\Gamma_2 = \Gamma_1^*$. If we remove the parts inside those spheres from $u$ and glue the boundary components $\Gamma_1$ and $\Gamma_2$ onto each other we obtain a new closed pre-foam $u'$ and the KL evaluations of $u$ and $u'$ are related by (see~\cite{KR-LG})
\begin{equation}\label{slN:eq:CN-KL}
\KL{u'}=\KL{\psi_{\Gamma_1,\Gamma^*_1}u}=\sum_j\KL{a_j^* u a_j}.
\end{equation}

%%%%%%%%%%%%%%%% End of Part 2 %%%%%%%%%%%%%%%%
%
%
%%%%%%%%%%%%%% Part 3 %%%%%%%%%%%%%%%%%
%
%
%
%
%
%%%%%%%%%%%%%%%%%%%
% local relations %
%%%%%%%%%%%%%%%%%%%
%
\section{The category $\foam$}
\label{slN:sec:foamN}
Recall that $\KL{u}$ denotes the KL evaluation of a closed 
pre-foam $u$.  
\begin{defn}
\label{slN:defn:foam}
The category $\foam$ is the quotient of the category $\PF$ by the 
kernel of $\KL{\ }$, i.e. by the following identifications: for any webs $\Gamma$, 
$\Gamma'$ 
and finite sets $f_i\in\Hom_{\PF}(\Gamma,\Gamma')$ 
and $c_i\in\bQ$ we impose the relations 
$$\sum\limits_{i}c_if_i=0\quad \Leftrightarrow\quad
\sum\limits_{i}c_i\KL{g'f_ig}=0,$$ 
for all 
$g\in\Hom_{\PF}(\emptyset,\Gamma)$ and  
$g'\in\Hom_{\PF}(\Gamma',\emptyset)$. The 
morphisms of $\foam$ are called \emph{foams}. 
\end{defn}

In the next two propositions we prove the ``principal'' relations in $\foam$. 
All other relations that we need are consequences of these and will be proved in 
subsequent lemmas and corollaries.     

\begin{prop}
\label{slN:prop:principal rels1}
The following identities hold in $\foam$: 

%%%%%%%%%%%%%%%%%%%
% dot conversion  %
%%%%%%%%%%%%%%%%%%%
 
(The \emph{dot conversion} relations) 
\begin{align*}
\figins{-8}{0.3}{plan-i} &= 0\quad\text{if}\quad i\geq N
\\[1.5ex]
\figins{-8}{0.3}{dplan-km} &= 0\quad\text{if}\quad k\geq N-1
\\[1.5ex]
\figins{-8}{0.3}{plan-pqr} &= 0\quad\text{if}\quad p\geq N-2
\end{align*}

%%%%%%%%%%%%%%%%%
% dot migration %
%%%%%%%%%%%%%%%%%

\bigskip

(The \emph{dot migration} relations)
\begin{align*}
\figins{-14}{0.45}{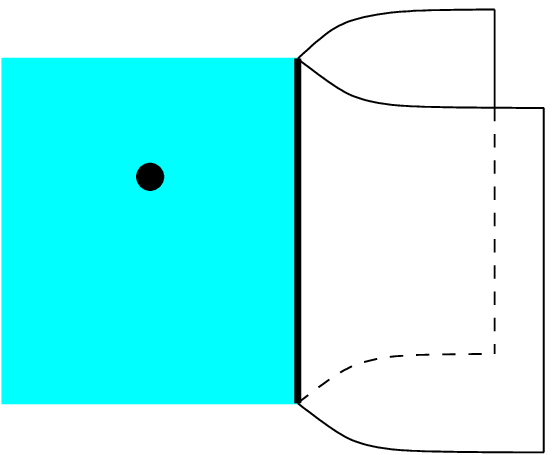}\ &=  \quad
\figins{-14}{0.45}{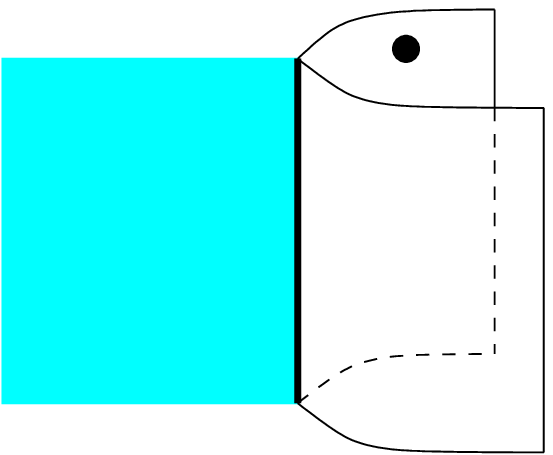} \ + \
\figins{-14}{0.45}{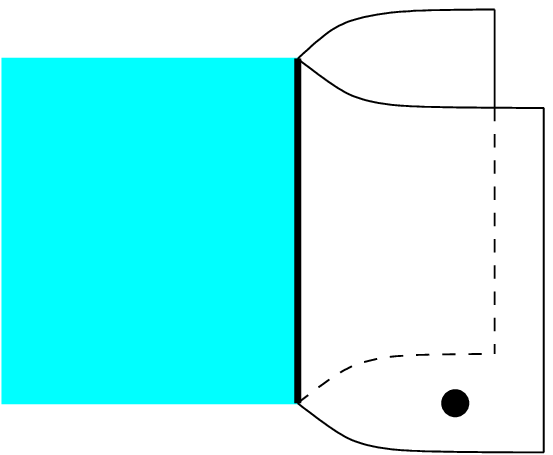}
\\[1ex]
\figins{-14}{0.45}{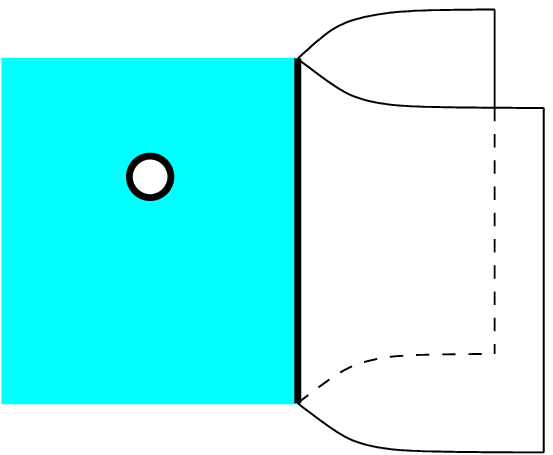}\ &=  
\quad
\figins{-14}{0.45}{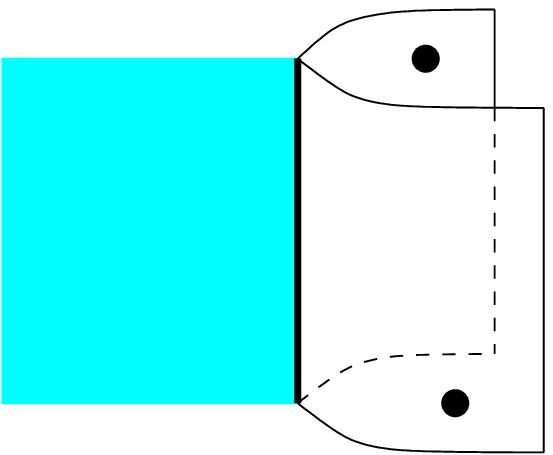}
\\[2ex]
\figins{-14}{0.45}{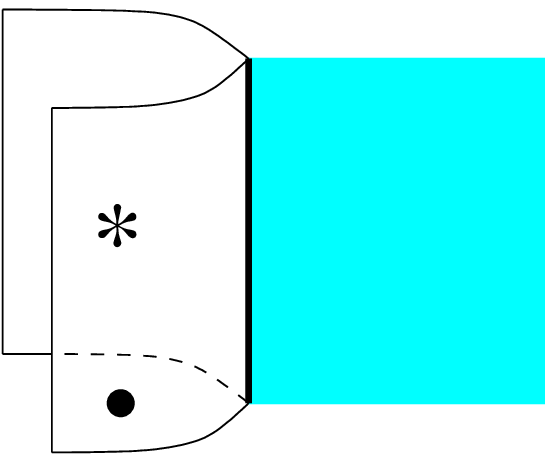} \ \ \ &=  \
\figins{-14}{0.45}{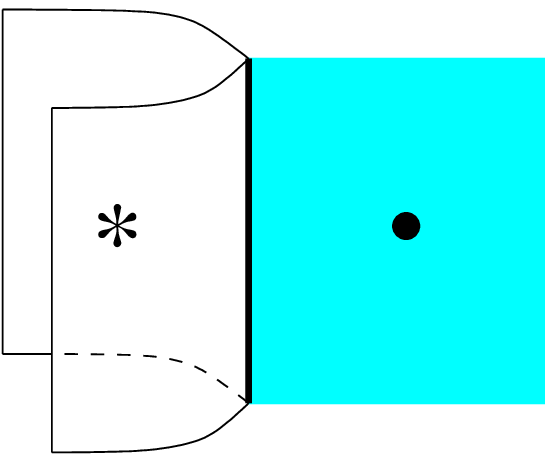} \ + \
\figins{-14}{0.45}{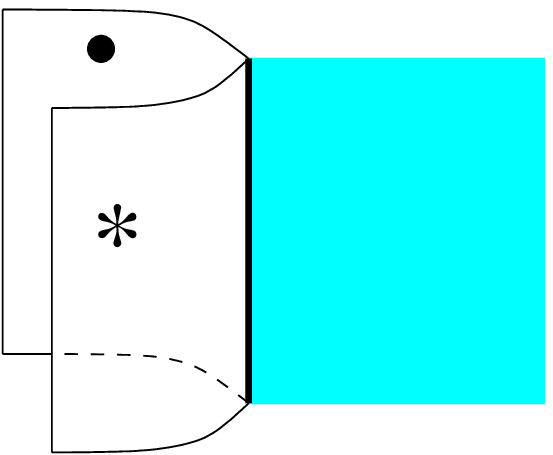} 
\\[1ex]
\figins{-14}{0.45}{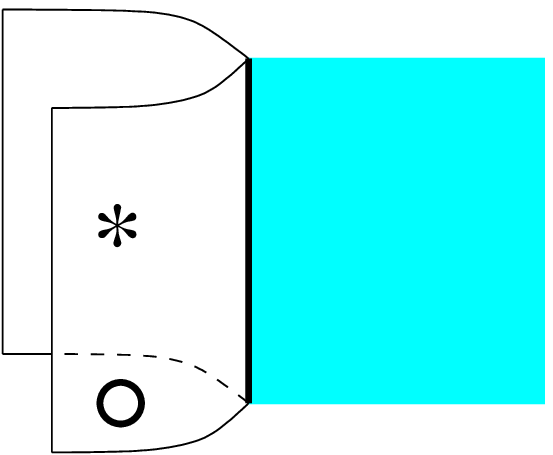} \ \ \ &=  \
\figins{-14}{0.45}{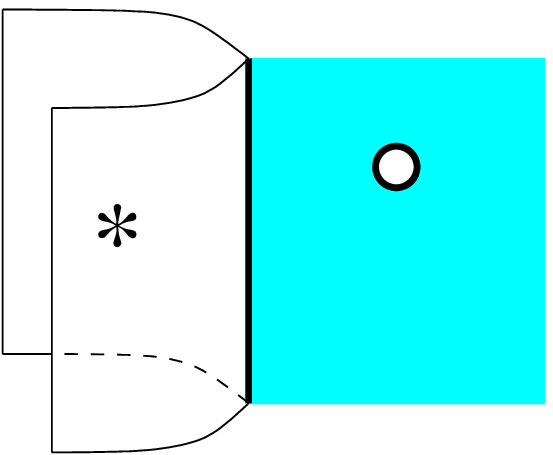} \ + \
\figins{-14}{0.45}{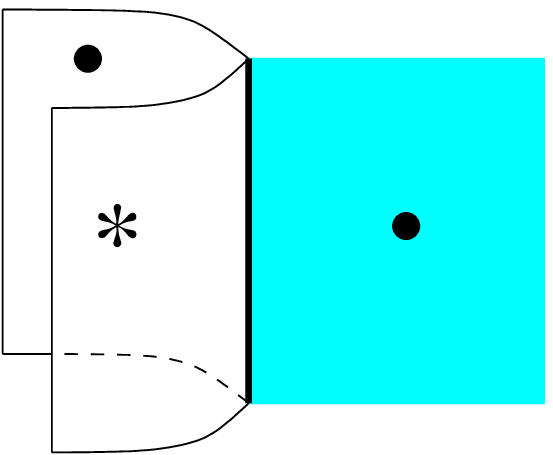}
\\[1ex]
\figins{-14}{0.45}{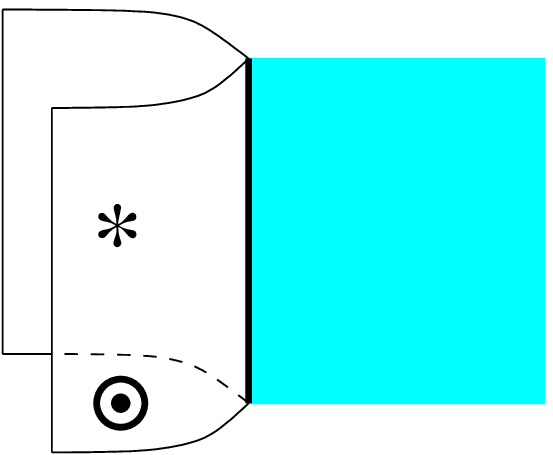} \ \ \ &=  \
\figins{-14}{0.45}{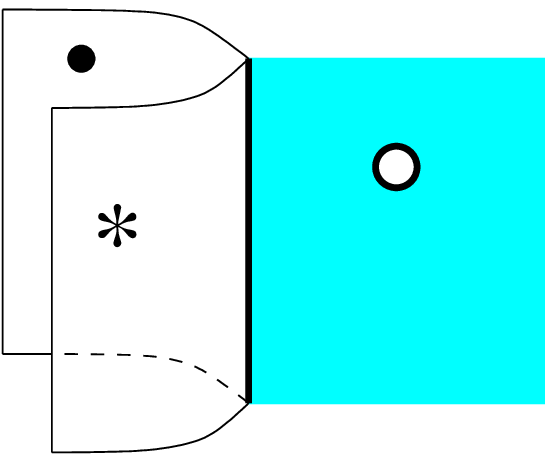}
\end{align*}

\bigskip

(The \emph{cutting neck} relations\footnote{The notation $\widehat{(i,j)}$, $\widehat{(i,j,k)}$ was explained in Subsection~\ref{slN:ssec:hflag}.}) 
$$
\figins{-27}{0.8}{cylinder}=
\sum\limits_{i=0}^{N-1}
\figins{-27}{0.8}{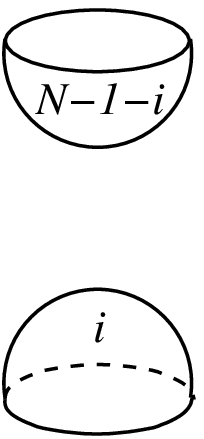}
\qquad\text{\emph{(CN$_1$)}}
$$
$$   
\figins{-28}{0.8}{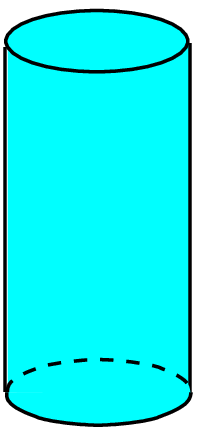}=-
\sum\limits_{0\leq j\leq i\leq N-2}
\figins{-28}{0.8}{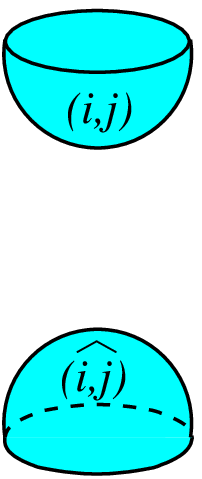}
\qquad\text{\emph{(CN$_2$)}}
\qquad\qquad
\figins{-28}{0.8}{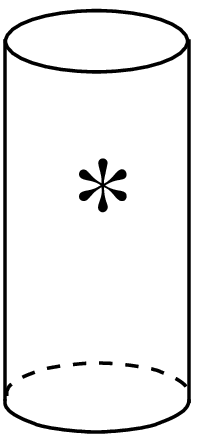}=
-\sum\limits_{0\leq k\leq j\leq i\leq N-3}
\figins{-28}{0.8}{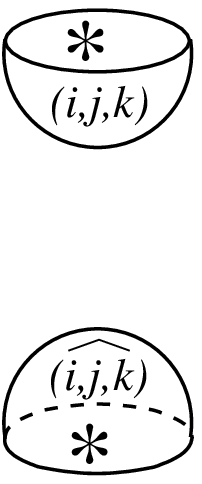}
\qquad\text{\emph{(CN$_*$)}}
$$

\bigskip

(The \emph{sphere} relations)
$$
\figins{-10}{0.35}{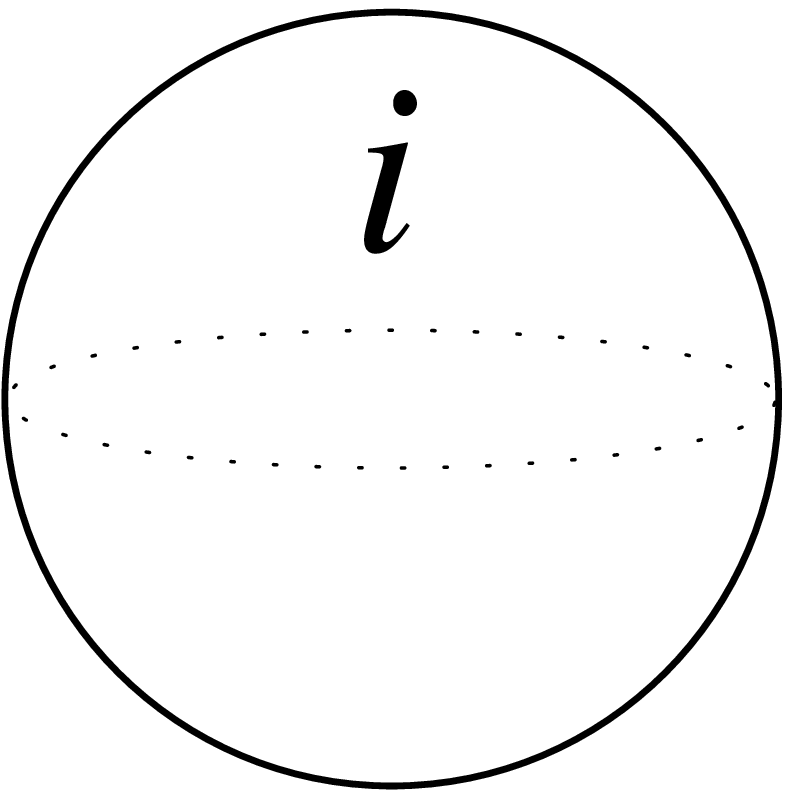}=
\begin{cases}1, & i=N-1 \\ 0, & \text{else}\end{cases}
\qquad 
\text{\emph{(S$_1$)}} \qquad\quad
\figins{-10}{0.35}{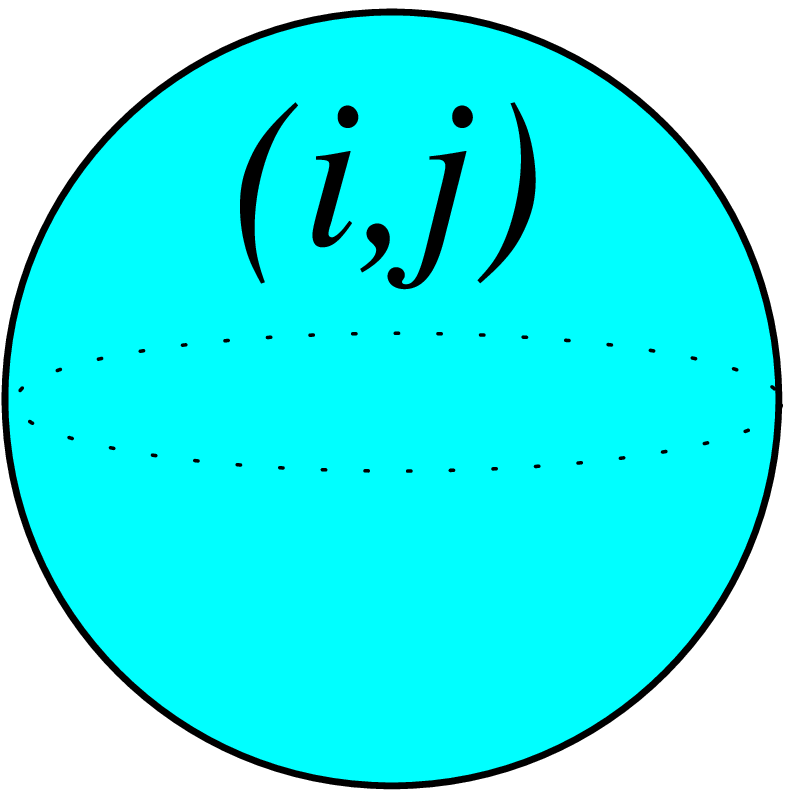}=
\begin{cases}-1, & i=j=N-2 \\ 0, & \text{else}\end{cases}
\qquad 
\text{\emph{(S$_2$)}}
$$

$$
\figins{-10}{0.35}{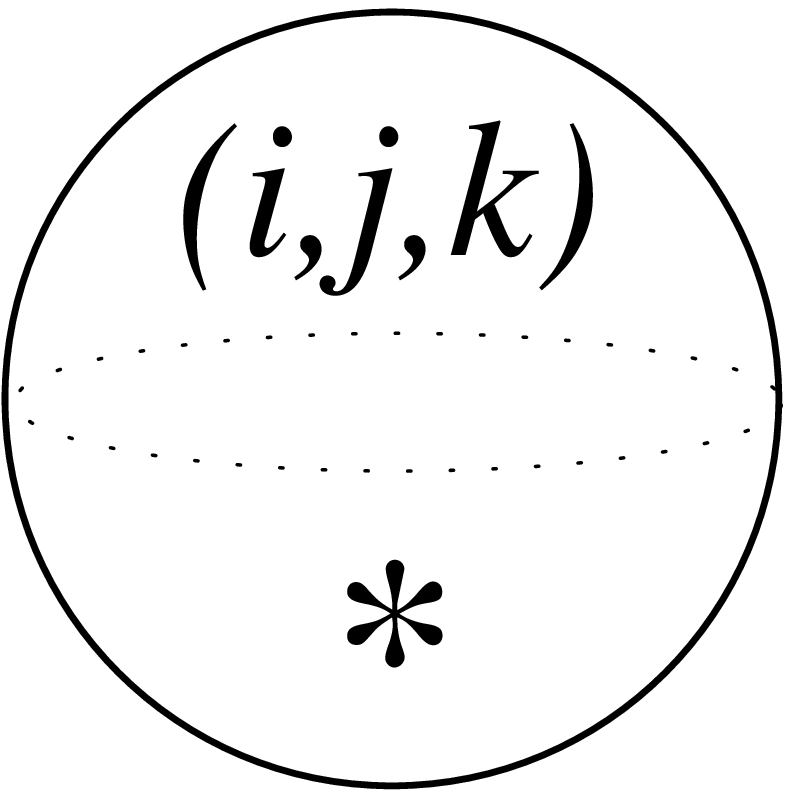}=
\begin{cases}-1, & i=j=k=N-3 \\ 0, & \text{else}.\end{cases}
\qquad
\text{\emph{(S$_*$)}}
$$

\bigskip

(The \emph{$\Theta$-foam} relations)
$$
\figins{-10}{0.35}{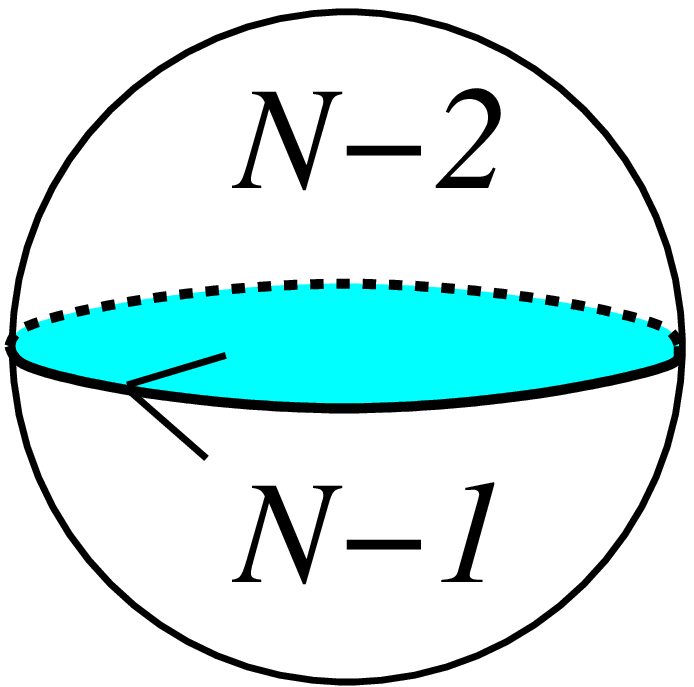} = 1  = - \
\figins{-10}{0.35}{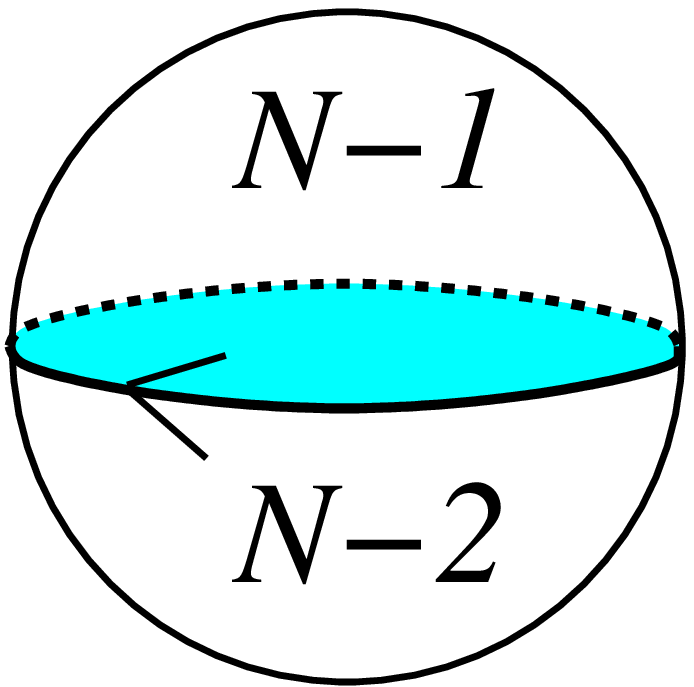} \quad
 (\ThetaGraph)
\qquad \ \text{and} \ \qquad
\figins{-16}{0.42}{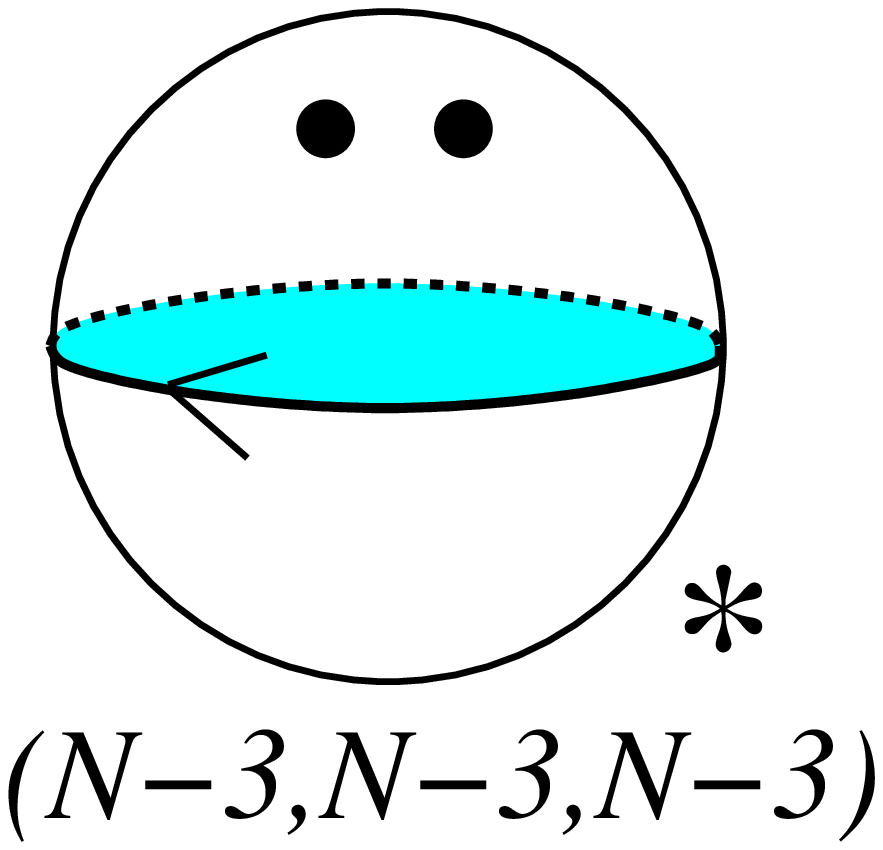} = 1  = -  
\figins{-10}{0.42}{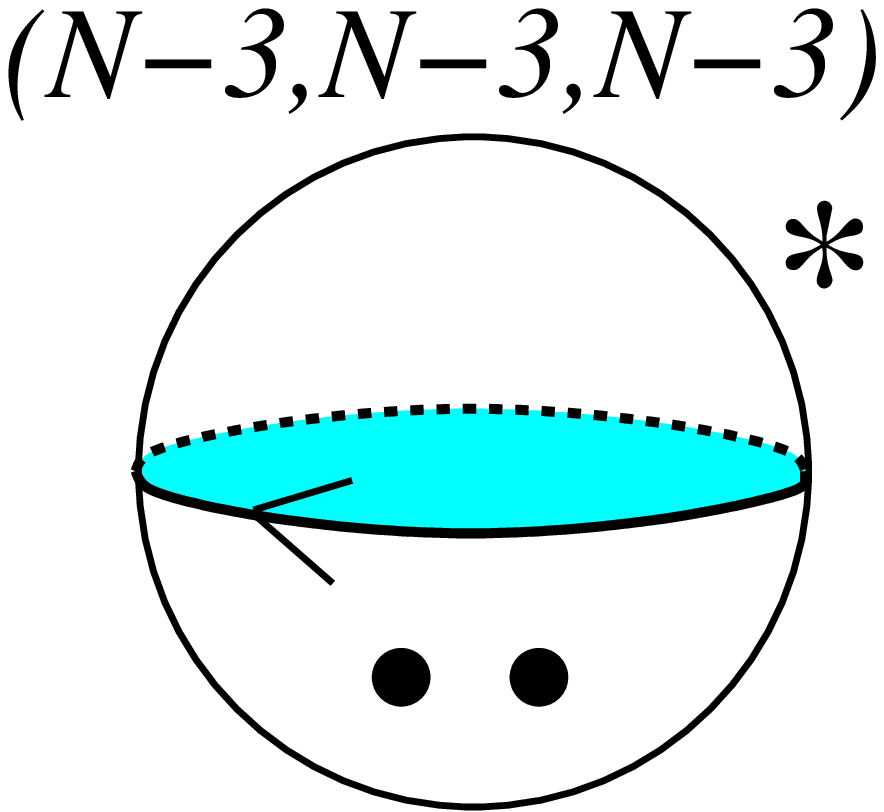}
 \quad
 (\ThetaGraph_*).$$
Inverting the orientation of the singular circle of $(\ThetaGraph_*)$ inverts the sign of the 
corresponding foam. A theta-foam with dots on the double facet can be transformed into a 
theta-foam with dots only on the other two facets, using the dot migration relations.

%%%%%%%%%%%%%%%%%%%
% the MP-relation %
%%%%%%%%%%%%%%%%%%%

\bigskip

(The \emph{Matveev-Piergalini} relation) 
\begin{equation}
\figins{-24}{0.8}{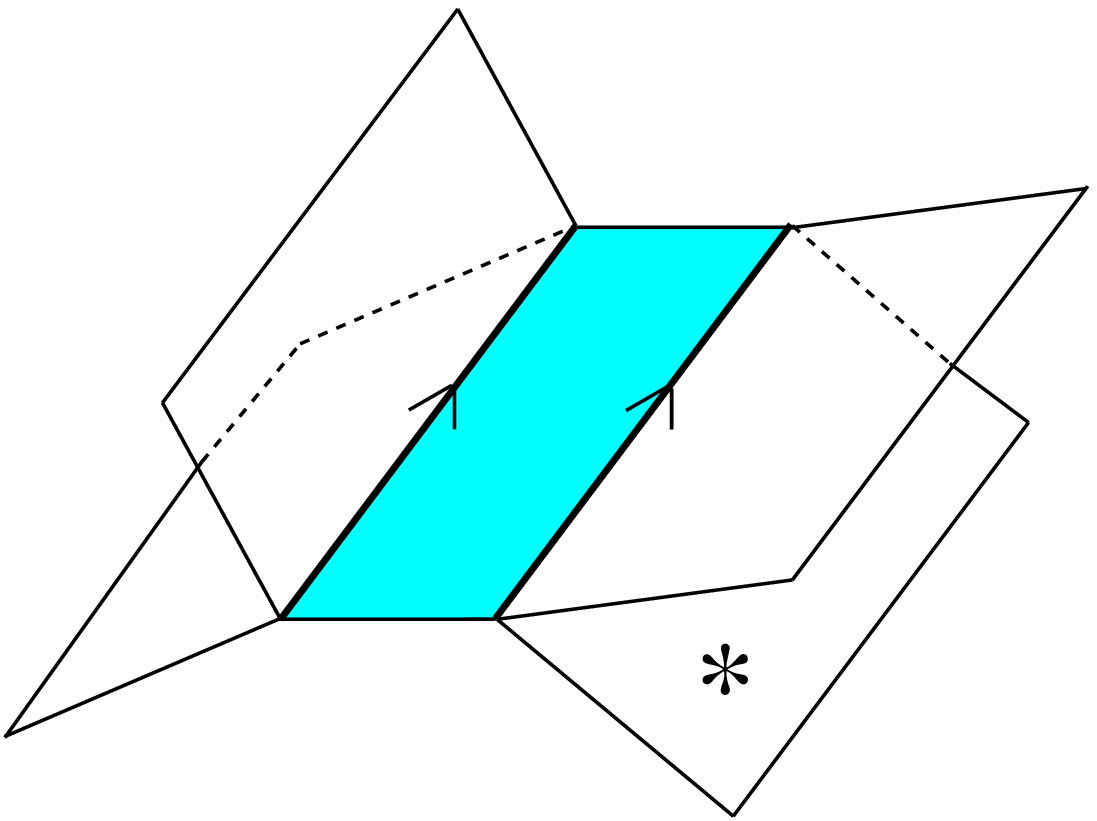} =
\figins{-24}{0.8}{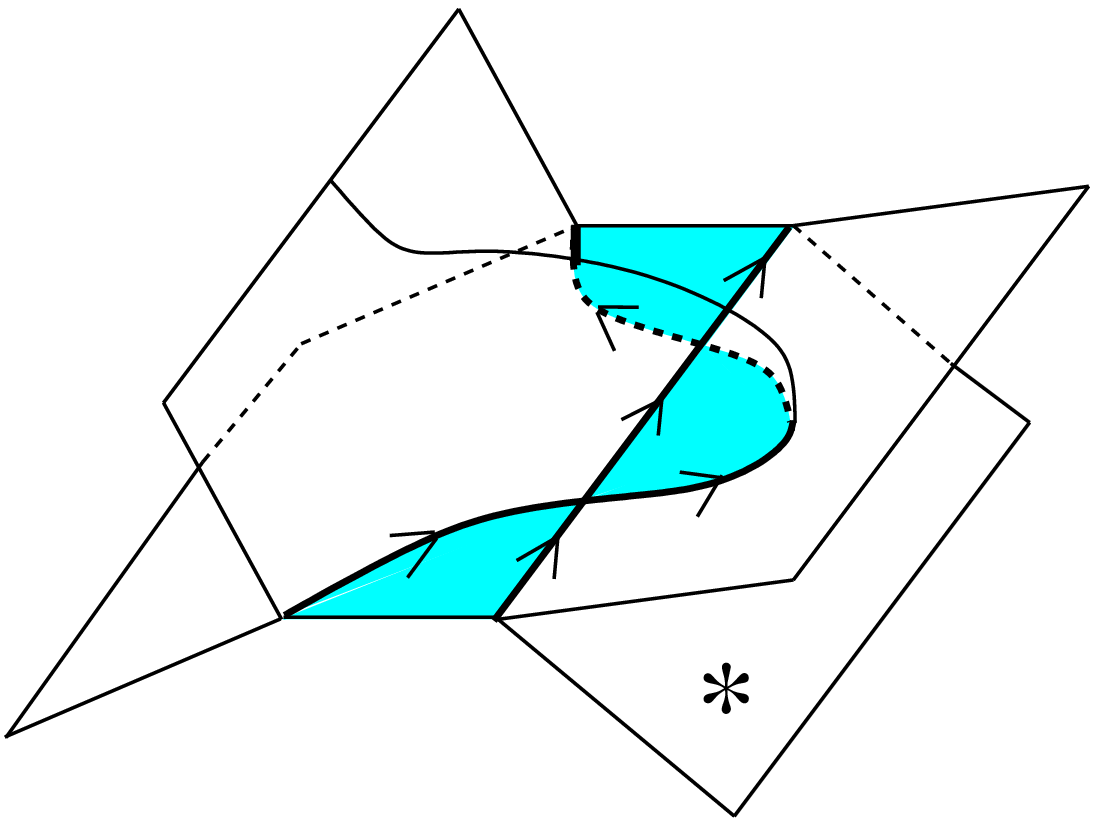},\quad
\figins{-24}{0.8}{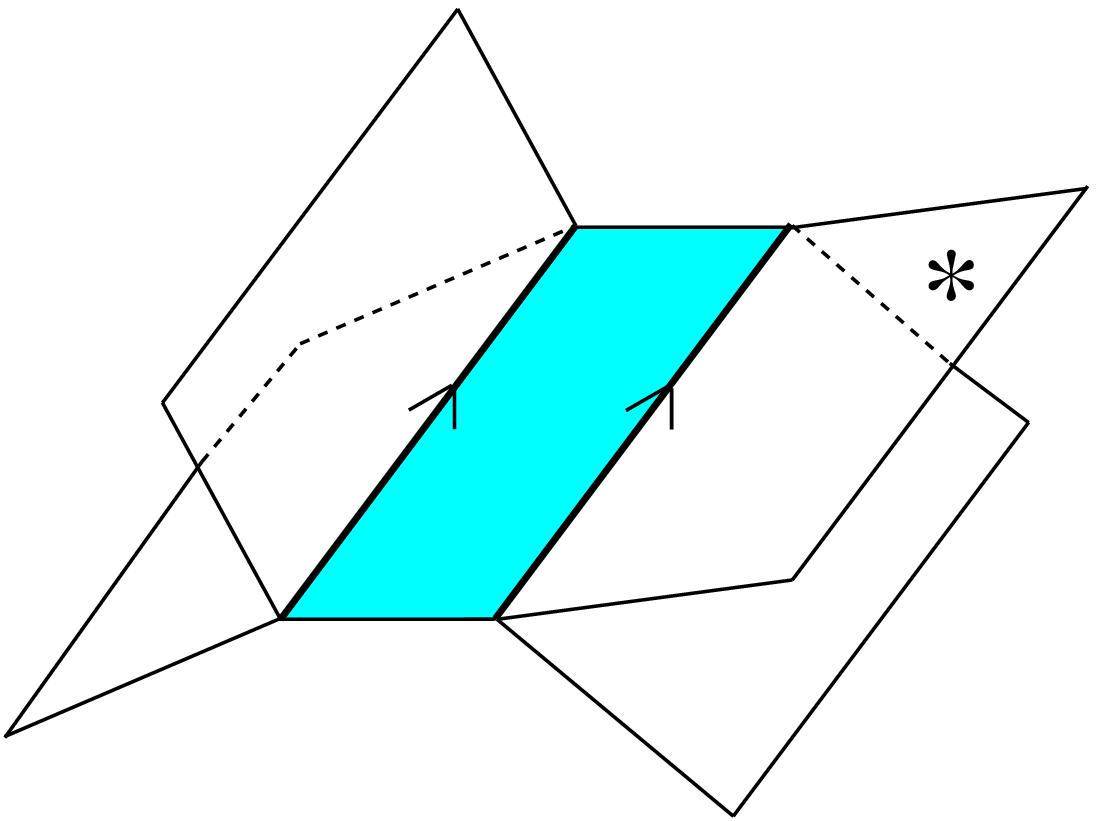} =
\figins{-24}{0.8}{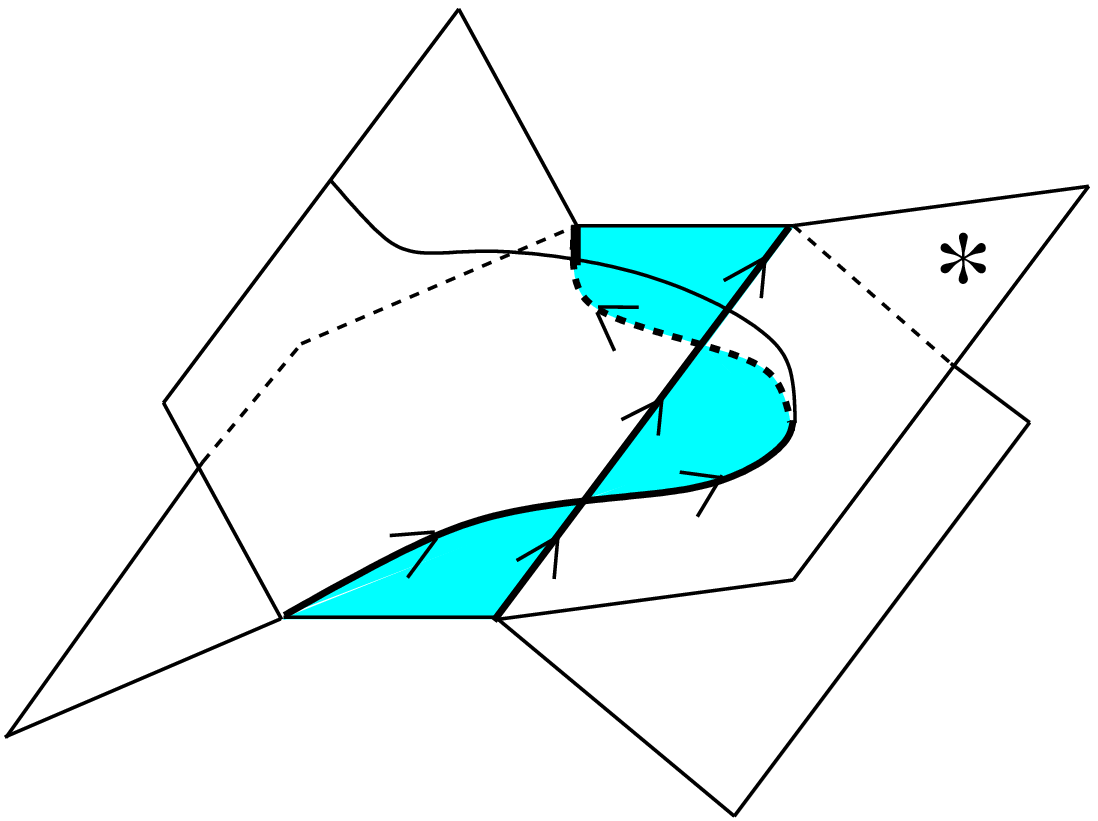}.
\tag{MP}\label{slN:eq:MP}
\end{equation}
\end{prop}

\begin{proof}
The dot conversion and migration relations, the sphere relations, 
the theta foam relations have already been proved in Section~\ref{slN:sec:KL}. 

The cutting neck relations are special cases of Equation~\eqref{slN:eq:CN-KL} where $a_j$ and $a_j^*$ can be read off from Equations~\eqref{slN:eq:db}.

The Matveev-Piergalini~\eqref{slN:eq:MP} relation is an immediate consequence of the choice of input 
for the singular vertices. Note that in this relation there are always two singular 
vertices of different type. The elements in the $\Ext$-groups associated to those two 
types of singular vertices are inverses of each other, which implies exactly the~\eqref{slN:eq:MP} relation by the glueing properties explained in Subsection~\ref{slN:sec:glueing}.  
\end{proof}

The following identities are a consequence of the dot and the theta relations. 

\begin{lem} 
\label{slN:lem:theta-pqrkli}
$$
\figins{-17}{0.55}{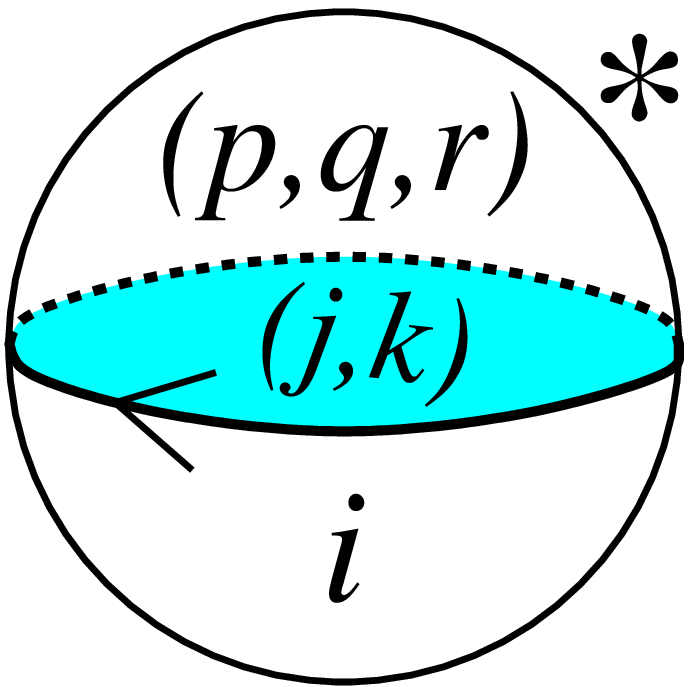}=
\begin{cases}
  -1&\text{if}\quad(p,q,r)=
(N-3-i,N-2-k,N-2-j)\\
-1&\text{if}\quad(p,q,r)=(N-3-k, N-3-j,N-1-i)\\
1&\text{if}\quad (p,q,r)=(N-3-k,N-2-i,N-2-j)\\
0&\text{else}
\end{cases}
$$
Note that the first three cases only make sense if 
\begin{align*}
&N-2\geq j\geq k\geq i+1
\geq 1\\
&N-1\geq i\geq j+2\geq k+2\geq 2\\
&N-2\geq j\geq i\geq k+1\geq 1
\end{align*}
respectively.  
\end{lem}

\begin{proof}
We denote the value of a theta foam by 
$\Theta(\pi_{p,q,r},\pi_{j,k},i)$. Since the $q$-degree of a non-decorated theta foam is 
equal to $-(N-1)-2(N-2)-3(N-3)=-(6N-14)$, we can have nonzero values of 
$\Theta(\pi_{p,q,r},\pi_{j,k},i)$ only if $p+q+r+j+k+i=3N-7$. Thus, if the 3-facet is not 
decorated, i.e. $p=q=r=0$, we have only four possibilities for the triple $(j,k,i)$ -- 
namely $(N-2,N-2,N-3)$, $(N-2,N-3,N-2)$, $(N-2,N-4,N-1)$ and 
$(N-3,N-3,N-1)$.
By Proposition~\ref{slN:prop:principal rels1} we have 
$$\Theta(\pi_{0,0,0},\pi_{N-2,N-2},N-3)=-1.$$
However by dot migration, Lemma~\ref{slN:lem1} and the fact that 
$\pi_{p,q,r}=0$ if $p\ge N-2$, we have
\begin{align*}
0 = \Theta(\pi_{N-2,N-2,N-3},\pi_{0,0},0) &= \Theta(\pi_{0,0,0},
\pi_{N-2,N-2},N-3)+\Theta(\pi_{0,0,0},\pi_{N-2,N-3},N-2),\\
0 =\Theta(\pi_{N-1,N-3,N-3},\pi_{0,0},0) &= \Theta(\pi_{0,0,0},
\pi_{N-2,N-3},N-2)+\Theta(\pi_{0,0,0},\pi_{N-3,N-3},N-1),\\
0 = \Theta(\pi_{N-1,N-2,N-4},\pi_{0,0},0) &= \Theta(\pi_{0,0,0},
\pi_{N-2,N-2},N-3)+\Theta(\pi_{0,0,0},\pi_{N-2,N-3},N-2) \\
 & \quad +\Theta(\pi_{0,0,0},\pi_{N-2,N-4},N-1).
\end{align*}
Thus, the only nonzero values of the theta foams, when the 3-facet is nondecorated are
\begin{gather*}
\Theta(\pi_{0,0,0},\pi_{N-2,N-2},N-3)=\Theta(\pi_{0,0,0},
\pi_{N-3,N-3},N-1)=-1, \\
\Theta(\pi_{0,0,0},\pi_{N-2,N-3},N-2)=+1.
\end{gather*}

Now we calculate the values of the general theta foam. Suppose first that $i\le k$. 
Then we have 
\begin{equation}
\Theta(\pi_{p,q,r},\pi_{j,k},i)=\Theta(\pi_{p,q,r},\pi_{i,i}\pi_{j-i,k-i},i)=
\Theta(\pi_{p+i,q+i,r+i},\pi_{j-i,k-i},0),
\label{slN:eq:th1}
\end{equation}
by dot migration. In order to calculate $\Theta(\pi_{x,y,z},\pi_{w,u},0)$
for $N-3\ge x\ge y\ge z \ge 0$ and $N-2 \ge w\ge u\ge 0$, we use Lemma~\ref{slN:lem1}. By 
dot migration we have
\begin{equation}
\Theta(\pi_{x,y,z},\pi_{w,u},0)=\sum_{(a,b,c)\sqsubset(x,y,z)}\Theta(\pi_{0,0,0},\pi_{w,u}\pi_{a,b},c).
\label{slN:eq:f1}
\end{equation}
Since $c\le p\le N-3$, a summand on the r.h.s. of~\eqref{slN:eq:f1} can be nonzero only 
for $c=N-3$ and $a$ and $b$ such that $\pi_{N-2,N-2}\in \pi_{w,u}\pi_{a,b}$, i.e. $a=N-2-u$ and $b=N-2-w$. Hence the value of~\eqref{slN:eq:f1} is equal to $-1$ if 
\begin{equation}
(N-2-u,N-2-w,N-3)\sqsubset(x,y,z),
\label{slN:eq:f2}
\end{equation}
and $0$ otherwise. Finally,~\eqref{slN:eq:f2} is equivalent to $x+y+z+w+u=3N-7$, $x\ge N-2-u\ge y$, $y\ge N-2-w\ge z$ and $x\ge N-3 \ge z$, and so we must have $u>0$ and
\begin{align*}
x &= N-3,\\
y &= N-2-u,\\
z &= N-2-w.
\end{align*}
Going back to~\eqref{slN:eq:f1}, we have that the value of theta is equal to $0$ if $l=i$, and in the case $l>i$ it is nonzero (and equal to $-1$) iff
\begin{align*}
p &= N-3-i,\\
q &= N-2-k,\\
r &= N-2-j,
\end{align*}
which gives the first family.

Suppose now that $k<i$. As in~\eqref{slN:eq:th1} we have
\begin{equation}
\Theta(\pi_{p,q,r},\pi_{j,k},i)=\Theta(\pi_{p+k,q+k,r+k},\pi_{j-k,0},i-k).
\label{slN:eq:th2}
\end{equation}
Hence, we now concentrate on $\Theta(\pi_{x,y,z},\pi_{w,0},u)$
for $N-3\ge x\ge y\ge z \ge 0$, $N-2 \ge w\ge 0$ and $N-1\ge u\ge 1$. Again, by using Lemma~\ref{slN:lem1} we have
\begin{equation}
\Theta(\pi_{x,y,z},\pi_{w,0},0)=\sum_{(a,b,c)\sqsubset(x,y,z)}\Theta(\pi_{0,0,0},\pi_{w,0}\pi_{a,b},u+c).
\label{slN:eq:f3}
\end{equation}
Since $a\le N-3$, we cannot have $\pi_{N-2,N-2}\in \pi_{w,0}\pi_{a,b}$ and we can have 
$\pi_{N-2,N-3}\in \pi_{w,0}\pi_{a,b}$ iff $a=N-3$ and $b=N-2-w$. In this case we have a nonzero 
summand (equal to $1$) iff $c=N-2-u$. Finally $\pi_{N-3,N-3}\in \pi_{w,0}\pi_{a,b}$ iff 
$a=N-3$ and $b=N-3-w$. In this case we have a nonzero summand (equal to $-1$) iff $c=N-1-u$. 
Thus we have a summand on the r.h.s. of~\eqref{slN:eq:f3} equal to $+1$ iff 
\begin{equation}
(N-3,N-2-w,N-2-u)\sqsubset (x,y,z),
\label{slN:eq:rel1}
\end{equation} 
and a summand equal to $-1$ iff 
\begin{equation}
(N-3,N-3-w,N-1-u)\sqsubset (x,y,z).
\label{slN:eq:rel2}
\end{equation} 
Note that in both above cases we must have $x+y+z+w+u=3N-7$, $x=N-3$ and $u\ge 1$. Finally, 
the value of the r.h.s of~\eqref{slN:eq:f3} will be nonzero iff exactly one of~\eqref{slN:eq:rel1} and~\eqref{slN:eq:rel2} holds.

In order to find the value of the sum on r.h.s. of~\eqref{slN:eq:f3}, we split the rest of the proof in three cases according to the relation between $w$ and $u$.

If $w\ge u$,~\eqref{slN:eq:rel1} is equivalent to $y\ge N-2-w$, $z\le N-2-w$, while~\eqref{slN:eq:rel2} is 
equivalent to $y\ge N-3-w$, $z\le N-3-w$ and $u\ge 2$. Now, we can see that the sum is nonzero 
and equal to $1$ iff $z=N-2-w$ and so $y=N-2-u$. Returning to~\eqref{slN:eq:th2}, we have that 
the value of $\Theta(\pi_{p,q,r},\pi_{j,k},i)$ is equal to $1$ for 
\begin{align*}
p &= N-3-k,\\
q &= N-2-i,\\
r &= N-2-j,
\end{align*}
for $N-2\ge j \ge i \ge k+1 \ge 1$, which is our third family.

If $w\le u-2$,~\eqref{slN:eq:rel1} is equivalent to $y\ge N-2-w$, $z\le N-2-u$ and $u\le N-2$ while~\eqref{slN:eq:rel2} is equivalent to $y\ge N-3-w$, $z\le N-1-u$. Hence, in this case we have that the total sum is nonzero and equal to $-1$ iff $y=N-3-w$ and $z=N-1-u$, which by returning to~\eqref{slN:eq:th2} gives that the value of  $\Theta(\pi_{p,q,r},\pi_{j,k},i)$ is equal to $-1$ for
\begin{align*}
p &= N-3-k,\\
q &= N-3-j,\\
r &= N-1-i,
\end{align*}
for $N-3\ge i-2\ge j \ge k \ge 0$, which is our second family.

Finally, if $u=w+1$~\eqref{slN:eq:rel1} becomes equivalent to $y\ge N-2-w$ and $z\le N-3-w$, 
while~\eqref{slN:eq:rel2} becomes $y\ge N-3-w$ and $z\le N-3-w$. Thus, in order to have a nonzero sum, 
we must have $y=N-3-w$. But in that case, because of the fixed total sum of indices, we 
would have $z=N-1-u=N-2-w>N-3-w$, which contradicts~\eqref{slN:eq:rel2}. Hence, in this case, 
the total value of the theta foam is $0$. 
\end{proof}
As a direct consequence of the previous theorem, we have
\begin{cor}
\label{slN:cor:theta-pqrkli}
For fixed values of $j$, $k$ and $i$, if $j\ne i-1$ and $k\ne i$, there is exactly one triple $(p,q,r)$ such that the value of $\Theta(\pi_{p,q,r},\pi_{j,k},i)$ is nonzero. Also, if $j=i-1$ or 
$k=i$, the value of $\Theta(\pi_{p,q,r},\pi_{j,k},i)$ is equal to $0$ for every triple $(p,q,r)$. 
Hence, for fixed $i$, there are $\binom{n-1}{2}$ 5-tuples $(p,q,r,j,k)$ such that 
$\Theta(\pi_{p,q,r},\pi_{j,k},i)$ is nonzero.

Conversely, for fixed $p$, $q$ and $r$, there always exist three different triples $(j,k,i)$ 
(one from each family), such that $\Theta(\pi_{p,q,r},\pi_{j,k},i)$ is nonzero.

Finally, for all $p$, $q$, $r$, $j$, $k$ and $i$, we have
$$\Theta(\pi_{p,q,r},\pi_{j,k},i)=\Theta(\hat{\pi}_{p,q,r},\hat{\pi}_{j,k},N-1-i).$$
\end{cor} 
The following relations are an immediate consequence of 
Lemma~\ref{slN:lem:theta-pqrkli}, 
Corollary~\ref{slN:cor:theta-pqrkli} and (CN$_i$), $i=1,2,*$.
\begin{cor} 
\label{slN:cor:bubble-star}

\begin{equation}
\label{slN:eq:bubble-pqri}
\figins{-18}{0.6}{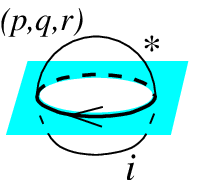}=
\begin{cases}
-\figins{-17}{0.4}{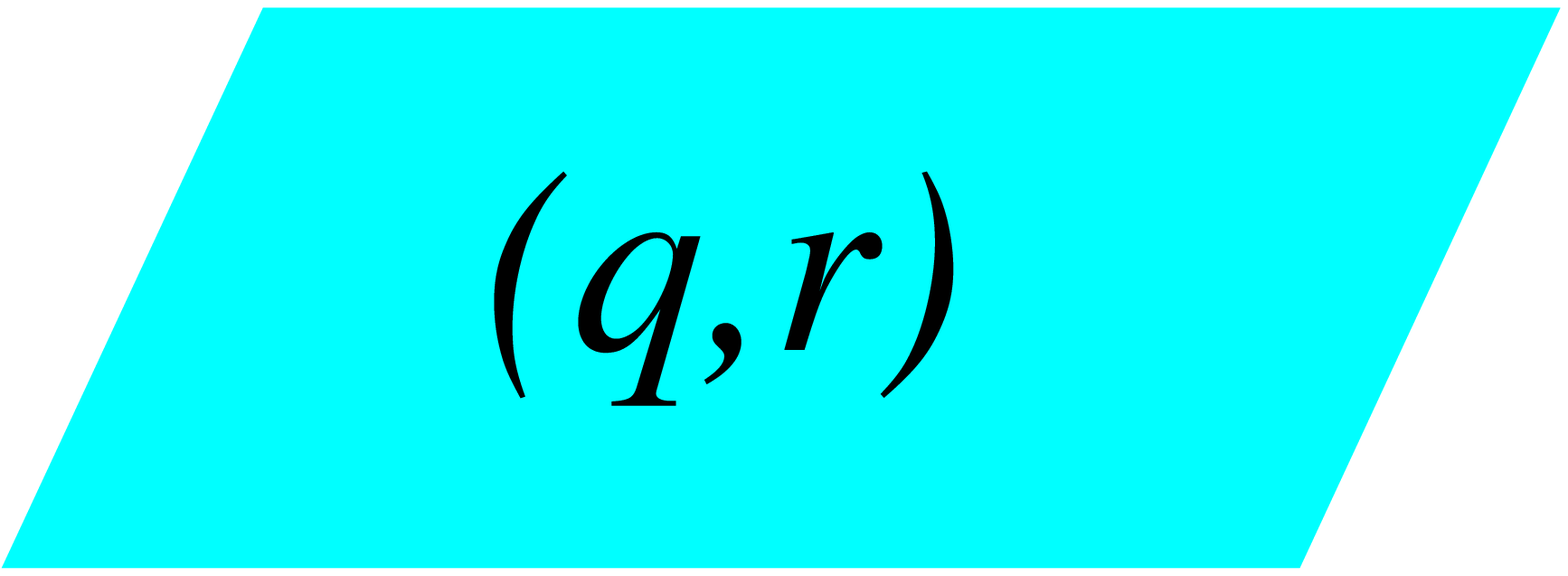}   &\text{if}\quad p=N-3-i\\ 
-\figins{-17}{0.4}{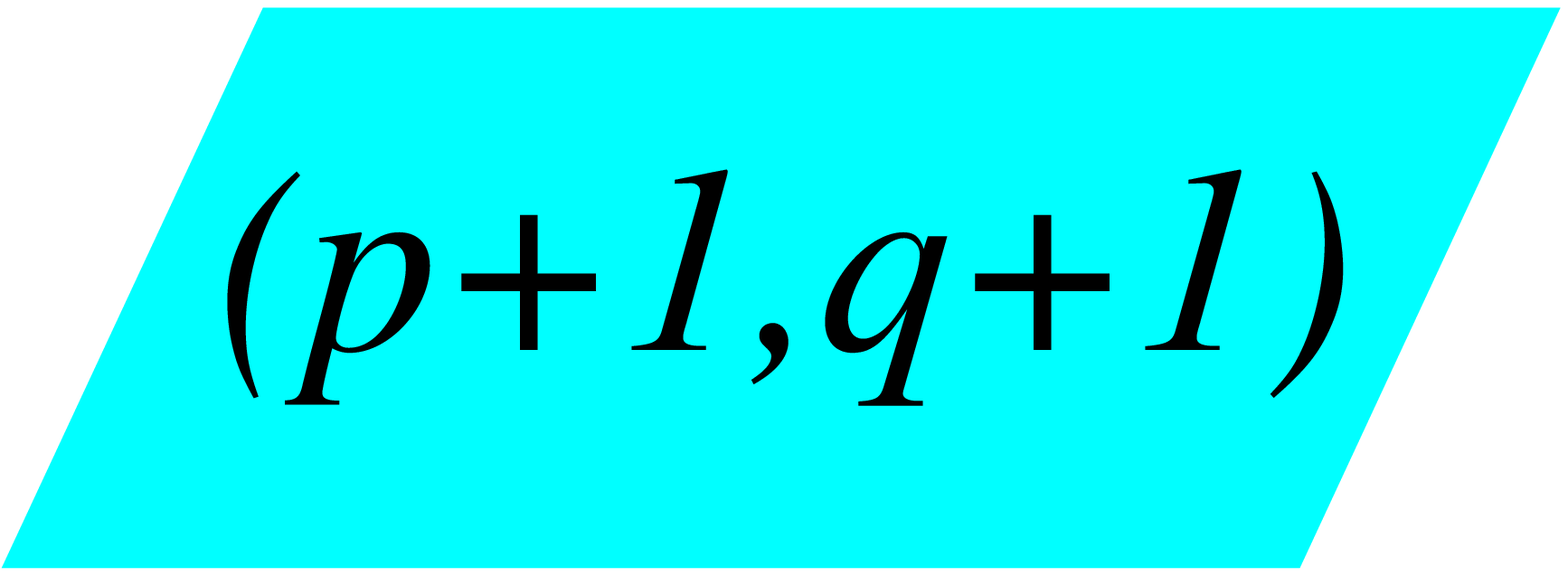} &\text{if}\quad r=N-1-i\\ 
 \figins{-17}{0.4}{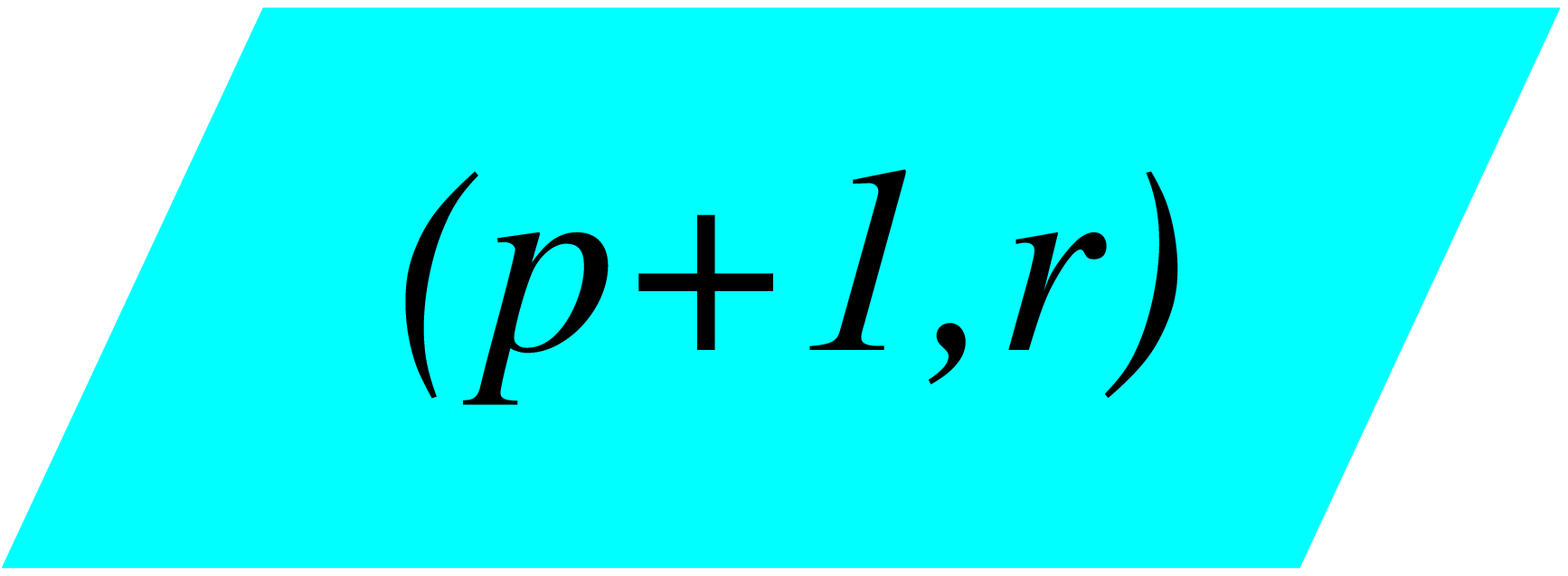}  &\text{if}\quad q=N-2-i\\
 \qquad 0                      &\text{else}
\end{cases}
\end{equation}

\bigskip

\begin{equation}
\label{slN:eq:bubble-ikm-star}%changed j->m
\figins{-18}{0.6}{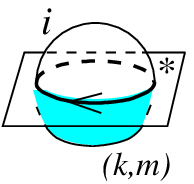}=
\begin{cases}
-\figins{-15}{0.4}{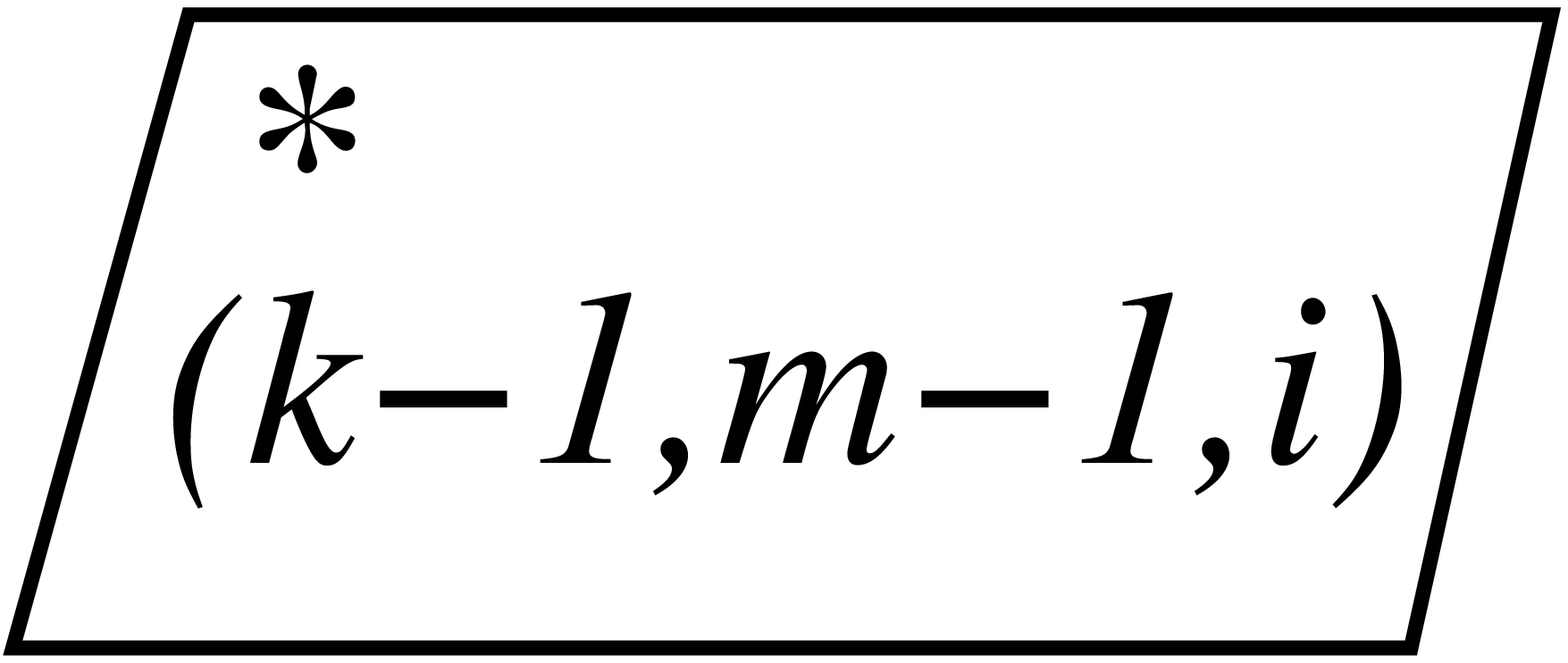} & \text{if}\quad N-2\geq k\geq m\geq i+1\geq 1\  \\
-\figins{-15}{0.4}{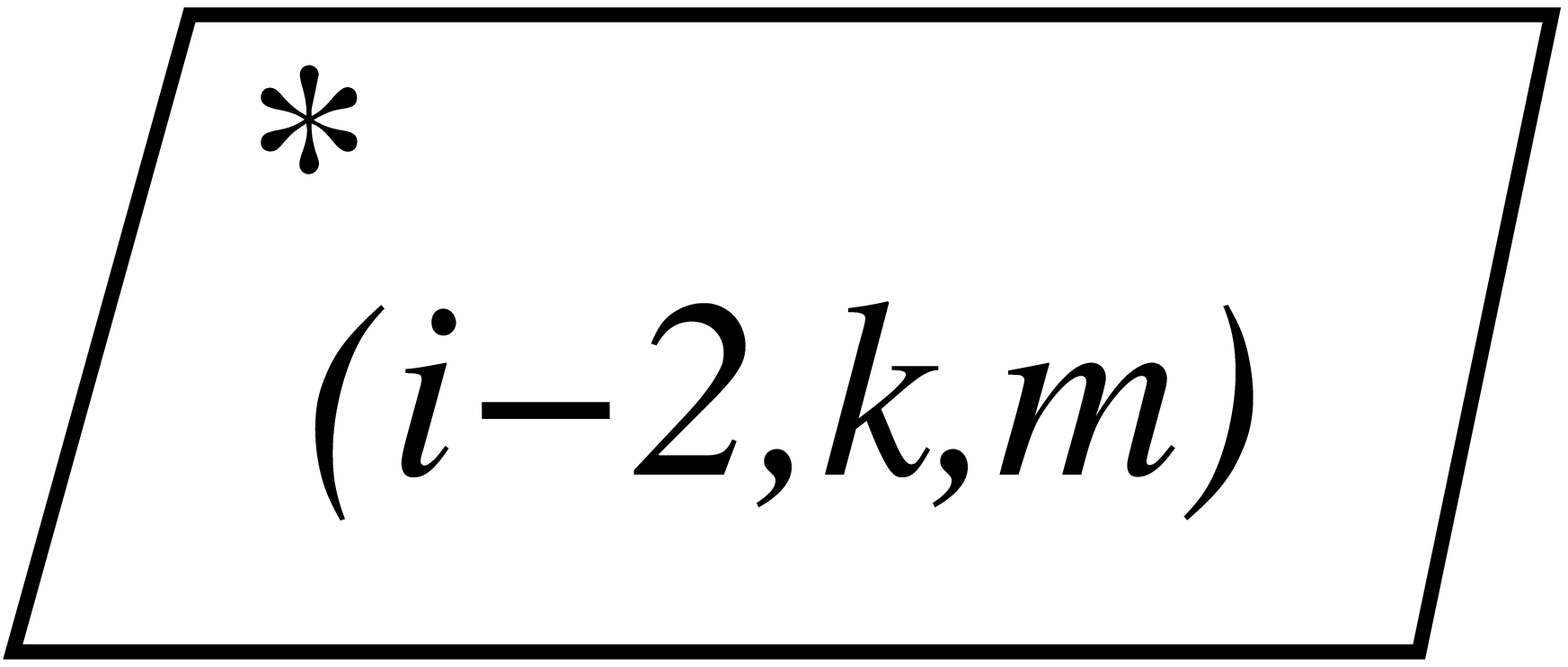} & \text{if}\quad N-1\geq i\geq k+2\geq m+2\geq 2\ \\ 
\ \ \ 
\figins{-15}{0.4}{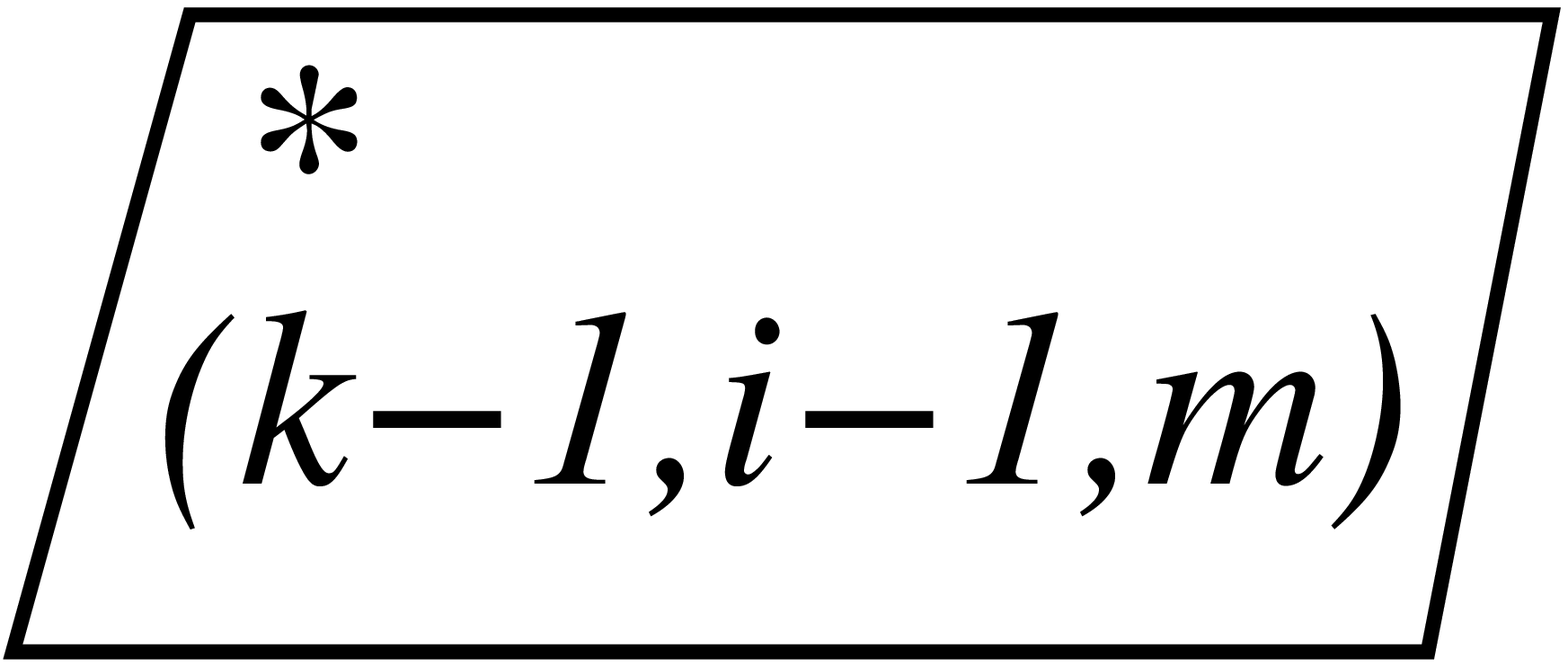} & \text{if}\quad N-2\geq k\geq i\geq m+1\geq 1\ \\ 
\qquad 0 &\text{else}
\end{cases}
\end{equation}

\begin{equation}
\label{slN:eq:bubble-pqrkm-star}
\figins{-18}{0.64}{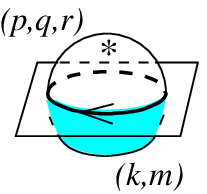}=
\begin{cases}\ \ \
- \figins{-13}{0.35}{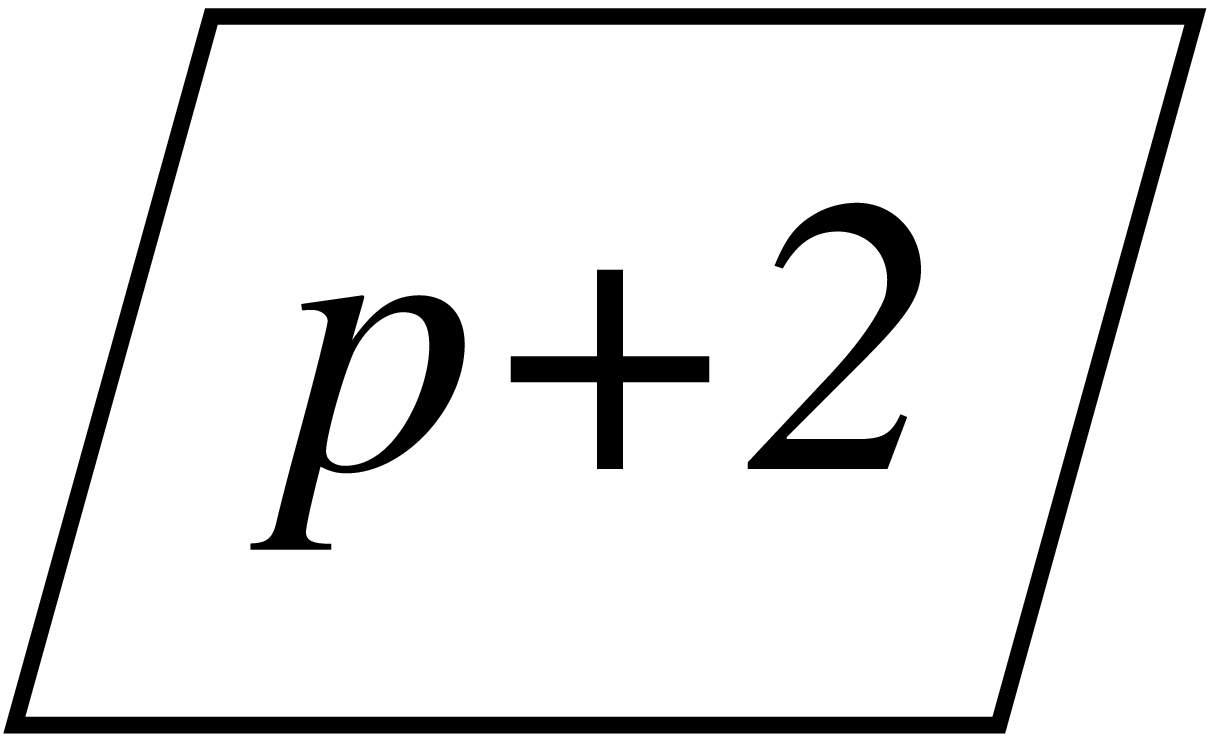} & \text{if}\quad q= N-2-m,\,\, r=N-2-k\  \\ \ \ \
- \figins{-13}{0.35}{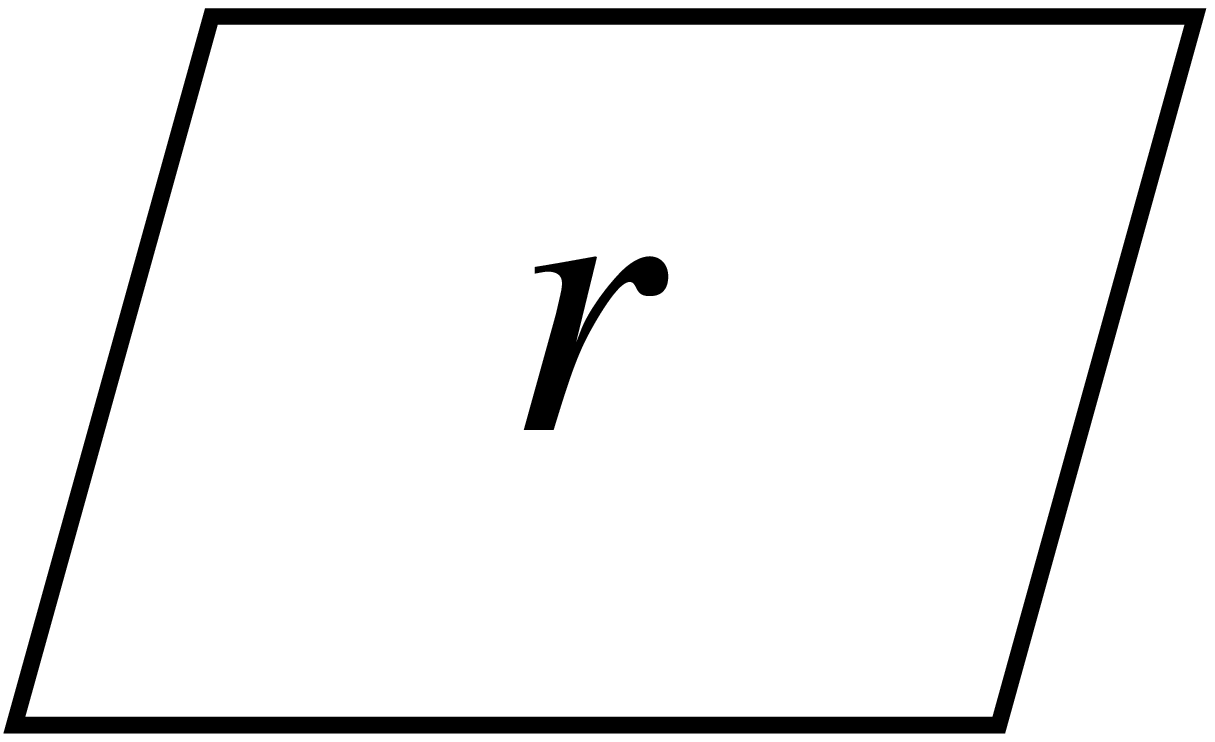} & \text{if}\quad p=N-3-m,\,\, q=N-3-k\ \\ 
\figins{-13}{0.35}{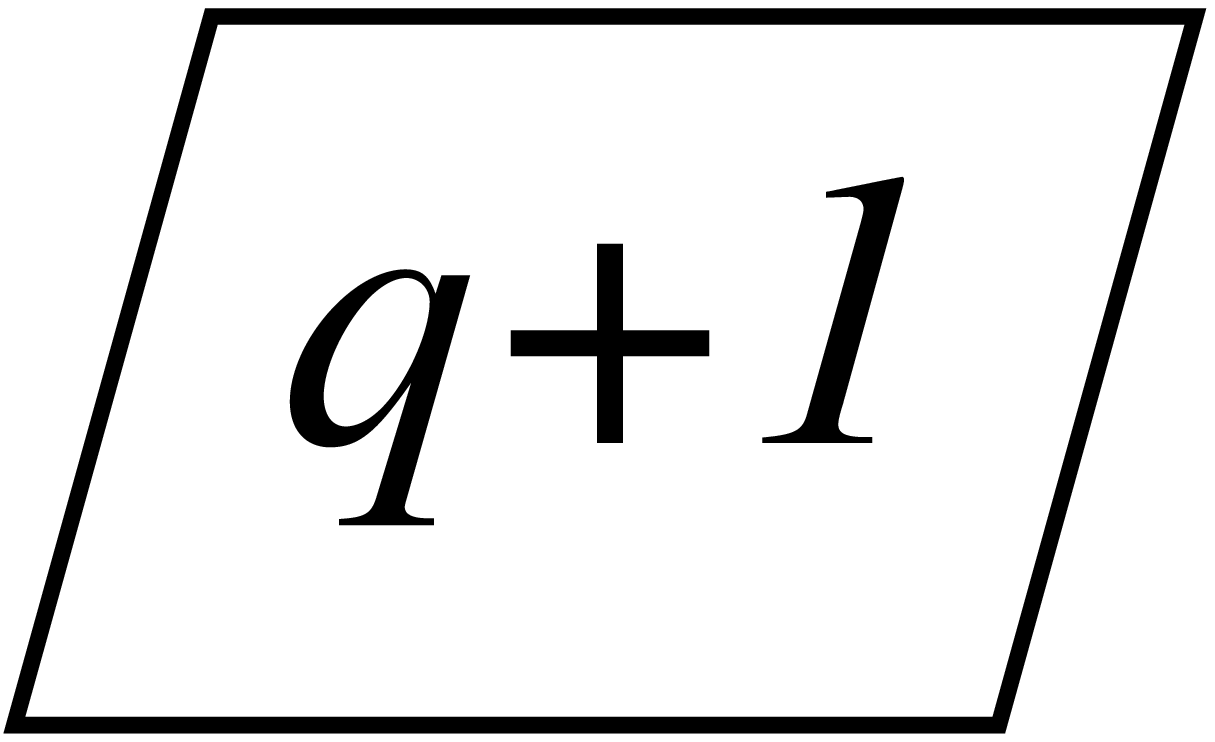} & \text{if}\quad p=N-3-m,\,\, r=N-2-k \\ 
\qquad 0 &\text{else}
\end{cases}
\end{equation}
\end{cor}

\begin{lem}
\label{slN:lem:theta-ijkl}
$$\figins{-17}{0.55}{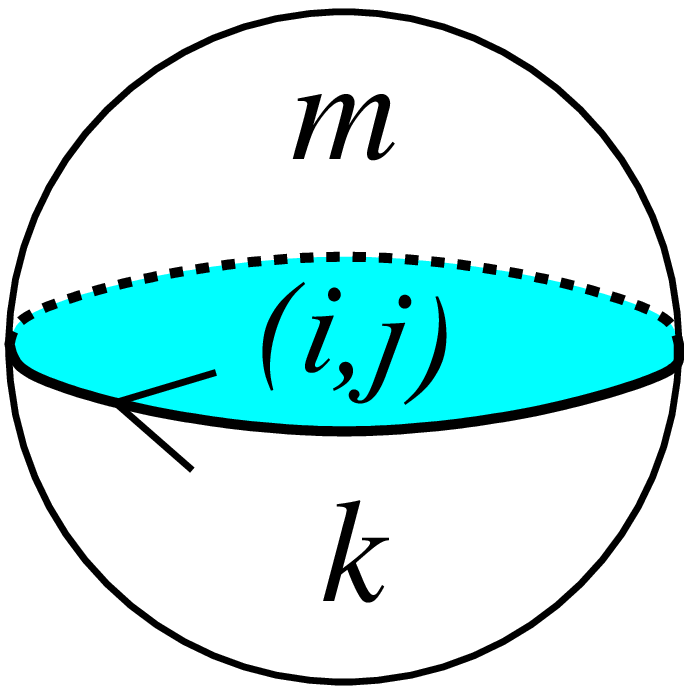}=
\begin{cases}
-1&\text{if}\quad m+j=N-1=i+k+1\\
+1&\text{if}\quad j+k=N-1=i+m+1, 
\end{cases}$$
\end{lem}
\begin{proof} By the dot conversion formulas, 
we get 
$$
\figins{-17}{0.55}{theta-mijkk}=\sum\limits_{\alpha =0}^{i-j}
\figins{-17}{0.55}{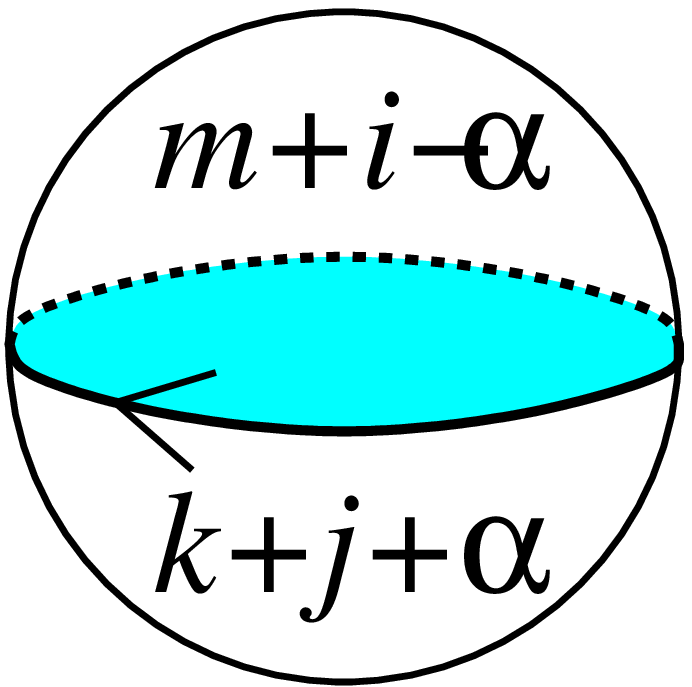}.$$
By $(\ThetaGraph)$ we have
$$\figins{-17}{0.55}{theta-miakja}=
\begin{cases}
-1&\text{if}\quad m+i-(N-1)=\alpha=N-2-(k+j)\\
+1&\text{if}\quad m+i-(N-2)=\alpha=N-1-(k+j)\\
0&\text{else}.
\end{cases}
$$
We see that, in the sum above, the summands for two consecutive values of $\alpha$ will cancel 
unless one of them is zero and the other is not. We see that the total sum is equal to $-1$ 
if the first non-zero summand is at $\alpha=i-j$ and $+1$ if the last non-zero summand 
is at $\alpha=0$.   
\end{proof}

The following bubble-identities are an immediate consequence of Lemma~\ref{slN:lem:theta-ijkl} and 
(CN$_1$) and (CN$_2$).

\begin{cor}
\label{slN:cor:bubble-12}
\begin{equation}
\label{slN:eq:bubble-ij}
\figins{-20}{0.6}{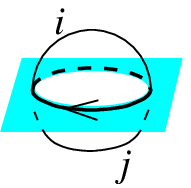}=
\begin{cases}
\quad -\figins{-13}{0.35}{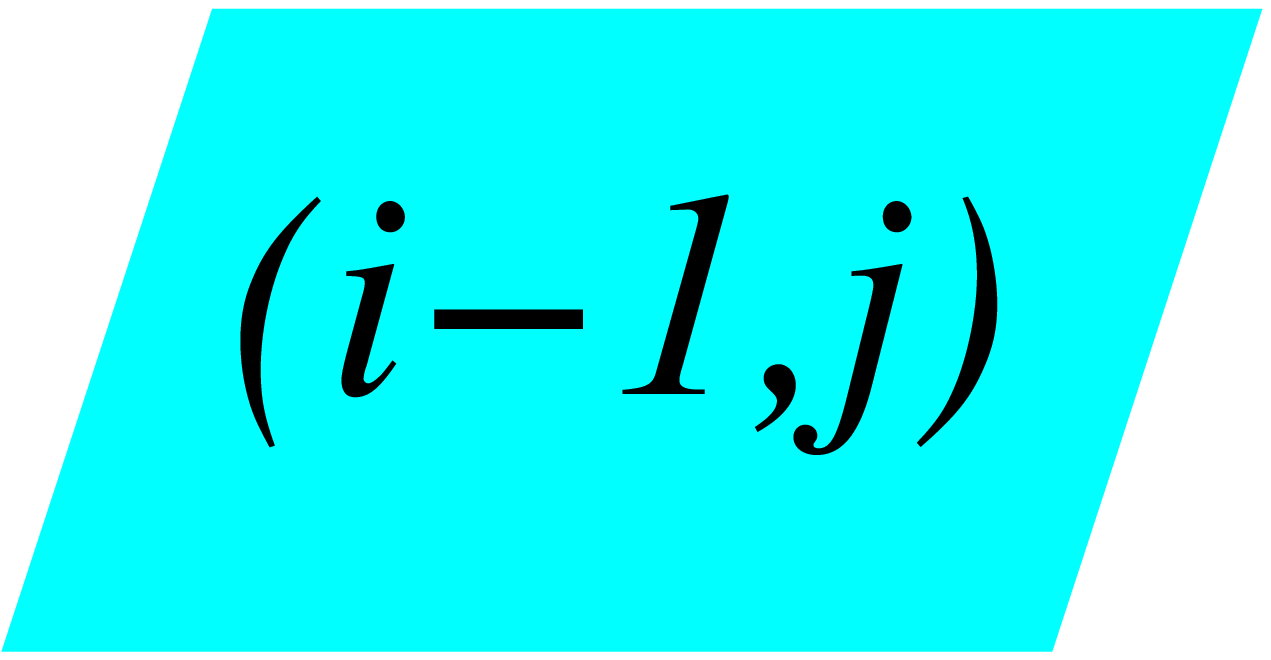} & \text{if}\quad i>j\geq 0 \\
\figins{-13}{0.35}{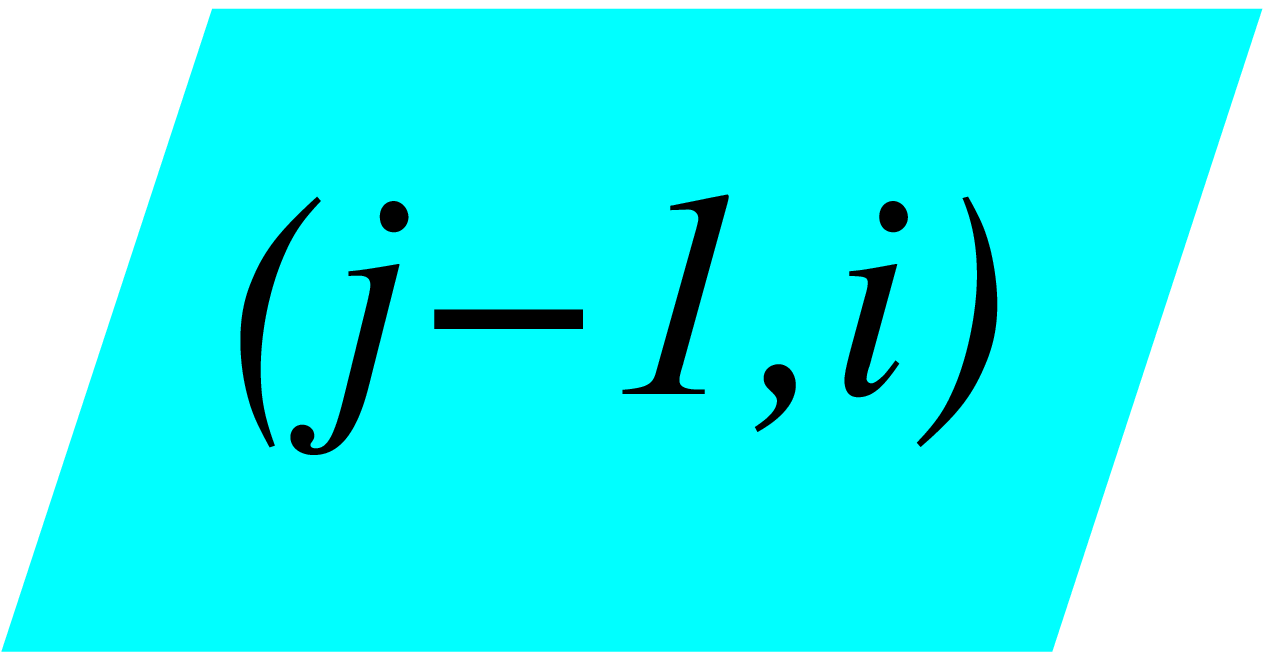} & \text{if}\quad j>i\geq 0 \\
\qquad 0 & \text{if}\quad i=j
\end{cases} 
\end{equation}

\begin{equation}
\label{slN:eq:bubble-ikj}
\figins{-18}{0.6}{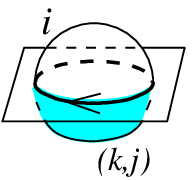}=
\begin{cases}
-   \figins{-8}{0.3}{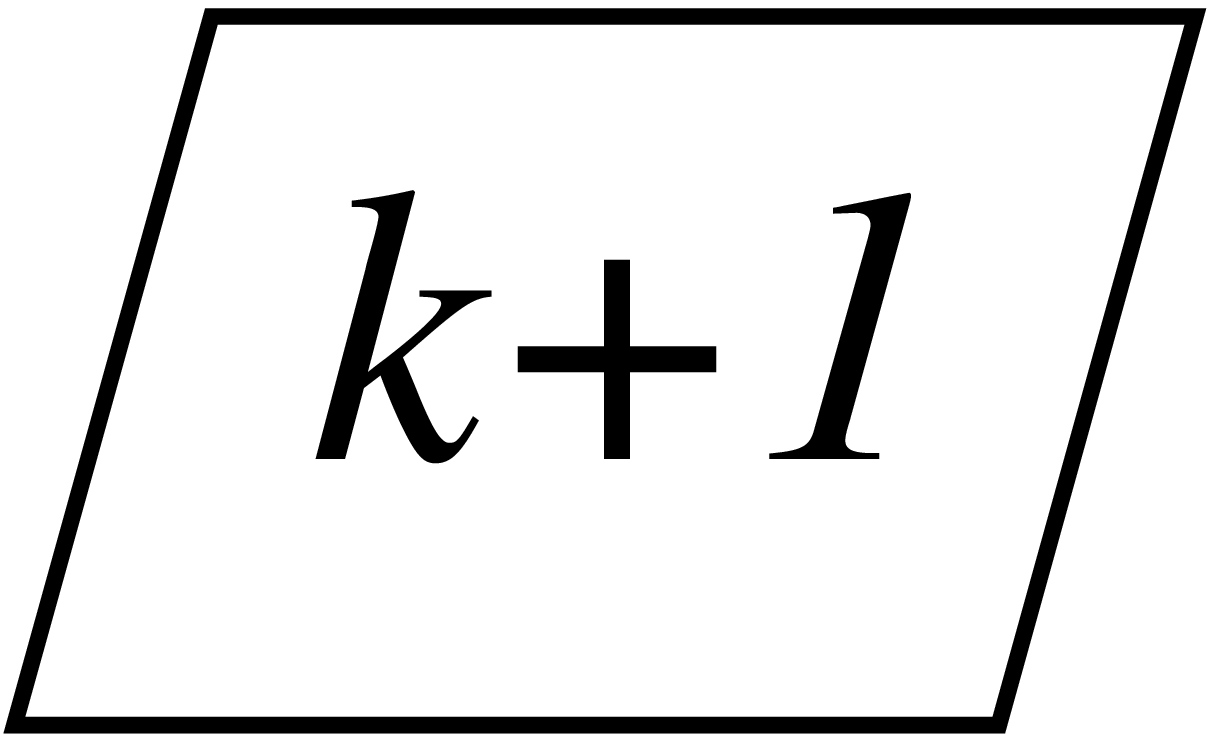} & \text{if }\  i+j=N-1 \\ \\
   \figins{-8}{0.3}{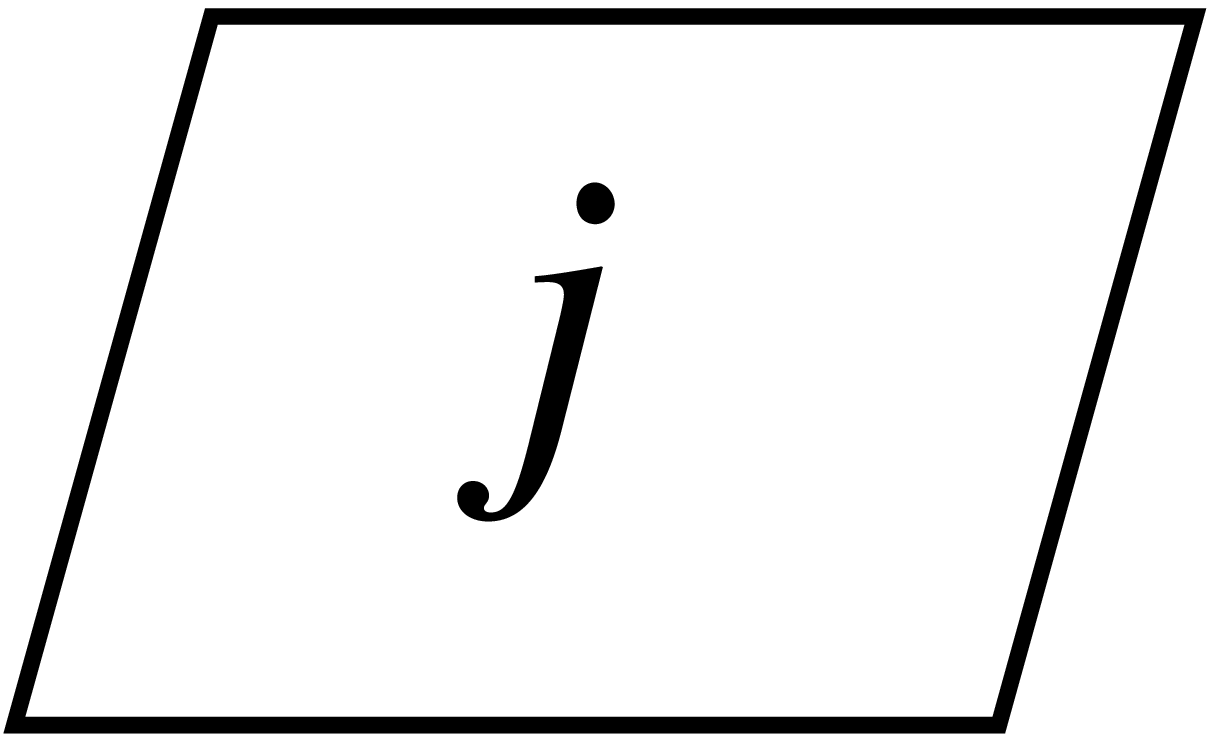} & \text{if }\ i+k=N-2   \\ \\
 \quad\     0 & \text{ else }
\end{cases}
\end{equation}
\end{cor}

The following identities follow easily from (CN$_1$), (CN$_2$),  
Lemma~\ref{slN:lem:theta-pqrkli}, Lemma~\ref{slN:lem:theta-ijkl} and their corollaries.

\begin{cor}
\label{slN:cor:tubes}
%%%%%%%%%
% 3cyls %
%%%%%%%%%

\begin{gather}
\figins{-22}{0.6}{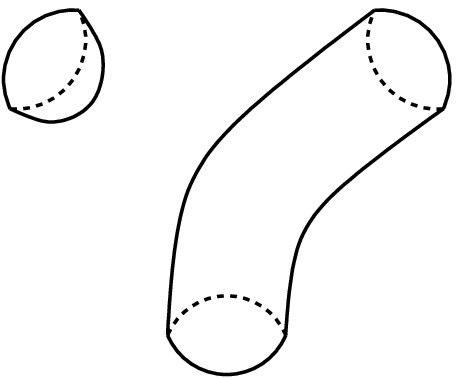} + \sum_{a+b+c=N-2} 
\figins{-22}{0.6}{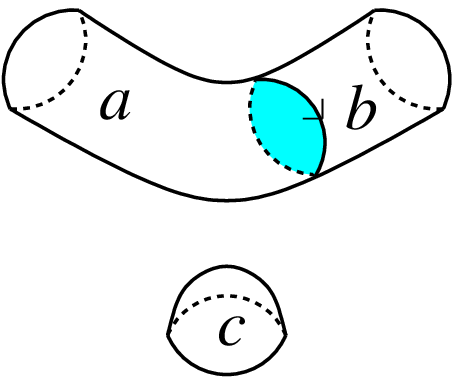} =
\figins{-22}{0.6}{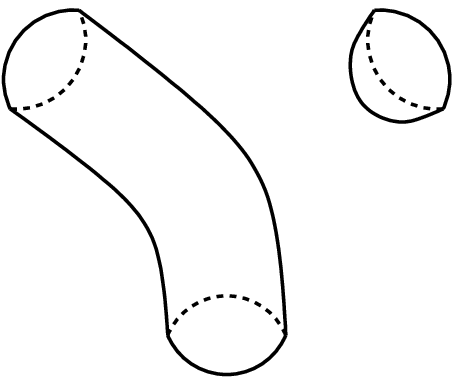}
\tag{3C}\label{slN:eq:3C}
\\[2ex]
%%%%%%
% RD %
%%%%%%
\figins{-8}{0.3}{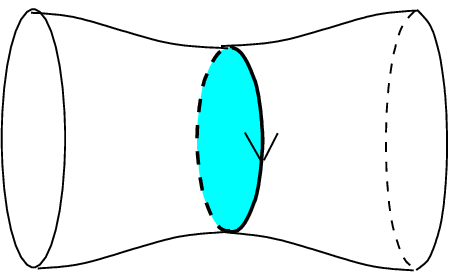} \ =  
\figins{-8}{0.3}{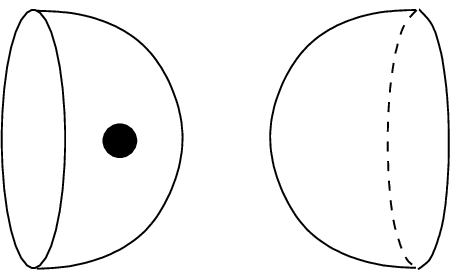} \ - \ 
\figins{-8}{0.3}{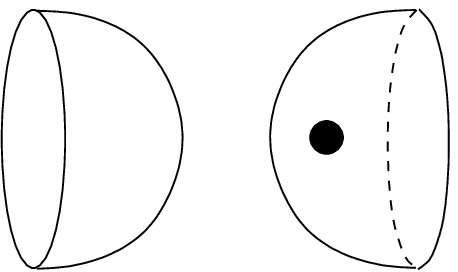}
\tag{RD$_1$}\label{slN:eq:RD1}
\\[2ex]
\figins{-8}{0.3}{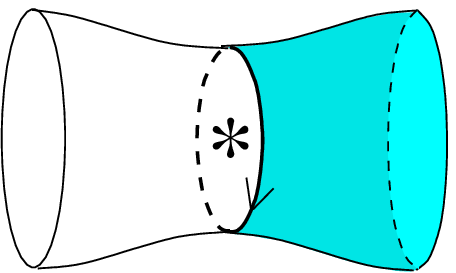} \ =  
\figins{-8}{0.3}{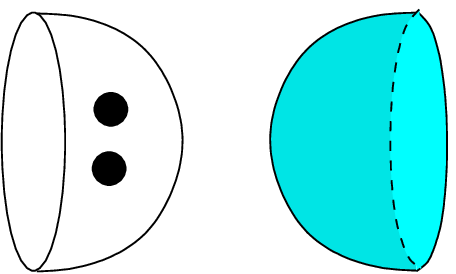} \ - \ 
\figins{-8}{0.3}{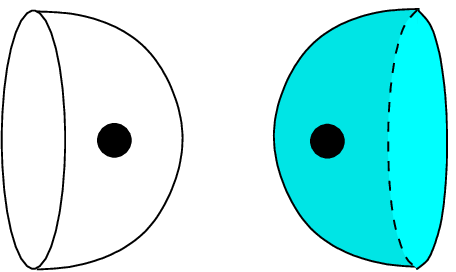} \ + \ 
\figins{-8}{0.3}{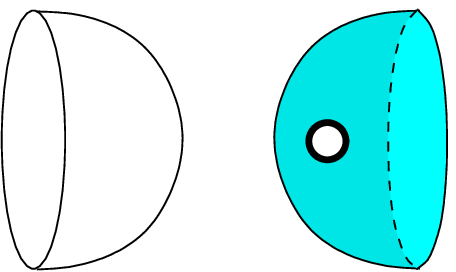}
\tag{RD$_2$}\label{slN:eq:RD2}
\\[2ex]
%%%%%%%%%%
% FatCyl %
%%%%%%%%%%
\figins{-8}{0.3}{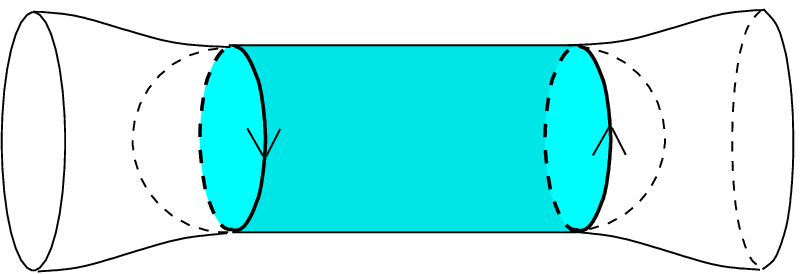} \ = - \ 
\figins{-8}{0.3}{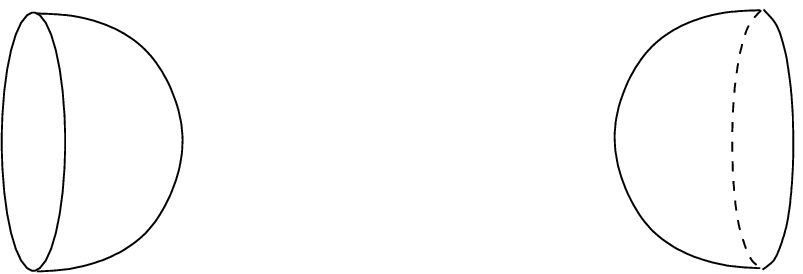}
\tag{FC}\label{slN:eq:FC}
\end{gather}
\end{cor}

\bigskip

Note that by the results above we are able to compute 
$\KL{u}$ combinatorially for any 
closed foam $u$ whose singular graph has no vertices, 
simply by using the cutting neck relations near all singular 
circles and evaluating the resulting spheres and theta 
foams. If the singular graph of $u$ has vertices, then we 
do not know if our relations are sufficient to evaluate $u$. 
We conjecture that they are sufficient, and that therefore 
our theory is strictly combinatorial, but we do not have 
a complete proof.  

\begin{prop}
\label{slN:prop:principal rels2}
The following identities hold in $\foam$: 

\medskip
%%%%%%%%%%%%%%%%%
% digon removal %
%%%%%%%%%%%%%%%%%
(The \emph{digon removal} relations)
\begin{align}
\figins{-20}{0.6}{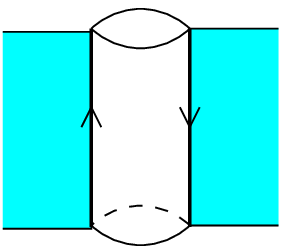}
&=
\figins{-20}{0.6}{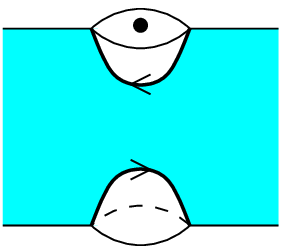}-
\figins{-20}{0.6}{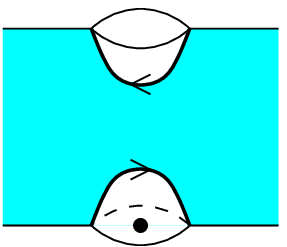}
\tag{DR$_1$}\label{slN:eq:DR1}
\\[1.5ex]
\figins{-27}{0.7}{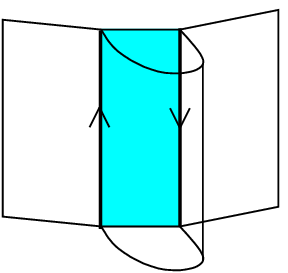} 
&= 
\sum_{a+b+c=N-2}\
\figins{-27}{0.7}{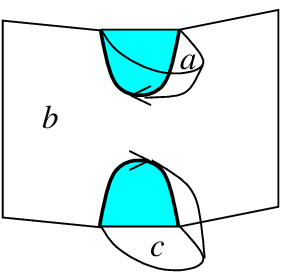} \ = \sum_{i=0}^{N-2}\
\figins{-27}{0.7}{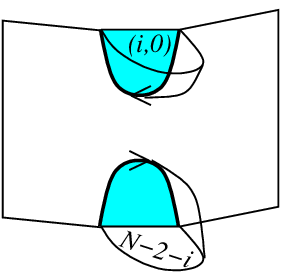}
\tag{DR$_2$}\label{slN:eq:DR2}
\\[1.5ex]
\figins{-30}{0.7}{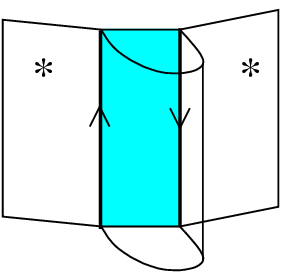} 
&=\
-\figins{-30}{0.7}{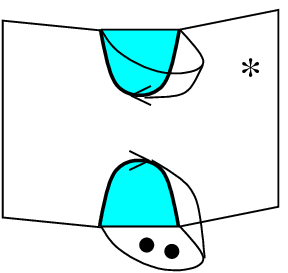}+
\figins{-30}{0.7}{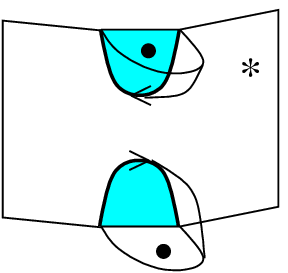}-
\figins{-30}{0.7}{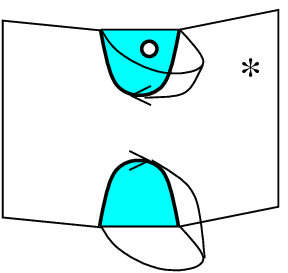}
\tag{DR$_{3_1}$}\label{slN:eq:DR31}
\\[1.5ex]
\figins{-30}{0.7}{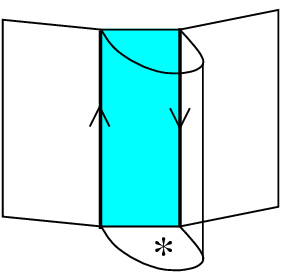} 
&= 
- \sum_{0\leq j\leq i\leq N-3}\
\figins{-30}{0.7}{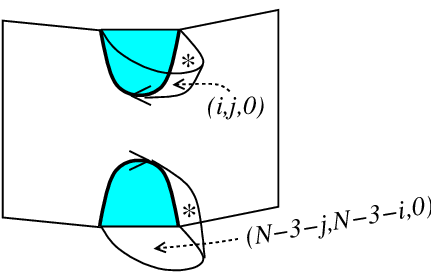}
\tag{DR$_{3_2}$}\label{slN:eq:DR32}
\\[1.5ex]
\figins{-20}{0.6}{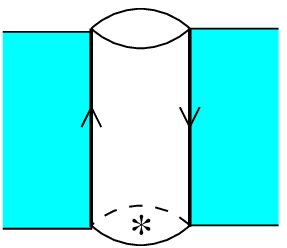}
&=
- \sum_{i=0}^{N-3} 
\figins{-20}{0.6}{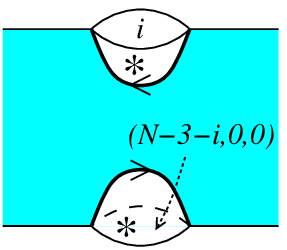}
\tag{DR$_{3_3}$}\label{slN:eq:DR33}
\end{align}

\bigskip

%%%%%%%%%%%%%%%%%%
% square removal %
%%%%%%%%%%%%%%%%%%
(The \emph{first square removal} relation)
\begin{equation}
\figins{-31}{0.9}{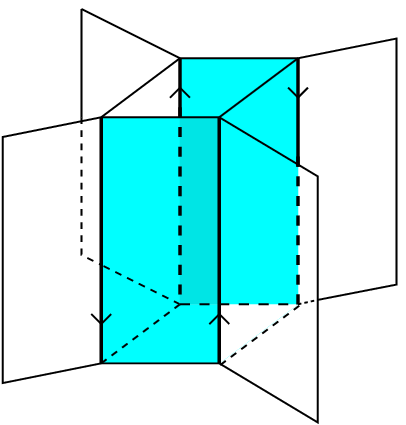}= 
-\;\figins{-31}{0.9}{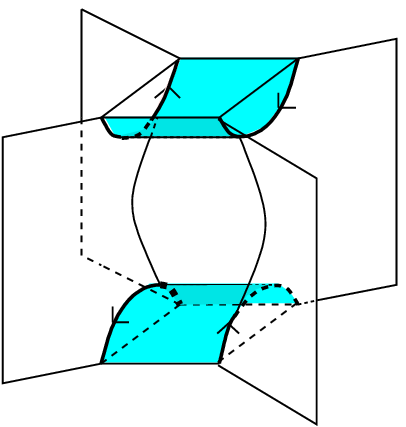}\quad + \sum_{a+b+c+d=N-3}\
\figins{-31}{0.9}{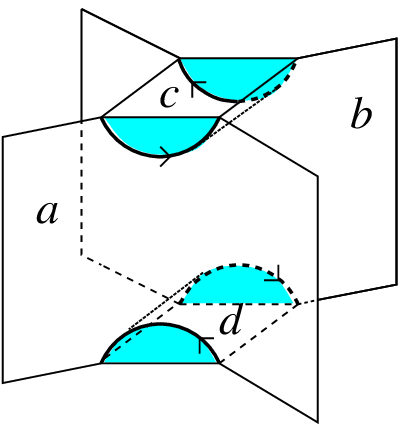}
\tag{SqR$_1$}\label{slN:eq:SqR1}
\end{equation}

\end{prop}

\begin{proof}
We first explain the idea of the proof. Naively one could try to consider 
all closures of the foams in a relation and compare their KL-evaluations. However, in practice 
we are not able to compute the KL-evaluations of arbitrary closed foams 
with singular vertices. Therefore we use a different strategy. We consider any foam in the 
proposition as a foam from $\emptyset$ to its boundary, rather than as a foam between one part 
of its boundary to another part. If $u$ is such a foam whose 
boundary is a closed web $\Gamma$, then by the properties explained in Section~\ref{slN:sec:KL} 
the KL-formula associates to $u$ an element in $\hy_{\mf{}}(\Gamma)$. By Definition~\ref{slN:defn:foam} and by the 
glueing properties of the KL-formula, as explained in Section~\ref{slN:sec:KL}, the induced 
linear map 
$\KLket{\ }\colon \Hom_{\foam}(\emptyset,\Gamma)\to \hy(\Gamma)$ is injective. 
The KL-formula also defines an inner product on $\Hom_{\foam}(\emptyset,\Gamma)$ by 
$$(u,v)\mapsto\KL{u\hat{v}}.$$
By $\hat{v}$ we mean the foam in $\Hom_{\foam}(\Gamma,\emptyset)$ obtained by rotating 
$v$ along the axis which corresponds to the $y$-axis (i.e. the horizontal axis parallel to the plane of the paper) in the original picture in this 
proposition. By the 
results in~\cite{KR} we know 
the dimension of $\hy(\Gamma)$. Suppose it is equal to $m$ 
and that we can find two sets of elements $u_i$ and $u_i^*$ in $\Hom_{\foam}(\emptyset,\Gamma)$, 
$i=1,2,\ldots,m$, such that 
$$\KL{u_i\widehat{u_j^*}} = \delta_{i,j}.$$ 
Then $\{u_i\}$ and $\{u_i^*\}$ are mutually dual bases of 
$\Hom_{\foam}(\emptyset,\Gamma)$ and 
$\KLket{\ }$ is an isomorphism. Therefore, two elements 
$f,g\in\Hom_{\foam}(\emptyset,\Gamma)$ are equal if and only if 
$$\KL{f\widehat{u_i}} =
\KL{g\widehat{u_i}},$$ 
for all $i=1,2,\ldots,m$ (alternatively one can use the $u_i^*$ 
of course). In practice this only helps if the l.h.s. and the r.h.s. of 
these $m$ equations can be computed, e.g. if $f,g$ and the $u_i$ are all foams with singular graphs without vertices. 
Fortunately that is the case for all the relations in this proposition.     
 
Let us now prove~\eqref{slN:eq:DR1} in detail. Note that the boundary of any of the 
foams in~\eqref{slN:eq:DR1}, 
denoted $\Gamma$, is homeomorphic to the web 
$$\figins{-17}{0.5}{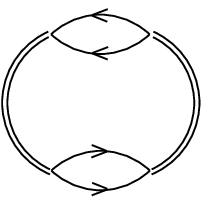}.$$ 
Recall that the dimension of 
$\hy(\Gamma)$ is equal to $2N(N-1)$ (see~\cite{KR}). For $0\leq i,j\leq 1$ 
and $0\leq m\leq k\leq N-2$, let $u_{i,j;(k,m)}$ 
denote the following foam 
$$\figins{-23}{0.8}{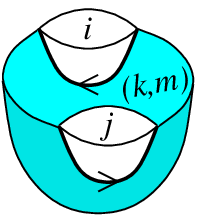}.$$
Let 
$$u_{i,j;(k,m)}^*=u_{1-j,1-i;(N-2-m,N-2-k)}.$$ 
From Equation~\eqref{slN:eq:bubble-ij} and the sphere relation (S$_2$) 
it easily follows that $\KL{ u_{i,j;(k,m)}\widehat{u_{r,s;(t,v)}^*}} =
\delta_{i,r}\delta_{j,s}\delta_{k,t}\delta_{m,v}$, where $\delta$ denotes the 
Kronecker delta. Note that there are exactly $2N(N-1)$ quadruples 
$(i,j;(k,l))$ 
which satisfy the conditions. Therefore the $u_{i,j;(k,m)}$ define a basis of 
$\hy(\Gamma)$ and 
the $u_{i,j;(k,m)}^*$ define its dual basis. In order to prove~\eqref{slN:eq:DR1} all we 
need to do next is check that 
$$\KL{(\text{l.h.s. of~\eqref{slN:eq:DR1}}) \widehat{u_{i,j;(k,m)}}} =
\KL{(\text{r.h.s. of~\eqref{slN:eq:DR1}}) \widehat{u_{i,j;(k,m)}}},$$ for all $i,j$ and $(k,m)$. This again follows easily from 
Equation~\eqref{slN:eq:bubble-ij} and the sphere relation (S$_2$).     

The other digon removal relations are proved in the same way. We do not repeat the  
whole argument again for each digon removal relation, 
but will only give the relevant mutually dual bases. 
For~\eqref{slN:eq:DR2}, note that $\Gamma$ is equal to 
$$\figins{-17}{0.5}{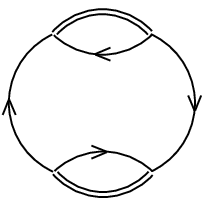}.$$ 
Let $u_{i,k,m}$ denote the foam
$$\figins{-23}{0.8}{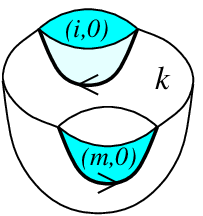}$$
for 
$0\leq i,m\leq N-2$ and $0\leq k\leq N-1$. The dual basis is defined by 
$$u_{i,k,m}^*=-u_{N-2-m,N-1-k,N-2-i},$$ 
for the same range of indices. Note that 
there are $N(N-1)^2$ possible indices, which corresponds exactly to the dimension of 
$\hy(\Gamma)$.

For~\eqref{slN:eq:DR31}, the web $\Gamma$ is equal to 
$$\figins{-17}{0.5}{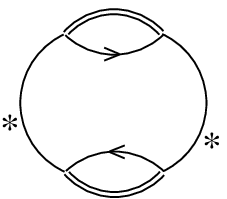}.$$
let $u_{i;(k,m);(p,q,r)}$ denote the foam 
$$\figins{-23}{0.8}{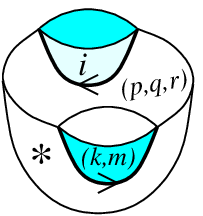}$$
for 
$0\leq i\leq 2$, $0\leq m\leq k\leq 1$ and $0\leq r\leq q\leq p\leq N-3$. The dual 
basis is given by 
$$u_{i;(k,m);(p,q,r)}^*=u_{2-(k+m);(1-\lfloor \frac{i}{2}\rfloor,
1-\lceil \frac{i}{2}\rceil);(N-3-r,N-3-q,N-3-p)},$$ 
for $0\leq t\leq s\leq 1$, $0\leq m\leq k\leq 1$ and $0\leq r\leq q\leq p\leq N-3$. 
Note that there are $3^2\binom{N}{3}$ possible indices, which corresponds exactly to 
the dimension of $\hy(\Gamma)$. 

For ~\eqref{slN:eq:DR32}, take $\Gamma$ to be 
$$\figins{-17}{0.5}{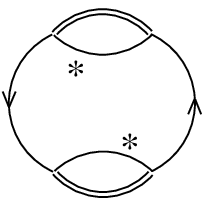}.$$
Let 
$u_{i;(k,m);(s,t)}$ denote the foam 
$$\figins{-23}{0.8}{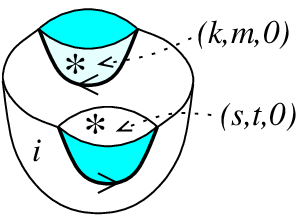}$$
for 
$0\leq i\leq N-1$, $0\leq m\leq k\leq N-3$ and $0\leq t\leq s\leq N-3$. The 
dual basis is given by  
$$u_{i;(k,m);(s,t)}^*=u_{N-1-i;(N-3-t,N-3-s);(N-3-m,N-3-k)},$$
for the same range of indices. Note that there are $N\binom{N-1}{2}^2$ indices, which 
corresponds exactly to the dimension of $\hy(\Gamma)$.

For~\eqref{slN:eq:DR33}, take $\Gamma$ to be 
$$\figins{-17}{0.5}{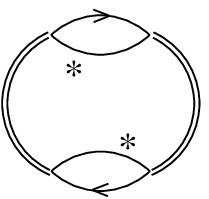}.$$ 
Let $u_{i,j;(k,m)}$ denote the foam 
$$\figins{-23}{0.8}{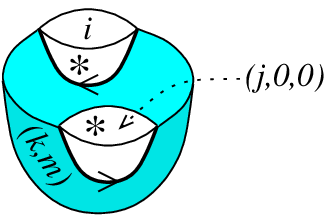}$$
for 
$0\leq i,j\leq N-3$ and $0\leq m\leq k\leq N-2$. Define 
$$u_{i,j;(k,m)}^*=u_{N-3-j,N-3-i;(N-2-m,N-2-k)},$$ 
for the same range of indices. Note that there are $(N-2)^2\binom{N}{2}$ indices, which 
corresponds exactly to the dimension of $\hy(\Gamma)$.

For~\eqref{slN:eq:SqR1}, the relevant web $\Gamma$ is equal to 
$$\figins{-20}{0.6}{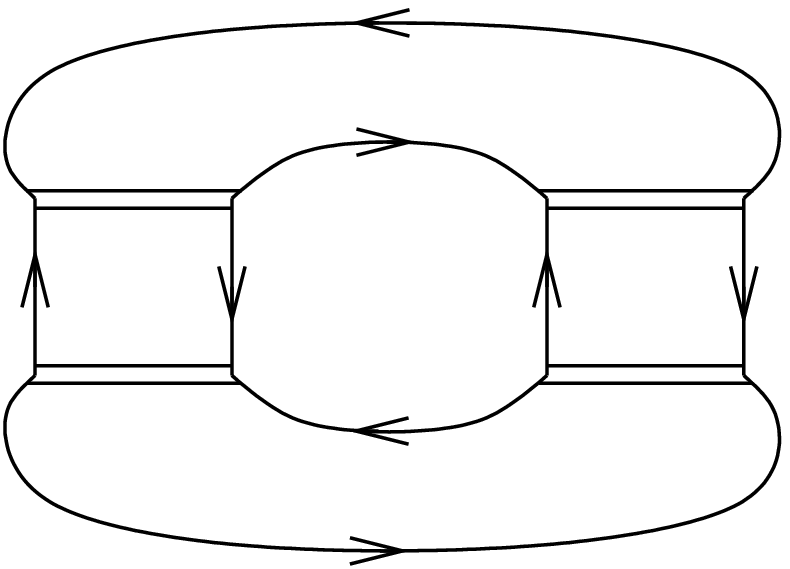}\ .$$ 
By the results in~\cite{KR} we know that the dimension of $\hy(\Gamma)$ is equal to $N^2+2N(N-2)+N^2(N-2)^2$. 
The proof of this relation is very similar, except that it is slightly harder to 
find the mutually dual bases in $\hy(\Gamma)$. The problem is that the 
two terms on the right-hand side of~\eqref{slN:eq:SqR1} are topologically distinct. Therefore 
we get four different types of basis elements, which are glueings of the upper or lower 
half of the first term and the upper or lower half of the second term. For 
$0\leq i,j\leq N-1$, let $u_{i,j}$ 
denote the foam 
$$\figins{0}{0.8}{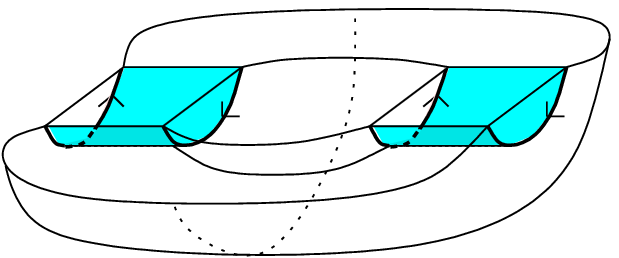}$$
with the top simple facet labelled by $i$ and the bottom one by $j$.
Take 
$$u_{i,j}^*=u_{N-1-j,N-1-i}.$$ 
Note that 
$$\KL{u_{i,j}\widehat{u_{k,m}^*}} =
\delta_{i,k}\delta_{j,m}$$ by the~\eqref{slN:eq:FC} relation in 
Corollary~\ref{slN:cor:tubes} and the Sphere 
Relation (S$_1$). 

For $0\leq i\leq N-1$ and $0\leq k\leq N-3$, let $v'_{i,k}$ denote the foam 
$$\figins{0}{0.8}{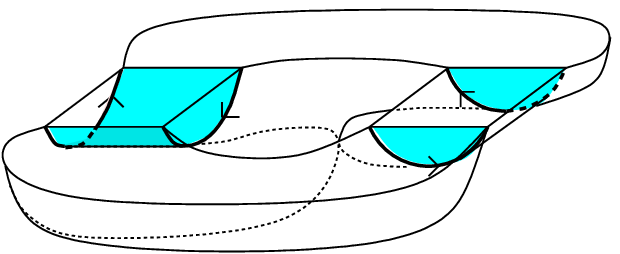}$$
with the simple square on the r.h.s. labelled by $i$ and the other simple facet by $k$. Note 
that the latter is only one facet indeed, because it has a saddle-point in the middle where the dotted lines meet.  
For the same range of indices, we define $w'_{i,k}$ by 
$$\figins{0}{0.8}{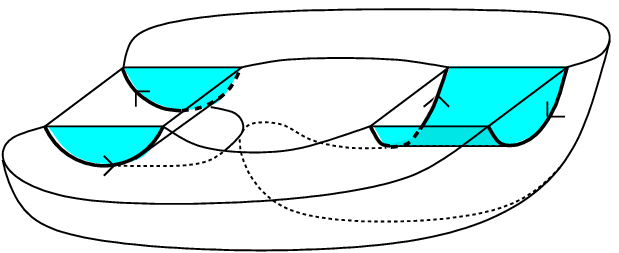}$$
with the simple square on the l.h.s. labelled by $i$ and the other simple facet by $k$.
The basis elements are now defined by 
$$v_{i,k}=\sum_{a+b+c=N-3-k}v'_{c,a+b+i}$$
and   
$$w_{i,j}=\sum_{a+b+c=N-3-j}w'_{c,a+b+i}.$$
The respective duals are defined by  
$$v_{i,k}^*=w'_{k,N-1-i}\quad\text{and}\quad w_{i,k}^*=v'_{k,N-1-i}.$$ 

We show that 
$$\KL{v_{i,j}\widehat{v_{k,m}^*}} = \delta_{i,k}\delta_{j,m}=
\KL{v_{i,j}\widehat{v_{k,m}^*}}$$
holds. First apply the ~\eqref{slN:eq:FC} relation of Corollary~\ref{slN:cor:tubes}. 
Then apply~\eqref{slN:eq:RD1} of the same corollary twice and finally use the sphere relation (S$_1$).  

For $0\leq i,j\leq N-1$ and $0\leq k,m\leq N-3$, let $s'_{i,j,k,m}$ denote the foam 
$$\figins{0}{0.8}{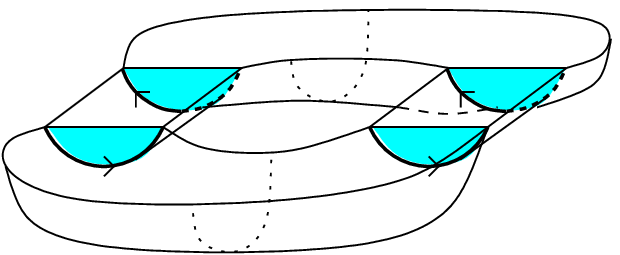}$$
with the simple squares labelled by $k$ and $m$, from left to right respectively, and 
the other two simple facets by $i$ and $j$, from front to back respectively. 
The basis elements are defined by 
$$s_{i,j,k,m}=\sum_{\substack{a+b+c=N-3-k \\ d+e+f=N-3-m}}
s'_{c,f,i+a+d,j+b+e}.$$ For the same range of indices, the dual elements of this shape 
are given by 
$$s_{i,j,k,m}^*=s'_{m,k,N-1-i,N-1-j}.$$

From~\eqref{slN:eq:RD1} of Corollary~\ref{slN:cor:tubes}, applied twice, and the sphere relation (S$_1$) 
it follows that  
$$\KL{s_{a,b,c,d}\widehat{s_{i,j,k,m}^*}} = \delta_{a,i}\delta_{b,j}
\delta_{c,k}\delta_{d,m}$$ holds. 

It is also easy to see that the inner product $\KL{\ }$ of a basis element and a dual 
basis element of distinct shapes, i.e. indicated by different letters above, gives zero. 
For example, consider 
$$\KL{u_{i,j} \widehat{v_{k,m}^*}}$$ for any valid choice 
of $i,j,k,m$. At the place where the two different shapes are glued, 
$$u_{i,j}\widehat{v_{k,m}^*}$$ 
contains a simple saddle with a simple-double bubble. 
By Equation~\eqref{slN:eq:bubble-ikj} that bubble kills 
$$\KL{u_{i,j} \widehat{v_{k,m}^*}},$$ 
because $m\leq N-3$. The same argument holds for 
the other cases. This shows that $\{u,v,w,s\}$ and $\{u^*,v^*,w^*,s^*\}$ 
form dual bases of $\hy(\Gamma)$, because the number of possible indices 
equals $N^2+2N(N-2)+N^2(N-2)^2$.

In order to prove~\eqref{slN:eq:SqR1} one now has to compute the inner product of the 
l.h.s. and the r.h.s. with any basis element of $\hy(\Gamma)$ and show that they are equal. We 
leave this to the reader, since the arguments one has to use are the same as we used above. 
\end{proof}

\begin{cor}{(The \emph{second square removal} relation)}
\begin{equation}
\figins{-34}{0.8}{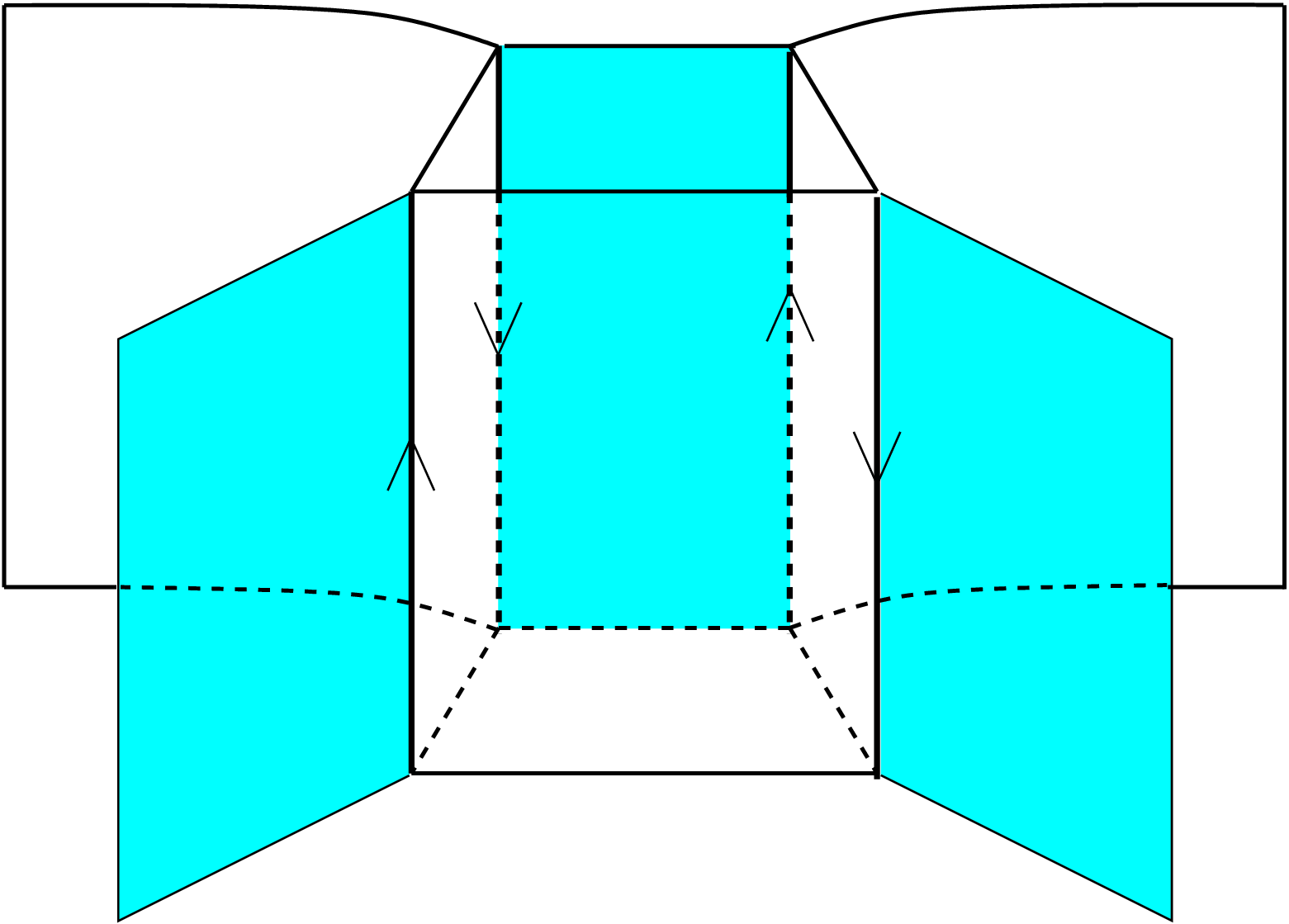}=
-\;\figins{-34}{0.8}{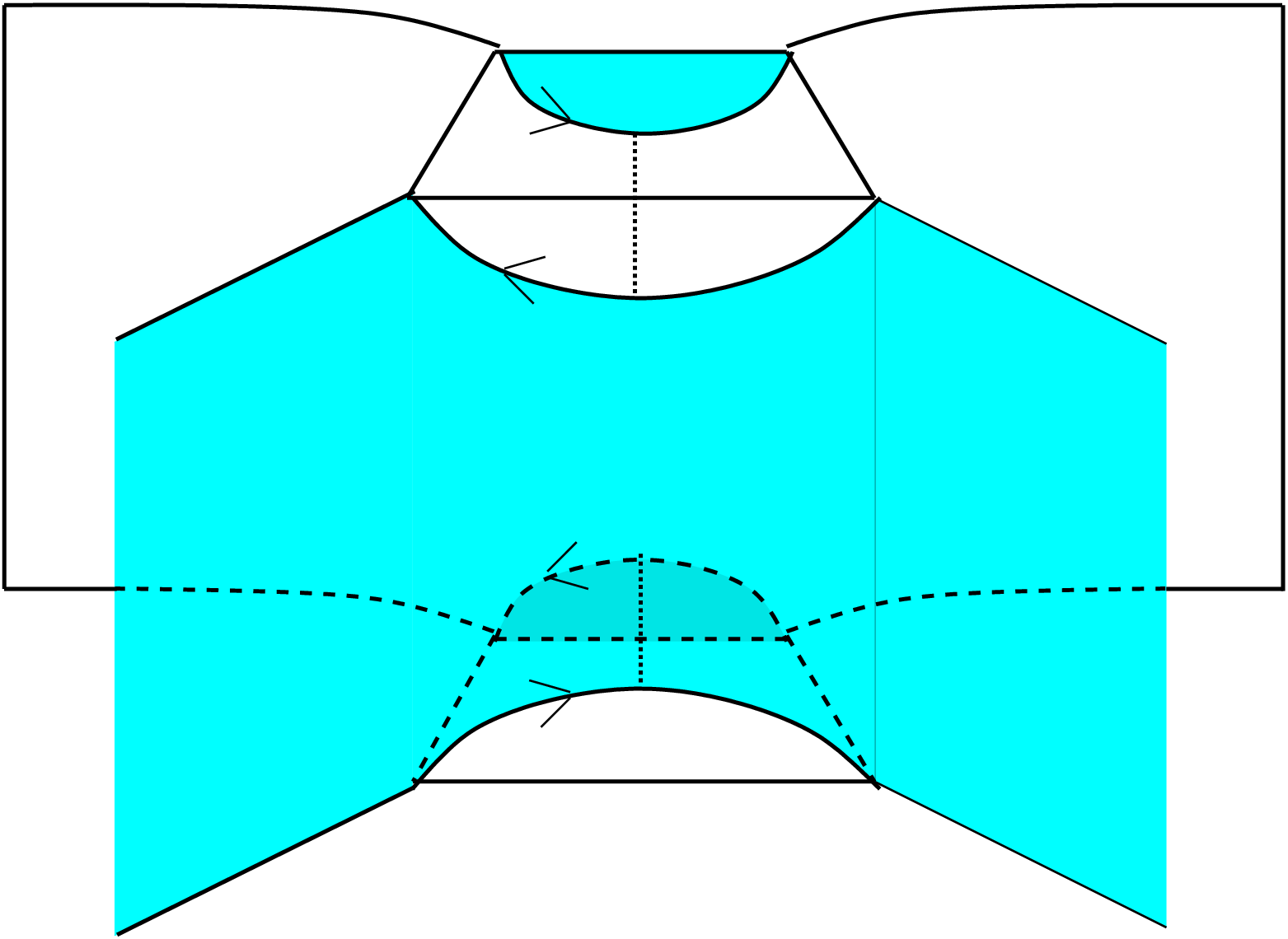}-
\figins{-34}{0.8}{sq_rem_b2}
\tag{SqR$_2$}\label{slN:eq:SqR2}
\end{equation}
\end{cor}

\begin{proof} 
Apply the Relation~\eqref{slN:eq:SqR1} to the simple-double square tube perpendicular to 
the triple facet of the second term on the r.h.s. of~\eqref{slN:eq:SqR2}. The simple-double square tube is identified with a circle in Figure~\ref{slN:fig:sqr2circ}. 
\begin{figure}[h!]
\centering
\figins{0}{0.8}{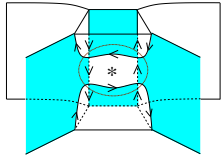}
\caption{A simple-double square tube around a triple facet} 
\label{slN:fig:sqr2circ}
\end{figure}
The first term 
on the r.h.s. of~\eqref{slN:eq:SqR1} yields minus the first term on the r.h.s. of~\eqref{slN:eq:SqR2} after applying 
the relations~\eqref{slN:eq:DR32},~\eqref{slN:eq:MP} and the Bubble Relation \eqref{slN:eq:bubble-pqri}.  The second term 
on the r.h.s. of~\eqref{slN:eq:SqR1}, i.e. the whole sum, yields the l.h.s. of~\eqref{slN:eq:SqR2} after applying 
the relations ~\eqref{slN:eq:DR33},~\eqref{slN:eq:MP} and the Bubble Relation~\eqref{slN:eq:bubble-pqri}. Note that 
the signs come out right because in both cases we get two bubbles with opposite orientations.  
\end{proof}

%%%%%%%%%%%%%%%%%%%%%%%%%%%%%%%%%%%%%%%%
%%%                                  %%%
%%%        Invariance                %%%
%%%                                  %%%
%%%%%%%%%%%%%%%%%%%%%%%%%%%%%%%%%%%%%%%%
\section{Invariance under the Reidemeister moves}
\label{slN:sec:invariance}

The proof of invariance under the Reidemeister moves follows in the same way as in Chapter~\ref{chap:univ3}. Let $\kom(\foam)$ and $\kom_{/h}(\foam)$ denote the category of complexes in $\foam$ 
and the same category modulo homotopies respectively.
As in Chapters~\ref{chap:KR} and~\ref{chap:univ3} we can take all different flattenings of $D$ to obtain 
an object in $\kom(\foam)$ which we call $\brak{D}$. The construction is indicated in Figure~\ref{slN:fig:geom-cplx}.
\begin{figure}[h]
\begin{align*}
\brak{\undercrossing} &=  0 \lra \underline{\brak{\orsmoothing}} 
     \stackrel{\figins{0}{0.2}{ssaddle}}{\longrightarrow} \brak{\DoubleEdge} \lra 0 \\
\brak{\overcrossing} &=  0 \lra \brak{\DoubleEdge}   
     \stackrel{\figins{0}{0.2}{ssaddle_ud}}{\longrightarrow} \underline{\brak{\orsmoothing}} \lra 0  
\end{align*}
\caption[Complex associated to a crossing]{Complex associated to a crossing. Underlined terms correspond to homological degree zero}
\label{slN:fig:geom-cplx}
\end{figure}

%%%%%%%%%%%%%%
% Invariance %
%%%%%%%%%%%%%%
%\subsection{Invariance}\label{slN:ssec:thm-inv}

\begin{thm}\label{slN:thm:inv}
The bracket $\brak{\;}$ is invariant in $\kom_{/h}(\foam)$ under the Reidemeister moves.
\end{thm}

\begin{proof}
%%%%%%%%%%%%%%%%%%
% Reidemeister I %
%%%%%%%%%%%%%%%%%%
\n\emph{Reidemeister I}: 
Consider diagrams $D$ and $D'$ that differ in a circular region, as in the 
figure below.
$$D=\figins{-13}{0.4}{lkink}\qquad
D'=\figins{-13}{0.4}{ReidI-1}\,$$
We give the homotopy between complexes $\brak{D}$ and $\brak{D'}$ in 
Figure~\ref{slN:fig:RI_Invariance}~\footnote{We thank Christian 
Blanchet for spotting a mistake in a previous version of this diagram.}.
\begin{figure}[h!]
$$\xymatrix@R=32mm{
  \brak{D}:
\\
  \brak{D'}:
}
\xymatrix@C=35mm@R=28mm{
 \figins{0}{0.3}{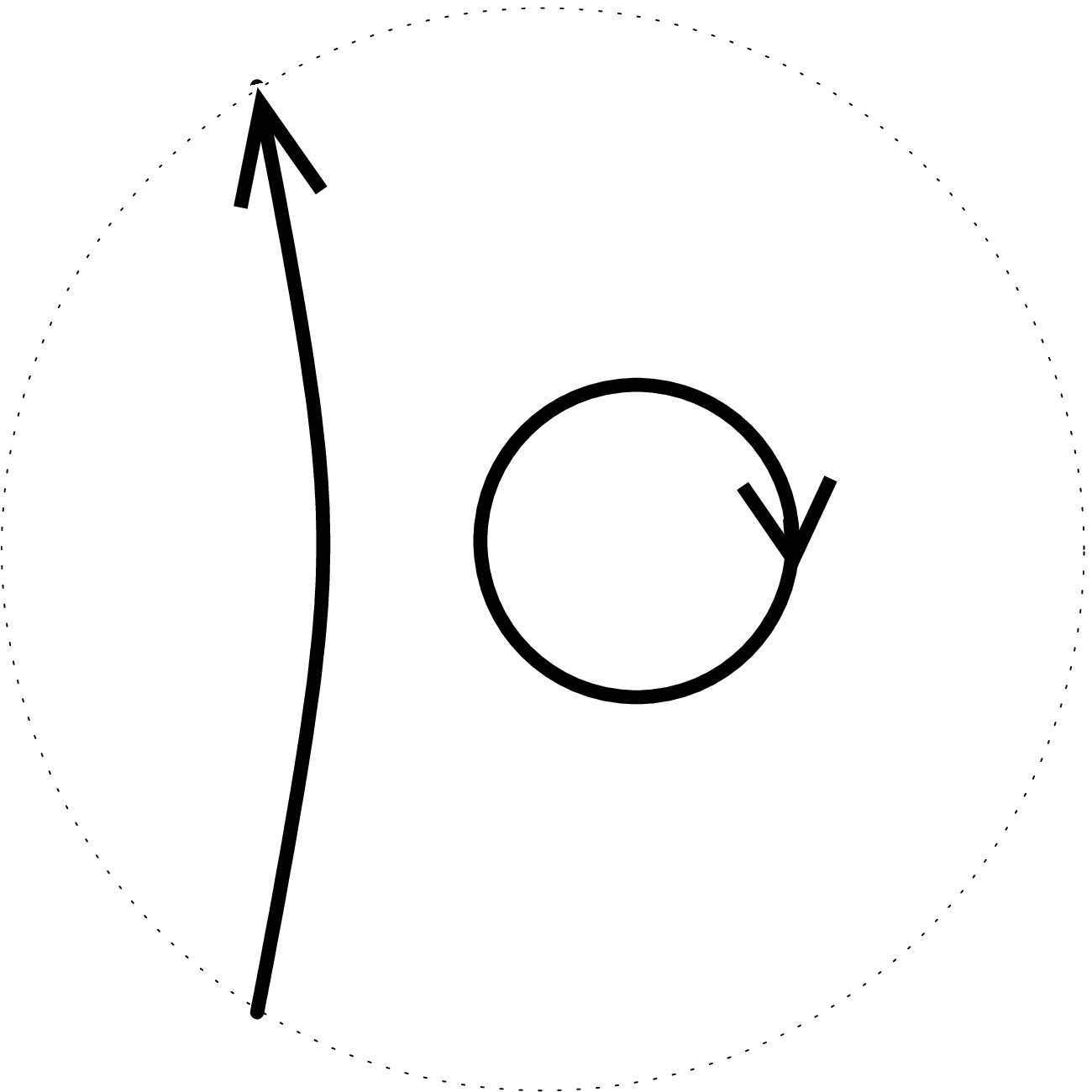}
  \ar@<4pt>[d]^{
        g^0 \;=\; \figins{-22}{0.6}{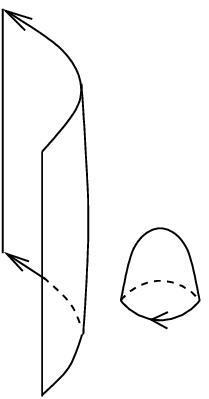} }
  \ar@<4pt>[r]^{
        d \;=\; \figins{-24}{0.6}{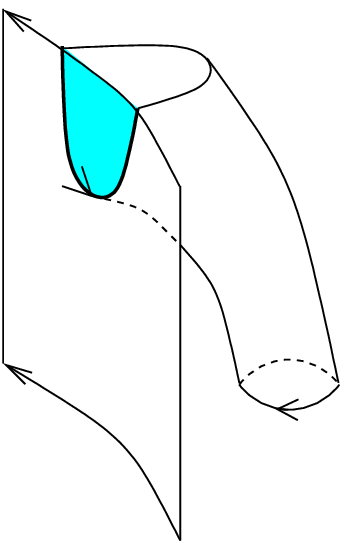} } &
  \figins{0}{0.3}{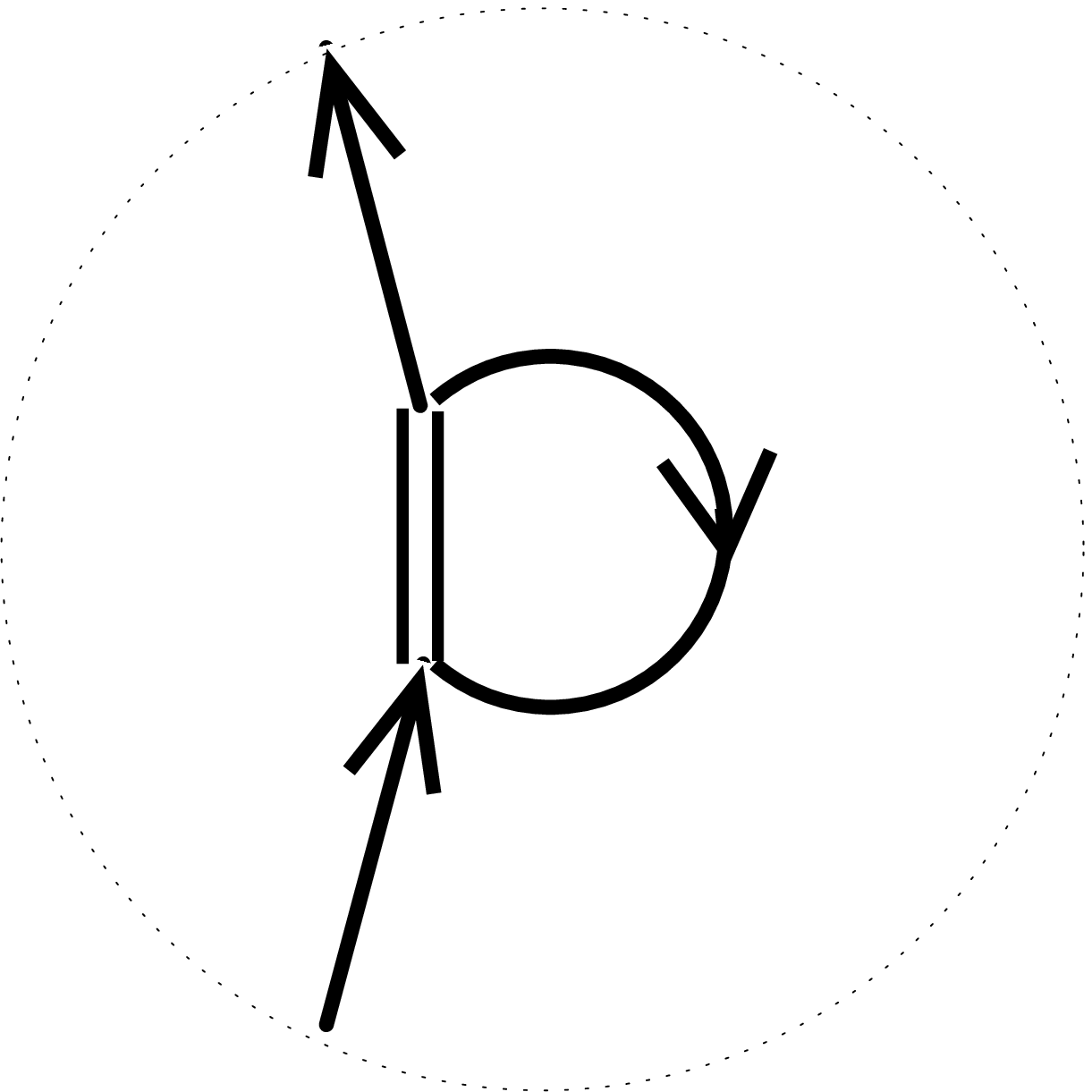} \ar@<2pt>[l]^{
     \quad h \;= \sum\limits_{a+b+c=N-2}\ \figins{-24}{0.6}{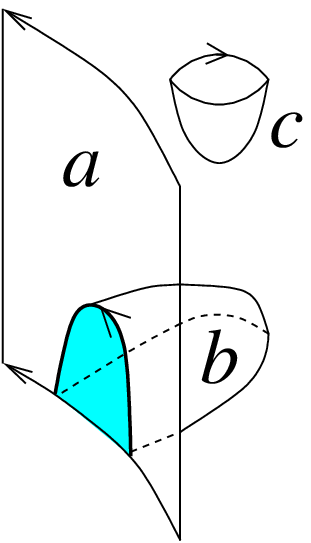} } \\
 \figins{0}{0.3}{ReidI-1}\ar[r]^0 
  \ar@<4pt>[u]^{
        f^0 \;= \sum\limits_{i=0}^{N-1}\ \figins{-24}{0.6}{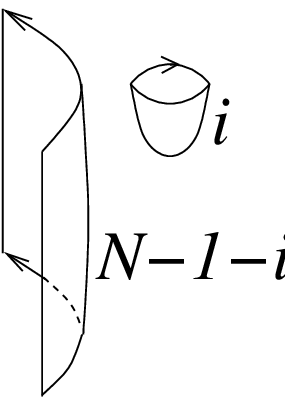} } &
  0 \ar@{<->}[u]_0 
}$$
\caption{Invariance under \emph{Reidemeister I}} 
\label{slN:fig:RI_Invariance}
\end{figure} 
\n  By the Sphere Relation (S$_1$), we get $g^0f^0=Id_{\brak{D'}^0}$.
To see that $df^0=0$ holds, one can use dot migration to get a new labelling of the same foam 
with the double facet labelled by $\pi_{N-1,0}$, which kills the foam by the dot conversion 
relations. The equality $dh=Id_{\brak{D}^1}$ follows from~\eqref{slN:eq:DR2}. To show that 
$f^0g^0+hd=Id_{\brak{D}^0}$, apply~\eqref{slN:eq:RD1} to $hd$ and then cancel all terms which appear 
twice with opposite signs. What is left is the sum of $N$ terms which is equal to 
$Id_{\brak{D}^0}$ by (CN$_1$). Therefore $\brak{D'}$  is homotopy-equivalent to $\brak{D}$.

\bigskip
%%%%%%%%%%%%%%%%%%%%
% Reidemeister IIa %
%%%%%%%%%%%%%%%%%%%%
\n\emph{Reidemeister IIa}:
Consider diagrams $D$ and $D'$ that differ in a circular region, as
in the figure below.
$$
D=\figins{-13}{0.4}{DReidIIa}\qquad
D'=\figins{-13}{0.4}{twoedges}$$
We only sketch the arguments that the diagram in 
Figure~\ref{slN:fig:RIIaInvariance} defines a homotopy 
equivalence between the complexes $\brak{D}$ and 
$\brak{D'}$: 
\begin{figure}[ht!]
\centering
\figwins{0}{5.0}{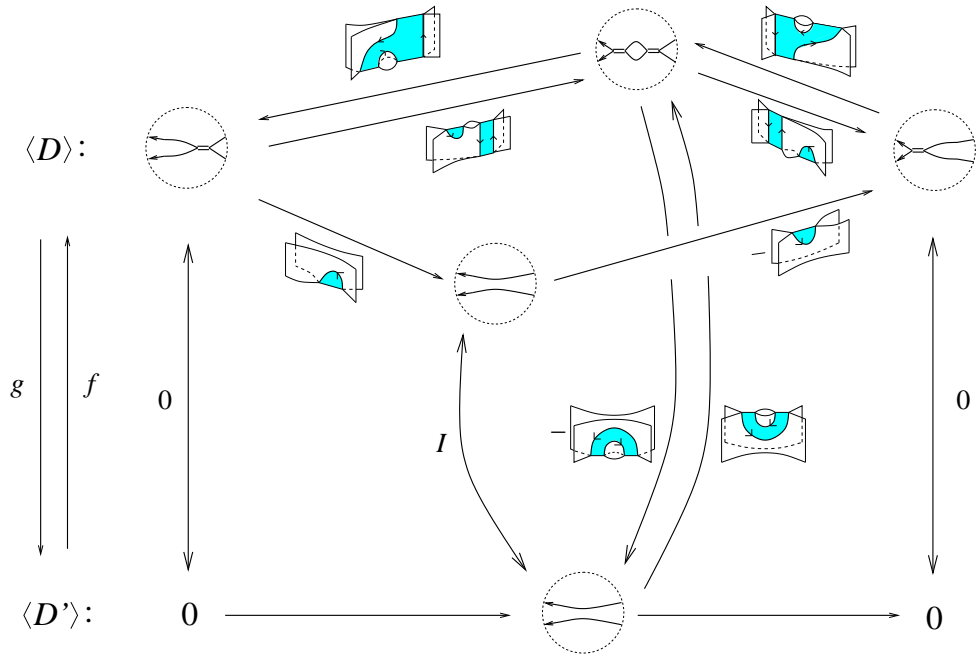}
\caption{Invariance under \emph{Reidemeister IIa}} 
\label{slN:fig:RIIaInvariance}
\end{figure}

\begin{itemize}
\item $g$ and $f$ are morphisms of complexes (use only isotopies); 
\item $g^1f^1=Id_{\brak{D'}^1}$  (uses equation \eqref{slN:eq:bubble-ij}); 
\item $f^0g^0+hd=Id_{\brak{D}^0}$ and $f^2g^2+dh=Id_{\brak{D}^2}$  (use isotopies);
\item $f^1g^1+dh+hd=Id_{\brak{D}^1}$  (use~\eqref{slN:eq:DR1}).
\end{itemize}

\bigskip
%%%%%%%%%%%%%%%%%%%%
% Reidemeister IIb %
%%%%%%%%%%%%%%%%%%%%
\n\emph{Reidemeister IIb}:
Consider diagrams $D$ and $D'$ that differ only in a circular region, as 
in the figure below.
$$
D=\figins{-13}{0.4}{DReidIIb}\qquad
D'=\figins{-13}{0.4}{twoedgesop}$$
Again, we sketch the arguments that the diagram in 
Figure~\ref{slN:fig:RIIbInvariance} defines a 
homotopy equivalence between the complexes $\brak{D}$ and $\brak{D'}$:
\begin{figure}[h!]
\centering
\figwins{0}{5.4}{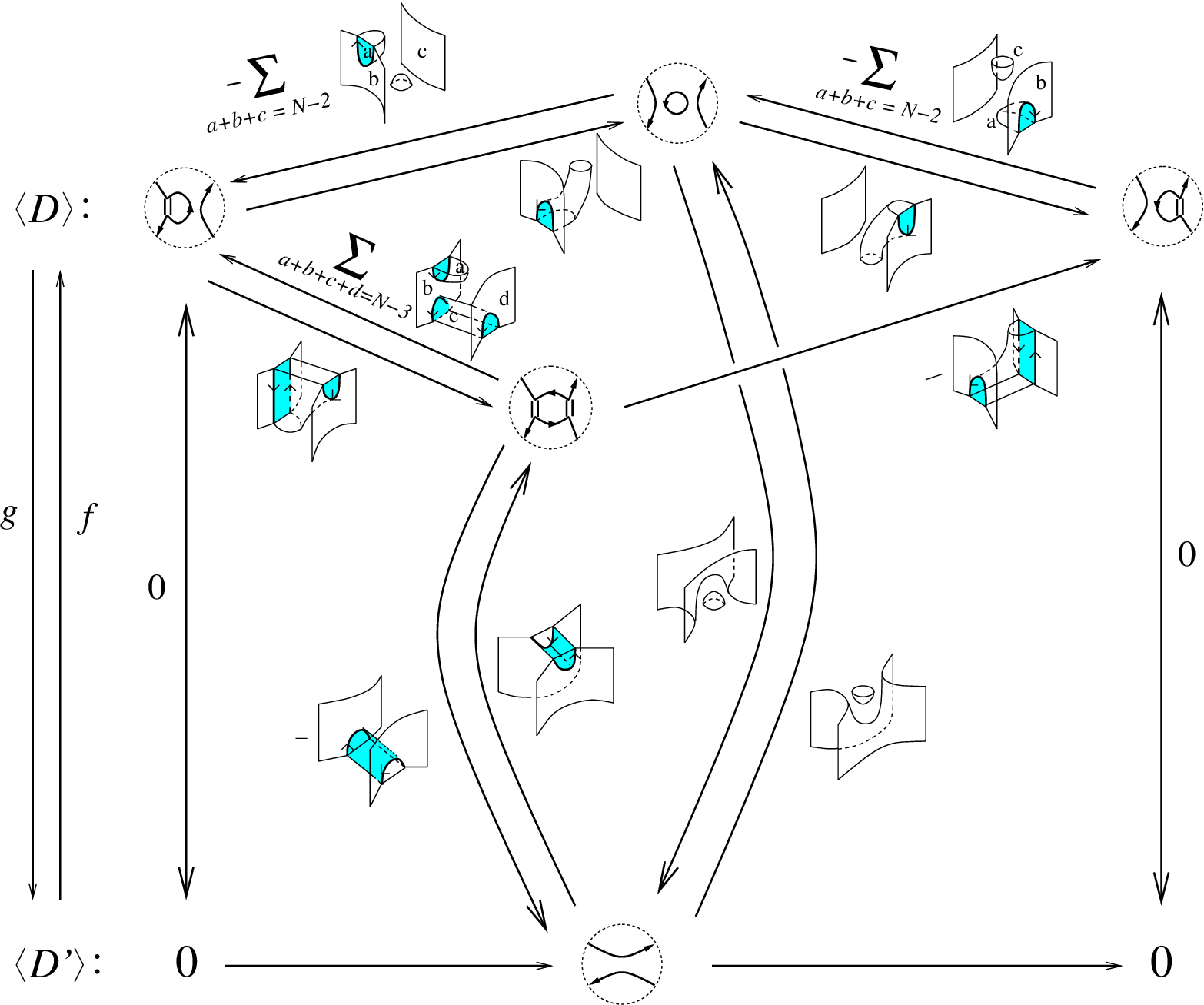}
\caption{Invariance under \emph{Reidemeister IIb}} 
\label{slN:fig:RIIbInvariance}
\end{figure}
\begin{itemize}
\item $g$ and $f$ are morphisms of complexes (use isotopies and~\eqref{slN:eq:DR2}); 
\item $g^1f^1=Id_{\brak{D'}^1}$  (use~\eqref{slN:eq:FC} and (S$_1$)); 
\item $f^0g^0+hd=Id_{\brak{D}^0}$ and $f^2g^2+dh=Id_{\brak{D}^2}$ (use~\eqref{slN:eq:RD1} and~\eqref{slN:eq:DR2});
\item $f^1g^1+dh+hd=Id_{\brak{D}^1}$ (use~\eqref{slN:eq:DR2},~\eqref{slN:eq:RD1},~\eqref{slN:eq:3C} and~\eqref{slN:eq:SqR1}.
\end{itemize}

\bigskip
%%%%%%%%%%%%%%%%%%%%
% Reidemeister III %
%%%%%%%%%%%%%%%%%%%%
\n\emph{Reidemeister III}:
Consider diagrams $D$ and $D'$ that differ only in a circular region, as
in the figure below.
$$
D =\figins{-13}{0.4}{D1ReidIII}\qquad
D'=\figins{-13}{0.4}{D2ReidIII}$$
\begin{figure}[ht!]
\centering
\figwins{0}{4.7}{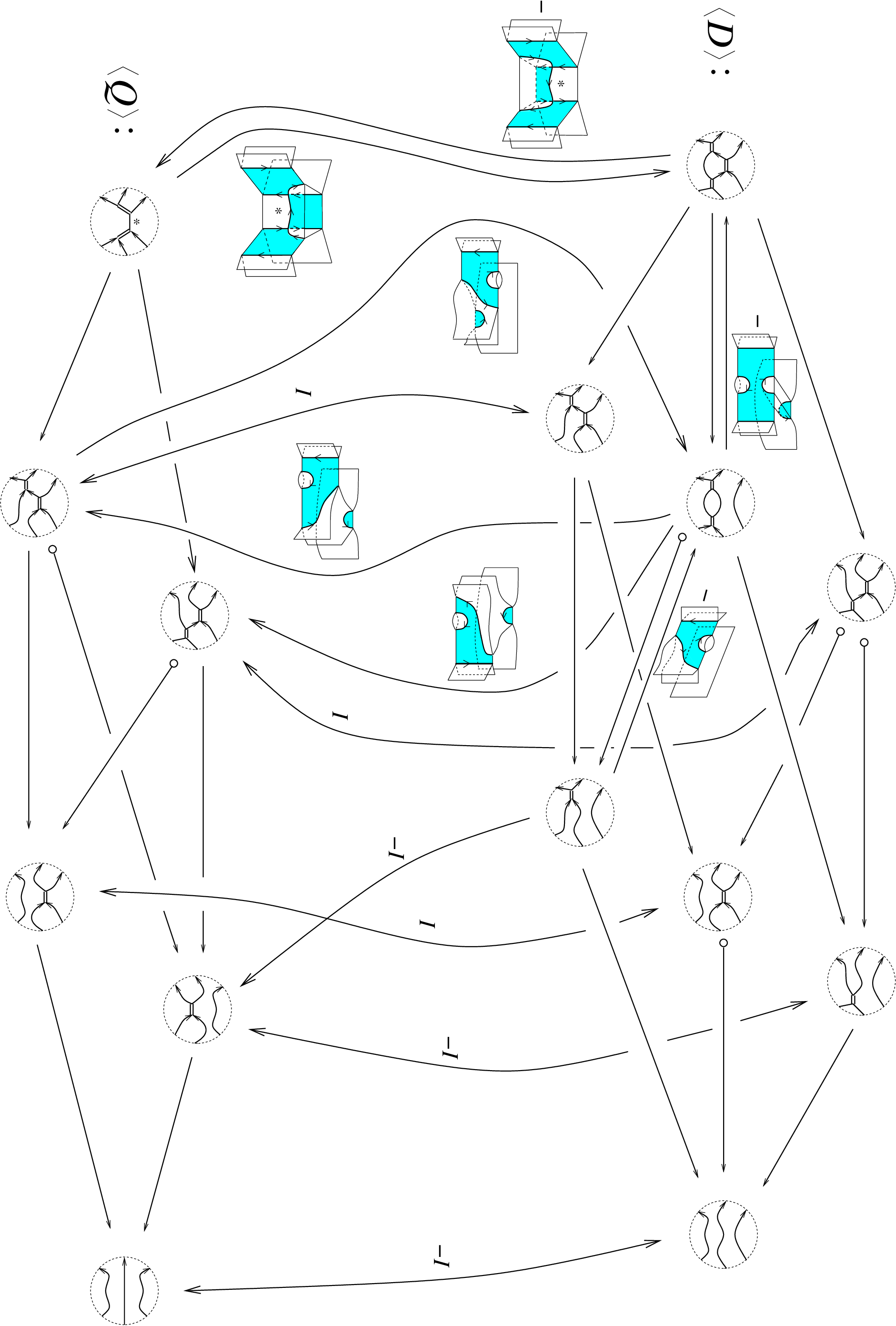} %\figwins{0}{4.9}{ReidIII}
\caption[Invariance under \emph{Reidemeister III}]{Invariance under \emph{Reidemeister III}. A circle attached to the tail of an arrow indicates that the corresponding morphism has a minus sign.}
\label{slN:fig:RIIIInvariance}
\end{figure}
In order
to prove that $\brak{D'}$ is homotopy equivalent to $\brak{D}$ 
we show that the latter is homotopy equivalent to a third complex denoted 
$\brak{Q}$ in Figure~\ref{slN:fig:RIIIInvariance}. The 
differential in $\brak{Q}$ in homological 
degree $0$ is defined by 
$$\figins{-10}{0.6}{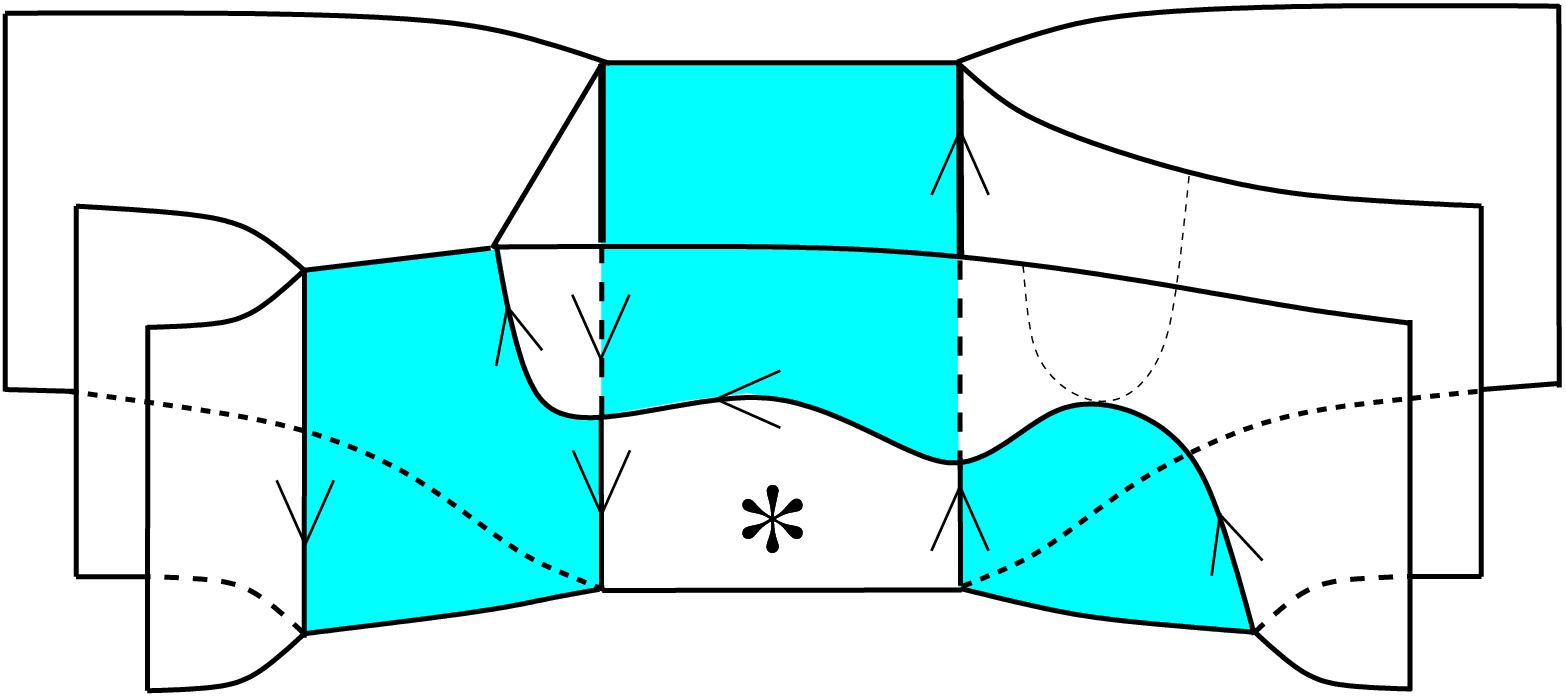}$$
for one summand and a similar foam for the other summand.   
By applying a symmetry relative to a horizontal axis crossing each diagram 
in $\brak{D}$ and $\brak{Q}$ we obtain a homotopy equivalence between 
$\brak{D'}$ and $\brak{Q'}$. It is easy to see that $\brak{Q}$ and 
$\brak{Q'}$ are isomorphic. In homological degree $0$ the isomorphism 
is given by the obvious foam with two singular vertices. In the other 
degrees the isomorphism is given by the identity (in degrees $1$ and $2$ 
one has to swap the two terms of course). The fact that this defines an 
isomorphism follows immediately from the~\eqref{slN:eq:MP} relation. We conclude 
that $\brak{D}$ and $\brak{D'}$ are homotopy equivalent.    
\end{proof}

From Theorem~\ref{slN:thm:inv} we see that we can use any diagram $D$ of $L$ to obtain the 
invariant in $\kom_{/h}(\foam)$ which justifies the notation $\brak{L}$ for $\brak{D}$.

%%%%%%%%%%%%%%%%%%%%%%%%%%%%%%%%%%%%%%%%
%%%                                  %%%
%%%        Functoriality             %%%
%%%                                  %%%
%%%%%%%%%%%%%%%%%%%%%%%%%%%%%%%%%%%%%%%%
\section{Functoriality}
\label{slN:sec:functoriality}
As in Chapter~\ref{chap:univ3}, the proof of the functoriality of our homology follows the same line of reasoning as in~\cite{bar-natancob} and~\cite{mackaay-vaz}.
As in those papers, it is clear that the construction and the results of the previous sections can be extended to the category of tangles, following a similar approach using open webs and foams with corners. A foam with corners should be considered as living inside a cylinder, as in~\cite{bar-natancob}, such that the intersection with the cylinder is just a disjoint set of vertical edges.  

The degree formula can be extended to the category of open webs and foams with corners by
\begin{defn}
Let $u$ be a foam with $d_\bdot$ dots of type $\bdot$, $d_\wdot$ dots of type 
$\wdot$ and $d_\bwdot$ dots of type $\bwdot$. Let $b_i$ be the number of vertical edges 
of type $i$ of the boundary of $u$. The $q$-grading of $u$ is given 
by
\begin{equation}
q(u)= -\sum_{i=1}^3 i(N-i) q_i(u) - 2(N-2)q_{\sk}(u) + 
\dfrac{1}{2}\sum_{i=1}^3 i(N-i)b_i+ 2d_\bdot + 4d_\wdot +6d_\bwdot.
\end{equation}
\end{defn}

Note that the KL formula also induces a grading on foams with corners, because 
for any foam $u$ between two (open) webs $\Gamma_1$ and $\Gamma_2$, it gives an element 
in the graded vector space $\Ext(M_1,M_2)$, 
where $M_i$ is the matrix factorization associated to $\Gamma_i$ in~\cite{KR}, for $i=1,2$. 
Recall that the $\Ext$ groups have a $\bZ/2\bZ\times\bZ$-grading. For foams there is 
no $\bZ/2\bZ$-grading, but the $\bZ$-grading survives.  

\begin{lem} For any foam $u$, the KL grading of $u$ is equal to $q(u)$.
\end{lem}
\begin{proof}
Both gradings are additive under horizontal and vertical glueing and 
are preserved by the 
relations in $\foam$. Also the degrees of the dots are the same in both 
gradings. 
Therefore it is enough to establish the equality between the gradings 
for the foams which generate $\foam$. For any foam without a singular graph the gradings are 
obviously equal, so let us concentrate on the singular cups and caps, the singular saddle point 
cobordisms and the cobordisms with one singular vertex in 
Figure~\ref{slN:fig:elemfoams}. To compute the degree of the singular cups and caps, for both gradings, one can use the digon removal 
relations. For example, let us consider the singular cup 
$$\figins{-10}{0.4}{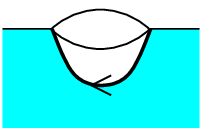}.$$ 
Any grading that preserves relation~\eqref{slN:eq:DR1} has to attribute the value of 
$-1$ to that foam, because the foam on the l.h.s. of~\eqref{slN:eq:DR1} has degree $0$, being an identity, 
and the dot on the r.h.s. has degree $2$. Similarly one can compute the degrees of the other 
singular cups and caps. To compute the degree of the singular saddle-point cobordisms, one can 
use the removing disc relations~\eqref{slN:eq:RD1} and~\eqref{slN:eq:RD2}. For example, the saddle-point cobordism in 
Figure~\ref{slN:fig:elemfoams} 
has to have degree $1$. Finally, using the~\eqref{slN:eq:MP} relation 
one shows 
that both foams on the r.h.s. in Figure~\ref{slN:fig:elemfoams} 
have to have degree $0$.          
\end{proof}

\begin{cor}
For any closed foam $u$ we have that $\KL{u}$ is zero if $q(u)\neq 0$.
\end{cor}

As in~\cite{mackaay-vaz} we have the following lemma, which 
is the analogue of Lemma 8.6 in~\cite{bar-natancob}:

\begin{lem}
For a crossingless tangle diagram $T$ we have that
$\Hom_{\foam}(T,T)$ is zero in negative degrees and $\bQ$ in degree zero.
\end{lem}

\begin{proof}
Let $T$ be a crossingless tangle diagram and $u\in\Hom_{\foam}(T,T)$.   
Recall that $u$ can be considered to be in a cylinder with vertical edges 
intersecting the latter. The boundary of $u$ consists of a disjoint union 
of circles (topologically speaking). By dragging these 
circles slightly into the interior of $u$ one gets a disjoint union of circles in 
the interior of $u$. Apply relation (CN$_1$) to each of these circles. We 
get a linear combination of terms in $\Hom_{\foam}(T,T)$ each of which 
is the disjoint union of the identity on $T$, possibly decorated with dots, 
and a closed foam, which can be evaluated by $\KL{\ }$. Note 
that the identity of $T$ with any number of dots has always 
non-negative degree. Therefore, if $u$ has negative degree, the closed 
foams above have negative degree as well and evaluate to zero. This shows 
the first claim in the lemma. If $u$ has degree $0$, the only terms 
which survive after evaluating the closed foams have degree $0$ as well and 
are therefore a multiple of the identity on $T$. This proves the second 
claim in the lemma. 
\end{proof}

The proofs of Lemmas~8.7-8.9 in~\cite{bar-natancob} are ``identical''. 
The proofs of Theorem 4 and Theorem 5 follow the same reasoning but 
have to be adapted as in~\cite{mackaay-vaz}. One has to use the 
homotopies of our Section~\ref{slN:sec:invariance} instead of the homotopies 
used in~\cite{bar-natancob}. 
Without giving further details, we state the 
main result. Let $\kom_{/\bQ^*h}(\foam)$ denote the category 
$\kom_{/h}(\foam)$ 
modded out by $\bQ^*$, the invertible rational numbers. Then

\begin{prop}
\label{slN:prop:func}
$\brak{\;}$ defines a functor $\Link\to 
\kom_{/\bQ^*h}(\foam )$.
\end{prop}

%%%%%%%%%%%%%%%%%%%%%%%%%%%%%%%%%%%%%%%%
%%%                                  %%%
%%%        Homology                  %%%
%%%                                  %%%
%%%%%%%%%%%%%%%%%%%%%%%%%%%%%%%%%%%%%%%%
\section{The $\ \sln$-link homology}
\label{slN:sec:taut-functor}

\begin{defn}
Let $\Gamma$, $\Gamma'$ be closed webs and $f\in\Hom_{\foam }(\Gamma,\Gamma')$. Define a functor 
$\mathcal{F}$ between the categories $\foam$ and the category $\V$ 
of $\bZ$-graded rational vector spaces and ${\bZ}$-graded linear maps as
\begin{enumerate}
\item
$\cF(\Gamma)=\Hom_{\foam}(\emptyset,\Gamma),$
\item $\cF(f)$ is the $\bQ$-linear map 
$\cF(f)\colon\Hom_{\foam }(\emptyset,\Gamma)\to\Hom_{\foam }(\emptyset,\Gamma')$ given by composition.
\end{enumerate} 
\end{defn}

\medskip

Note that $\cF$ is a tensor functor and that the degree of $\cF(f)$ equals $q(f)$. Note also that 
$\cF(\unknot)\cong \hy(\cp{N-1})\{-N+1\}$ and 
$\cF(\figins{-2.5}{0.16}{dble-circ})\cong \hy(\cG_{2,N})\{-2N+4\}$.

The following are a categorified version of the relations in Figure~\ref{KR:fig:moy}.

\begin{lem}[MOY decomposition]\label{slN:lem:moy-F}
We have the following decompositions under the functor $\cF$:
\begin{enumerate}
\item \label{slN:eq:decomp-dr1}
$\cF\left(
\figins{-8}{0.3}{digon-up}\right)\cong \cF\left(
\figins{-8}{0.3}{dbedge-up}\right)\left\{-1\right\}\bigoplus \ \cF\left(
\figins{-8}{0.3}{dbedge-up}\right)\left\{1\right\}$.
\vspace{1em}
\item \label{slN:eq:decomp-dr2}
$\cF\left(
\figins{-8}{0.3}{dbedge-dig}\right)\cong\bigoplus\limits_{i=0}^{N-2}\cF\left(
\figins{-8}{0.3}{edge-up}\right)\left\{2-N+2i\right\}$.
\vspace{1em}
\item \label{slN:eq:decomp-sqr1}
$\cF\left(
\figins{-8}{0.3}{square}\right)\cong\cF\left(
\figins{-8}{0.3}{twoedges-lr}\right)\bigoplus\left(\bigoplus\limits_{i=0}^{N-3} \ \cF\left(
\figins{-8}{0.3}{twoedges-ud}\right)\left\{3-N+2i\right\}\right)$.
\vspace{1em}
\item \label{slN:eq:decomp-sqr2}
$\cF\left(\figins{-20}{0.65}{moy5-1}\right) \bigoplus
\cF\left(\figins{-20}{0.65}{moy5-21}\right)\cong 
\cF\left(\figins{-20}{0.65}{moy5-2}\right)\bigoplus
\cF\left(\figins{-20}{0.65}{moy5-11}\right)
$.
\end{enumerate}
\end{lem}

\begin{proof}
%%%%%%%%
% MOY1 %
%%%%%%%%
\emph{\eqref{slN:eq:decomp-dr1}}: Define grading-preserving maps 
\begin{gather*}
\varphi_0\colon\cF\left(
\figins{-8}{0.3}{digon-up}\right)\left\{1\right\} 
\to \cF\left(\figins{-8}{0.3}{dbedge-up}\right)
\qquad
\varphi_1\colon\cF\left(
\figins{-8}{0.3}{digon-up}\right)\left\{-1\right\}  
\to \cF\left(\figins{-8}{0.3}{dbedge-up}\right)
\\[2.0ex]
\psi_0\colon\cF\left(
\figins{-8}{0.3}{dbedge-up}\right) \to
\cF\left(
\figins{-8}{0.3}{digon-up}\right)\left\{1\right\} 
\qquad
\psi_1\colon \cF\left(
\figins{-8}{0.3}{dbedge-up}\right) \to
\cF\left(
\figins{-8}{0.3}{digon-up}\right)\left\{-1\right\} 
\end{gather*}
as
$$
\varphi_0=  \cF\left(\figins{-10}{0.3}{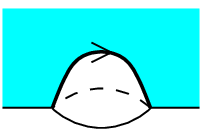}\right),\ \ \ 
\varphi_1=\cF\left(\figins{-10}{0.3}{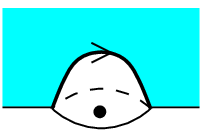} \right),\ \ \
\psi_0=\cF\left(\figins{-6}{0.3}{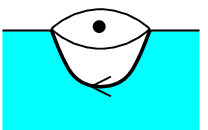}\right),\ \ \
\psi_1=-\cF\left( \figins{-6}{0.3}{scup}\right).
$$
The bubble identities imply that $\varphi_i\psi_j=\delta_{i,j}$ (for $i,j=0,1$) and from the~\eqref{slN:eq:DR1} relation it follows that $\psi_0\varphi_0+\psi_1\varphi_1$ is the identity map in $\cF\left(\figwhins{-3.5}{0.17}{0.5}{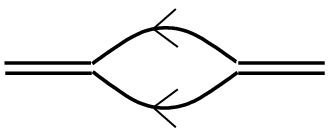}\right)$.

%%%%%%%%
% MOY2 %
%%%%%%%%
\n\emph{\eqref{slN:eq:decomp-dr2}}: Define grading-preserving maps 
$$\varphi_i\colon\cF\left(
\figins{-8}{0.3}{dbedge-dig}\right)\left\{N-2-2i\right\}
\to \cF\left(\figins{-8}{0.3}{edge-up}\right),
\qquad 
\psi_i\colon\cF\left(
\figins{-8}{0.3}{edge-up}\right)
\to\cF\left(
\figins{-8}{0.3}{dbedge-dig}\right)\left\{N-2-2i\right\},
$$
for $0\leq i\leq N-2$, as
$$
\varphi_i=  \cF\left(\figins{-24}{0.65}{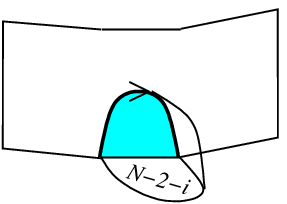}\right),\qquad
\psi_i=\sum\limits_{j=0}^{i}\cF\left(\figins{-22}{0.65}{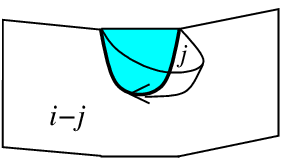}\right).
$$

We have $\varphi_i\psi_k=\delta_{i,k}$ and $\sum\limits_{i=0}^{N-2}\psi_i\varphi_i=\id\left(\cF\left(\figins{-4}{0.2}{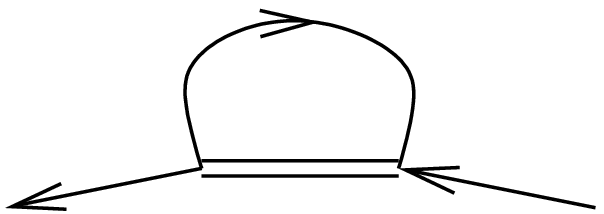} \right)\right)$. The first assertion is straightforward and can be checked using the~\eqref{slN:eq:RD1} and (S$_1$) relations and the second is immediate from the~\eqref{slN:eq:DR2} relation, which can be written as  

$$
\figins{-32}{0.8}{digonfid_b} = \sum_{i=0}^{N-2}\sum_{j=0}^{i}\
\figins{-32}{0.8}{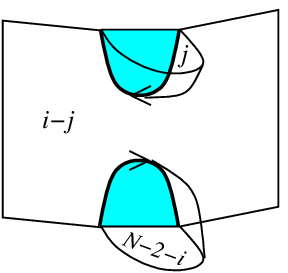}
\rlap{\hspace{15ex}\text{(DR$_2$)}}
$$

\bigskip

%%%%%%%%
% MOY3 %
%%%%%%%%
\n\emph{\eqref{slN:eq:decomp-sqr1}}: Define grading-preserving maps 
$$\varphi_i\colon\cF\left(
\figins{-8}{0.3}{square}\right)\left\{N-3+2i\right\}
\to \cF\left(\figins{-8}{0.3}{twoedges-ud}\right),
\quad 
\psi_i\colon\cF\left(
\figins{-8}{0.3}{twoedges-ud}\right)
\to\cF\left(
\figins{-8}{0.3}{square}\right)\left\{N-3+2i\right\},
$$
for $0\leq i\leq N-3$, and
$$\rho\colon\cF\left(
\figins{-8}{0.3}{square}\right)
\to \cF\left(\figins{-8}{0.3}{twoedges-lr}\right),
\qquad 
\tau\colon\cF\left(
\figins{-8}{0.3}{twoedges-lr}\right)
\to\cF\left(
\figins{-8}{0.3}{square}\right),
$$
as
\begin{align*}
\varphi_i &=  \cF\left(\figins{-22}{0.7}{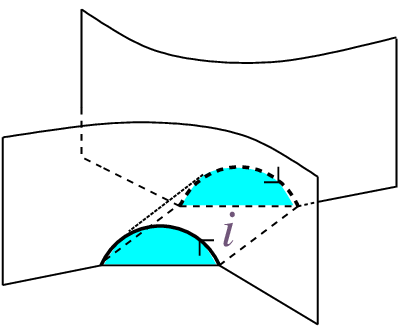}\right)
&\psi_i &=\sum\limits_{a+b+c=N-3-i}\cF\left(\figins{-22}{0.7}{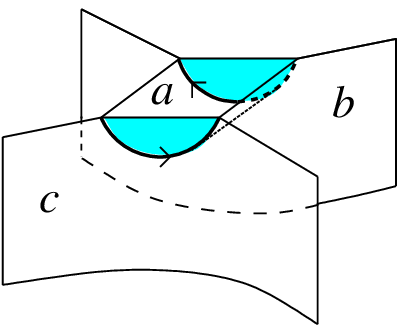}\right)
\\
\rho &= \cF\left(\figins{-22}{0.7}{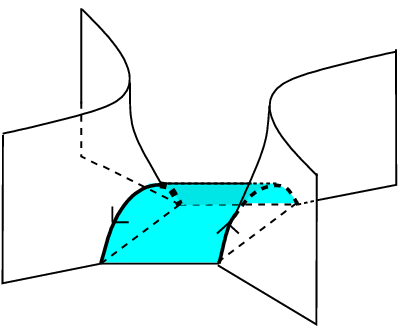}\right)
&\tau &= -\cF\left(\figins{-22}{0.7}{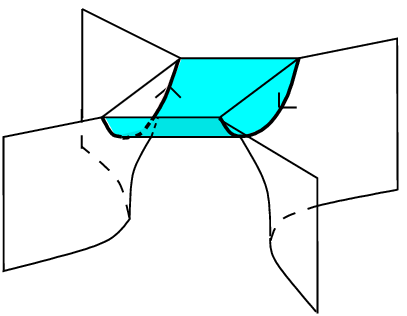}\right)
\end{align*}

Checking that $\varphi_i\psi_k=\delta_{i,k}$ for $0\leq i,k\leq N-3 $, 
$\varphi_i\tau=0$ and $\rho\psi_i=0$, for $0\leq i\leq N-3$, and 
$\rho\tau=1$ is left to the reader. From the~\eqref{slN:eq:SqR1} relation it 
follows that 
$\tau\rho+\sum\limits_{i=0}^{N-3}\psi_i\varphi_i=\id\left(\cF\left(\figins{-4.5}{0.2}{square} 
\right)\right)$.

%%%%%%%%
% MOY4 %
%%%%%%%%
\n\emph{Direct Sum Decomposition~\eqref{slN:eq:decomp-sqr2}}: We prove direct decomposition~\eqref{slN:eq:decomp-sqr2} showing that

$$\xymatrix@R=1mm{
\cF\left(\figins{-20}{0.65}{moy5-1}\right)\cong
\cF\left(\figins{-20}{0.65}{moy5-11}\right)\bigoplus
\cF\left(\figins{-20}{0.65}{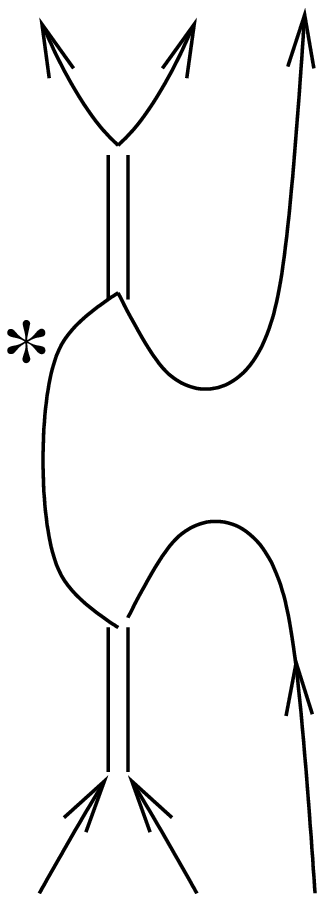}\right) &
\cF\left(\figins{-20}{0.65}{moy5-2}\right)\cong
\cF\left(\figins{-20}{0.65}{moy5-21}\right)\bigoplus
\cF\left(\figins{-20}{0.65}{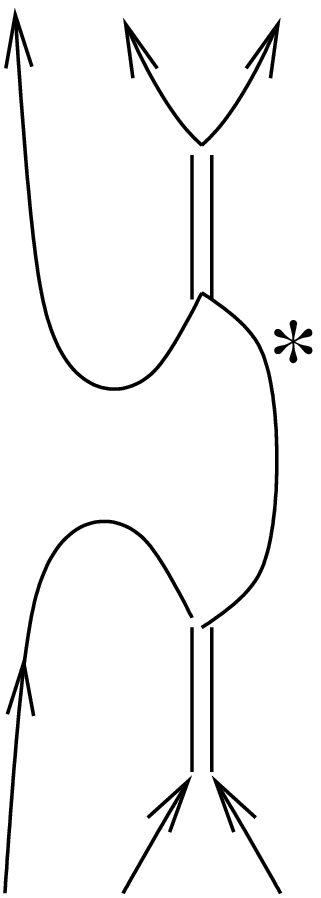}\right). \\
a) & b) 
}$$
Note that this suffices because the last term on the r.h.s. of a) is 
isomorphic to the last term on the r.h.s. of b) by the~\eqref{slN:eq:MP} relation. 

To prove $a)$ we define grading-preserving maps 
\begin{align*}
\varphi_0\colon\cF\left(
\figins{-15}{0.5}{moy5-1}\right)\to 
\cF\left(
\figins{-15}{0.5}{moy5-11}\right)
\qquad 
\varphi_1\colon\cF\left(
\figins{-15}{0.5}{moy5-1}\right)
\to\cF\left(
\figins{-15}{0.5}{moy5-22}\right)
\\[2.0ex]
\psi_0\colon\cF\left(
\figins{-15}{0.5}{moy5-11}\right)\to 
\cF\left(
\figins{-15}{0.5}{moy5-1}\right)
\qquad 
\psi_1\colon\cF\left(
\figins{-15}{0.5}{moy5-22}\right)
\to\cF\left(
\figins{-15}{0.5}{moy5-1}\right)
\end{align*}
by
\begin{align*}
\varphi_0 &= -\cF\left(\figins{-19}{0.6}{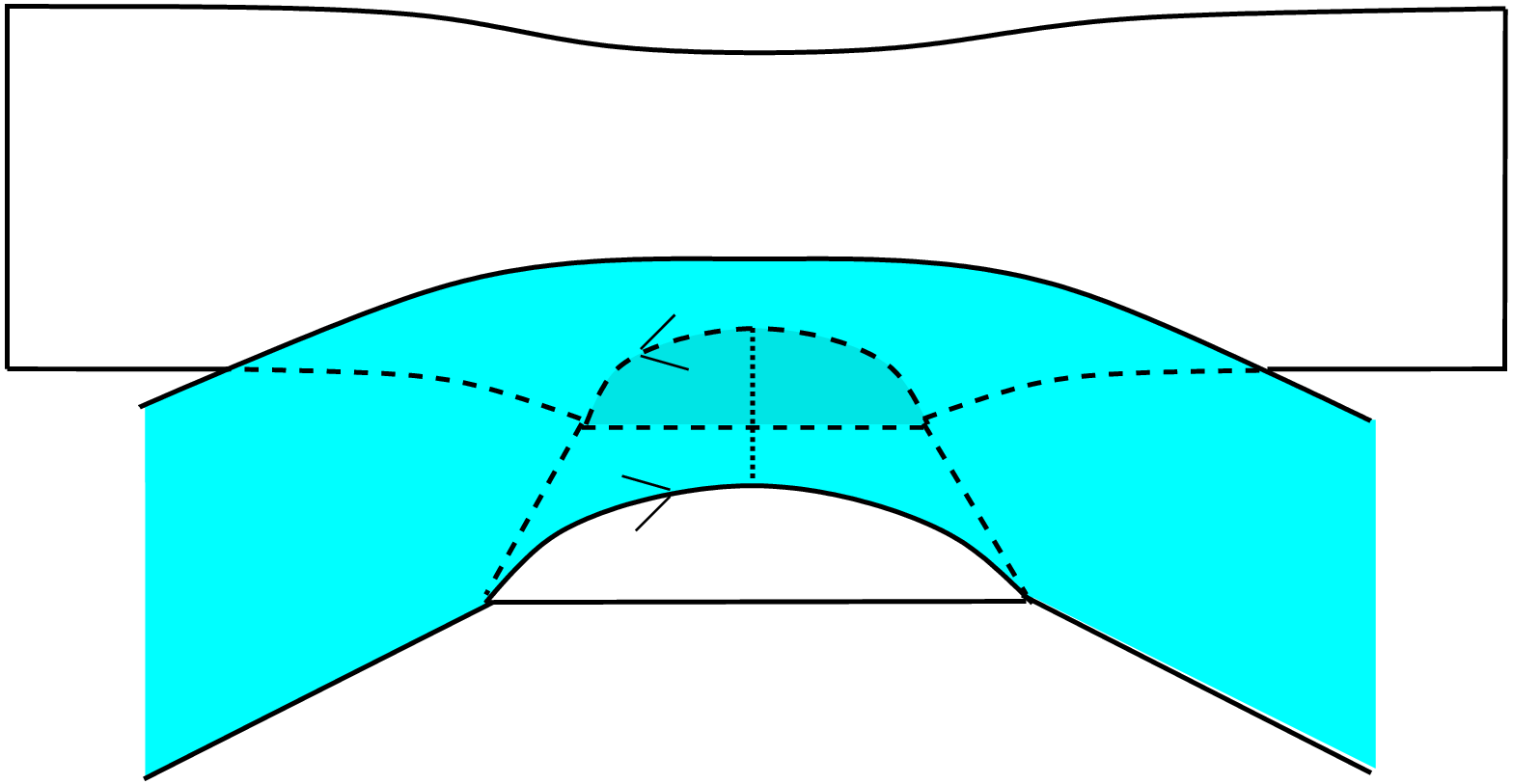}\right)
&\varphi_1 &= -\cF\left(\figins{-19}{0.6}{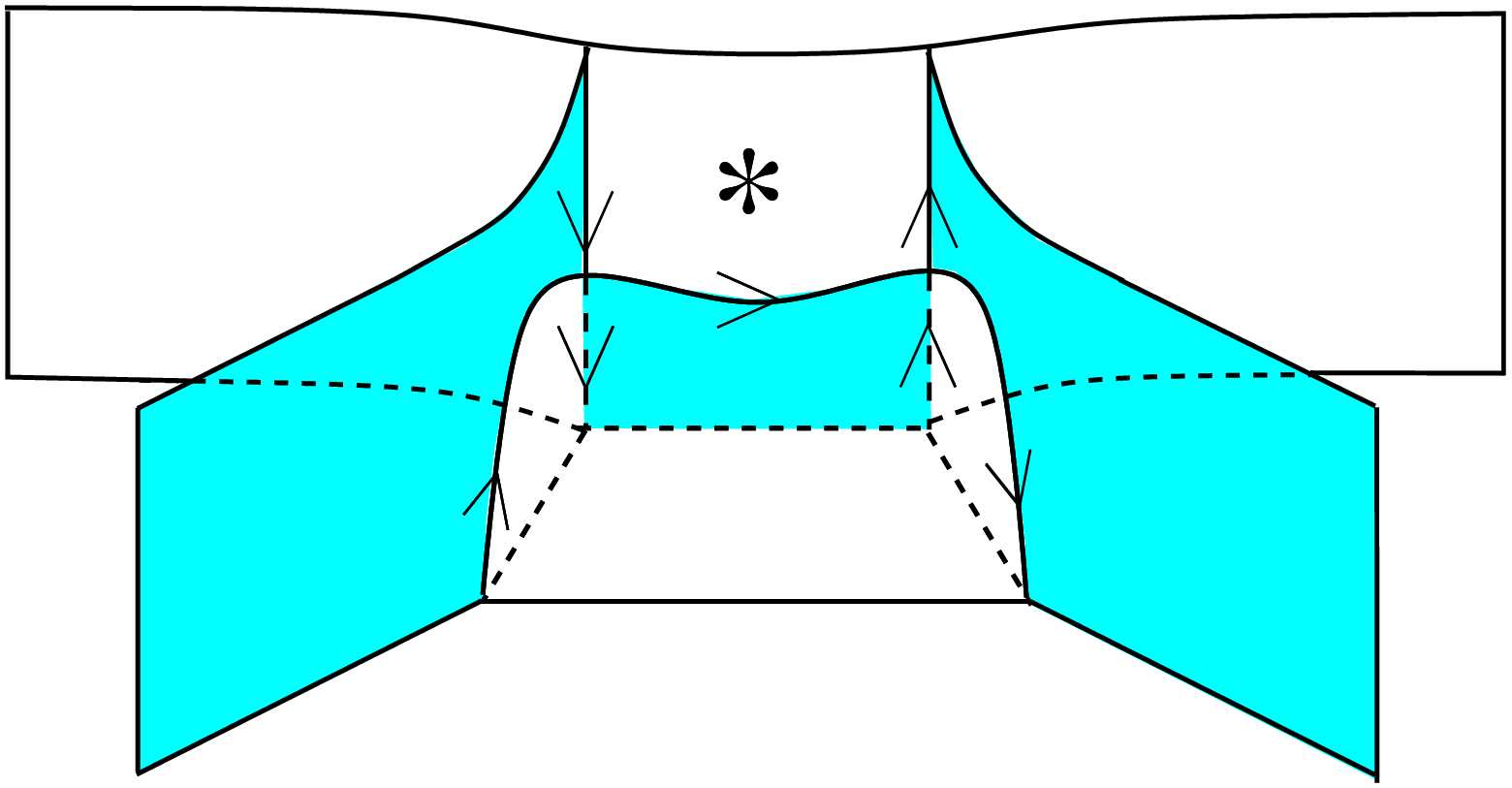}\right) 
\\[2ex]
\psi_0 &= \cF\left(\figins{-19}{0.6}{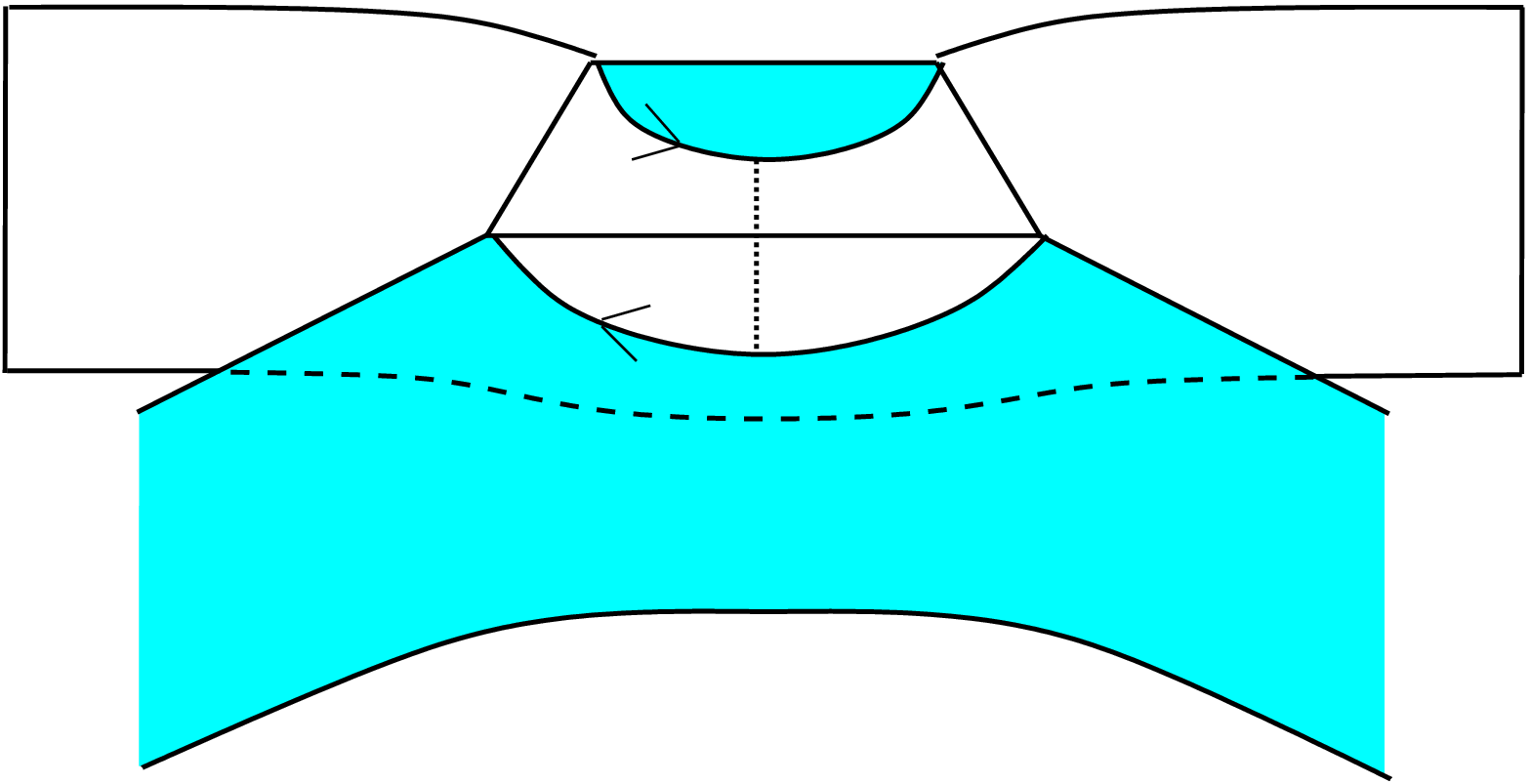}\right)
&\psi_1 &= \cF\left(\figins{-19}{0.6}{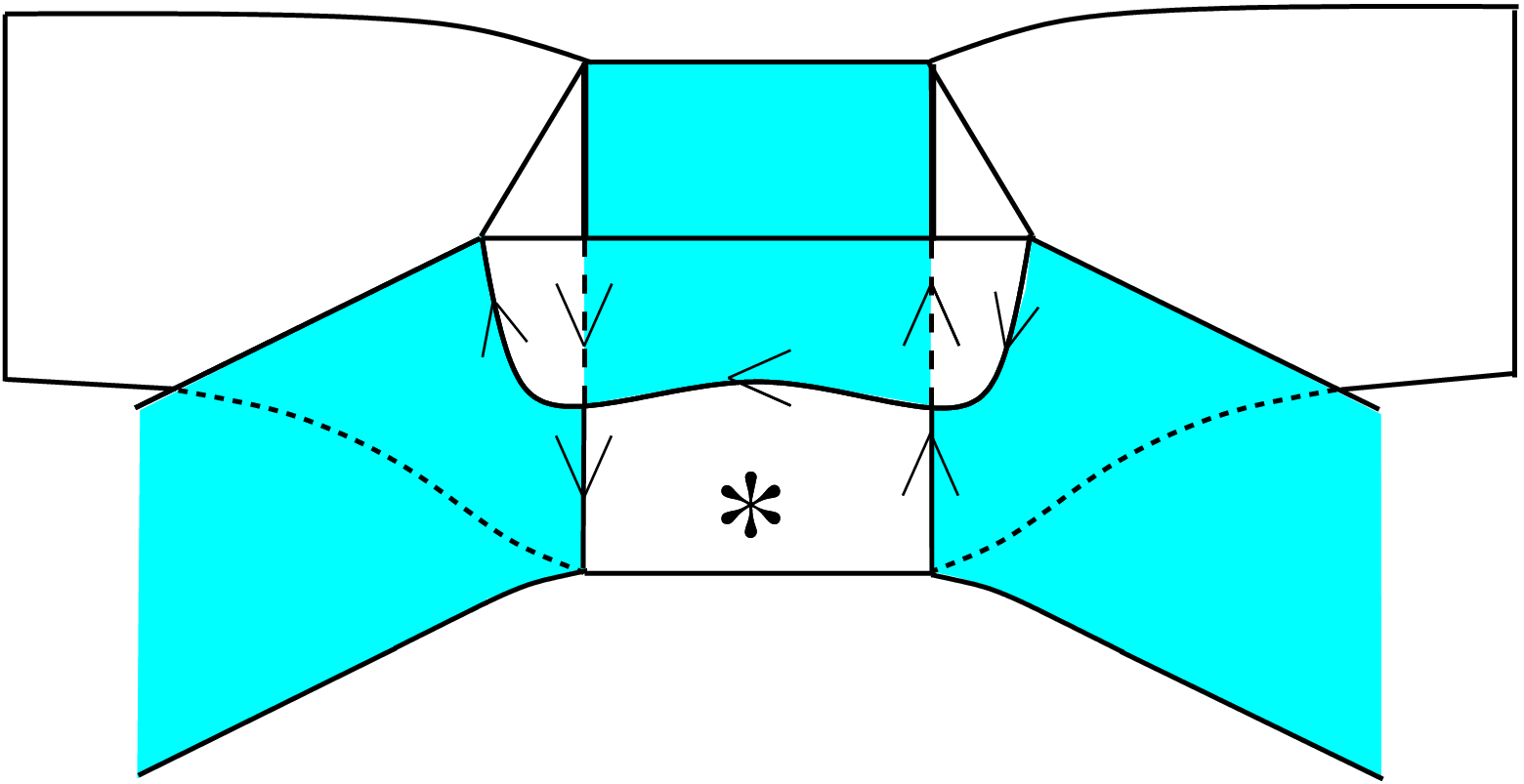}\right)
\end{align*}

We have that $\varphi_i\psi_j=\delta_{i,j}$ for $i,j = 0,1 $ (we leave the details to the reader). From the~\eqref{slN:eq:SqR2} relation it follows that $$\psi_0\varphi_0+\psi_1\varphi_1=
\id\left(\cF\left(\figins{-15}{0.5}{moy5-1} \right)\right).$$ Applying a symmetry to all diagrams 
in decomposition $a)$ gives us decomposition $b)$.
\end{proof}

In order to relate our construction to the $\sln$ polynomial we need to introduce shifts. We denote by $\{n\}$ an upward shift in the $q$-grading by $n$ and by $[m]$ an upward shift in the homological grading by $m$. 
\begin{defn}\label{slN:def:gradedcomplx}
Let $\brak{L}_i$ denote the $i$-th homological degree of the complex $\brak{L}$. We define the 
$i$-th homological degree of the complex $\cF(L)$ to be  
$$\cF_i(L)=\cF\brak{L}_i[-n_-]\{ (N-1)n_+ - Nn_-  + i \},$$ 
where $n_+$ and $n_-$ denote the number of positive and negative crossings in the diagram used to 
calculate $\brak{L}$.
\end{defn}

We now have a homology functor $\Link\to\V$ which 
we still call $\cF$. Definition~\ref{slN:def:gradedcomplx}, 
Theorem~\ref{slN:thm:inv} and Lemma~\ref{slN:lem:moy-F} imply that
\begin{thm}
For a link $L$ the graded Euler characteristic of 
$\hy\left(\cF(L)\right)$ equals $P_N(L)$, the $\sln$ polynomial of $L$. 
\end{thm}

The MOY-relations are also the last bit that we need in 
order to show the following theorem. 
\begin{thm} For any link $L$, the bigraded complex $\cF(L)$ is 
isomorphic to the Khovanov-Rozansky complex $\KR(L)$ in~\cite{KR}.  
\end{thm}
\begin{proof} The map $\KLket{\ }$ 
defines a grading preserving linear injection 
$\cF(\Gamma)\to \KR(\Gamma)$, for 
any web $\Gamma$. Lemma~\ref{slN:lem:moy-F} implies that the 
graded dimensions of $\cF(\Gamma)$ and 
$\KR(\Gamma)$ are equal, so $\KLket{\ }$ is 
a grading preserving linear isomorphism, for any web $\Gamma$.   

To prove the theorem we would have to show that 
$\KLket{\ }$ commutes 
with the differentials. We call  
$$
\figins{0}{0.7}{ssaddle} \qquad
$$
the \emph{zip} and 
$$
\figins{0}{0.7}{ssaddle_ud} \qquad
$$
the \emph{unzip}. Note that both the zip and the unzip 
have $q$-degree $1$. Let $\Gamma_1$ be the source web of the zip and 
$\Gamma_2$ its target web, and let $\Gamma$ be the theta web, which is 
the total boundary of the zip where the vertical edges have been pinched. 
The graded dimension of 
$\Ext(\Gamma_1,\Gamma_2)$ is equal to 
$$q^{2N-2}\qdim(\Gamma)=q+q^2(\ldots),$$
where $(\ldots)$ is a polynomial in $q$. Therefore the differentials in the 
two complexes 
commute up to a scalar. By the removing disc relation~\eqref{slN:eq:RD1} we see that if the ``zips'' 
commute up to $\lambda$, then the ``unzips'' commute up to $\lambda^{-1}$. If $\lambda\ne 1$, 
we have to modify our map between the two complexes slightly, in order to get an honest morphism of complexes. We 
use Khovanov's idea of ``twist equivalence'' in~\cite{khovanovfrob}. For a given 
link consider the hypercube of resolutions. If an arrow in the hypercube 
corresponds to a zip, multiply $\KLket{(\text{target})}$ by $\lambda$, where 
target means the target of the arrow. 
If it corresponds to an unzip, multiply $\KLket{(\text{target})}$ by $\lambda^{-1}$.
This is well-defined, because all squares in the hypercube (anti-)commute. By definition this 
new map commutes with the differentials and therefore proves that the two complexes are 
isomorphic.   
\end{proof}

We conjecture that the above isomorphism actually extends to link cobordisms, 
giving a projective natural isomorphism between the two projective link 
homology functors. Proving this would require a careful comparison between the 
two functors for all the elementary link cobordisms.

%%%%%%%%%%%% End of Part 3 %%%%%%%%%%
%
%
%%%%%%%%%%%%%% Part 4 %%%%%%%%%%%%%%

%%%%%%%%%%%%%%%%%%%%%%%%%%%%%%%%%%%%%%%%
%%%                                  %%%
%%% (2,m)-torus links & torsion      %%%
%%%                                  %%%
%%%%%%%%%%%%%%%%%%%%%%%%%%%%%%%%%%%%%%%%
\section{The homology of the $(2,m)$-torus links and torsion}
\label{slN:sec:2m-torus}
The relations in Proposition~\ref{slN:prop:principal rels1} are defined over $\bZ$. This and the fact that all the homomorphisms considered in Section~\ref{slN:sec:KL} can be defined over $\bZ$ suggests that we have found a theory that is defined over the integers. 

\begin{conj}\label{slN:conj:int-homol}
The $\sln$-link homology defined in the previous sections is integral. 
\end{conj}
This means that the KL evaluation of every closed foam is an integer. Assuming Conjecture~\ref{slN:conj:int-homol} we can compute the integral $\sln$ homology for $(2,m)$-torus links. Although the (rational) Khovanov-Rozansky $\sln$ homology for $(2,m)$-torus knots has already been obtained in~\cite{rasmussen2-bridge} (see also~\cite{stosic-thicktknots, stosic-tlinks} for a treatment of torus knots and torus links) since $(2,m)$-torus links are 2-bridge, the results in this section have some interest because we prove that the integral $\sln$ homology has torsion by giving an explicit example. It is well known that the original Khovanov $\mathfrak{sl}(2)$-homology has torsion~\cite{shumakovitch-torsion} of various orders (see~\cite{asaeda-przytycki, bar-natanfast, shumakovitch-torsion} for some examples). We obtain torsion of order $N$ for the $\sln$-homology of $(2,m)$-torus links. We expect there to be examples of other torsion as well.

We start with a result which in some sense is a generalization of the \emph{reduction lemma} of~\cite{bar-natanfast} and can be proved as in~\cite{bar-natanfast}. Let $R$ be a ring.
Consider the following bicomplex $K$ of $R$-modules and $R$-module homomorphisms:
$$\xymatrix@R=3mm@C=6mm{
  \ar[r]  & 
I \ar[ddd]^\eta \ar[rr]^\phi && 
G\ar[rr]^\xi\ar[dd]^{f_2}
  \ar@/_2.4pc/[dddd]_>(0.3){f_1} && 
H\ar[rr]^\zeta\ar[dd]_{g_1}\ar@/^2.4pc/[dddd]^<(0.2){g_2} && 
J\ar[r]\ar[ddd]^\kappa & 
\\ \\
  & & &  b_1 \ar[rr]^\psi\ar'[dr][ddrr]^>(0.6)\gamma 
  & & b_2 \ar[drr]|\hole^<(0.6)\mu & & \\
   \ar[r] & 
C \ar[urr]|\hole^>(0.4)\alpha\ar[drr]^>(0.4)\beta 
  & & \oplus & & \oplus & &  
F \ar[r]& \\
  & & & 
D \ar[rr]^\varepsilon\ar[uurr]^>(.3){\delta}&  & 
E \ar[urr]^<(0.6)\nu & & 
}$$

\begin{lem}\label{slN:lem:red}
  If $\psi$ is an isomorphism then the total complex of the bicomplex $K$ is homotopy equivalent to the total complex of the bicomplex
$$\xymatrix@R=10mm@C=7mm{
  \ar[r]  & 
I \ar[d]^\eta \ar[rr]^\phi && 
G\ar[rr]^\xi\ar[d]^{f_1} && 
H\ar[rr]^\zeta\ar[d]^{g_2-\gamma\psi^{-1}g_1} && 
J\ar[r]\ar[d]^\kappa & \\  \ar[r]  & 
C \ar[rr]^\beta &&  
D \ar[rr]^{\varepsilon -\gamma\psi^{-1}\delta}  && 
E\ar[rr]^\nu && 
F \ar[r] & 
}$$
\end{lem}

\medskip

We now use Lemma~\ref{slN:lem:red} to compute the integral $\sln$ homology of the torus link $T_{2,m}$.
Let $\sigma =\overcrossing$. Recall that $\brak{\overcrossing}=\brak{\orsmoothing}\xra{d}\brak{\DoubleEdge}\{1\}$. Then by~\eqref{slN:eq:DR1} we have an isomorphism
$$\brak{\sigma^2}=
\figins{-33}{1}{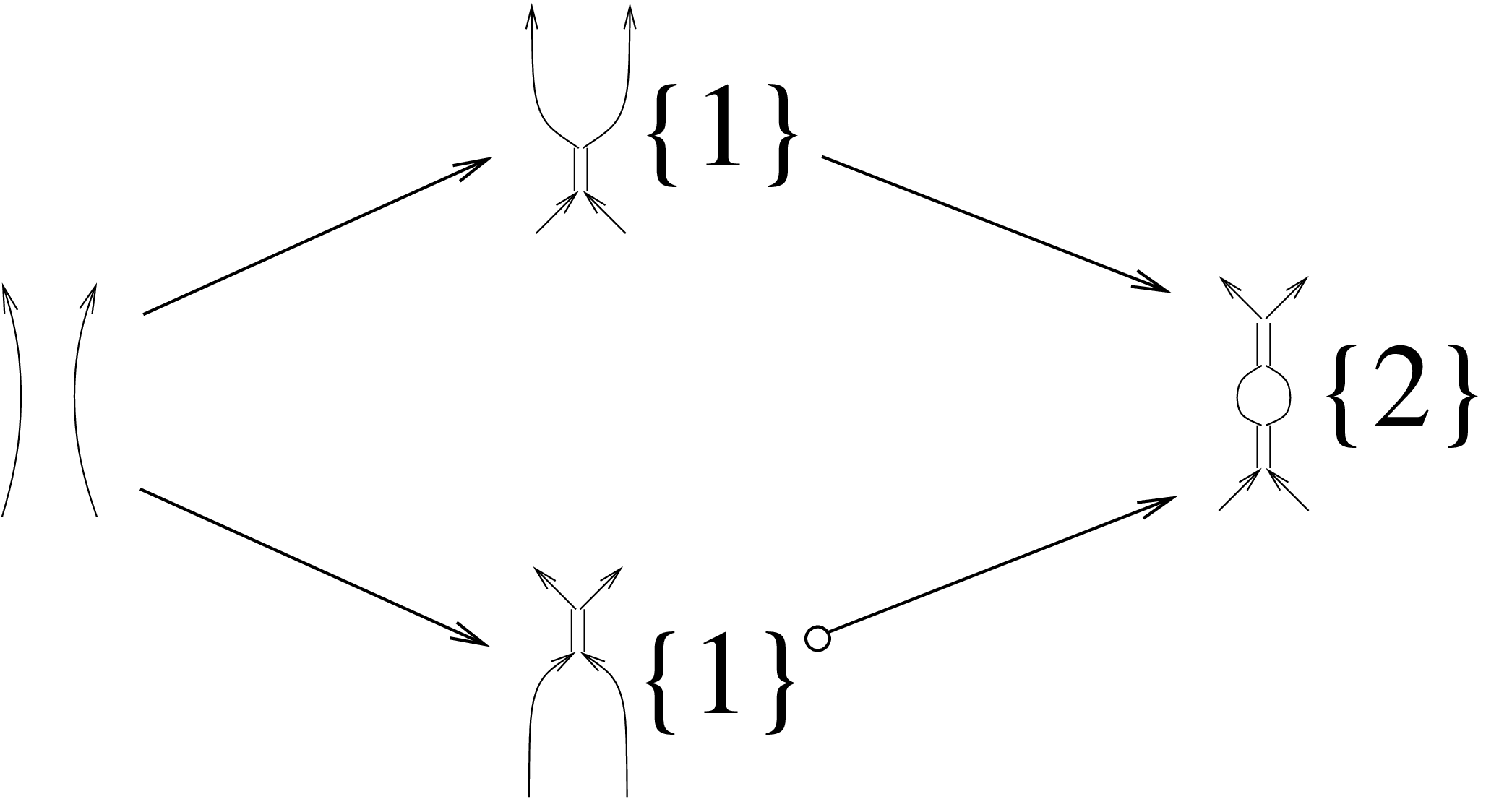}\cong\ 
\figins{-33}{1}{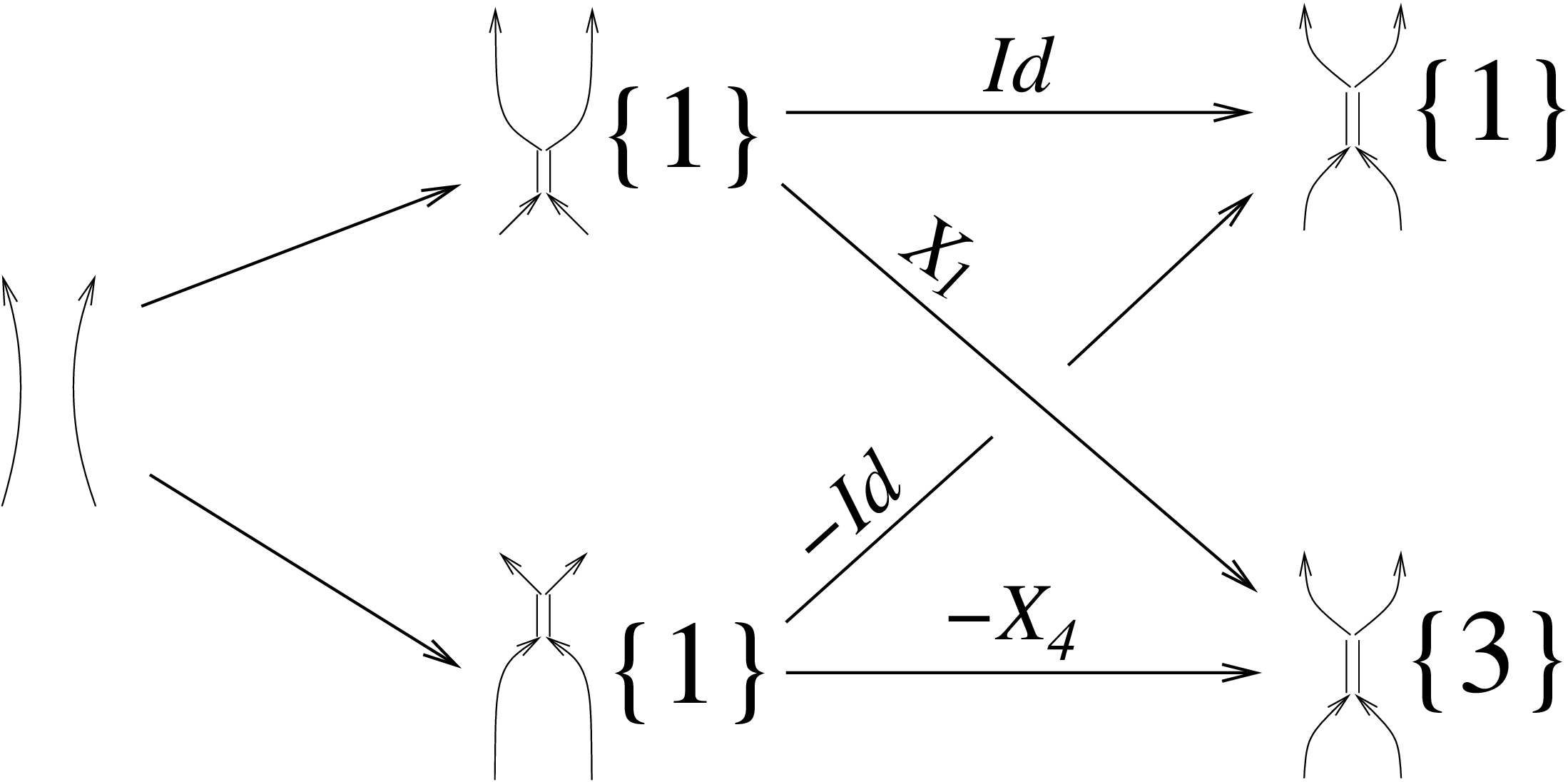}
$$
(the notation for the maps $X_i$ is explained in Figure~\ref{slN:fig:edge-lab}) and by Lemma~\ref{slN:lem:red} the last complex is homotopy equivalent to the complex
$$
\figins{-8}{0.3}{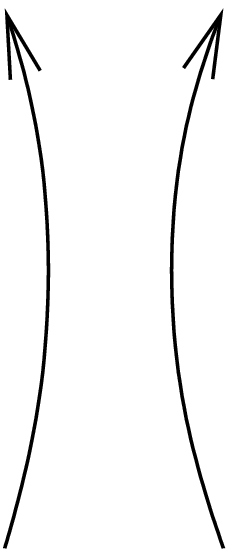} \xra{\ d\ }
\figins{-8}{0.3}{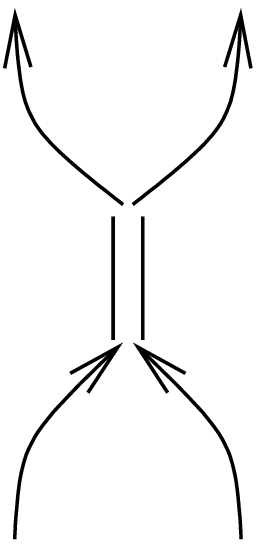}\{1\} \xra{\ X_1-X_4\ }
\figins{-8}{0.3}{sig2g2}\{3\}.
$$

Adding another crossing we get the bicomplex
$$
\brak{\sigma^3}\cong
\xymatrix@C=12mm{
\figins{-8}{0.3}{sig2g1}\quad\ \ \ar[r]^{d\ \ }\ar@<-1.6ex>[d]^{d} & 
\figins{-8}{0.3}{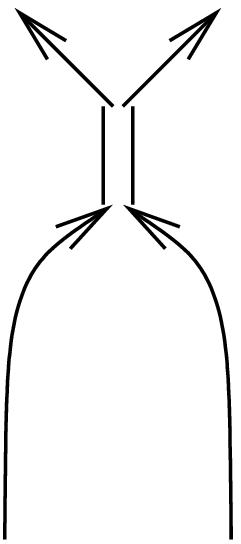}\{1\}\ar[r]^{\ \ X_1-X_4\quad }\ar@<-1.6ex>[d]^{-d} & 
\figins{-8}{0.3}{dumbelltop}\{3\}\ar@<-1.6ex>[d]^{d} \\
\figins{-8}{0.3}{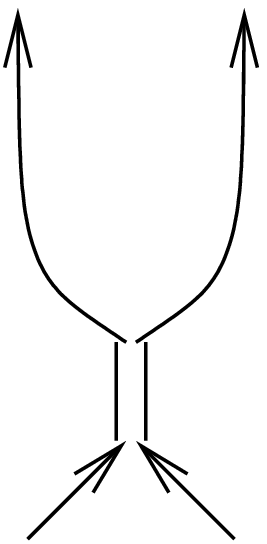}\{1\}\ar[r]^{\ d \ } & 
\figwhins{-8}{0.3}{0.14}{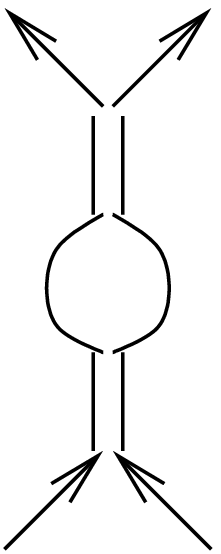}\{2\}\ar[r]^{\ \ X_1-X_5\quad } &
\figwhins{-8}{0.3}{0.14}{sig2dig}\{4\}
}
$$
with $\ \figins{-4}{0.18}{sig2g1}\ $ in homological degree 0 and where the map $X_5$ is $\figins{-6}{0.26}{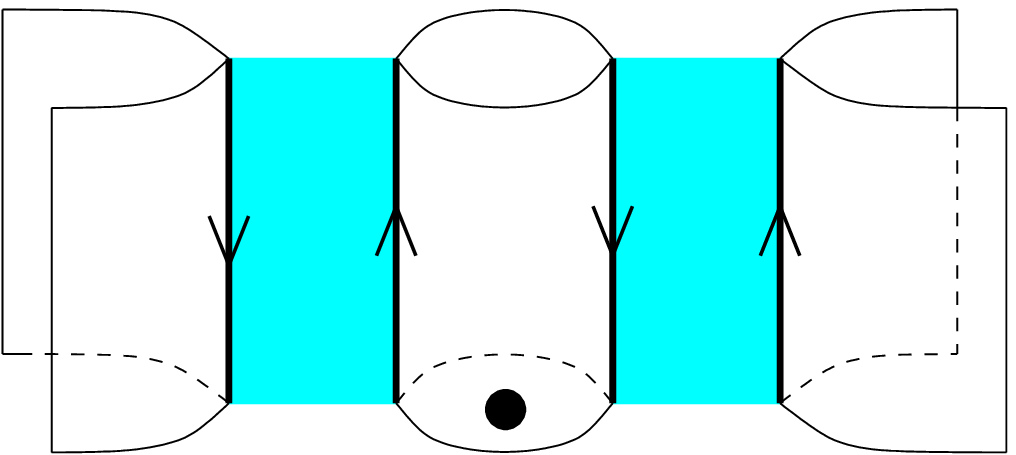}$. Using~\eqref{slN:eq:DR1} and Equation~\eqref{slN:eq:bubble-ij} we find that this bicomplex is isomorphic to the bicomplex
$$\xymatrix@R=3mm@C=6mm{
\figins{-8}{0.3}{sig2g1}\quad\ \
\ar@<-1.6ex>[ddd]^{d} \ar[rr]^{d} && 
\figins{-8}{0.3}{sig2g2}\{1\} 
\ar[rr]^{X_1-X_4}
\ar@<-1.6ex>[dd]^{-X_4}
\ar@<-1.8ex>@/_2.4pc/[dddd]_>(0.3){-1} && 
\figins{-8}{0.3}{sig2g2}\{3\}
%\ar[rr]
\ar@<-1.6ex>[dd]_{1}
\ar@/^2.4pc/[dddd]^<(0.42){X_4} %&& 
%0\ar[ddd]^0  
\\ \\  & &  
\figins{-8}{0.3}{sig2g2}\{3\} 
\ar[rr]^{-1}
\ar'[dr][ddrr]^>(-0.6){\mspace{-10.0mu}-X_2}
  & & 
\figins{-8}{0.3}{sig2g2}\{3\} 
%\ar[drr]|\hole^<(0.6){} & & 
\\ 
\figins{-8}{0.3}{sig2g2}\{1\} 
\ar@<0.5ex>[urr]|\hole^>(0.45){X_1}
\ar@<-1.6ex>[drr]^>(0.4){1}
  & & \oplus\ & & \oplus\ %& &  
%0  
\\   & & 
\figins{-8}{0.3}{sig2g2}\{1\} 
\ar[rr]^{X_1X_2}
\ar[uurr]^>(.35){X_1}&  & 
\figins{-8}{0.3}{sig2g2}\{5\} 
%\ar[urr]^<(0.6){} & 
}$$
which by Lemma~\ref{slN:lem:red} is homotopy equivalent to the bicomplex
$$
\xymatrix@C=12mm{
\figins{-8}{0.3}{sig2g1}\quad\ \ 
\ar[r]^{d\ \ }
\ar@<-1.6ex>[d]^{d} & 
\figins{-8}{0.3}{sig2g2}\{1\}
\ar[r]^{\ \ X_1-X_4\quad }
\ar@<-1.6ex>[d]^{-1} & 
\figins{-8}{0.3}{sig2g2}\{3\}
\ar@<-1.6ex>[d]^{X_4-X_2} 
\\
\figins{-8}{0.3}{sig2g2}\{1\}
\ar[r]^{\ 1 \ } & 
\figwhins{-8}{0.3}{0.14}{sig2g2}\{1\}
\ar[r]^{\ \ 0\quad } &
\figwhins{-8}{0.3}{0.14}{sig2g2}\{5\}
}
$$
\begin{figure}[ht!]
\labellist
\small\hair 2pt
\pinlabel $_1$ at -15 110
\pinlabel $_2$ at  99 110
\pinlabel $_3$ at  99 0
\pinlabel $_4$ at -15 0
\endlabellist
\centering
\figs{0.23}{dbedge}\
\raisebox{-3pt}{
\raisebox{11pt}{\ ,\qquad $X_1\leftrightarrow$}\ \figs{0.22}{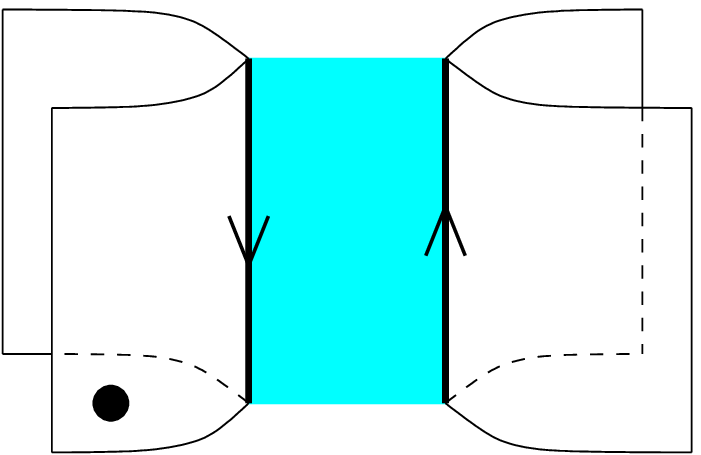}
\raisebox{11pt}{\ ,\quad $X_4\leftrightarrow$}\ \figs{0.22}{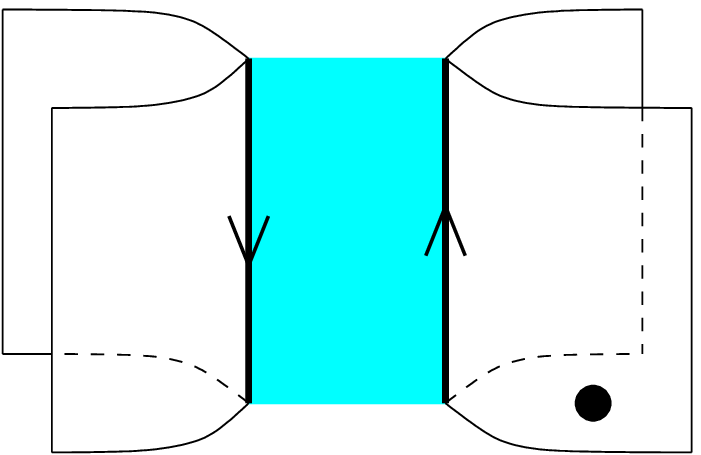}
\raisebox{11pt}{\ ,\quad $X_2\leftrightarrow$}\ \figs{0.22}{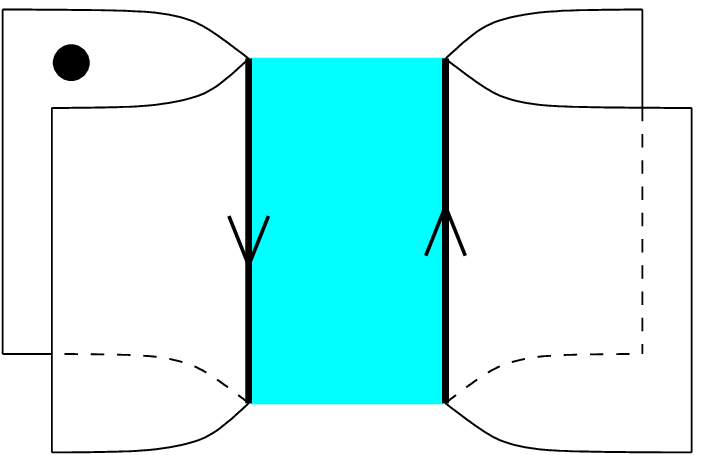}
}
\caption{Edge labelling and notation}
\label{slN:fig:edge-lab}
\end{figure}

\medskip

\n Using the fact that the first map in the second row is an isomorphism we see that
$$
\brak{\sigma^3}
\cong
\figins{-8}{0.3}{sig2g1} \xra{\ d\ }
\figins{-8}{0.3}{sig2g2}\{1\} 
\xra{\ X_1-X_4\ }
\figins{-8}{0.3}{sig2g2}\{3\}
\xra{\ X_4-X_2\ }
\figins{-8}{0.3}{sig2g2}\{5\}
,
$$ 
where $\ \figins{-4}{0.18}{sig2g1}\ $ is in homological degree 0.

\medskip

\n Using induction and Lemma~\ref{slN:lem:red} one can prove that
\begin{lem}
Up to an overall shift of $(N-1)m$,
\begin{align*}
\brak{\sigma^m} &\cong
\figins{-5}{0.2}{orsmoothing}
\xra{\ d\ }
\figins{-5}{0.2}{dbedge}\{1\}
\xra{X_1-X_4}
\figins{-5}{0.2}{dbedge}\{3\}
\xra{X_4-X_2}
\figins{-5}{0.2}{dbedge}\{5\}
\xra{X_1-X_4}
\figins{-5}{0.2}{dbedge}\{7\}
\xra{X_4-X_2}
\cdots \\[2.5ex]   &\quad\ \cdots 
\xra{X_1-X_4}
\figins{-5}{0.2}{dbedge}\{2j-1\}
\xra{X_4-X_2}
\figins{-5}{0.2}{dbedge}\{2j+1\}
\xra{X_1-X_4}
\cdots \\[2.5ex]  &\quad\ \cdots 
\begin{cases}
\xra{X_4-X_2}
\figins{-5}{0.2}{dbedge}\{2m+1\}, & \text{$m$ odd}\\[2ex]
\xra{X_1-X_4}
\figins{-5}{0.2}{dbedge}\{2m-1\}, & \text{$m$ even}\\
\end{cases}
\end{align*}
\n with $\orsmoothing$ in homological degree zero.
\end{lem}
\bigskip

For the $(2,m)$-torus link we have that $X_1=X_4$ and the complex splits into a direct sum of 
$$
\figins{-6}{0.25}{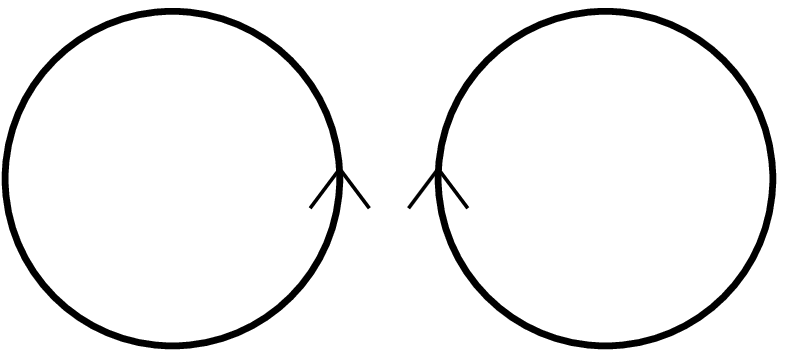}
\xra{d}
\figins{-6}{0.25}{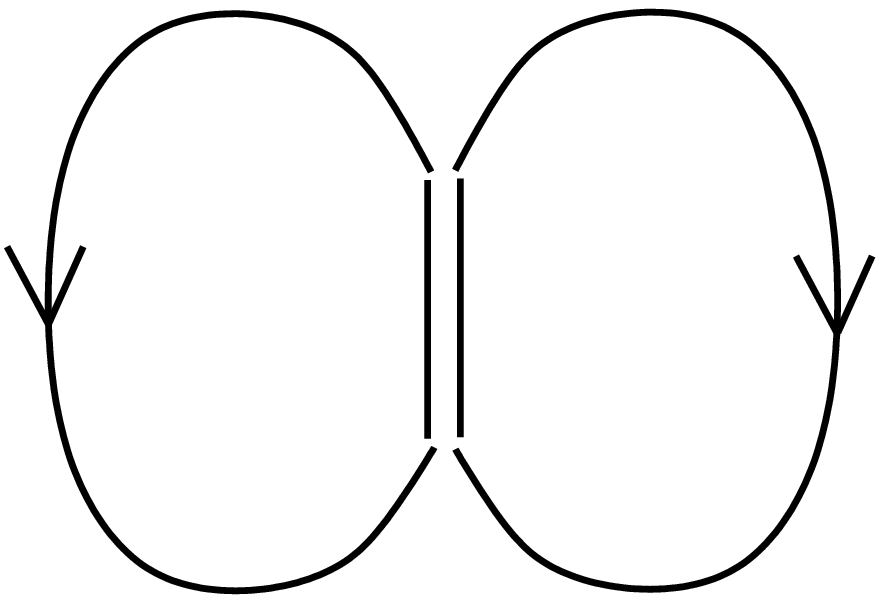}\{1\}
$$ 
with $\figins{-5}{0.2}{twocircles}$ in homological degree 0, and copies of 
$$
\figins{-6}{0.25}{TWeb}\{2j-1\}\xra{X_1-X_2}
\figins{-6}{0.25}{TWeb}\{2j+1\},\qquad 
\begin{cases}
2\leq j\leq m-1, & \text{$m$ odd} \\
2\leq j\leq m-2, & \text{$m$ even}
\end{cases}$$ 
with $\figins{-5}{0.2}{TWeb}$ in homological degree $j$, where we are using the notation
$$
X_1\leftrightarrow\figins{-20}{0.6}{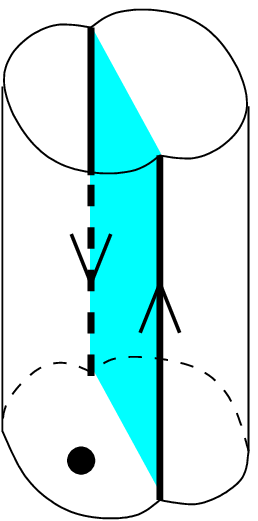} \quad
X_2\leftrightarrow\figins{-20}{0.6}{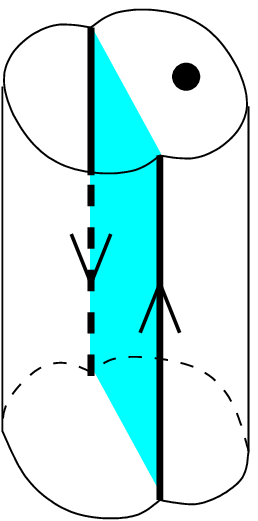}\ .$$

For $m$ even we have an extra summand consisting of the complex 
\begin{equation}
\label{slN:eq:toruslast}
0\to\figins{-6}{0.25}{TWeb}\{2m-1\}\to 0
\end{equation}
with $\figins{-5}{0.2}{TWeb}$ in homological degree $m$.

\begin{cor}
For the $(2,m)$-torus link $T_{2,m}$ we have
\begin{align*}
\hy^0(T_{2,m})  &\cong \bZ[X]/(X^N) \{(N-1)(m-2)\}, \\
\hy^1(T_{2,m})  &=      0, \\
\hy^{2j}(T_{2,m})  &\cong \bZ/(N\bZ) \{(N-1)m+4j\}
         \oplus\biggl(\bigoplus\limits_{i=0}^{N-2}\bZ\{(N-1)m+2(2j-i)-2\}\biggr) \\
\hy^{2j+1}(T_{2,m}) &\cong  \bigoplus\limits_{i=0}^{N-2}\bZ\{(N-1)m+2(2j-i)+2N-2\},\qquad 2j+1\leq m,
\end{align*}
and, if m is even
$$\hy^m(T_{2,m})\cong \hy(Fl_{2,3,N},\bZ)\{(N+1)m-2N+2\}.$$
Note that in even homological degree there is always torsion for $0< 2j < m$.
\end{cor}

\begin{proof}
The complex $\figins{-4}{0.18}{twocircles}\xra{\ d\ }\figins{-5}{0.2}{TWeb}\{1\}$ can be simplified using~\eqref{slN:eq:DR2} to remove the digon on the right in the diagram in homological degree 1. We obtain 
$$
\figins{-5}{0.2}{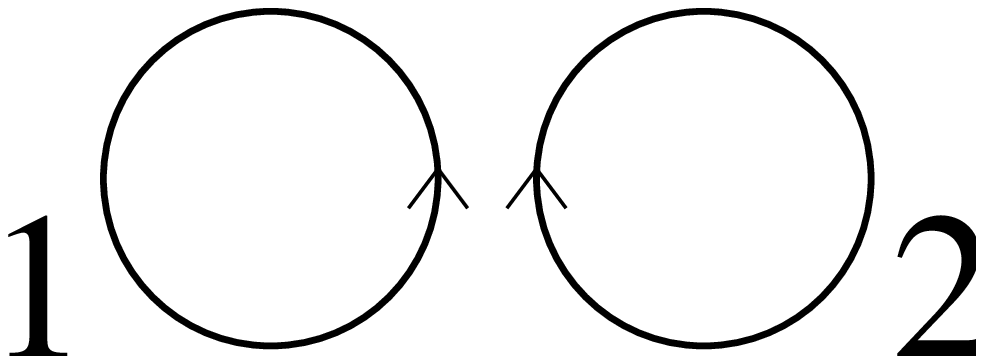}
\xra{\ d\ }
\bigoplus\limits_{i=0}^{N-2}\
\figins{-5}{0.2}{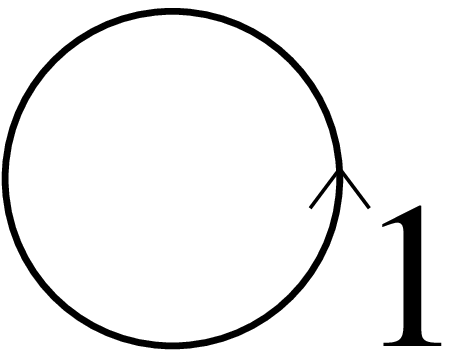}\{3-N-2i\}$$
where
$d\colon\figins{-4}{0.18}{twocircles12}\to
\figins{-4}{0.18}{onecircle1}\{3-N-2i\}$
is the foam
$$
\figins{-17}{0.55}{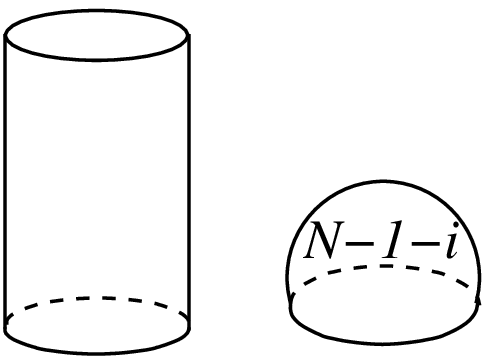}\
-\ \
\figins{-17}{0.55}{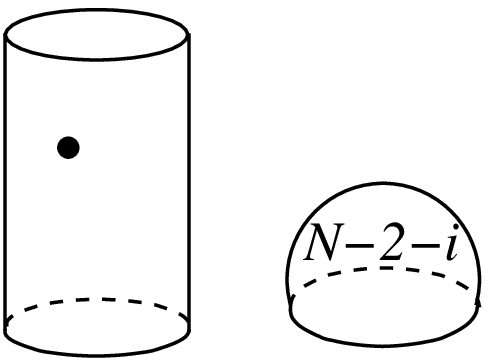}\ .
$$

Note that $\ \figins{-4}{0.18}{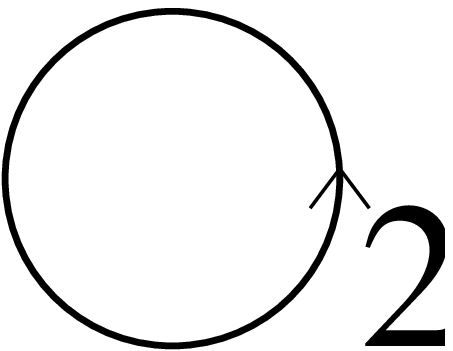}\ $ decompose into a direct sum of empty webs
$\bigoplus\limits_{j=0}^{N-1}\emptyset\{N-1-2j\}$ with inclusions
$$\psi_j\colon\emptyset\{N-1-2j\}\to\figins{-5}{0.2}{onecircle2},\quad
\psi_j=\figins{-5}{0.2}{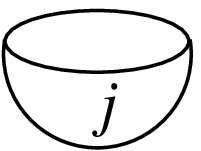}.$$
This is called \emph{delooping} in~\cite{bar-natanfast}.
Composing $\psi$ with $d$ and using Lemma~\ref{slN:lem:red} once again we get the complex
$$0\to\figins{-4}{0.2}{onecircle1}\{-N+1\}\to 0$$
with $\figins{-4}{0.18}{onecircle1}\{-N+1\}$ in homological degree zero. This gives $H^1(T_{2,m})=0$. Taking the norma\-lization into account we get $\hy^0(T_{2,m})\cong \bZ[X]/X^N \{(N-1)(m-2)\}$.

Consider the complex
$$
\figins{-6}{0.25}{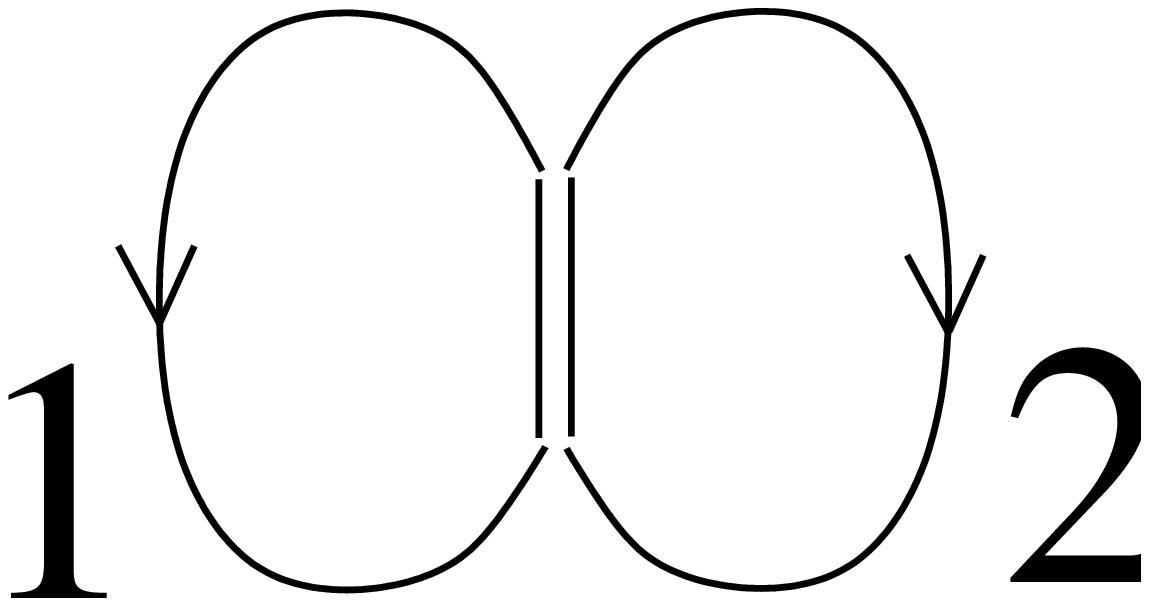}\xra{X_1-X_2}
\figins{-6}{0.25}{TWeb12}\{2\}.
$$
Remove the digon labeled 2 using~\eqref{slN:eq:DR2} and by Lemma~\ref{slN:lem:red} we can simplify the complex to obtain
$$0\to\figins{-4}{0.2}{onecircle1}\{2-N\}
\xra{N\ \figins{-7}{0.35}{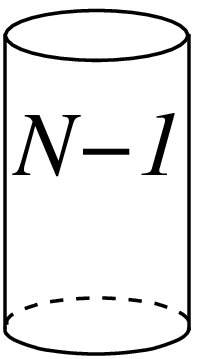}\ } \figins{-4}{0.2}{onecircle1}\{N\}\to 0.$$
Using the decomposition of a circle into a direct sum of empty webs we get that the only non\-zero map is $\emptyset\{1\}\xra{\ N\ }\emptyset\{1\}$. This gives the $N$-torsion. The claim for $\hy^{2j}(T_{2,m})$ and $\hy^{2j+1}(T_{2,m})$ follows easily from this observation.

The assertion for $\hy^m(T_{2,m})$ when $m$ is even follows from Equation~\eqref{slN:eq:toruslast}.
\end{proof}

\newpage\

%%%%%%%%%%%%%%%%%%%%%%%%%%%%%%%%%%%%%%%%

%%%%%%%%%%% End of Chapter %%%%%%%%%%%%%
%
%
%
%
%%%%%%%%%%%%%%%%%%%%%%% Backmatter %%%%%%%%%%%%%%%%%%%%%%%%%
%
%
%
\nocite{*}
\bibliographystyle{plain}
\bibliography{biblio}
\addcontentsline{toc}{chapter}{Bibliography}
\end{document}

%% file: defs.tex
\newlength  {\defbaselineskip}
\newcommand {\setlinespacing}   [1] {\setlength{\baselineskip}{#1 \defbaselineskip}}
\oddsidemargin 18pt \evensidemargin 18pt      \setlinespacing{1.5}
\newdimen\captionwidth\captionwidth=\hsize

\newcommand {\ie}                   {\emph{i.e.} }
\newtheorem{thm}{Theorem}[section] % [chapter]
\newtheorem{lem}[thm]{Lemma}
\newtheorem{cor}[thm]{Corollary}
\newtheorem{prop}[thm]{Proposition}

\newtheorem{conj}{Conjecture}

\theoremstyle{definition}
\newtheorem{defn}[thm]{Definition}

\newcommand{\figins}[3] 
{\raisebox{#1pt}{\includegraphics[height=#2 in]{figs/#3}}}
\newcommand{\figwins}[3] 
{\raisebox{#1pt}{\includegraphics[width=#2 in]{figs/#3}}}
\newcommand{\figwhins}[4] 
{\raisebox{#1pt}{\includegraphics[height=#2 in, width=#3 in]{figs/#4}}}
\newcommand{\figs}[2] 
{{\includegraphics[scale =#1]{figs/#2}}}
\newcommand{\cp}[1]{\bC\mbox{P}^{#1} }

\newcommand{\brak}[1]{\langle #1\rangle}
\newcommand{\n}{\noindent}
\newcommand{\Foam}{\mathbf{Foam}}
\newcommand{\foam}{\mathbf{Foam}_{N}}
\newcommand{\lfoam}{\mathbf{Foam}_{/\ell}}
\newcommand{\Qlfoam}{\mathbf{Foam}_{\ell'}}

\newcommand{\PF}{\mathbf{Pre-foam}}
\newcommand{\sln}{\mathfrak{sl}(N)}
\newcommand{\slt}{\mathfrak{sl}(3)}
\newcommand{\V}{\mathbf{Vect}_{\bZ}}

\newcommand{\abcModgr}{\mathbf{\bZ[a,b,c]-{\text Mod}_{\text{gr}}}}
\newcommand{\abcModbg}{\mathbf{\bZ[a,b,c]-{\text Mod}_{\text{bg}}}}
\newcommand{\Modbg}{\mathbf{Mod_{bg}}}
\newcommand{\Link}{\mathbf{Link}}
\newcommand{\kom}{\mathbf{Kom}}
\newcommand{\mf}[1]{\mathbf{MF}_{#1}}
\newcommand{\hmf}[1]{\mathbf{HMF}_{#1}}

\newcommand{\quotient}[2]{#1\raisebox{-1.5pt}{\text{$/#2$}}}
\newcommand{\KL}[1]{\brak{#1}_{KL}}

\newcommand{\KLket}[1]{\ket{#1}_{KL}}
\DeclareMathOperator{\KR}{KR}
\DeclareMathOperator{\HKR}{HKR}
\DeclareMathOperator{\hy}{H}
\DeclareMathOperator{\KH}{HKh}
\DeclareMathOperator{\gr}{gr}
\DeclareMathOperator{\sk}{s_\gamma}

\DeclareMathOperator{\Hom}{Hom}
\DeclareMathOperator{\Ext}{Ext}
\DeclareMathOperator{\End}{End}
\DeclareMathOperator{\Ker}{Ker}

\DeclareMathOperator{\Image}{Im}
\DeclareMathOperator{\id}{Id}

\DeclareMathOperator{\qdim}{qdim}

\DeclareMathOperator{\sgn}{sgn}
\DeclareMathOperator{\Tr}{Tr}
\DeclareMathOperator{\str}{STr}

\newcommand{\bZ}{\mathbb{Z}}
\newcommand{\bQ}{\mathbb{Q}}
\newcommand{\bR}{\mathbb{R}}
\newcommand{\bC}{\mathbb{C}}

\newcommand{\cD}{\mathcal{D}}
\newcommand{\cE}{\mathcal{E}}
\newcommand{\cF}{\mathcal{F}}
\newcommand{\cG}{\mathcal{G}}

\newcommand{\cU}{\mathcal{U}}
\newcommand{\cV}{\mathcal{V}}
\newcommand{\ra}{\rightarrow}
\newcommand{\xra}[1]{\xrightarrow{#1}}
\newcommand{\lra}{\longrightarrow}

\newcommand{\bra}[1]{\langle\,#1\,|}
\newcommand{\ket}[1]{|\,#1\,\rangle}

\newcommand{\kh}[3]{\KH^{{#1},{#2}}({#3},\bC)}

\newcommand{\lk}{\mbox{{\em lk}}}
\newcommand{\uht}[2]{U_{a,b,c}^{{#1}}({#2},\bC)}

\newcommand{\uhtabc}[5]{U_{{#1},{#2},{#3}}^{{#4}}({#5},\bC)}

\newcommand{\uhtgen}[2]{U_{a,b,c}^{{#1}}({#2})}
\newcommand{\qbin}[2]{\genfrac{[}{]}{0pt}{}{#1}{#2}}
\catcode`\@=11
\long\def\@makecaption#1#2{%
    \vskip 10pt
    \setbox\@tempboxa\hbox{%
\small{#1: }\ignorespaces #2}%
    \ifdim \wd\@tempboxa >\captionwidth {%
        \rightskip=\@captionmargin\leftskip=\@captionmargin
        \unhbox\@tempboxa\par}%
      \else
        \hbox to\hsize{\hfil\box\@tempboxa\hfil}%
    \fi}
\newdimen\@captionmargin\@captionmargin=2\parindent
\newdimen\captionwidth\captionwidth=\hsize
\catcode`\@=12